\documentclass[11pt]{amsart} 
\usepackage{amsthm,amsbsy,amsmath,amssymb,amscd,amsfonts,array,mathrsfs,verbatim,enumerate,xypic,enumitem,color,stmaryrd}
\usepackage[all]{xy}
\xyoption{arc} 

\newtheorem*{unnumberedthm}{Theorem}
\newtheorem{thm}{Theorem}[subsection]
\newtheorem{prop}[thm]{Proposition}     
\newtheorem{lem}[thm]{Lemma}
\newtheorem{cor}[thm]{Corollary}

\theoremstyle{definition}

\newtheorem{defn}[thm]{Definition}

\newtheorem{example}[thm]{Example} 
\newtheorem{rem}[thm]{Remark}

\DeclareFontFamily{OT1}{rsfs}{}
\DeclareFontShape{OT1}{rsfs}{n}{it}{<-> rsfs10}{}
\DeclareMathAlphabet{\curly}{OT1}{rsfs}{n}{it}

\makeatletter \@addtoreset{equation}{subsection} \makeatother  
\makeatletter \@addtoreset{thm}{subsection} \makeatother  

\makeatletter \@addtoreset{equation}{subsection} \makeatother  
\makeatletter \@addtoreset{thm}{subsection} \makeatother  

\newcommand{\DAb}{{\bf DAb}} 
\newcommand{\C}{{\bf C}} 
\newcommand{\Qco}{{\bf Qco}} 
\newcommand{\Alg}{{\bf Alg}} 
\newcommand{\An}{{\bf An}} 
\newcommand{\ClassicalFans}{{\bf ClassicalFans}} 
\newcommand{\Ab}{{\bf Ab}} 
\newcommand{\Sets}{{\bf Sets}} 
\newcommand{\Mon}{{\bf Mon}} 
\newcommand{\MonSch}{{\bf MonSch}} 
\newcommand{\Mod}{{\bf Mod}} 
\newcommand{\LRS}{{\bf LRS}}  
\newcommand{\RS}{{\bf RS}}  
\newcommand{\Sch}{{\bf Sch}} 
\newcommand{\LMS}{{\bf LMS}} 
\newcommand{\MS}{{\bf MS}}   
\newcommand{\Top}{{\bf Top}}   
\newcommand{\Fans}{{\bf Fans}} 

\renewcommand{\1}{{\bf 1}}

\renewcommand{\AA}{\mathbb{A}} 
\newcommand{\NN}{\mathbb{N}} 
\newcommand{\ZZ}{\mathbb{Z}} 
\newcommand{\RR}{\mathbb{R}} 
\newcommand{\QQ}{\mathbb{Q}} 

\newcommand{\CC}{\mathbb{C}}
\newcommand{\GG}{\mathbb{G}}
\newcommand{\FF}{\mathbb{F}}
\newcommand{\UU}{\mathbb{U}}

\newcommand{\p}{\mathfrak{p}}
\newcommand{\q}{\mathfrak{q}}

\newcommand{\F}{\curly{F}}

\newcommand{\M}{\mathcal{M}} 

\renewcommand{\O}{\mathcal{O}} 

\renewcommand{\u}{\underline}

\newcommand{\ov}{\overline}
\newcommand{\into}{\hookrightarrow}
\newcommand{\be}{\begin{eqnarray*}}
\newcommand{\ee}{\end{eqnarray*}}

\newcommand{\bne}[1]{\begin{eqnarray} \label{#1} }
\newcommand{\ene}{\end{eqnarray}}
\newcommand{\xym}{\xymatrix}
\newcommand{\bp}{\begin{pmatrix}}
\newcommand{\ep}{\end{pmatrix}}
\newcommand{\slot}{ \hspace{0.05in} {\rm \_} \hspace{0.05in} } 
\newcommand{\dirlim}{\displaystyle \lim_{ \longrightarrow } \,} 

\newcommand{\invlim}{\displaystyle \lim_{ \longleftarrow } \,} 

\newcommand{\threerightarrows}{
        \mathrel{\vcenter{\mathsurround0pt
                \ialign{##\crcr
                        \noalign{\nointerlineskip}$\rightarrow$\crcr
                        \noalign{\nointerlineskip}$\rightarrow$\crcr
                        \noalign{\nointerlineskip}$\rightarrow$\crcr
                }
        }}
}

\newcommand{\Hom}{\operatorname{Hom}}   
   
\newcommand{\Ext}{\operatorname{Ext}}
\renewcommand{\H}{\operatorname{H}}
\newcommand{\CechH}{\check{\operatorname{H}}}    

\newcommand{\R}{\operatorname{R}}

\newcommand{\sHom}{\curly Hom}            
\newcommand{\sExt}{\curly Ext}            

\newcommand{\Ker}{\operatorname{Ker}}
\newcommand{\Cok}{\operatorname{Cok}}

\newcommand{\ord}{\operatorname{ord}}
\newcommand{\Id}{\operatorname{Id}}
\newcommand{\Spec}{\operatorname{Spec}}
\newcommand{\Proj}{\operatorname{Proj}}

\newcommand{\Span}{\operatorname{Span}}

\newcommand{\Stab}{\operatorname{Stab}}
\newcommand{\rk}{\operatorname{rank}}
\renewcommand{\Im}{\operatorname{Im}}
\newcommand{\row}{\operatorname{row}}     

\def\arraystretch{1.2} 

\numberwithin{equation}{subsection}

\begin{document}

\author{W.~D.~Gillam}
\address{Department of Mathematics, Bogazici University, Bebek, Istanbul 34342}
\email{wdgillam@gmail.com}

\date{\today}
\title[Fans]{Fans}

\begin{abstract}  The category of (abstract) fans is to the category of monoids what the category of schemes is to the category of rings:  a fan is obtained by gluing spectra of monoids along open embeddings.  Here we study the basic algebraic geometry of fans: coherent sheaves, group fans, affine maps, flat maps, proper maps, and so forth.  This study is motivated by the fact that the category of fans lies over all other reasonably geometric categories (schemes, differentiable spaces, analytic spaces, etc) as well as the ``log" versions of all such categories.  \end{abstract}

\maketitle

\arraycolsep=2pt 
\def\arraystretch{1.2}

\section*{Introduction}  The purpose of this paper is to provide a basic reference for the theory of (abstract) \emph{fans}.  One motivation for studying fans is the following analogy with the basic constructions of algebraic geometry:  To any (commutative) ring $A$, one associates $\Spec A$, which is a topological space equipped with a sheaf of rings $\O_A$ with local stalks.  As a set, $\Spec A$ is the set of prime ideals of $A$; the topology is the ``Zariski topology" where the basic opens are the sets of primes not containing some given element of $A$.  This provides a contravariant, fully faithful functor from rings to locally ringed spaces.  The objects in the essential image of this functor are called \emph{affine schemes} and one then defines a \emph{scheme} to be a locally ringed space admitting an open cover by affine schemes.  The point is that all of these constructions make sense when ``(commutative) ring" is replaced by ``(commutative) monoid," so that \emph{fans} are to monoids what schemes are to rings.  This general idea has been around for a long time and is certainly not due to the present author (see \cite{Ogus}, \cite{D1}, \cite{D2}, \cite{K}, \cite{CC}, and the references therein).  The theory of \emph{monoid schemes} from \cite{CDH}, while similar to our theory of fans, also differs from it in several respects (\S\ref{section:monoidschemes}). 

The \emph{classical fans} familiar from the theory of toric varieties can be viewed as fans in this abstract sense (\S\ref{section:classicalfans})---they are roughly analogous to varieties among schemes, so that the category of (abstract) fans has certain technical advantages over the category of classical fans, just as the category of schemes is, in many ways, ``nicer" than the category of varieties.  We can try to push the analogy \begin{center} monoids : rings :: fans : schemes \end{center} as far as possible.  For example, there is a notion of \emph{modules} over a monoid and a corresponding notion of \emph{quasi-coherent sheaves} on fans (\S\ref{section:coherentsheaves}).  There are (relative) $\Spec$ functors, $\Proj$ functors, blowups, and so forth (\S\ref{section:relativeSpecandProj}).  For many kinds of morphisms of schemes there are analogous kinds of morphisms of fans:  affine morphisms (\S\ref{section:affinemaps}), separated and proper morphisms (\S\ref{section:propermaps}), flat morphisms (\S\ref{section:flatmaps}), etc.  In general the theory of fans works best when we restrict our attention of \emph{fine} monoids, which are the monoids arising as finitely generated submonoids of (abelian) groups.  

We are also motivated to study fans by virtue of the fact that fans can be ``realized" in any ``reasonably geometric" category, so that fans can serve as a testing ground for constructions in a wide variety of geometric contexts.  This is made precise in \cite[\S4.5]{GM1}, where we defined a ``category of spaces" and the $2$-category of such categories, then we proved that the category of (locally finite type) fans is the initial object in this $2$-category \cite[4.5.12]{GM1}.  For now it suffices to note that topological spaces, locally ringed spaces, schemes, analytic spaces, and differentiable spaces are acceptable ``categories of spaces," as are the ``log" versions of all these categories, so that the category of fans admits ``nice" (preserving inverse limits, among other things) functors to all of these categories.  In algebraic geometry, this general nonsense specializes to the well-known theory of toric varieties.  In differential geometry, this construction seems to have been overlooked, though it can be quite useful.  For example, the ``differential realization" of a smooth classical fan is a manifold with corners; one can study, say, various Morse-theoretic constructions on this manifold with corners via toric (i.e.\ ``fan") techniques.

One can use these realization functors to study the category of fans itself.  For example, one can consider the class of maps of fans $f$ for which some given geometric realization of $f$ has a particular property.  For example, one could define a map of fans to be affine iff its scheme realization is affine (in fact it turns out that this is \emph{not} the most natural notion of an ``affine morphism" of fans).  Sometimes this works better than others, mostly because some geometric realizations of $f$ are better reflections of $f$ than others: the topological space realization is fairly ``coarse" from this perspective, while various algebraic realizations of $f$ are in fact ``fully faithful" in some sense.  It is often desirable to translate a property of some realization of $f$ into a property defined in terms of $f$ itself, for then one suspects that other realizations of $f$ will enjoy an analogous property in the appropriate geometric category.  For example, we prove the following ``valuative criterion for properness" in \S\ref{section:propermaps}:

\begin{unnumberedthm} The scheme realization of a map of fine fans is proper iff the map of fans is quasi-compact and has the unique right lifting property with respect to the map of fans $\Spec \ZZ \to \Spec \NN$. \end{unnumberedthm}

\noindent In general, the realization of a proper map of fans will be ``proper" for an appropriate notion of ``proper" in the geometric setting where the map is realized.

There are often two routes to proving a given statement about monoids/fans with an analogy in rings/schemes.  One is to translate the proof of the analogous statement for rings/schemes into a proof for monoids/fans.  The other route is to just use the result in rings/schemes directly in conjunction with properties of the scheme realization of fans to establish the desired result for fans.  When possible, I have endeavored to give purely ``fan-theoretic" proofs of theorems about fans, but in some cases I have found it more expedient to resort to direct application of scheme-theoretic results.  This happens a fair amount in the treatment of flatness in \S\ref{section:flatmaps}, where the following results are established:

\begin{unnumberedthm} A flat map of locally finite type fans is open on topological spaces. \end{unnumberedthm} 

\begin{unnumberedthm} A flat surjective map of locally finite type fans is an effective descent morphism. \end{unnumberedthm}

In \S\ref{section:groupobjects} we study group objects and actions in the category of fans.  The situation is surprisingly simple---it turns out (Lemma~\ref{lem:groupobjects}) that every connected group fan is isomorphic to the ``diagonalizable" group fan associated to an abelian group.  One can form quotients by actions of finite, diagonalizable group fans, and they behave much like finite quotients in algebraic geometry (Theorem~\ref{thm:quotients}).

Some of the most technically difficult results are in \S\ref{section:reducedfibers}.  There we characterize the ``primes" $p$ and the maps of fine fans whose realizations have ``reduced characteristic $p$ fibers" (Corollary~\ref{cor:reducedfibers}), as well as the integral monoids $P$, ideals $I \subseteq P$, and fields $k$ for which $k[P]/k[I]$ is reduced.  These results are not particularly hard in characteristic zero, but in positive characteristic one needs to make use of some ``structure theory" of fine monoids, concerning the existence of splittings $P \cong P^* \oplus \ov{P}$ (\S\ref{section:splittings}).  In \S\ref{section:finitemaps} we study finite and quasi-finite and maps and show, among other things, that an injective, quasi-finite map of fine monoids induces an open embedding on topological spaces which is a homeomorphism iff the map is finite.

In \S\ref{section:CZEmaps}, we study \emph{CZE} maps of fans, which are a sort of analog of \'etale maps of schemes.  Several characterizations of these maps are given in Theorem~\ref{thm:CZEmaps}---these make use of the results discussed in the previous paragraph.  We compute the CZE cohomology groups of any fan with coefficients in any connected group fan in Theorem~\ref{thm:CZEcohomology}.  The result is reminiscent of the interpretation of Hilbert's ``Theorem 90" in terms of \'etale cohomology.  The basic theory of CZE maps laid out here can be used to set up a useful theory of stacks over the category of fans---this is pursued in \cite{GM2}.

\subsection*{Acknowledgements} The author is supported by a Marie Curie / T\"{U}BITAK co-funded Brain Circulation Scheme fellowship.

\newpage

\setcounter{tocdepth}{2}
\tableofcontents

\newpage

\section{Monoids and Spec} \label{section:monoidsandspec} Before getting into the ``algebraic geometry" of fans, we have to develop a certain amount of ``commutative algebra" for monoids.

\subsection{Monoids} \label{section:monoids} In this paper a \emph{monoid} is a set $P$ equipped with an associative, commutative binary operation $+$ with a (necessarily unique) additive identity element $0$.  A \emph{morphism} of monoids is a map of sets respecting $+$ and $0$.  The category of monoids is denoted $\Mon$; it has all direct and inverse limits.  An element $p$ of a monoid $P$ is called a \emph{unit} iff there is some (necessarily unique) element $-p \in P$ so that $p+(-p)=0$.  The units of a monoid form a group $P^*$ which is also a submonoid of $P$.  A \emph{group} is a monoid $P$ for which $P=P^*$, so that groups form a full subcategory $\Ab \subseteq \Mon$ of monoids.  (All \emph{groups} in this paper are abelian.)  

A monoid $P$ is called \dots \begin{enumerate}[label=\dots] \item \emph{finitely generated} iff there is a surjection $\NN^n \to P$ for some finite $n$. \item \emph{integral} iff $P \to P^{\rm gp}$ is injective. \item \emph{fine} iff it is finitely generated and integral. \item \emph{sharp} iff $P^* = \{ 0 \}$. \item \emph{saturated} iff it is integral and $P  = \{ p \in P^{\rm gp} : np \in P \subseteq P^{\rm gp} {\rm \; for \; some  \; } n \in \ZZ_{>0} \}$.  \item \emph{fs} iff it is fine and saturated. \item \emph{torsion-free} iff it is integral and $P^{\rm gp}$ is torsion-free \item \emph{toric} iff it is fs and torsion-free. \end{enumerate}

Every monoid $P$ admits a map $P \to P^{\rm gp}$ to a group (resp.\ a map $P \to P^{\rm int}$ to an integral monoid, $P \to P^{\rm sat}$ to a saturated monoid, $P \to \ov{P} = P/P^*$ to a sharp monoid, $P \to P^{\rm tf}$ to a torsion-free monoid) through which any other map from $P$ to a group (resp.\ an integral monoid, \dots) factors uniquely.  We have used the map $P \to P^{\rm gp}$ in defining integral and saturated above.  These constructions, described further in \S\ref{section:specinvariance}, provide left adjoints to the inclusions of various full subcategories of $\Mon$ (namely: groups, integral monoids, \dots); in particular all of these constructions preserve direct limits.

For any monoid $P$ and any submonoid $S \subseteq P$, there is a map of monoids $P \to S^{-1} P$, called the \emph{localization} of $P$ at $S$, which is initial among monoid homomorphisms out of $P$ taking elements of $S$ to units.  Evidently $P^{-1}P = P^{\rm gp}$.

\subsection{Modules, ideals, and faces} \label{section:modules} A \emph{module} over a monoid $P$ is a set $M$ equipped with an \emph{action map} \be P \times M & \to & M \\  (p,m) & \mapsto & p+m \ee satisfying $p+(q+m) = (p+q)+m$ and $0+m = m$ for every $p,q \in P$, $m \in M$.  A \emph{map} of $P$-modules is a map of sets respecting the action maps.  The category of $P$-modules is denoted $\Mod(P)$.  It is \emph{not} an abelian category.  It has all direct and inverse limits.  

\begin{defn} \label{defn:finitelygenerated} A subset $S \subseteq M$ of a $P$-module $M$ is said to \emph{generate} $M$ (resp.\ is called a \emph{basis for} $M$) iff every $m \in M$ can be written $m = p+s$ for some (resp.\ a unique pair) $p \in P$, $s \in S$.  $M$ is called \emph{finitely generated} (resp.\ \emph{free}) iff some finite subset $S \subseteq M$ generates $M$ (resp.\ $M$ has a basis).  A $P$-module $M$ is called \emph{flat} iff it can be expressed as a filtered direct limit of free $P$-modules. \end{defn}

\begin{example} \label{example:flatmodule} Let $f : A \to B$ be a map of (abelian) groups, $S$ a subset of $B$.  Then $S$ generates $B$ as an $A$-module iff the quotient projection $B \to \Cok f$ takes $S$ surjectively onto $\Cok f$.  Similarly, if $f$ is injective, then $S$ is a basis for $B$ as an $A$-module iff $B \to \Cok f$ takes $S$ bijectively onto $\Cok f$.  (Of course one can always find such an $S$ by taking $S$ to be the image of any set-theoretic section of the surjective map of sets $B \to \Cok f$.) \end{example}

\begin{defn} \label{defn:face}  The addition for $P$ is an action map making $P$ a module over itself.  An \emph{ideal} is a $P$-submodule of $P$.  An ideal $I$ is called \emph{prime} iff its complement $F := P \setminus I$ is a submonoid of $P$, which is equivalent to the usual condition $p+q \in I \implies p$ or $q \in I$.  A \emph{face} is a submonoid of $P$ whose complement is a prime ideal (equivalently: for all $p,p' \in P$ if $p+p' \in F$, then $p,p' \in F$), so that taking complements defines a bijection between prime ideals of $P$ and faces of $P$.  We use the notation $F \leq P$ to mean that $F$ is a face of $P$.  \end{defn}

For any monoid $P$, the units $P^*$ form the smallest face of $P$ and $P$ itself is the largest face of $P$.  As in the case of rings, the preimage of a prime ideal (resp.\ face) under a monoid homomorphism is again a prime ideal (resp.\ face).

\begin{lem} \label{lem:facesandlocalization} Let $P$ be a monoid, $S$ a submonoid of $P$, $l : P \to S^{-1}P$ the localization map.  Then $F := l^{-1}((S^{-1}P)^*)$ is the smallest face of $P$ containing $S$.  In particular, $S = F$ if $S$ is a face of $P$. \end{lem}

\begin{proof} $F$ is a face of $P$ (which clearly contains $S$) since it is the preimage of the face $(S^{-1}P)^* \leq S^{-1}P$ under $l$.  If $F'$ is any face of $P$ containing $S$ and $f \in F$, then $l(f) \in (S^{-1}P)^*$, so $f + p + s \in S \subseteq F'$ for some $p \in P$, $s \in S$, hence $f$ (and also $p,s$) is (are) in $F'$ since $F'$ is a face. \end{proof}

The next lemma is a fundamental fact about fine monoids which we shall need at various points.

\begin{lem} \label{lem:duality} Let $P$ be a fine monoid, $F \subseteq P$ a face.  There exists a monoid homomorphism $h : P \to \NN$ such that $h^{-1}(0) = F$. \end{lem}

\begin{proof} Replacing $P$ with $P/F$, we reduce to proving that for any fine, sharp monoid $P$, there is a map $h : P \to \NN$ for which $h^{-1}(0) = \{ 0 \}$.  This is \cite[2.2.2]{Ogus} or \cite[3.1.20]{Ols}. \end{proof}  

\begin{example} One cannot replace ``fine" with ``finitely generated" in Lemma~\ref{lem:duality}.  The ``unique" monoid $P$ with two elements not isomorphic to $\ZZ / 2 \ZZ$ provides a counterexample because any map $h : P \to \NN$ will factor through $P^{\rm int} = 0$, so we have $h^{-1}(0) = P$ for any such $h$, so the lemma will fail for $F = P^* = \{ 0 \} \neq P$.  \end{example}

\begin{defn} \label{defn:flat} If $h : Q \to P$ is a monoid homomorphism, then $P$ becomes a $Q$-module with action $q+p := h(q) + p$.  We say that $h$ is \emph{flat} (resp.\ \emph{finite}) iff this makes $P$ a flat (resp.\ finitely generated) $Q$-module. \end{defn}

It is easy to see from this definition that a composition of flat (resp.\ finite) monoid homomorphisms is flat (resp.\ finite).   As an exercise with our definition of flatness, the reader can prove the following by analogy with the case of rings:

\begin{prop} \label{prop:flat} The localization $Q \to S^{-1} Q$ of a monoid $Q$ at any submonoid $S \subseteq Q$ is flat. \end{prop}

See \cite[\S5]{G2} for an extensive discussion of modules over monoids and flatness.  Here are a few basic points:  The category $\Mod(P)$ has an evident notion of bilinear maps and a corresponding notion of tensor products, making it symmetric monoidal.  In particular, a map of monoids $Q \to P$ gives rise to a \emph{restriction of scalars} functor $\Mod(P) \to \Mod(Q)$ and a corresponding left adjoint \emph{extension of scalars} functor $\Mod(Q) \to \Mod(P)$, denoted $M \mapsto M \otimes_Q P$.  This extension of scalars functor takes free (resp.\ finitely generated) modules to free (resp.\ finitely generated) modules and commutes with filtered direct limits, so it also takes flat modules to flat modules.  As in the case of rings, the tensor product of modules is compatible with the pushout of monoids in the sense that if $h_i : Q \to P_i$ ($i=1,2$) are monoid homomorphisms, then the monoid $P_1 \oplus_Q P_2$ is naturally isomorphic, as a $P_2$ module, to the extension of scalars $P_1 \otimes_Q P_2$ along $Q \to P_2$ of $P_1 \in \Mod(Q)$.  It follows that flat (resp.\ finite) maps of monoids are stable under pushout.

\subsection{Splittings} \label{section:splittings} 

\begin{defn} \label{defn:split} A monoid $P$ for which the natural map $P^* \to P^{\rm gp}$ is injective is called \emph{quasi-integral}.  For a quasi-integral monoid $P$, a \emph{splitting} of $P$ is a section of the sharpening map $P \to \ov{P} := P/P^*$.  A quasi-integral monoid admitting a splitting is called \emph{split}. \end{defn}

Every integral monoid is quasi-integral.  A splitting $s : \ov{P} \to P$ of a quasi-integral monoid $P$ gives rise to an isomorphism of monoids \bne{sharpeningsectionmap} (\subseteq,s) : P^* \oplus \ov{P} & \to & P \ene over $\ov{P}$ (and conversely).  (Quasi-integrality is needed to show that it is injective.  Any monoid of the form $P^* \oplus \ov{P}$ with $P^*$ a group and $\ov{P}$ sharp is quasi-integral.)  

The sharpening map $P \to \ov{P}$ is the cokernel of the inclusion $P^* \into P$ and groupification preserves direct limits, so for any quasi-integral monoid $P$, we obtain an exact sequence \bne{charseq} & 0 \to P^* \to P^{\rm gp} \to \ov{P}^{\rm gp} \to 0, \ene called the \emph{characteristic sequence} of $P$.  If $s$ is a splitting of $P$, then $s^{\rm gp}$ provides a splitting of \eqref{charseq}.  If $t : \ov{P}^{\rm gp} \to P^{\rm gp}$ is a splitting of \eqref{charseq} and $P$ (hence also $\ov{P}$) is integral, then one checks that $t$ takes $\ov{P} \subseteq \ov{P}^{\rm gp}$ into $P \subseteq P^{\rm gp}$, yielding a splitting of $P$.  For $P$ integral, we thus obtain a bijective correspondence between splittings of $P$ and splittings of the characteristic sequence \eqref{charseq} of $P$.

\begin{example} As we shall see in Corollary~\ref{cor:saturatedlocalization}, $\ov{P}^{\rm gp}$ is free when $P$ is fs, so any fs monoid is split.  In particular, any toric monoid is split. \end{example}

\begin{example} Suppose $\ov{P}$ is a sharp, integral monoid and \bne{ovPSES} & 0 \to A \to B \to \ov{P}^{\rm gp} \to 0 \ene is an exact sequence of abelian groups.  If we define a monoid $P$ by the \emph{cartesian} diagram $$ \xym{ P \ar[d] \ar[r] & \ov{P} \ar[d] \\ B \ar[r] & \ov{P}^{\rm gp} } $$ then one can check that $P$ is an integral monoid (in fact a fine monoid if $\ov{P}$ is fine and $A$ is finitely generated) with sharpening $\ov{P}$ whose characteristic sequence is \eqref{ovPSES}.  This gives a means of producing non-split fine monoids.  For example, if we let $\ov{P}$ be the submonoid of $\ZZ \oplus \ZZ/2 \ZZ$ generated by $(1,0)$ and $(1,1)$, then $\ov{P}$ is a fine, sharp monoid with $\ov{P}^{\rm gp} = \ZZ \oplus \ZZ / 2 \ZZ$.  If we then take the non-trivial extension of $\ov{P}^{\rm gp}$ by $\ZZ / 2\ZZ$ and perform the above construction we obtain a non-split fine monoid $P$ with the given $\ov{P}$ as its sharpening. \end{example}

The rest of \S\ref{section:splittings} is devoted to showing that every fine monoid is ``not too far from being split."  This basic result about the structure of fine monoids will be useful in \S\ref{section:reducedfibers}.  We begin with a simple lemma:

\begin{lem} \label{lem:splittinggroups}  For any exact sequence of finitely generated abelian groups \bne{FGAGs} & 0 \to A \to B \to C \to 0, \ene there is an injective map of abelian groups $A \into A'$ with the following properties: \begin{enumerate} \item The cokernel of $A \into A'$ is finite and the integers $|A_{\rm tor}|$ and $|A_{\rm tor}'|$ have the same prime divisors. \item The exact sequence $$ 0 \to A' \to B' \to C \to 0 $$ obtained by pushing out \eqref{FGAGs} along $A \to A'$ splits. \end{enumerate} \end{lem}

\begin{proof}  Let us refer to an injective group homomorphism $A \into A'$ with those two properties as a \emph{solution for} \eqref{FGAGs}.  By applying $\Hom( \slot, A)$ to the exact sequence $$ 0 \to C_{\rm tor} \to C \to C/C_{\rm tor} \to 0 $$ and using the fact that $C/C_{\rm tor}$ is free (by the structure theory of finitely generated abelian groups), we find that the natural map \be \Ext^1(C,A) & \to & \Ext^1(C_{\rm tor},A) \ee is an isomorphism.  Hence, if we have a solution $A \into A'$ for the exact sequence $$ 0 \to A \to B'' \to C_{\rm tor} \to 0$$ obtained by pulling back \eqref{FGAGs} along the inclusion $C_{\rm tor} \into C$ then $A \into A'$ will also be a solution for \eqref{FGAGs}.  This reduces us to the case where $C = C_{\rm tor}$ is finite, which we assume for the remainder of the proof.  By applying $\Hom(C, \slot)$ to the exact sequence $$ 0 \to A_{\rm tor} \to A \to A/A_{\rm tor} \to 0 $$ and using the fact that $A/A_{\rm tor}$ is free, we find that the natural map \be \Ext^1(C,A_{\rm tor}) & \to & \Ext^1(C,A) \ee is an isomorphism.  Hence the exact sequence \eqref{FGAGs} is ``the" pushout of some exact sequence \bne{torSES} & 0 \to A_{\rm tor} \to B''' \to C \to 0 \ene along the inclusion $A_{\rm tor} \into A.$  In this situation, if $A_{\rm tor} \into A''$ is a solution for \eqref{torSES}, then the map $A \into A'$ obtained by pushing out $A_{\rm tor} \into A''$ along $A_{\rm tor} \into A$ will be a solution for \eqref{FGAGs}.  This reduces us to the case where $A = A_{\rm tor}$ is finite, which we assume for the remainder of the proof.  By the structure theory of finite abelian groups, we can find an isomorphism \be C & \cong & \ZZ / p_1^{e_1} \ZZ \oplus \cdots \oplus \ZZ / p_m^{e_m} \ZZ \ee for some primes $p_1,\dots,p_m$ (not necessarily distinct) and some positive integers $e_1,\dots,e_m$.  For $i \in \{ 1, \dots, m \}$, let \bne{sumseq} 0 \to A \to B_1 \to \ZZ / p_i^{e_i} \ZZ \to 0 \ene be the exact sequence obtained from \eqref{FGAGs} by pulling back along the summand inclusion $\ZZ / p_i^{e_i} \ZZ \into C$.  Suppose we have a solution $A \into A_i$ for each \eqref{sumseq}.  Then since $\Ext^1( \slot, A)$ preserves finite direct sums, we see that the nautral map $A \into A' := A_1 \oplus \cdots \oplus A_m$ is a solution to \eqref{FGAGs}.  This reduces us to the case (assumed from now on) where $C = \ZZ / p^e \ZZ$ for some prime $p$ and some positive integer $e$.  Choose an isomorphism \be A & \cong & \ZZ / \ell_1^{f_1} \ZZ \oplus \cdots \oplus \ZZ / \ell_n^{f_n} \ZZ \ee for primes $\ell_1,\dots,\ell_n$ and positive integers $f_1,\dots,f_n$, so that the $\ell_i$ are precisely the primes dividing $|A|=|A_{\rm tor}|$.  For $i \in \{ 1, \dots, n \}$, let \bne{sumseq2} 0 \to \ZZ / \ell_i^{f_i} \to B_i' \to C \to 0 \ene be the exact sequence obtained by pushing out \eqref{FGAGs} along the projection $A \to \ZZ / \ell_i^{f_i}$.  Suppose we have a solution $\ZZ / \ell_i f^i \into A_i'$ for each \eqref{sumseq2} (so $\ell_i$ is the only prime dividing $|A_i'|$).  Then since $\Ext^1(C, \slot)$ preserves finite direct sums, the natural map \be (A \into A') & := & \oplus_{i=1}^n ( \ZZ / \ell_i^{f_i} \ZZ \into A_i') \ee will be a solution for \eqref{FGAGs}.  It remains only to treat the case where $A = \ZZ / \ell^f \ZZ$ for some prime $\ell$ and some positive integer $f$, so that \eqref{FGAGs} is of the form \bne{FGAGSeqspecial} & 0 \to \ZZ / \ell^f \ZZ \to B \to \ZZ / p^e \ZZ \to 0. \ene  By applying $\Hom( \slot, G)$ to the obvious exact sequence $$0 \to \ZZ \to \ZZ \to \ZZ / p^e \ZZ \to 0$$ we obtain an isomorphism \bne{Extiso} \Ext^1(\ZZ/p^e \ZZ, G) & = & G/p^e G \ene natural in the abelian group $G$.  In particular, if the primes $p$ and $\ell$ are not equal, then \eqref{FGAGSeqspecial} must split because then $p^e ( \ZZ / \ell^f \ZZ) = \ZZ / \ell^f \ZZ$, so $\Id : A \to A$ will be a solution for \eqref{FGAGSeqspecial}.  It remains only to treat the case $p=\ell$.  Here the map \bne{Extsolution} p^e = \ell^e : \ZZ / \ell^f \ZZ & \into & \ZZ / \ell^{e+f} \ZZ \ene will be a solution for \eqref{FGAGSeqspecial} because of the natural isomorphism \eqref{Extiso} and the fact that the image of \eqref{Extsolution} is manifestly contained in the subgroup $p^e( \ZZ / \ell^{e+f} \ZZ)$.  \end{proof}

\begin{thm} \label{thm:splittingmonoids} For any fine monoid $Q$ we can find an injective map of fine monoids $Q \into P$ with the following properties: \begin{enumerate} \item \label{splittingmonoids1} The map $Q^* \into P^*$ has finite cokernel and the integers $|Q^*_{\rm tor}|$ and $|P^*_{\rm tor}|$ have the same prime divisors. \item \label{splittingmonoids2} The diagram of monoids $$ \xym{ Q^* \ar[r] \ar[d] & P^* \ar[d] \\ Q \ar[r] & P }$$ is a pushout diagram.  In particular, if we choose a subset $S \subseteq P^*$ mapping bijectively onto $P^*/Q^*$ under the projection, then $P$ is a free $Q$-module with basis $S$. \item \label{splittingmonoids3} The monoid $P$ is split.  That is, $P \cong P^* \oplus \ov{P}$ as monoids over $\ov{P}$. \end{enumerate} \end{thm}

\begin{proof} Choose $Q^* \into A'$ as in Lemma~\ref{lem:splittinggroups} with the exact sequence there given by the characteristic sequence of $Q$.  Define $Q \into P$ to be the pushout of $Q^* \into A'$ along $Q^* \into Q$.  It is straightforward to check that $P$ is fine, that $A'=P^*$, $\ov{Q} = \ov{P}$, and that the characteristic sequence of $P$ is nothing but the pushout of the characteristic sequence of $Q$ along $Q^* \into A' = P^*$, which splits by the choice of $Q^* \into A'$.  The desired properties of $Q \into P$ follow. \end{proof}

\subsection{Monoid algebra} \label{section:monoidalgebra}  Given a monoid $P$, the \emph{monoid algebra} $\ZZ[P]$ is the ring with underlying additive abelian group equal to the free abelian group on symbols $[p]$ for $p \in P$.  Its multiplication law is characterized by $[p][p']=[p+p']$ for $p,p' \in P$.  This construction yields a functor \bne{Zslot} \ZZ[ \slot ] : \Mon & \to & \An \ene from monoids to rings left adjoint to the obvious forgetful functor the other way, so that we have a bijection \bne{monoidalgebraadjunction} \Hom_{\An}(\ZZ[P],A) & = & \Hom_{\Mon}(P,A) \ene natural in $P \in \Mon$, $A \in \An$.  For a fixed monoid $P$, we have a similar functor \bne{Zslot2} \ZZ[ \slot ] : \Mod(P) & \to & \Mod(\ZZ[P]) \ene from $P$-modules to $\ZZ[P]$-modules left adjoint to the forgetful functor.  The functors \eqref{Zslot} and \eqref{Zslot2} are clearly faithful.

The functor \eqref{Zslot2} clearly takes free (resp.\ finitely generated) modules to free (resp.\ finitely generated) modules and commutes with filtered direct limits, hence it takes flat modules to flat modules.  This functor also clearly takes takes ideals of $P$ to ideals of $\ZZ[P]$.  We can use \eqref{Zslot2} to prove that when $P$ is a finitely generated monoid, every submodule of a finitely generated $P$-module is itself finitely generated:  If not there would be some infinite strictly ascending chain and then $\ZZ[ \slot ]$ of it would be an infinite strictly ascending chain in a finitely generated module over the noetherian ring $\ZZ[P]$, which can't happen.

\subsection{Locally monoidal spaces} \label{section:locallymonoidalspaces}  Here we shall introduce and study the category of \emph{locally monoidal spaces}, which is the obvious ``monoidal" analog of the category of locally ringed spaces.

\begin{defn} \label{defn:local} A map of monoids $h : Q \to P$ is called \emph{local} iff $h^{-1}(P^*) = Q^*$. \end{defn}

\begin{thm} \label{thm:localdirectlimits} Let $i \mapsto P_i$ be a functor to monoids (direct limit system of monoids) with direct limit $P$.  Suppose every transition map $P_i \to P_j$ for the direct limit system is a \emph{local} map of monoids.  Then the structure maps $P_i \to P$ to the direct limit are also local and $P$ is also the direct limit of $i \mapsto P_i$ in the category of monoids with \emph{local} maps as the morphisms. \end{thm}

\begin{proof} See \cite[1.2.8]{GM1}.  To prove this, one really has to get into the details of the construction of direct limits of monoids. \end{proof}

It is impossible to overestimate the importance of Theorem~\ref{thm:localdirectlimits}---there is no analogous statement for rings because even a tensor product of two fields will not generally be a local ring.

\begin{defn} \label{defn:locallymonoidalspace}  A \emph{monoidal space} (also called a \emph{locally monoidal space}) $X = (X,\M_X)$ is a pair consisting of a topological space $X$ and a sheaf of monoids $\M_X$ on $X$, called the \emph{structure sheaf} of $X$.  A \emph{morphism of monoidal spaces} $f : X \to Y$ is a map of spaces $f : X \to Y$ (abuse of notation) together with a map of sheaves of monoids $f^\dagger : f^{-1} \M_Y \to \M_X$ on $X$.  Such a morphism is called a \emph{morphism of locally monoidal spaces} iff the map of monoids $f^\dagger_x : \M_{Y,f(x)} \to \M_{X,x}$ is local for each $x \in X$.  The category whose objects are monoidal spaces and whose morphisms are morphisms of monoidal spaces (resp.\ morphisms of locally monoidal spaces) is denoted $\MS$ (resp.\ $\LMS$).  The categories $\MS$ and $\LMS$ have the same objects, but we usually refer to a monoidal space as a ``locally monoidal space" when we think of it as an object of $\LMS$. \end{defn}

\begin{defn} \label{defn:embedding} A map $f : X \to Y$ of locally monoidal spaces (resp.\ locally ringed spaces) is called a \emph{strict embedding} (resp.\ \emph{open embedding}) iff $f$ is an embedding (resp.\ open embedding) on the level of topological spaces and $f^\dagger: f^{-1}\M_Y \to \M_X$ (resp.\ $f^\sharp : f^{-1} \O_Y \to \O_X$) is an isomorphism.  \end{defn}

\begin{rem} \label{rem:embedding} If $Y$ is a locally monoidal space and $U$ is a subspace of its topological space, then there is an obvious strict embedding of locally monoidal spaces \be i : U := (U,\O_Y|U) & \to & Y,\ee which is an open embedding iff $U$ is open.  In any case, the $\LMS$-morphism $i$ has the following ``universal property" completely analogous to that of the ``subspace topology:"  For any $X \in \LMS$, the map $g \mapsto ig$ is a bijection from $\Hom_{\LMS}(X,(U,\O_Y|U))$ to the set of $\LMS$-morphisms $f : X \to Y$ for which $f(X) \subseteq U$ on the level of topological spaces.  If $f : X' \to X$ is a map of locally monoidal spaces with image $f(X') = U$ on the level of topological spaces, then $f$ is an embedding iff $X'$ is isomorphic, as a locally monoidal space over $X$, to $(U,\O_X|U)$.  In particular, two strict embeddings $X',X'' \to X$ with the same image are (uniquely) isomorphic as locally monoidal spaces over $X$.  So strict embeddings (resp.\ open embeddings)  to $X$ are ``the same thing" as subsets (resp.\ open subsets) of the underlying space of $X$.  Similar remarks remain valid when ``monoidal" is replaced by ``ringed."  \end{rem}

\subsection{Inverse limits} \label{section:inverselimits}  Like the category of ringed spaces, the category $\MS$ has all inverse limits and these limits can be constructed in the ``obvious" way.  ``The" inverse limit of a functor $i \mapsto X_i$ to $\MS$ is constructed as follows:  First form ``the" inverse limit $X$ in the category $\Top$ topological spaces.  Let $\pi_i : X \to X_i$ be the projection.  Endow $X$ with the structure sheaf $\M_X := \dirlim \pi_i^{-1} \M_{X_i}$, so we have a natural map \bne{invlimitmap} \pi_i^\dagger : \pi_i^{-1} \M_{X_i} & \to & \M_X \ene for each $i$, thus we can view $\pi_i : X \to X_i$ as a morphism of monoidal spaces.  It is easy to see that the $X$ thus constructed serves as ``the" inverse limit of $i \mapsto X_i$.  In particular, we see that the underlying space functor \bne{underlyingspacefunctor} \MS & \to & \Top \ene preserves inverse limits.

Now suppose our functor $i \mapsto X_i$ is actually a functor to $\LMS$.  Since $\LMS$ is a subcategory (though not a full one) of $\MS$, we can view our functor as a functor to $\MS$ and compute its inverse limit $X$ in $\MS$ as above.  By construction of $\M_X$ and the fact that direct limits of sheaves commute with stalks, the stalk of \eqref{invlimitmap} at $x \in X$ is the structure map \bne{invlimitmapstalk} \M_{X_i,\pi_i(x)} & \to & \dirlim \M_{X_i,\pi_i(x)} \ene to the direct limit, which is \emph{local} (Definition~\ref{defn:local}) by Theorem~\ref{thm:localdirectlimits} because the transition maps for the direct limit system in question are local since $i \mapsto X_i$ is a functor to $\LMS$.  It follows easily that $X$ is also the inverse limit of $i \mapsto X_i$ in $\LMS$.  In other words, the inclusion \bne{LMStoMS} \LMS & \to & \MS \ene preserves inverse limits!  This is certainly not true of the analogous inclusion $\LRS \to \RS$ of locally ringed spaces into ringed spaces!  This is one major reason why the geometry of monoids is much simpler than the geometry of rings.

\subsection{Spec} \label{section:Spec} For any monoid $P$ there is an associated locally monoidal space $\Spec P = (\Spec P,\M_P)$, contravariantly functorial in $P$, so that $\Spec$ may be viewed as a functor \bne{SpecLMS} \Spec : \Mon^{\rm op} & \to & \LMS. \ene  In analogy with the ring situation, the functor \eqref{SpecLMS} is right adjoint to the functor \bne{LMStoSpec} \LMS & \to & \Mon^{\rm op} \\ \nonumber X & \mapsto & \M_X(X), \ene so we have a bijection \bne{Specadjunction} \Hom_{\LMS}(X,\Spec P) & = & \Hom_{\Mon}(P,\M_X(X)) \ene natural in $X \in \LMS$ and $P \in \Mon^{\rm op}$.  In particular, for each monoid $P$, we obtain a natural map \bne{globalsections} P & \to & \M_P(\Spec P) \ene (in fact it is an isomorphism, as we shall see momentarily) by considering the image of the identity map under \eqref{Specadjunction} in the case $X = \Spec P$.  Since the ``constant sheaf" functor is left adjoint to the global section functor, we can view \eqref{globalsections} as a map of sheaves of monoids \bne{PtoMP} \u{P} & \to & \M_P \ene on $\Spec P$, where $\u{P}$ denotes the constant sheaf associated to $P$.  

The constructions of $\Spec P$ and the adjunction isomorphism \eqref{SpecLMS} are completely parallel to the analogous constructions with (commutative) rings, so we will content ourselves with a brief review.  We can view the points of $\Spec P$ either as prime ideals of $P$, or faces of $P$ (Definition~\ref{defn:face}).  \emph{We will use the latter convention}, but we will avoid confusion by writing $F$, $G$, \dots for faces and $\p$, $\q$, \dots for prime ideals, as well as writing ``$F \leq P$" instead of ``$F \in \Spec P$".   As in the case of rings, the topology on $\Spec P$ is the ``Zariski topology," whose basic open subsets are those of the form \bne{Up} U_p & := & \{ F \leq P : p \in F \} \ene for $p \in P$.  By definition of a face, we clearly have \bne{basicopenintersection} U_{p+p'} & = & U_p \cap U_{p'} \\ \nonumber U_0 & = & \Spec P. \ene  The complementary basic closed sets \be Z_p & := & \{ F \leq P : p \notin F \} \ee hence satisfy \bne{basicclosedunion} Z_{p+p'} & = & Z_p \cup Z_{p'} . \ene  If $h : Q \to P$ is a map of monoids, then one checks that for a face $F \leq P$, $(\Spec h)(F) := h^{-1}(F)$ is a face of $\Spec Q$.  The resulting map of topological spaces \be \Spec h : \Spec P & \to & \Spec Q \ee is continuous because the preimage of the basic open subset $U_q \subseteq \Spec Q$ under $\Spec h$ is the basic open subset $U_{h(q)}$ of $\Spec P$:  \be (\Spec h)^{-1}(U_q) & = & U_{h(q)}. \ee

The sheaf of monoids $\M_P$ and the map \eqref{globalsections} are characterized by the fact that there are isomorphisms \bne{basicopensections} \M_P(U_p) & = & P[-p] \ene for each $p \in P$ under which \begin{enumerate} \item the composition of \eqref{globalsections} and the restriction map $\M_P(\Spec P) \to \M_P(U_p)$ is identified with the localization map $P \to P[-p]$.  \item the restriction map $\M_P(U_p) \to \M_P(U_{p'})$ associated to an inclusion $U_{p'} \subseteq U_p$ of basic open subsets is identified with the localization map $P[-p] \to P[-p'].$ \end{enumerate}  Here $P[-p]$ denotes the initial object among monoids under $P$ where $p$ becomes invertible.

Here are some basic facts about $\Spec P$:

\begin{prop} \label{prop:Spec} Let $P$ be a monoid. \begin{enumerate} \item \label{opensubsets} A subset of $\Spec P$ is open iff it is of the form \be \{ F \in \Spec P : F \cap S \neq \emptyset \} & = & \cup_{p \in S} U_p \ee for some subset $S \subseteq P$. \item \label{stalkMP} For $F \in \Spec P$, the stalk of $\M_P$ at $F$ is given by $\M_{P,F}  =  F^{-1} P$ and the stalk of \eqref{PtoMP} at $F$ is the localization map $P \to F^{-1}P$.  \item \label{basicopens} For any $p \in P$, the map of locally monoidal spaces $\Spec P[-p] \to \Spec P$ induced by the localization map $P \to P[-p]$ is an open embedding (Definition~\ref{defn:embedding}) with image $U_p$.  \item \label{Specclosure} For faces $F$, $G$ of $P$, one has $F \in \{ G \}^-$ in $\Spec P$ iff $F \subseteq G$.  In particular, the point $P$ of $\Spec P$ is dense in $\Spec P$ and hence $\Spec P$ is an irreducible topological space.  \item \label{irreducibles} For any $F \in \Spec P$, the subset \be Z_F & := & \{ G \in \Spec P : G \subseteq F \} \ee of $\Spec P$ is the closure of $F$ in $\Spec P$, hence it is an irreducible closed subset of $\Spec P$.  Any non-empty irreducible closed subset of $\Spec P$ is equal to $Z_G$ for a unique $G \in \Spec P$.  In other words, $\Spec P$ is \emph{sober} in the sense that any non-empty irreducible closed subset of $\Spec P$ is the closure of a unique point.  \item \label{generalization} When $F \in \{ G \}^-$, the generalization map $\M_{P,F} \to \M_{P,G}$ is identified with the localization map $F^{-1}P \to G^{-1}P = (F^{-1}G)^{-1}(F^{-1}P)$.  \item \label{generalizationprimes}  For $F \in \{ G \}^-$, let $F_G$ be the preimage of $\M_{P,G}^*$ in $\M_{P,F}$ under the generalization map $\M_{P,F} \to \M_{P,G}$.  Then $G \mapsto F_G$ is a bijection between the set of generalizations $G$ of $F$ in $\Spec P$ and the set $\Spec \M_{P,F}$ of faces of $\M_{P,F}$.  In particular, the point $P \in \Spec P$ has no (non-trivial) generalizations in $\Spec P$ and is the unique point of $\Spec P$ with this property.  \item \label{uniqueclosedpoint} $P^* \in \Spec P$ is the unique closed point of $\Spec P$ and any $F \in \Spec P$ is a generalization of $P^*$.  $\Spec P$ itself is the only open subset of $\Spec P$ containing of $P^*$.  In particular, $\Spec P$ is quasi-compact.  \item \label{localizationembedding} For any submonoid $S \subseteq P$, $\Spec S^{-1}P \to \Spec P$ is a strict embedding (but not necessarily an open embedding) with image equal to the set of $G \in \Spec P$ with $S \leq G$.  \end{enumerate} \end{prop}

\begin{proof}  All of these statements have analogs for $\Spec A$, $A$ a ring, and can be proved in essentially the same manner.  We can therefore leave most of the proofs to the reader.  The only possible exception is \eqref{irreducibles}, for which some explanation may be helpful.  The first sentence of \eqref{irreducibles} and the uniqueness assertion in \eqref{irreducibles} are consequences of \eqref{Specclosure}, so the only issue is to show that a non-empty irreducible closed subset $Z$ of $\Spec P$ is equal to $Z_G$ for some $G \in \Spec P$.  If we believe this is true, then we know $G$ will have to be defined by $G = \sum_{F \in Z} F$, so we need to show that this $G$ will do.  

(The same arguments that we are about to give will actually also work to prove that any irreducible closed subset $Z$ of $\Spec A$ ($A$ a ring) is of the form $$ \{ \q \in \Spec A : \p \subseteq \q \} $$ for a unique $\p \in \Spec A$, and that we can take $\p = \cap_{\q \in Z} \q$.  However, the reader might not be familiar with this line of reasoning because there are other ways of characterizing the irreducible closed subsets of $\Spec A$ that don't necessarily make much sense in the monoid situation.)   

First of all we need to show that $G$ is actually a face of $P$.  Suppose not.  Then we can find $p \in P \setminus G$, $p' \in P$, $F_1,\dots,F_n \in Z$, and $f_1 \in F_1, \dots, f_n \in F_n$ such that $p+p'=f_1+\cdots+f_n$.  Formula \eqref{basicclosedunion} then yields an equality \bne{closedsubsetequality} Z_p \cup Z_{p'} & = & Z_{f_1} \cup \cdots \cup Z_{f_n} \ene of closed subspaces of $\Spec P$.  Since $p \notin G$, $p$ is not in any $F$ in $Z$, so we have $Z \subseteq Z_p$.  Using \eqref{closedsubsetequality}, we find that $Z \subseteq Z_{f_1} \cup \cdots \cup Z_{f_n}$, so by irreducibility of $Z$ we conclude that $Z \subseteq Z_{f_i}$ for some $i$.  But this is absurd because $F_i \in Z$ is \emph{not} in $Z_{f_i}$ since $f_i \in F_i$.

Now that we know $G \leq P$, the equality $Z = Z_G$ will follow immediately from \eqref{Specclosure} and the definition of $G$, provided we can show that $G \in Z$.  Suppose not.  Then since $Z$ is closed in $\Spec P$ and $G$ is a point of $\Spec P$ not in $Z$, we can find some $p \in P$ so that the basic closed subset $Z_p$ contains $Z$ but not $G$.  Since $G \notin Z_p$, we have $p \in G$ so we can write $p = f_1+\cdots+f_n$ for some $F_1,\dots,F_n \in Z$, $f_i \in F_i$.  Using the irreducibility of $Z$ again, we reach the same contradiction as in the previous paragraph. \end{proof}

\begin{rem} \label{rem:ZF} Let $P$ be a monoid, $F \in \Spec P$.  The irreducible closed subset $Z_F$ of $\Spec P$ in Proposition~\ref{prop:Spec}\eqref{irreducibles} is clearly in bijective correspondence with $\Spec F$ by regarding a face of $F$ as a face of $P$.  Furthermore, the topology on $Z_F$ inherited from $\Spec P$ coincides (under this obvious bijection) with the usual ``Zariski" topology on $\Spec F$ because, for $p \in P$, the intersection $U_p \cap Z_F$ is either empty (when $p \in P \setminus F$), or equal to the corresponding basic open subset $U_p$ of $\Spec F$ (when $p \in F \subseteq P$).   

It is important to understand, however, that this ``closed embedding" of $\Spec F$ into $\Spec P$ is purely topological---there does not actually exist, in general, any monoid homomorphism $r : P \to F$ for which $\Spec r$ is a homeomorphism from $\Spec F$ to $\Spec P$.  Indeed, in contrast with the situation for rings, it is not generally possible to write a closed subset of $\Spec P$ as the image of $\Spec (h : P \to Q)$ for a monoid homomorphism $h$.  For example, the closed point of $\Spec \NN$ does not arise in this manner (exercise!).  Similarly, a closed subset of $\Spec P$ need not even be the underlying space of any fan (Example~\ref{example:closedsubspacesoffans}).   We will describe a kind of crude substitute for this in \S\ref{section:boundaryconstruction}.  An alternative approach (discussed in \S\ref{section:monoidschemes}) is to work with ``pointed monoids," where it \emph{does} become possible to define a map like ``$r$" above.  \end{rem}

\begin{rem} \label{rem:SpecP} Proposition~\ref{prop:Spec}\eqref{uniqueclosedpoint} implies that the global section functor $\Gamma(\Spec P, \slot)$ is just the ``stalk at the closed point $P^*$" functor; in particular it preserves direct limits and finite inverse limits.  In particular, combining this with Proposition~\ref{prop:Spec}\eqref{stalkMP} we find $$ \Gamma(\Spec P,\M_P)  =  \M_{P,P^*}  = (P^*)^{-1} P = P. $$ \end{rem}

Combining the above remark and the adjunction formula \eqref{Specadjunction}, we obtain:

\begin{prop} \label{prop:Specfullyfaithful} The functor $\Spec : \Mon^{\rm op} \to \LMS$ is fully faithful. \end{prop}

\begin{prop} \label{prop:faces} Let $P$ be a monoid, $S$ a set of generators of $P$.  For each face $F$ of $P$, $F$ is generated as a monoid by $S \cap F$ and $F \mapsto S \cap F$ is an injective map from $\Spec P$ to the set of subsets of $S$. \end{prop}

\begin{proof} The point is that if we write a non-zero element $f \in F$ as $f = \sum_{i=1}^n a_i s_i$ for positive integers $a_i$ and $s_i \in S$, then all the $s_i$ are also in $F$ because $F$ is a face. \end{proof}

\begin{cor} \label{cor:faces} If $P$ is finitely generated (resp.\ fine), then $\Spec P$ is finite and for every face $F \leq P$, both $F$ and $F^{-1}P$ are finitely generated (resp.\ fine). \end{cor}

\begin{rem} \label{rem:finitegeneration} Even though every face of a finitely generated monoid is finitely generated, it is not true that every submonoid of a finitely generated monoid is finitely generated.  For example, the submonoid \be S & := & \{ (a,b) \in \NN^2 : a=0 {\rm \; or \; } b>0 \} \ee of $\NN^2$ is not finitely generated. \end{rem}

\begin{prop} \label{prop:Specfinite} Let $P$ be a monoid for which $\Spec P$ is finite.  \begin{enumerate} \item As in any finite topological space, the topology of $\Spec P$ is determined by specialization, so that a subset is open (resp.\ closed) iff it is stable under generalization (resp.\ specialization).  \item \label{UF} As in any finite topological space, every point $F \in \Spec P$ has a smallest open neighborhood $U_F$ given by \be U_F & = & \{ G \in \Spec P : F \in \{ G \}^- \} \ee and $\{ U_F : F \in \Spec P \}$ is a basis for $\Spec P$.  \item \label{UFisUp} Since $\{ U_p : p \in P \}$ is a basis for $\Spec P$, each $U_F$ must be equal to $U_p$ for some $p \in P$.  \item \label{facelocalizationembedding} For every face $F \in \Spec P$, the map $\Spec F^{-1}P \to \Spec P$ is an open embedding with image $U_F$. \item \label{uniqueopenpoint} The face $P$ is the unique open point of $\Spec P$. \end{enumerate} \end{prop}

\begin{proof} This follows easily from Proposition~\ref{prop:Spec}. \end{proof}

\begin{rem} \label{rem:Specfinite} If one assumes $P$ is finitely generated, then one can see ``directly" (without using Proposition~\ref{prop:Specfinite}\eqref{UFisUp}) that $U_F$ ($F \leq P$) is one of the basic opens $U_p$, as follows.  By Corollary~\ref{cor:faces}, $F$ is finitely generated, say by $p_1,\dots,p_n$.  Then $U_F = U_p$ for $p := p_1+\cdots+p_n$ because any face containing $p$ must contain $p_1,\dots,p_n$ and must therefore contain $F$. \end{rem}

The basic properties of inverse limits in $\MS$ and $\LMS$ discussed in \S\ref{section:inverselimits} ensure that the $\Spec$ functor ``commutes with inverse limits" in the strongest possible sense:

\begin{thm} \label{thm:Specinverselimits} Let $i \mapsto P_i$ be a direct limit system of monoids with direct limit $P$.  Let $e_i : P_i \to P$ be the structure map to the direct limit.  Then the map of topological spaces \bne{topmap} \invlim \Spec e_i : \Spec P & \to & \invlim \Spec P_i \ene is a homeomorphism, the natural map \bne{monmap} \dirlim (\Spec e_i)^{-1} \M_{P_i} & \to & \M_P \ene of sheaves of monoids on $\Spec P$ is an isomorphism, and $\Spec P$ is the inverse limit of $i \mapsto \Spec P_i$ in both $\LMS$ and $\MS$. \end{thm}

\begin{proof} Certainly $\Spec P$ is the inverse limit of $i \mapsto \Spec P_i$ in $\LMS$ because the $\Spec$ functor \eqref{SpecLMS} is a right adjoint, so it preserves inverse limits.  Since $\LMS \to \MS$ also preserves inverse limits (\S\ref{section:inverselimits}), we see that $\Spec P$ is also the inverse limit of $i \mapsto \Spec P_i$ in $\MS$, hence the other assertions follow from the explicit construction of inverse limits in $\MS$ discussed in \S\ref{section:inverselimits}.  \end{proof}

\begin{example} \label{example:SpecN}  The prototypical example to keep in mind is $\Spec \NN$.  The monoid $\NN$ has only the two obvious faces $\{ 0 \} = \NN^*$ and $\NN$, so that $\Spec \NN$ has only the two obvious points: the closed point $0 = \{ 0 \}$ and the generic point $\NN$.  The generic point is the image of the unique point of $\Spec \ZZ$ under the map \be \Spec \ZZ & \to & \Spec \NN \ee induced by the inclusion $\NN \into \ZZ$.  Evidently then, the topological space $\Spec \NN$ is the two point ``Sierpinski space" where $ 0  \in \{ \NN \}^-$.  The corresponding generalization map on stalks $\M_{ \NN, 0 } \to \M_{ \NN, \NN}$ is just the obvious inclusion $\NN \into \ZZ$.  On a finite topological space, the generalization maps on stalks uniquely determine a sheaf, so the category of sheaves on $\Spec \NN$ is just the category of maps of sets.  With this understanding, the structure sheaf $\M_{\NN}$ \emph{is} the map $\NN \to \ZZ$.

To give an $\LMS$-morphism $g : \Spec \ZZ \to X$ is to give a point $x' \in X$ and a local map of monoids $g^\dagger : \M_{X,x'} \to \ZZ$.  This locality is equivalent to saying that $\M_{X,x'}$ is a group; we usually write $g$ instead of $g^\dagger$.  To give a map $g : \Spec \NN \to X$ is to give points $x = g(0)$ and $x' = g(\NN) \in X$ with $x \in \{ x' \}^-$, and a commutative diagram of monoids $$ \xym{  \M_{X,x} \ar[d]_{g^\dagger_{0}} \ar[r] & \M_{X,x'} \ar[d]^{g^\dagger_{\NN}} \\ \NN \ar[r] & \ZZ } $$ where the vertical arrows are local maps of monoids.  This latter condition is equivalent to saying that $(g^\dagger_{0})^{-1}(0) = \{ 0 \}$ and $\M_{X,g(\NN)}$ is a group.  We will usually just write $g$ for both $g^\dagger_0$ and $g^\dagger_{\NN}$.  To lift a map $g : \Spec \ZZ \to X$ corresponding to $x' \in X$ and $g : \M_{X,x'} \to \ZZ$ to a map $g : \Spec \NN \to X$ is to a give a point $x \in \{ x' \}^-$ so that the composition $$ \M_{X,x} \to \M_{X,x'} \to \ZZ $$ of the generalization map and $g$ yields a local map $g : \M_{X,x} \to \NN$.  Compare the description \cite[II.4.4]{H} of maps out of a trait in algebraic geometry.  Notice that a monoid homomorphism $h : P \to \NN$ is local iff $\Spec \NN$ takes the closed point to the closed point. \end{example}

\subsection{Spec invariance} \label{section:specinvariance}  The topological space underlying $\Spec P$ turns out to be a rather crude invariant of $P$.  The map of spaces $\Spec h$ is a homeomorphism for a wide variety of monoid homomorphisms $h$.  We now summarize the main examples.

\begin{lem} \label{lem:specembedding} If $h : Q \to P$ is a surjective map of monoids, then $\Spec h$ is an embedding (not necessarily closed!) on the level of topological spaces. \end{lem}

\begin{proof} If $F,G \leq P$ are faces with $h(F)=h(G)$ then surjectivity of $h$ implies $F=G$, so $\Spec h$ is one-to-one.  To see that it is an embedding, note that for \emph{any} monoid homomorphism $h$ and any $q \in Q$, we have $(\Spec h)^{-1}(U_q) = U_{h(q)}$ where $U_q \subseteq \Spec Q$ and $U_{h(q)} \subseteq \Spec P$ are the basic open subsets defined in \eqref{Up}.  When $h$ is surjective, this proves that the open sets $(\Spec h)^{-1}(U_q)$ are a basis for the topology on $\Spec P$. \end{proof}

\begin{example} \label{example:notclosed}  Consider the addition map $h=(1,1): \NN^2 \to \NN$.  The image of $\Spec h$ consists of the two points $h^{-1}(0)=0, h^{-1}(\NN)=\NN^2 \in \Spec \NN^2$.  This set of two points is not closed because the other two points $\NN \oplus 0$ and $0 \oplus \NN$ of $\Spec \NN^2$ are in its closure by Proposition~\ref{prop:Spec}\eqref{Specclosure}. \end{example}

\begin{thm} \label{thm:sharpening} Suppose $h : Q \to P$ is a surjective map of monoids with the following property:  For any $q_1,q_2 \in Q$ with $h(q_1) = h(q_2)$, there is an $n \in \ZZ_{>0}$ and a $u \in Q^*$ such that $nq_1 = nq_2+u$.  Then $h(F) \leq P$ for any $F \leq Q$ and $\Spec h$ is a homeomorphism with inverse $F \mapsto h(F)$.  \end{thm}

\begin{proof} To see that $h(F)$ is a face of $P$, suppose $p_1+p_2 \in h(F)$.  Pick $q_i \in Q$ with $h(q_i) = p_i$ ($i=1,2$) and $f \in F$ with $p_1+p_2=h(f)$.  Then $h(q_1+q_2) = h(f)$, so by assumption $nq_1+nq_2 = nf+u$ for some $n \in \ZZ_{>0}$, $u \in Q^*$.  Since $nf+u \in F$ ($u \in F$ because $u+(-u) =0 \in F$ and $F$ is a face) and $F$ is a face, $q_1,q_2 \in F$, hence $p_i = h(q_i) \in F$.  By Lemma~\ref{lem:specembedding}, $\Spec h$ is an embedding, so to see that it is a homeomorphism, we just need to check that it is surjective.  For any $F \in \Spec Q$, we just saw that $h(F) \in \Spec P$ and the same argument shows that $(\Spec h)(h(F)) = h^{-1}(h(F)) = F$. \end{proof}

\begin{cor} \label{cor:sharpening} For any monoid $P$, the sharpening map $P \to \ov{P} = P/P^*$ satisfies the hypotheses of Theorem~\ref{thm:sharpening}, hence it induces a homeomorphism $\Spec \ov{P} \to \Spec P$ natural in $P$. \end{cor}

Combining \eqref{cor:sharpening} and Corollary~\ref{cor:faces}, we obtain:

\begin{cor} \label{cor:Specfinite} $\Spec P$ is finite whenever $\ov{P}$ is finitely generated. \end{cor}

\begin{cor} \label{cor:tf}  For a monoid $P$, let $P^{\rm tf}$ be the image of $P$ in $P^{\rm gp} / P^{\rm gp}_{\rm tor}$.  Then $(P^{\rm tf})^{\rm gp} = P^{\rm gp} /  P^{\rm gp}_{\rm tor}$ and the surjection $P \to P^{\rm tf}$ is initial among maps from $P$ to a torsion-free monoid (\S\ref{section:monoids}).  If $P$ is integral then \begin{enumerate} \item \label{tflocal} $P \to P^{\rm tf}$ is local.  \item \label{tfhomeo} $P \to P^{\rm tf}$ induces a homeomorphism $\Spec P^{\rm tf} \to \Spec P$ with inverse $F \mapsto F^{\rm tf}$. \item \label{tflocalization} for any face $F \leq P$, the natural map \be F^{-1} P & \to & (F^{\rm tf})^{-1} P^{\rm tf} \ee is local and induces an isomorphism \be (F^{-1} P)^{\rm tf} & \to & (F^{\rm tf})^{-1} P^{\rm tf}. \ee \end{enumerate} \end{cor}

\begin{proof} The formula for $(P^{\rm tf})^{\rm gp}$ and the universality statement are straightforward.  Suppose, for the remainder of the proof, that $P$ integral.  If $p \in P$ maps to a unit in $P^{\rm tf}$, then we can write $p+p'=u$ in $P^{\rm gp}$ for some $p' \in P$, $u \in P^{\rm gp}_{\rm tor}$.  If we let $n \in \ZZ_{>0}$ be the order of $u$, then we find $np+np'=0$ in $P^{\rm gp}$, hence $np+np'=0$ in $P$ since $P$ is integral, hence $p \in P^*$.  This proves \eqref{tflocal}.  \eqref{tfhomeo} follows from Theorem~\ref{thm:sharpening} because if $p,p' \in P$ have the same image in $P^{\rm tf}$, then we have an equality $p=p'+u$ in $P^{\rm gp}$ for some $u \in P^{\rm gp}_{\rm tor}$, hence $np=np'$ in $P^{\rm gp}$ for $n$ the order of $u$, hence $np=np'$ in $P$ because $P$ is integral.  The explicit formula for the inverse of the homeomorphism will also follow from that theorem provided we can show that $F^{\rm tf}$ is the image of $F$ in $P^{\rm tf}$ for any $F \leq P$.  From the universal property, we obtain a commutative square $$ \xym{ F \ar[r] \ar[d] & P \ar[d] \\ F^{\rm tf} \ar[r] & P^{\rm tf} } $$ where the vertical arrows are surjective, so we reduce to proving that $F^{\rm tf} \to P^{\rm tf}$ is injective.  Indeed, if $f,f' \in F$ have the same image in $P^{\rm tf}$, then we can write $f=f'+u$ in $P^{\rm gp}$ for some $u \in P^{\rm gp}_{\rm tor}$.  But then $u = f-f'$ must actually be in $F^{\rm gp}_{\rm tor} \subseteq P^{\rm gp}_{\rm tor}$, so $f$ and $f'$ have the same image in $F^{\rm tf}$.  In light of \eqref{tflocal}, it is enough to prove the isomorphy statement in \eqref{tflocalization}.  This can be done either by checking that both monoids in question can be described as the submonoid of $P^{\rm gp} / P^{\rm gp}_{\rm tor}$ consisting of elements that can be written as the image of $p-f \in P^{\rm gp}$ for some $p \in P$, $f \in F$, or by checking that both monoids are initial among maps from $P$ to a torsion-free monoid which map $F$ into the units. \end{proof}

\begin{defn} \label{defn:dense} A map of monoids $h : Q \to P$ is called \emph{dense} iff, for every $p \in P$, there is an $n \in \ZZ_{>0}$ and a $q \in Q$ such that $np = h(q)$. \end{defn}

\begin{thm} \label{thm:dense}  {\bf (Gordan's Lemma)}.  Let $h : Q \to P$ be a map of monoids.  \begin{enumerate} \item \label{dense1} If $h$ is finite (Definition~\ref{defn:flat}), then $P^{\rm gp} / Q^{\rm gp}$ is finite.  \item \label{dense2} If $P$ is finitely generated and $h$ is dense, then $h$ is finite. \item \label{dense3} If $Q$ is finitely generated, $P$ is integral, and $h$ is finite, then $P$ is fine and $h$ is dense. \item \label{dense4} If $Q$ is fine, $P$ is integral, $h$ is dense, and $P^{\rm gp}$ is finitely generated, then $h$ is finite. \item \label{dense5} If $Q$ is a fine monoid and $Q \into G$ is any injective monoid homomorphism from $Q$ to a finitely generated abelian group $G$, then the saturation $P$ of $Q$ in $G$ is a finitely generated $Q$-module. \end{enumerate} \end{thm}

\begin{proof} \cite[Theorem~1.6.2]{GM1} \end{proof}

\begin{thm} \label{thm:densespecinvariance} Let $Q \into P$ be a dense, injective map of monoids.  For any face $F \leq Q$, the set \be F' & := & \{ p \in P : np \in F {\rm \; for \; some \;} n \in \ZZ_{>0} \} \ee is a face of $P$.  The map $\Spec P \to \Spec Q$ is a homeomorphism with inverse $F \mapsto F'$.  \end{thm}

\begin{proof} To see that $F' \leq P$ suppose, for a contradiction, that there is some $p \in P$, $p' \in P \setminus F'$ such that $p+p' \in F'$.  Using density and the definition of $F'$, we can find an $n \in \ZZ_{>0}$ such that $np,np' \in Q$ and $n(p+p') \in F$.  Since $p' \notin F'$, $np' \notin F$, so $np+np' \in F$ contradicts $F \leq Q$.

Next we show that $G = (G \cap Q)'$ for any $G \leq P$.  For the containment $\subseteq$, suppose $g \in G$.  Since $Q \into P$ is dense, there is an $n \in \ZZ_{>0}$ so that $ng \in Q$.  Since $ng$ is also in $G$ we have $g \in (G \cap Q)'$ as desired.  For the containment $\supseteq$, suppose $g \in (G \cap Q)'$.  Then there is $n \in \ZZ_{>0}$ so that $ng \in G \cap Q$.  Since $g \in P$, $ng \in G$ and $G \leq P$, we conclude that $g \in G$.

We next show that $F = F' \cap Q$ for any $F \leq Q$.  The containment $\subseteq$ is clear.  For the opposite containment, suppose $f \in F' \cap Q$.  Then there is an $n \in \ZZ_{>0}$ so that $nf \in F$.  But $f \in Q$ and $F \leq Q$, so this implies $f \in F$.

It remains only to show that the continuous bijection $G \mapsto G \cap Q$ is open.  For this it suffices to show that the image of a basic open set $U_p$ of $\Spec P$ is open in $\Spec Q$.  This image is \be V_p & = & \{ G \cap P : p \in G \}. \ee  But $U_{np} = U_p$ for any $n \in \ZZ_{>0}$, so by density we can assume $p \in Q$.  Then we see that \be V_p & = & \{ G \cap P : p \in G \cap Q \} \\ & = & \{ F \leq Q : p \in F \} \ee (using the surjectivity of $G \mapsto G \cap Q$ established above) is just one of the usual basic opens in $\Spec Q$. \end{proof}

\begin{cor} \label{cor:sat} For any integral monoid $P$, let \be P^{\rm sat} & := & \{ p \in P^{\rm gp} : np \in P {\rm \; for \; some \;} n \in \ZZ_{>0} \}. \ee  For a general monoid $P$, we set $P^{\rm sat} := (P^{\rm int})^{\rm sat}$.  For any monoid $P$, the natural map $P \to P^{\rm sat}$ is initial among monoid homomorphisms from $P$ to a saturated monoid.  If $P$ is integral then \begin{enumerate} \item \label{satlocal} the map $P \to P^{\rm sat}$ is injective and local.  \item \label{sathomeo} the map $P \to P^{\rm sat}$ induces a homeomorphism $\Spec P^{\rm sat} \to \Spec P$ with inverse $F \mapsto F'$, where \be F' & := & \{ p \in P^{\rm sat} : np \in F {\rm \; for \; some \;} n \in \ZZ_{>0} \} . \ee \item \label{satlocalization} for every face $F \leq P$, the natural map \be F^{-1}P & \to & (F')^{-1} P^{\rm sat} \ee is local and the induced map \be (F^{-1}P)^{\rm sat} & \to & (F')^{-1} P^{\rm sat} \ee is an isomorphism. \end{enumerate}  The map $P \into P^{\rm sat}$ is finite when $P$ is fine.  \end{cor}

\begin{proof} The universality statement is obvious.  Suppose for the remainder of the proof that $P$ is integral.  Clearly $P \to P^{\rm sat}$ is injective, by definition of $P^{\rm sat}$.  If an element $p \in P$ becomes a unit in $P^{\rm sat}$ then $p+p' = 0$ in $P^{\rm gp}$ for some $p' \in P^{\rm sat} \subseteq P^{\rm gp}$.  Since $np' \in P$ for some $n \in \ZZ_{>0}$, we find that $np$ is a unit in $P$, hence $p$ itself is a unit.  This proves \eqref{satlocal}.  \eqref{sathomeo} follows from Theorem~\ref{thm:sharpening}.  For \eqref{satlocalization}, it is enough to prove the second statement in light of \eqref{satlocal}.  Both monoids are submonoids of $P^{\rm gp}$ so the issue is just to show that the containment $(F^{-1}P)^{\rm sat} \subseteq (F')^{-1} P^{\rm sat}$ is an equality.  An element of $(F')^{-1} P^{\rm sat}$ is an element of $P^{\rm gp}$ which can be written as $p-f'$ for some $p \in P^{\rm sat}$, $f' \in F'$, so there are positive integers $a,b$ such that $ap \in P$, $bf' \in F$.  If we set $n := ab$, then $n(p-f') = b(ap) - a(bf')$ is manifestly in $F^{-1} P$, hence $p-f' \in (F^{-1}P)^{\rm sat}$.  Use Theorem~\ref{thm:dense}\eqref{dense5} for the last statement. \end{proof}

\begin{rem} In the situation of Corollary~\ref{cor:sat}, even with $P$ fine, the face $F'$ of $P^{\rm sat}$ associated to a face $F \leq P$ need not be equal to the saturation $F^{\rm sat}$ of the monoid $F$ because $F' \subseteq P^{\rm gp}$ need not be contained in $F^{\rm gp}$.  On the other hand, $F'$ is clearly contained in the \emph{saturation} of $F^{\rm gp}$ in $P^{\rm gp}$ (the preimage of the torsion subgroup of $P^{\rm gp} / F^{\rm gp} = (P/F)^{\rm gp}$ under the projection from $P^{\rm gp}$).  In general, $F^{\rm gp}$ need not be saturated in $P^{\rm gp}$, but we will see in Corollary~\ref{cor:saturatedlocalization} that $F^{\rm gp}$ \emph{is} a saturated subgroup of $P^{\rm gp}$ whenever $P$ is a saturated monoid and $F \leq P$, so in that case, we \emph{do} have $F'=F^{\rm sat}$ in Corollary~\ref{cor:sat}. \end{rem}

\begin{cor} \label{cor:finiteinvariance} Let $h : Q \to P$ be a finite, injective map of fine monoids.  Then $\Spec h$ is a homeomorphism. \end{cor}

\begin{proof} $h$ is dense by Theorem~\ref{thm:dense}\eqref{dense3}. \end{proof}

\begin{cor} \label{cor:trc} For any monoid $P$, we have $(P^{\rm sat})^{\rm tf} = (P^{\rm tf})^{\rm sat} =: P^{\rm trc}$.  The map $P \to P^{\rm trc}$ is initial among maps from $P$ to a saturated, torsion-free monoid (\S\ref{section:monoids}).  If $P$ is integral then \begin{enumerate} \item \label{trclocal} $P \to P^{\rm trc}$ is local. \item \label{trchomeo} $P \to P^{\rm trc}$ induces a homeomorphism $\Spec P^{\rm trc} \to \Spec P$. \item \label{trclocalization} for any face $F \leq P$, if $G \leq P^{\rm trc}$ is the corresponding face of $P^{\rm trc}$ (cf.\ \eqref{trchomeo}), then the natural map \be F^{-1}P & \to & G^{-1} P^{\rm trc} \ee is local and the induced map \be (F^{-1}P)^{\rm trc} & \to & G^{-1} P^{\rm trc} \ee is an isomorphism. \end{enumerate}  If $P$ is fine, then $P^{\rm trc}$ is toric and $P \to P^{\rm trc}$ is finite. \end{cor}

\begin{proof} For the first statement, use Corollary~\ref{cor:tf} and Corollary~\ref{cor:sat} to check that both monoids have the indicated universal property.  Combine those corollaries for the last statements. \end{proof}

\subsection{Cones} \label{section:cones}  The well-known cone construction which we will review in this section provides a crucial link between the geometry of monoids and convex geometry.  We shall use this construction as an expedient means of establishing some basic facts about the topological space $\Spec P$ for a fine monoid $P$ (Corollary~\ref{cor:SpecPforPfine}).

\begin{defn} \label{defn:cone} Let $\RR_{\geq 0} := \{ \lambda \in \RR : \lambda \geq 0 \}$ be the additive submonoid of $\RR$ consisting of non-negative real numbers, $V$ an $\RR$ vector space.  A subset $\sigma \subseteq V$ is called a \emph{cone} iff and any $\RR_{\geq 0}$-linear combination of elements of $\sigma$ is again in $\sigma$.  (In particular, we require that the ``empty linear combination" $0$ be in $\sigma$.)  Every cone is a monoid under addition.  

For a cone $\sigma$, the \emph{dual cone} $\sigma^\lor$ and \emph{complementary subspace} $\sigma^\perp$ of $\sigma$ are the cone in (resp.\ subspace of) the dual space $V^\lor$ of linear functionals on $V$ defined by \be \sigma^\lor & := & \{ f \in V^\lor : f(v) \geq 0 {\rm \; for \; all \;} v \in \sigma \} \\ \sigma^\perp & := & \{ f \in V^\lor : f(v) = 0 {\rm \; for \; all \;} v \in \sigma \}. \ee  Every subset $S \subseteq V$ is contained in a smallest cone $\sigma(S)$, defined as the set of all $\RR_{\geq 0}$-linear combinations of elements of $S$, or equivalently as the intersection of all cones containing $S$.  A subset $G$ of a cone $\sigma$ is said to \emph{generate} $\sigma$ \emph{(as a cone)} iff $\sigma(G)=\sigma$.  A cone $\sigma$ is called \emph{finitely generated (as a cone)} iff there is a finite subset $G \subseteq \sigma$ which generates $\sigma$ (as a cone).  The \emph{dimension} of a cone $\sigma$, denoted $\dim \sigma$, is the dimension of the $\RR$ vector space $\Span \sigma = \sigma^{\rm gp}$.  The \emph{interior} of a cone $\sigma$, denoted $\sigma^\circ$, is the complement of the proper faces of $\sigma$: \be \sigma^\circ & := & \sigma \setminus \cup_{\tau < \sigma} \tau. \ee

When $V=N \otimes \RR =: N_{\RR}$ for an abelian group $N$ and $\sigma$ is a cone in $V$, the monoid of \emph{integral points} of $\sigma$, denoted $N \sqcap \sigma$, is defined to be the set of $n \in N$ such that the image of $n$ in $V$ is in $\sigma$, so that we have a cartesian diagram of monoids \bne{integralpoints} & \xym{ N \sqcap \sigma \ar[r] \ar[d] & \sigma \ar[d]^{\subseteq} \\ N \ar[r] & N_{\RR} } \ene defining $N \sqcap \sigma$.  In particular $N_{\rm tor} \subseteq N \sqcap \sigma$ and $N \sqcap \sigma = N \cap \sigma$ when $N \to N \otimes \RR$ is injective (i.e.\ when $N$ is torsion-free).  Note that we can define $N \sqcap \sigma$ in the same manner even if $\sigma \subseteq N_{\RR}$ is not a cone---we shall occassionally use this notation.  A cone $\sigma$ in $N_{\RR}$ is called \emph{rational} iff there are $n_1,\dots,n_k \in N \sqcap \sigma$ such that the images of the $n_i$ in $\sigma$ generate $\sigma$ (as a cone).  For a monoid $P$, the \emph{(intrinsic) cone over} $P$, denoted $\sigma(P)$, is the cone in $P^{\rm gp} \otimes \RR$ generated by the image of $P \to P^{\rm gp} \otimes \RR$.  \end{defn}

There is a possible notational conflict here when a monoid $P$ is contained in some abelian group $N$ larger than its groupification: ``$\sigma(P)$" could mean either the intrinsic cone over $P$, or the cone generated by $P$ in $N$.  Generally speaking, if there is clearly some such $N$ to refer to, then $\sigma(P)$ will have the latter meaning, but we will always be clear about the meaning of $\sigma(P)$ in situations where there might be a chance of confusion.

\begin{rem} \label{rem:integralpoints}  Suppose $N$ is a finitely generated abelian group and we choose a splitting of the projection $N \to \ov{N} := N / N_{\rm tor}$, so that $N \cong \ov{N} \oplus N_{\rm tor}$.  Then $N_{\RR} = \ov{N}_{\RR}$ and for any cone $\sigma$ in this vector space, $N \sqcap \sigma$ is identified with $(\ov{N} \cap \sigma) \oplus N_{\rm tor}$.  Since $N_{\rm tor}$ is a finite subgroup of the group of units in $N \sqcap \sigma$, this often allows us to reduce questions about $N \sqcap \sigma$ to the case where $N$ is a lattice. \end{rem}

Here are the basic facts about cones:

\begin{thm} \label{thm:cones} Let $V$ be a finite dimensional $\RR$ vector space, $\sigma$ a finitely generated cone in $V$, $n := \dim \sigma$. \begin{enumerate} \item \label{FarkasTheorem} {\bf (Farkas' Theorem)} The dual cone $\sigma^\lor$ is finitely generated and $\sigma^{\lor \lor} = \sigma$.  In other words, there are linear functionals $f_1,\dots,f_k : V \to \RR$ such that \be \sigma & = & \bigcap_{i=1}^k \{ v \in V : f_i(v) \geq 0 \} \ee and any linear functional $f : V \to \RR$ which is non-negative on $\sigma$ is an $\RR_{\geq 0}$ linear combination of the $f_i$. \item \label{facesofcones} Any face $\tau \leq \sigma$ is itself a finitely generated cone and there is some $f \in \sigma^\lor$ such that $\tau  = \sigma \cap f^\perp$.  For any $\tau \leq \sigma$, the subcone $\sigma^\lor \cap \tau^\perp$ of $\sigma^\lor$ is a face of $\sigma^\lor$ and we have $\tau^\lor$ is the localization of $\sigma^\lor$ at this face: $\tau^\lor = (\sigma^\lor \cap \tau^\perp)^{-1} \sigma^\lor$.  \item \label{finitelymanyfaces} Any cone $\sigma$ has finitely many faces.  \item \label{facesandfacets} Any proper face $\tau < \sigma$ is the intersection of the facets (codimension one faces) of $\sigma$ containing $\tau$.  (In particular such a $\tau$ must be of smaller dimension than $\sigma$.)  Any codimension two face is contained in exactly two facets.  \item \label{conetopology} $\sigma$ is a closed subspace of $V$ and the topological boundary of $\sigma$ in $\Span \sigma$ is the union of the facets of $\sigma$.  \item \label{inclusionreversingbijection} The map $\tau \mapsto (\tau^{-1} \sigma)^{\lor} = \sigma^\lor \cap \tau^\perp$ is an inclusion-reversing bijection from the set of faces of $\sigma$ to the set of faces of $\sigma^\lor$ with inverse $\rho \mapsto (\sigma^\lor + (-\rho))^\lor = (\rho^{-1} \sigma^{\lor})^\lor$. \item \label{sharpcriterion} $\sigma$ is sharp iff $\Span \sigma^\lor = V^\lor$.  \item \label{spanformula} For any $\tau \leq \sigma$, we have $\tau = \sigma \cap \Span \tau$.  \item \label{integralpointininterior} The interior $\sigma^\circ$ of $\sigma$ is non-empty and dense in $\sigma$.  \item \label{catenary} The space $\Spec \sigma$ is finite and catenary of dimension $n$.  In other words, any strictly increasing chain $\tau_1 < \cdots < \tau_m$ of faces of $\sigma$ can be extended to a strictly increasing chain $\{ 0 \} = \tau'_0 < \tau'_1 < \cdots < \tau'_n = \sigma$ which cannot be extended nontrivially.  \item \label{cones:lemma1} If $\tau$ is a finitely generated (as a cone) subcone of $\sigma$ such that $\tau^\circ \cap \sigma^\circ = \emptyset$, then $\tau$ is contained in a proper face of $\sigma$.  \item \label{sumandintersection} If $\sigma'$ is another finitely generated cone, then \be \sigma+\sigma' & := & \{ v+v' : v \in \sigma, v' \in \sigma' \} \ee and $\sigma \cap \sigma'$ are both finitely generated cones and we have \be (\sigma \cap \sigma')^\lor & = & \sigma^\lor + (\sigma')^\lor. \ee  \item \label{intersectionfaces} If $\sigma'$ is another finitely generated cone, then a subset $\rho$ of $\sigma \cap \sigma'$ is a face iff $\rho=\tau \cap \tau'$ for faces $\tau \leq \sigma$, $\tau' \leq \sigma'$.  We have $\sigma^\circ \cap (\sigma')^\circ \subseteq (\sigma \cap \sigma')^\circ$ with equality iff $\sigma^\circ \cap (\sigma')^\circ \neq \emptyset$. \item \label{cones:lemma2} If $x \in \sigma^\circ$ and $y \in \Span \sigma \setminus \sigma$, then \be \{ t \in [0,1] : (1-t)x + ty \in \sigma \} & = & [0,s] \ee for some $s \in (0,1)$ and $(1-s)x+sy$ is in some $\sigma' < \sigma$.  If, furthermore, $x \in N$, $y \in N$, and $\sigma$ is rational, then $s \in (0,1) \cap \QQ$.  \end{enumerate} If $V=N \otimes \RR$ for a finitely generated abelian group $N$ and $\sigma$ is rational, then we can replace ``finitely generated" with ``rational" throughout and the theorem still holds; we also have $N \sqcap \sigma^\circ \neq \emptyset$.  \end{thm}

\begin{proof} Statements \eqref{FarkasTheorem}-\eqref{sharpcriterion} and the fact that $\Spec \sigma$ is finite can be found in \cite[Pages 9-14]{F}.  Note, however, that Fulton takes ``$\tau = \sigma \cap \Ker f$ for some $f \in \sigma^\lor$" as the \emph{definition} of a face.  It is easy to see that every face in Fulton's sense is a face in our sense (Definition~\ref{defn:face}), but we need to check that every face $\tau$ of the monoid $\sigma$ in the sense of Definition~\ref{defn:face} is actually a face in Fulton's sense.  We first check that such a $\tau$ is a cone:  Since $\tau$ is closed under addition, we just need to check that $\lambda v \in \tau$ whenever $v \in \tau$ and $\lambda \in \RR_{\geq 0}$.  Take $n \in \NN$ with $n \geq \lambda$.  Then $\lambda v + (n-\lambda)v = nv$ is in $\tau$, and the two summands are in $\sigma$, so $\lambda v \in \tau$ since $\tau$ is a face.  Next we check that \eqref{spanformula} holds for such a $\tau$---the only issue is the containment $\sigma \cap \Span \tau \subseteq \tau$.  Since $\tau$ is a cone, we have $\tau^{\rm gp} = \Span \tau$, so any element of $\Span \tau$ can be written as $v-w$ with $v,w \in \tau$.  If such an element were in $\sigma \setminus \tau$, then $(v-w)+w = v \in \tau$ would contradict the assumption that $\tau$ is a face of the monoid $\sigma$.  Next we consider the image $\ov{\sigma}$ of $\sigma$ under the projection $\pi : V \to V / \Span \tau =: \ov{V}$.  The cone $\ov{\sigma}$ is sharp:  If it weren't, there would be $v \in \sigma \setminus \Span \tau$, $w \in \sigma$ such that $v+w \in \Span \tau$---but then $v+w$ would be in $\tau$ by the previous thing we proved, contradicting the assumption that $\tau \leq \sigma$.  Since $\sigma \subseteq V$ is a finitely generated cone, $\ov{\sigma}$ is a finitely generated cone in $\ov{V}$ so by \eqref{FarkasTheorem} we can find generators $\ov{f}_1,\dots,\ov{f}_l$ for $\ov{\sigma}^\lor \subseteq \ov{V}^\lor$.  Set $\ov{f} := \ov{f}_1+\cdots+\ov{f}_l$.  For any non-zero $\ov{v} \in \ov{\sigma}$, we must have $f(\ov{v})>0$ (otherwise we'd have $\ov{f}_i(\ov{v})=0$ for every $i$, so $-\ov{v} \in \ov{\sigma} = \ov{\sigma}^{\lor \lor}$, contradicting the fact that $\ov{\sigma}$ is sharp).  Set $f = \ov{f} \pi$, so $f \in \sigma^\lor$.  Again making use of \eqref{spanformula}, we check easily that $\tau = \sigma \cap \Ker f$.   

For \eqref{integralpointininterior}, take $v_1,\dots,v_k \in \sigma$ generating $\sigma$ as a cone (take the $v_i$ in the image of $N \sqcap \sigma \to \sigma$ if $\sigma$ is rational).  Then $v := v_1+\cdots+v_k$ must be in the interior of $\sigma$ because any face of $\sigma$ containing $v$ must contain all the $v_i$ by definition of a face.  To see that $\sigma^\circ$ is dense in $\sigma$, take any $x \in \sigma^\circ$, $y \in \sigma$.  Then $f(t) := (1-t)x+ty$ defines a continuous map $f : [0,1] \to \sigma$ with $f(1)=y$.  On the other hand, we must have $f(t) \in \sigma^\circ$ for $t \in [0,1)$, otherwise some non-zero multiple of $x$ would be in a proper face of $\sigma$ and then $x$ itself would be in that proper face since every face is itself a cone.

The catenary property in \eqref{catenary} follows from \eqref{facesofcones} and \eqref{facesandfacets}.

For \eqref{cones:lemma1}, suppose, toward a contradiction, that there is no $\sigma' < \sigma$ containing $\tau$.  Then for each $\sigma' < \sigma$, we can choose some $v(\sigma') \in \tau \setminus \sigma'$.  By \eqref{integralpointininterior}, there is some $v' \in \tau^\circ$.  Since $\sigma'$ has finitely many faces by \eqref{finitelymanyfaces}, we can consider the element $v := v' + \sum_{\sigma' < \sigma} v(\sigma')$ of $\tau$.  It follows from the defition of a face that $v \in \tau^\circ \cap \sigma^\circ$, contradicting $\tau^\circ \cap \sigma^\circ = \emptyset$.

For \eqref{sumandintersection} we first note that if $v_1,\dots,v_k$ (resp.\ $v_1',\dots,v'_l$) generate $\sigma$ (resp.\ $\sigma'$), then $v_1,\dots,v_l,v_1',\dots,v'_l$ clearly generate $\sigma+\sigma'$, thus we see that a sum of two finitely generated (resp.\ rational) cones is again finitely generated (resp.\ rational).  Since $\sigma=\sigma^{\lor \lor}$ and similarly for $\sigma'$ by \eqref{FarkasTheorem}, we see as a matter of definitions that \bne{dualformula} (\sigma^\lor + (\sigma')^\lor)^\lor & = & \sigma \cap \sigma'. \ene  So the previous thing we proved and \eqref{FarkasTheorem} (applied to $\sigma^\lor + (\sigma')^\lor$) together imply that $\sigma \cap \sigma'$ is finitely generated (and rational when $\sigma$ and $\sigma'$ are rational).  Taking duals in \eqref{dualformula} and again using \eqref{FarkasTheorem} yields the formula in \eqref{sumandintersection}.

For \eqref{intersectionfaces}, first suppose $\rho \leq \sigma \cap \sigma'$.  By \eqref{facesofcones} and the formula for $(\sigma \cap \sigma')^\lor$ in \eqref{sumandintersection}, we have \be \rho & = & \sigma \cap \sigma' \cap \Ker(f+f') \ee for some $f \in \sigma^\lor$, $f' \in (\sigma')^\lor$.  Since $f,f' \in (\sigma \cap \sigma')^\lor$, we have \be \sigma \cap \sigma' \cap \Ker(f+f') & = & \sigma \cap (\Ker f) \cap \sigma' \cap (\Ker f') \ee so $\tau := \sigma \cap \Ker f$ and $\tau' := \sigma' \cap \Ker f'$ are as desired.  It is obvious that $\tau \cap \tau' \leq \sigma \cap \sigma'$ when $\tau \leq \sigma$ and $\tau' \leq \sigma'$.  The fact that $\sigma^\circ \cap (\sigma')^\circ \subseteq (\sigma \cap \sigma')^\circ$ is clear from the description of the faces of $\sigma \cap \sigma'$ just established.  To see that this containment is an equality when $\sigma^\circ \cap (\sigma')^\circ \neq \emptyset$, suppose, toward a contradiction, that there is some $v \in (\sigma \cap \sigma')^\circ$ not in, say, $\sigma^\circ$.  Then $v \in \tau$ for some $\tau < \sigma$ and $\tau \cap \sigma' \leq \sigma \cap \sigma'$.  But we can't have $\tau \cap \sigma' < \sigma \cap \sigma'$ because $v \in (\sigma \cap \sigma')^\circ$, so we must have $\tau \cap \sigma' = \sigma \cap \sigma'$.  In particular $(\sigma')^\circ$ cannot intersect the interior of $\sigma$, contradicting $\sigma^\circ \cap (\sigma')^\circ \neq \emptyset$.

For \eqref{cones:lemma2} we can assume, after possibly replacing $V$ with $\Span \sigma$, that $V=\Span \sigma$.  Let  $f_1,\dots,f_k$ be generators for $\sigma^\lor$ as in \eqref{FarkasTheorem} (we can and do take the $f_i$ in $\Hom(N,\ZZ) \subseteq N_{\RR}^\lor$ when $\sigma$ is rational).  We can assume that no $f_i$ is zero, hence $\sigma \cap f_i^\perp < \sigma$ for every $i$ on dimension grounds.  Since $x \in \sigma^\circ$, it follows that $f_i(x) > 0$ for every $i$.  Define \be F_i(t) & := & f_i((1-t)x+ty) \\ & = & (1-t)f_i(x) + t f_i(y). \ee  The set \be A & := & \{ i \in \{ 1, \dots, k \} : f_i(y) < 0 \} \ee is non-empty since $\sigma = \sigma^{\lor \lor}$ and $y \notin \sigma$.  For $i \notin A$ we clearly have $F_i(t) \geq 0$ for all $t \in [0,1]$.  For $i \in A$, $F_i$ is a strictly decreasing continuous function $F_i : [0,1] \to \RR$ with $F_i(0) > 0$ and $F_i(1) < 0$, and $F_i(t_i)=0$ for \bne{ti} t_i & = & \frac{ f_i(x) }{ f_i(x)-f_i(y) } \in (0,1), \ene so $F_i^{-1}(\RR_{\geq 0}) = [0,t_i]$.  Using $\sigma = \sigma^{\lor \lor}$, we see that $s := \min \{ t_i : i \in A \}$ will be as desired.  We have $s = t_i$ for some $i \in A$ and $(1-s)x+sy \in \sigma \cap f_i^\perp$ for any such $i$.  Note that each $f_i(x)$, $f_i(y)$ is an integer when $x,y \in N$ and $\sigma$ is rational, hence each $t_i$ ($i \in A$) is rational by \eqref{ti} and hence $s$ is rational. \end{proof}

Our next result summarizes some facts about the relationship between monoids and cones.

\begin{thm} \label{thm:monoidsandcones} Let $P$ be an integral monoid, $N$ an abelian group, $N_{\RR} := N \otimes \RR$. \begin{enumerate} \item \label{monoidsandcones0} If $P$ is fine, $\sigma(P)$ is rational. \label{monoidsandcones1} The map $P \to \sigma(P)$ induces a homeomorphism $\Spec \sigma(P) \to \Spec P$ with inverse $F \mapsto \sigma(F)$.  We have $P^{\rm gp} \sqcap \sigma(P) = P^{\rm sat}$. \item \label{monoidsandcones2} For any cone $\sigma$ in $N_{\RR}$, the monoid $N \sqcap \sigma$ of integral points of $\sigma$ is integral and saturated.  If $N$ is finitely generated and $\sigma$ is rational, then $N \sqcap \sigma$ is fs and $\sigma = \sigma(N \sqcap \sigma)$. \item \label{monoidsandcones3} If $N$ is finitely generated, $\sigma$ is a rational cone in $N_{\RR}$, and $\tau \leq \sigma$, then ``taking integral points commutes with localization" in the sense that \be N \sqcap (\tau^{-1} \sigma) & = & (N \sqcap \tau)^{-1} (N \sqcap \sigma). \ee \end{enumerate} \end{thm}

\begin{proof} \eqref{monoidsandcones0}:  Indeed, any set of generators for $P$ will also generate $\sigma(P)$ as a cone.

\eqref{monoidsandcones1}:  Any integral monoid $P$ is the filtered direct limit of its fine (i.e.\ finitely generated) submonoids.  The constructions $P \mapsto P^{\rm sat}$, $P \mapsto \sigma(P)$, and $P \mapsto P^{\rm gp} \sqcap \sigma(P)$ clearly commute with filtered direct limits, as does $P \mapsto \Spec P$ (Theorem~\ref{thm:Specinverselimits}).  A filtered direct limit of faces is a face.  We thus reduce to the case where $P$ is fine, which we now assume.

The map $P^{\rm gp} \to P^{\rm gp} \otimes \RR$ factors as $$P^{\rm gp} \to P^{\rm gp} / P^{\rm gp}_{\rm tor} \to P^{\rm gp} \otimes \RR,$$ where the first map is surjective and the second is injective.  Hence $P \to P^{\rm gp} \otimes \RR$ factors as $$P \to P^{\rm tf} \to P^{\rm gp} \otimes \RR,$$ where the first map is the surjection $P \to P^{\rm tf}$ of Corollary~\ref{cor:tf} and the second map is injective.  In particular we have $\sigma(P) = \sigma(P^{\rm tf})$ and $P \to \sigma(P)$ factors as the composition of $P \to P^{\rm tf}$ and $P^{\rm tf} \to \sigma(P^{\rm tf})$.  The map $\Spec P^{\rm tf} \to \Spec P$ is a homeomorphism by Corollary~\ref{cor:tf}, so we reduce to proving that $\Spec$ of $P^{\rm tf} \to \sigma(P^{\rm tf})$ is a homeomorphism.

This reduces us to the case where $P$ is fine \emph{and} $P^{\rm gp}$ is torsion-free, so that $P \subseteq \sigma(P)$.  Set $V := P^{\rm gp} \otimes \RR$.  Let $F$ be a face of $P$.  By Lemma~\ref{lem:duality}, there is a monoid homomorphism $h : P \to \NN$ such that $F = h^{-1}(0)$.  The linear functional $u := h^{\rm gp} \otimes \RR : V \to \RR$ is non-negative on $\sigma(P)$, zero on $\sigma(F)$, and strictly positive on any $p \in P \setminus F \subseteq V$.   It follows that \bne{coneformula} \sigma(F) & = & \sigma(P) \cap \Ker u, \ene that $\sigma(F)$ is a face of $\sigma(P)$, and that $\sigma(F) \cap P = F$.  It remains to show that $\tau = \sigma(\tau \cap P)$ for every $\tau \leq \sigma(P)$.  This follows from the fact that any such $\tau$ is a rational cone (Theorem~\ref{thm:cones}\eqref{facesofcones}).

Next we check that $P' := P^{\rm gp} \sqcap \sigma(P)$ is equal to $P^{\rm sat}$ when $P$ is fine.  The containment $P' \supseteq P^{\rm sat}$ is clear, so $(P')^{\rm gp} = P^{\rm gp}$.  To see that $P' \subseteq P^{\rm sat}$, we need to show that the inclusion $P \subseteq P'$ is dense in the sense of Definition~\ref{defn:finite}.  By Theorem~\ref{thm:dense}\eqref{dense3} it suffices to show that the inclusion $P \subseteq P'$ is finite, which can be seen by the usual ``Gordan's Lemma argument" as follows:  Let $p_1,\dots,p_n$ be generators for $P$.  Then the set $S$ of all $p \in P^{\rm gp}$ such that \be p & = & \lambda_1p_1+\cdots+\lambda_np_n \ee in $P^{\rm gp} \otimes \RR$ for some $\lambda_1,\dots,\lambda_n \in [0,1] \subseteq \RR$ is finite (because $P^{\rm gp}$ is finitely generated and $[0,1]^n$ is compact) and generates $P'$ as an $P$ module because, for any $p' \in P'$, we can write \be p' & = & \lambda_1p_1+\cdots+\lambda_np_n \ee in $P^{\rm gp} \otimes \RR$ for some $\lambda_1,\dots,\lambda_n \in \RR_{\geq 0}$ and, if we take $a_i \in \NN$ maximal with respect to $a_i \leq \lambda_i$, then we have \be p'  & = & (a_1p_1+\cdots+a_np_n)+((\lambda_1-a_1)p_1+\cdots+(\lambda_n-a_n)p_n) \ee with the first summand in $P$ and the second in $S$.

The first statement in \eqref{monoidsandcones2} is obvious and the second statement follows from the Gordan's Lemma argument described above.

For \eqref{monoidsandcones3}, the only difficulty is the containment \bne{trickycontainment} N \sqcap ( \tau^{-1} \sigma) & \subseteq & (N \sqcap \tau)^{-1} (N \sqcap \sigma). \ene  We can easily reduce to the case where $N$ is a lattice (which we now assume) by the method discussed in Remark~\ref{rem:integralpoints}.  (The point is that both sides clearly contain $N_{\rm tor}$, so it is enough to prove the containment modulo $N_{\rm tor}$.)  By Theorem~\ref{thm:cones}\eqref{inclusionreversingbijection}, $(\tau^{-1} \sigma)^\lor$ is a face of $\sigma^\lor$.  Now we apply \cite[Proposition~2, Page 13]{F} with $\tau \leq \sigma$ there given by our $(\tau^{-1} \sigma)^\lor \leq \sigma^\lor$ and use the duality in Theorem~\ref{thm:cones}\eqref{FarkasTheorem} to see that there is some $u \in N \cap \sigma = N \cap \sigma^{\lor \lor}$ so that \bne{Fulton1} (\tau^{-1} \sigma)^\lor & = & \sigma^\lor \cap u_{\RR}^\perp \\ \label{Fulton2} N \cap (\tau^{-1} \sigma) & = & ( N \cap \sigma) + \NN(-u). \ene  Formula \eqref{Fulton1} implies that $u \in N \cap \tau$, so \eqref{Fulton2} implies the containment \eqref{trickycontainment}. \end{proof}

\begin{cor} \label{cor:saturatedlocalization} Let $P$ be a saturated monoid, $F$ a face of $P$.  Then $F$, $P/F$, and $F^{-1}P$ are saturated and $P^{\rm gp} / F^{\rm gp}$ is torsion-free. \end{cor}

\begin{proof} The fact that $F$ is saturated is an easy exercise with the definitions.  The fact that $P/F$ is saturated is a slightly harder exercise with the definitions (see \cite[Lemma~1.1.3]{GM1}).  Since $P/F$ is sharp and saturated, $(P/F)^{\rm gp} = P^{\rm gp} / F^{\rm gp}$ is torsion-free.

The tricky part is to show that $F^{-1}P$ is saturated.

Using Theorem~\ref{thm:dense}\eqref{dense5}, we see that $P$ is the filtered direct limit of the monoids $Q$, where $Q$ runs over the fs submonoids of $P$.  Similarly, the face $F$ is easily seen to be the filtered direct limit of the faces $F \cap Q$ and $F^{-1} P$ is the filtered direct limit of the $(F \cap Q)^{-1} Q$.  Since it is easy to see that a filtered direct limit of saturated monoids is saturated, this reduces us to the case where $P$ is fs.  In this case, we see that \be P & = & P^{\rm gp} \sqcap \sigma(P) \\ F & = & F^{\rm gp} \sqcap \sigma(F) \ee by using Theorem~\ref{thm:monoidsandcones}\eqref{monoidsandcones1} and the fact that $P$ and $F$ are saturated.  Since $P^{\rm gp} / F^{\rm gp}$ is free we can see easily that \be F & = & P^{\rm gp} \sqcap \sigma(F) \ee by choosing a splitting $P^{\rm gp} \cong F^{\rm gp} \oplus P^{\rm gp}/F^{\rm gp}$.

By construction, the monoid $P^{\rm gp} \sqcap \sigma(F)^{-1} \sigma(P)$ is clearly saturated, but on the other hand, by Theorem~\ref{thm:monoidsandcones}\eqref{monoidsandcones3}, we have \be P^{\rm gp} \sqcap \sigma(F)^{-1} \sigma(P) & = & (P^{\rm gp} \sqcap \sigma(F))^{-1}(P^{\rm gp} \cap \sigma(P)) \\ & = & F^{-1} P. \ee  \end{proof}

\begin{cor} \label{cor:SpecPforPfine} Let $P$ be a fine monoid.  Then: \begin{enumerate} \item $\Spec P$ is a finite, sober, catenary topological space of dimension $n = \rk P^{\rm gp}$. \item Any proper face $F < P$ is of dimension $\dim F := \rk F^{\rm gp}$ strictly less than $n$.  \item Any codimension two face of $\Spec P$ is contained in precisely two facets of $\Spec P$. \item Any face of $P$ is the intersection of the facets containing it. \end{enumerate} \end{cor}

\begin{proof}  By Theorem~\ref{thm:monoidsandcones}, the cone $\sigma(P) \subseteq P^{\rm gp} \otimes \RR$ is rational and $\Spec \sigma(P) \to \Spec P$ is a homeomorphism.  The assertions then follow from corresponding properties of the rational cone $\sigma(P)$ in Theorem~\ref{thm:cones}. \end{proof}

\subsection{Spec pathologies}  In this section we give some examples of unusual behaviour exhibited by $\Spec P$ when $P$ is not fine.

\begin{example} \label{example:SpecPnonintegral} The paradigm example of a non-integral monoid is the finitely generated monoid \be P & := & \langle a,b : a+b = a \rangle. \ee  The map $P \to P^{\rm int}$ is the map $P \to \NN$ given by $b \mapsto 0$, $a \mapsto 1$.  The faces of $P$ are $0$, $\NN b$, and $P$ itself.  In particular, $\NN b$ is a codimension zero face of $P$.  We will discuss $P$ further in Example~\ref{example:stratification}. \end{example}

\begin{example} \label{example:SpecPinfinite} Let $P_n$ be the submonoid of $\NN^2$ generated by the first $n+1$ elements in the list $$(0,1), (1,1), (2,1), \dots $$ and let $P$ be the union (filtered direct limit) of all the $P_n$.  This $P$ is an integral monoid with $P^{\rm gp} = \ZZ^2$.  We can alternatively describe $P$ as: \be P & = & \{ (a,b) \in \NN^2 : (a,b)=(0,0) {\rm \; or \;} b>0 \} . \ee  The faces of $\Spec P$ are $0$, $0 \oplus \NN$, and $P$ itself.  The intrinsic cone of $P$ is \be \sigma(P) & = & \{ (x,y) \in \RR^2_{\geq 0} : (x,y)=(0,0) {\rm \; or \;} y>0 \} . \ee  Note that the codimension two face $0$ of $P$ is contained in only one facet.  Each $P_n$ is fine and has two facets, but we can see that one of these facets is ``lost" in passing to the direct limit $P$.  Similarly, the integral monoid \be Q & := \{ (a,b) \in \ZZ \oplus \NN : (a,b)=(0,0) {\rm \; or \;} b>0 \} \ee has only the obvious faces $0$, $Q$.  \end{example}

\begin{example} \label{example:SpecPpathology} Let $I$ be a non-empty set and let $P := \prod_I \NN = \NN^I = \{ p : I \to \NN \}$ be the monoid of all functions from $I$ to $\NN$, under pointwise addition.  For $p \in P$, let \be K(p) & := & \{ i \in I : p(i)=0 \} \ee denote the ``kernel" of $p$.  Clearly we have \bne{kernelformula} K(0) & = & I \\ \nonumber K(p+p') & = & K(p) \cap K(p') \ene for all $p,p' \in P$.  For $A \subseteq I$, let $p(A) \in P$ be the characteristic function of $A$, so $K(p(A))=I \setminus A$.

Recall that a \emph{filter} on $I$ is a family $\F$ of subsets of $I$ satisfying the properties: \begin{enumerate} \item $\emptyset \notin \F$, $I \in \F$ \item $A,B \in \F$ implies $A \cap B \in \F$ \item If $A \in \F$ and $A \subseteq B \subseteq I$, then $B \in \F$. \end{enumerate}  An \emph{ultrafilter} is a maximal filter, or, equivalently, a filter $\F$ satisfying the additional property: \begin{enumerate} \setcounter{enumi}{3} \item Whenever $I = A \coprod B$, either $A$ or $B$ is in $\F$. \end{enumerate}  Let $\alpha(I)$ (resp.\ $\beta(I)$) denote the set of filters (resp.\ ultrafilters) on $I$.  

For $\F \in \alpha(I)$, it follows from \eqref{kernelformula} that \bne{Fformula} F(\F) & := & \{ p \in P : K(p) \in \F \} \ene is a face of $P$.  We have $F(\F) \leq F(\F')$ iff $\F \subseteq \F'$.  Set \be S & := & \{ p \in P : K(p) = \emptyset \} \\ Z & := & \{ F \leq P : F \cap S = \emptyset \} \ee so that $Z$ is a closed subspace of $\Spec P$ (Proposition~\ref{prop:Spec}\eqref{opensubsets}).  Given $F \in Z$, let \be \F(F) & := & \{ A \subseteq I : \exists p \in F, \; A \supseteq K(p) \}. \ee  Then $\F(F)$ is a filter on $I$ and we have $F \leq F(\F(F))$.  It follows that $Z = \cup_{\F \in \beta(I)} Z_{F(\F)}$ is the union of the irreducible closed subsets $Z_{F(\F)} = \{ F(\F) \}^{-}$ (cf.\ Proposition~\ref{prop:Spec}\eqref{irreducibles}), over all ultrafilters $\F$.  The faces $F(\F)$ for $\F \in \beta(I)$ are precisely the maximal points of $Z$ (in the ordering by inclusion of faces).  In particular, if $I$ is infinite, $Z$ will have infinitely many maximal points.  It seems to be rather pathological for $\Spec P$ to have a closed subset with infinitely many maximal points.  

We claim that $\F(F(\F)) = \F$ for any $\F \in \alpha(I)$.  The containment $\F(F(\F)) \subseteq \F$ follows from the definitions.  Given $A \in \F$, we have $p(I \setminus A) \in F(\F)$ and $K(p(I \setminus A)) = A$, which implies $A \in \F(F(\F))$, thus we establish the opposite containment.

Topologize $\beta(I)$ by viewing it as the Stone-\v{C}ech compactification of the discrete space $I$.  The subsets \be W(A) & := & \{ \F \in \beta(I) : A \in \F \} \quad \quad (A \subseteq I) \ee are closed subsets of $\beta(I)$ forming a basis for the closed sets in the topology of $\beta(I)$ and satisfy: \bne{WAformula} W(\emptyset) & = & \emptyset \\ \nonumber W(I) & = & \beta(I) \\ \nonumber W(A \cup B) & = & W(A) \cup W(B) \\ \nonumber W(A) & = & \beta(I) \setminus W(I \setminus A). \ene (The last two formulas use the fact that the $\F \in \beta(I)$ are \emph{ultra}filters.)  In particular, the $W(A)$ are also open and form a basis for the open subsets of $\beta(I)$.  The formula \eqref{Fformula} defines a function \bne{Fmap} F : \beta(I) & \to & Z \subseteq \Spec P. \ene  The function \eqref{Fmap} is one-to-one because we noted that $\F(F(\F))=\F$ above.  It is continuous because, for any $p \in P$, \be F^{-1}(U_p \cap Z) & = & \{ \F \in \beta(I) : F(\F) \in U_p \} \\ & = & \{ \F \in \beta(I) : p \in F(\F) \} \\ & = & \{ \F \in \beta(I) : K(p) \in \F \} \\ & = & W(K(p)) \\ & = & \beta(I) \setminus W(I \setminus K(p)) \ee is open in $\beta(I)$.  In particular $$F^{-1}(U_{p(A)} \cap Z) = W(K(p(A))) = W( I \setminus A),$$ so each basic open subset $W(A)$ of $\beta(I)$ is the preimage of an open subset of $Z$ under $F$.  We conclude that \eqref{Fmap} is an embedding.  Note that it is not a \emph{closed} embedding, even though $\beta(I)$ is compact Hausdorff.  \end{example}

\section{Fans} \label{section:fans} Having dispensed with the necessary ``commutative algebra" in \S\ref{section:monoidsandspec} we can now define our main objects of study!

\begin{defn} \label{defn:fan} A locally monoidal space (\S\ref{section:locallymonoidalspaces}) $X$ isomorphic to $\Spec P$ (\S\ref{section:Spec}) for some monoid (resp.\ integral monoid, finitely generated monoid, fine monoid, \dots) $P$ is called an \emph{affine fan} (resp.\ \emph{integral affine fan}, \emph{finite type affine fan}, \emph{fine affine fan}, \dots).  A locally monoidal space $X$ with an open cover $\{ U_i \}$ such that each $(U_i,\M_X|U_i)$ is an affine fan (resp.\ integral affine fan, finite type affine fan, fine affine fan, \dots) is called a \emph{fan} (resp.\ \emph{integral fan}, \emph{locally finite type fan}, \emph{fine fan}, \dots).  A morphism of fans is defined to be a morphism of locally monoidal spaces, so that fans form a full subcategory $\Fans \subseteq \LMS$.  \end{defn}

\begin{prop} \label{prop:Specequivalence}  The functor \eqref{SpecLMS} yields an equivalence of categories between $\Mon^{\rm op}$ and the full subcategory of $\LMS$ consisting of affine fans.  \end{prop}

\begin{proof} The functor in question is fully faithful by Proposition~\ref{prop:Specfullyfaithful} and essentially surjective by definition of an affine fan. \end{proof}

\begin{prop} If $X=(X,\M_X)$ is a locally monoidal space and $\{ U_i \}$ is an open cover of its underlying space, then $X$ is a fan iff each locally monoidal space $(U_i,\M_X|U_i)$ is a fan. \end{prop}

\begin{proof} If each $(U_i,\M_X|U_i)$ can be covered by affine fans, then so can $X$, so one implication is clear.  For the other implication, we just need to prove that if $U$ is an open subspace of (the space underlying) a fan $(X,\M_X)$, then $(U,\M_X|U)$ is a fan.  In light of what we just proved, the question is local, so we can assume $X = \Spec P$ is an affine fan.  Then $U$ can be covered by basic opens $U_p$ and each $(U_p,\M_X|U_p)$ is an affine fan because it is isomorphic to $\Spec P[-p]$ by Proposition~\ref{prop:Spec}\eqref{basicopens}. \end{proof}

\begin{example} As in the case of schemes, if $X$ is an affine fan and $U \subseteq X$ is an open subspace, then the fan $(U,\M_X|U)$ need not be affine---one can even use essentially the same example as in schemes:  Take $P := \NN^2$, $X := \Spec P$, so the points of $X$ are the four faces $P^* = \{ (0,0) \}$, $F_1 := \NN \times \{ 0 \}$, $F_2 := \{ 0 \} \times \NN$, and $P$ of $P$.  Let $U$ be the complement of the unique closed point $P^*$ of $X$ (cf.\ Proposition~\ref{prop:Spec}\eqref{uniqueclosedpoint}).  The fan $(U,\M_X|U)$ cannot be affine because its topological space has \emph{two} closed points: $F_1$ and $F_2$.  For schemes one argues that the complement $U$ of the origin in $X = \AA^2$ can't be affine because the restriction map $\O_X(X) \to \O_X(U)$ is an isomorphism (by Hartog's Theorem, or direct calculation), so if $U$ were affine it would have to be all of $\AA^2$.  One could also make the same argument in our fans example:  Since the basic opens $U_{F_1}$ of $U_{F_2}$ (cf.\ Proposition~\ref{prop:Specfinite}\eqref{UF}) of $X$ cover $U$ and have intersection $U_P$, the monoid $\M_X(U)$ is the equalizer of $$ \M_X(U_{F_1}) \times \M_X(U_{F_2}) \rightrightarrows \M_X(U_P). $$  Since $U_F$ is the smallest open containing the point $F$, this diagram is the same as $$ \M_{X,F_1} \times \M_{X,F_2} \rightrightarrows \M_{X,P} $$ which, by the stalk formula of Proposition~\ref{prop:Spec}\eqref{stalkMP}, is the same as $$ F_1^{-1} P \times F_2^{-1} P \rightrightarrows P^{-1}P, $$ which is the diagram $$ (\ZZ \times \NN) \times (\NN \times \ZZ) \rightrightarrows \ZZ \times \ZZ ,$$ which has equalizer $\NN^2 = \M_X(X)$. \end{example}

Here are some basic facts about fans:

\begin{prop} \label{prop:fans} Let $X$ be a fan. \begin{enumerate} \item \label{fansaresober}  The topological space underlying $X$ is sober (any non-empty irreducible closed subset is the closure of a unique point).  \item \label{fangeneralization} When $x$ and $y$ are points of $X$ with $x \in \{ y \}^-$, the generalization map $\M_{X,x} \to \M_{X,y}$ is identified with the localization of $\M_{X,x}$ at the face $F_y \leq \M_{X,x}$ given by the preimage of $\M_{X,y}^*$.  If we fix $x$, then the assignment $y \mapsto F_y$ defines a bijection from the set of generalizations of $x$ in $X$ to the set $\Spec \M_{X,x}$ of faces of $\M_{X,x}$.  In particular, if $\M_{X,x}$ is a group, then $x$ itself is the only generalization of $x$ in $X$.  \item \label{locallyquasicompact} Every point $x \in X$ has a neighborhood $U$ in $X$ containing a unique closed point and a unique point with no (non-trivial) generalizations.  This $U$ can be chosen so that $U$ itself is the only open subset of $U$ containing the closed point of $U$.  Such a $U$ is quasi-compact.  \end{enumerate} \end{prop}

\begin{proof} These things can all be checked locally, so they follow from Proposition~\ref{prop:Spec}. \end{proof}

\subsection{Affine fans}  Here we make a few simple observations about affine fans.  Some results in this section are noteworthy because the analogous statements for rings (like the analog of Proposition~\ref{prop:affineopensubfan}\eqref{aos3}) are false.

\begin{prop} \label{prop:affineopensubfan} Let $P$ be a monoid.  \begin{enumerate} \item \label{aos0} For any $F \in \Spec P$, we have \be \{ G \in \Spec P : F \leq G \} & = & \{ G \in \Spec P : F \in \{ G \}^- \} \\ & = & \Im( \Spec F^{-1}P \to \Spec P). \ee  Write $U_F$ for this subset of $\Spec P$.  (It is not generally open as the misleading notation might suggest.)  \item \label{aos1} Suppose $U \subseteq \Spec P$ is an open subset of $\Spec P$ for which the fan $(U,\M_P|U)$ is affine.  Then there is a unique $F \in \Spec P$ such that $U=U_F$ (in particular, this particular $U_F$ \emph{will} be open).  Furthermore, there is a natural isomorphism \be (U,\M_P|U) & = & \Spec F^{-1}P \ee of (affine) fans over $\Spec P$.  Said another way:  If $h : Q \to P$ is a map of monoids such that $\Spec h : \Spec P \to \Spec Q$ is an open embedding of locally monoidal spaces (Definition~\ref{defn:embedding}), then $h$ identifies $P$ with the localization of $Q$ at $h^{-1}(P^*)$.   \item \label{aos2} If $\Spec P$ is finite, then for an open subset $U \subseteq \Spec P$, the fan $(U,\M_P|U)$ is affine iff $U = U_F$ for some $F \in \Spec P$.  \item \label{aos3} If $P$ is finitely generated, then for an open set $U \subseteq \Spec P$, the fan $(U,\M_P|U)$ is affine iff $U$ is equal to a basic open set $U_p$ for some $p \in P$. \end{enumerate} \end{prop}

\begin{proof} \eqref{aos0}: The two equalities are immediate from \eqref{Specclosure} and \eqref{localizationembedding} in Proposition~\ref{prop:Spec}. 

\eqref{aos1}: Suppose $U \subseteq \Spec P$ is affine.  Then by Proposition~\ref{prop:Spec}\eqref{uniqueclosedpoint} it has a unique closed point $F \in U \subseteq \Spec P$ and if $G \in U$ is any other point of $U$, then $F$ is in the closure of $\{ G \}$ in $U$.  Since $U$ is open in $\Spec P$, it is, in particular, stable under generalization in $\Spec P$, so $U_F \subseteq U$.  On the other hand, if $G \in U$, then $F$ is in the closure of $\{ G \}$ in $U$, hence $F$ is also in the closure of $\{ G \}$ in $\Spec P$ (here we don't even need that $U \subseteq \Spec P$ is \emph{open}, just that it is a subspace), so $U \subseteq U_F$.  If $U_F = U_G$ then we have $F \leq G$ and $G \leq F$, hence $F=G$.  The ``Furthermore" follows from Proposition~\ref{prop:Spec}\eqref{localizationembedding} (with $S$ there equal to our $F$) because we already know $U=U_F$ is open.

\eqref{aos2}: In light of the previous part, we just need to prove that each $U_F$ is open in $\Spec P$---here we use Proposition~\ref{prop:Specfinite}\eqref{UF}---and that $(U_F,\M_P|U_F)$ is affine, which follows from Proposition~\ref{prop:Spec}\eqref{localizationembedding} (take $S=F$ there).

\eqref{aos3}: The fans $(U_p,\M_P|U_p)$ are affine by Proposition~\ref{prop:Spec}\eqref{basicopens}, so, in light of the previous part, we just need to show that each $U_F$ is equal to some $U_p$, which we noted in Remark~\ref{rem:Specfinite}. \end{proof}

\subsection{Integration, saturation, etc.} \label{section:integration}  In this section we will clarify the relationships between the types of fans (locally finite type fans, fine fans, \dots) defined in Definition~\ref{defn:fan}.  We will also show that the inclusions of various full subcategories of $\Fans$ into various other such full subcategories have right adjoints.

\begin{lem} \label{lem:integralfan} For a fan $X$ the following are equivalent: \begin{enumerate} \item \label{integral1} $\M_{X,x}$ is integral for every $x \in X$. \item \label{integral2} $\M_X(U)$ is integral for every open subset $U$ of $X$. \item \label{integral3} $X$ is an integral fan (Definition~\ref{defn:fan}). \end{enumerate} \end{lem}

\begin{proof} To see that \eqref{integral1} implies \eqref{integral2}, suppose $p+q=p'+q$ for some $p,p',q \in \M_X(U)$.  Then we have the same equality in each stalk $\M_{X,x}$, $x \in U$, hence $p=p'$ in each such $\M_{X,x}$ because each such $\M_{X,x}$ is integral, hence $p=p'$ since they have the same image in each stalk and $\M_X$ is a sheaf.  To see that \eqref{integral2} implies \eqref{integral1}, one need only check that a filtered direct limit of integral monoids is integral, which is easy.  To see that \eqref{integral3} implies \eqref{integral1}, suppose $X$ is integral and $x \in X$.  Then we can find a neighborhood $U$ of $x$ in $X$ such that $U \cong \Spec P$ for some integral monoid $P$, hence $\M_{X,x} \cong \M_{P,F}$ for some $F \in \Spec P$.  We have $\M_{P,F} = F^{-1}P$ by Proposition~\ref{prop:Spec}\eqref{stalkMP}, and any localization of an integral monoid is integral, so \eqref{integral1} holds.  To see that \eqref{integral2} implies \eqref{integral3}, consider any $x \in X$.  Since $X$ is a fan, $x$ has an affine neighborhood $U \cong \Spec P$ in $X$.  Then $\M_X(U) \cong \M_P(\Spec P)$ and $\M_P(\Spec P)=P$ (cf.\ Remark~\ref{rem:SpecP}), so $P$ is integral because $\M_X(U)$ is integral by \eqref{integral2}. \end{proof}

\begin{rem} \label{rem:integralLMS} The proof of Lemma~\ref{lem:integralfan} shows that the conditions \eqref{integral1} and \eqref{integral2} in that lemma are equivalent for any monoidal space $X$.  In the rest of \S\ref{section:integration} it will be useful to have a notion of \emph{integral monoidal space}, which we define to be a monoidal space satisfying those two equivalent conditions. \end{rem}

\begin{defn} \label{defn:locallyfinite} A topological space $X$ is called \emph{locally finite} iff every $x \in X$ has a finite open neighborhood (equivalently, $X$ has an open cover by finite spaces). \end{defn}

\begin{rem} \label{rem:locallyfinite} Any point $x$ of a locally finite topological space $X$ has a smallest open neighborhood \be U_x & = & \{ y \in X : x \in \ov{ \{ y \} } \} \ee given by the set of generalizations of $x$.  A subset $U$ of a locally finite space $X$ is open iff it is stable under generalization.  Suppose that $X$ is a locally finite space satisfying the condition:  $x \in \ov{ \{ y \} }$ and $y \in \ov{ \{ x \} }$ iff $x=y$.  (This latter condition holds for any sober space, and is hence satisfied by the space underlying any fan by Proposition~\ref{prop:fans}\eqref{fansaresober}.)  Then $\{ x \}$ is a closed subset of $U_x$ for every $x \in X$; in particular $\{ x \}$ is locally closed in $X$. \end{rem}

\begin{prop} \label{prop:locallyfinitetypefans} A fan $X$ is locally finite type (resp.\ fine, fs, toric) iff its underlying space is locally finite and the monoid $\M_{X,x}$ is finitely generated (resp.\ fine, \dots) for each $x \in X$.  If $x$ is any point of a fan $X$ with locally finite underlying topological space, then the smallest open neighborhood $U_x$ of $x$ in $X$ is an affine fan naturally isomorphic to $\Spec \M_{X,x}$. \end{prop}

\begin{proof} Suppose $X$ is locally finite type (resp.\ \dots).  Then, by definition, every point $x \in X$ has a neighborhood in $X$ isomorphic to $\Spec P$ for some finitely generated (resp.\ \dots) monoid $P$.  In particular, this neighborhood is \emph{homeomorphic} to $\Spec P$ and hence is finite by Corollary~\ref{cor:faces}.  If $F \leq P$ is the face of $P$ corresponding to $x$ under this isomorphism, then $\M_{X,x}$ is isomorphic to $\M_{P,F}$, which is equal to $F^{-1}P$ by Proposition~\ref{prop:Spec}\eqref{stalkMP}.  This monoid is finitely generated (resp.\ fine) by Corollary~\ref{cor:faces} (resp.\ Corollary~\ref{cor:faces}, Corollary~\ref{cor:tf}, Corollary~\ref{cor:sat}, Corollary~\ref{cor:trc}).  

For any fan $X$ with locally finite topological space, we have $\M_X(U_x) = \M_{X,x}$ (by definition of $U_x$ and stalks) and hence we obtain a natural map $U_x \to \Spec \M_{X,x}$ from the universal property of $\Spec$ (\S\ref{section:Spec}).  The question of whether this map is an isomorphism (and indeed, the entire construction of this map) is local on $X$, so we reduce to the case where $X$ is an affine fan, in which case the desired result is Proposition~\ref{prop:Specfinite}\eqref{facelocalizationembedding}.  Obviously this fact implies the ``if" in the first sentence of the proposition. \end{proof}

\begin{prop} \label{prop:integration} Let $P$ be a monoid, $\pi : P \to P^{\rm int}$ the natural map, $F \leq P$ a face of $P$.  The following are equivalent: \begin{enumerate} \item \label{integration1} The natural map of monoids $F^{-1}P \to (F^{-1}P)^{\rm int}$ is local. \item \label{integration2} $F$ is in the image of $\Spec \pi$ and if $G$ is the unique (since $\pi$ is surjective) face of $P^{\rm int}$ such that $F=\pi^{-1}G$, then the natural map $(F^{-1}P)^{\rm int} \to G^{-1}P^{\rm int}$ is an isomorphism.  \end{enumerate} \end{prop}

\begin{proof} To see that \eqref{integration1} implies \eqref{integration2} notice that we have a commutative diagram $$ \xym{ P \ar[r]^-{\pi} \ar[d] & P^{\rm int} \ar[d] \\ F^{-1}P \ar[r] & (F^{-1}P)^{\rm int} } $$ where the bottom horizontal arrow is local by \eqref{integration1}.  Using this and Lemma~\ref{lem:facesandlocalization}, we see that the preimage of the face $((F^{-1}P)^{\rm int})^*$ of $(F^{-1}P)^{\rm int}$ in $P$ is $F$.  On the other hand, the commutativity of the diagram ensures that $F = \pi^{-1}(G)$, where $G$ is the face of $P^{\rm int}$ given by the preimage of $((F^{-1}P)^{\rm int})^*$ under the right vertical arrow.

To see that \eqref{integration2} implies \eqref{integration1} suppose $F=\pi^{-1}G$ for some face $G \leq P^{\rm int}$.  Then certainly $F^{-1}P \to G^{-1}P^{\rm int}$ is local.  Also, $G^{-1}P^{\rm int}$ is integral, so by the universal property, this local map factors through $F^{-1}P \to (F^{-1}P)^{\rm int}$ via a map $(F^{-1}P)^{\rm int} \to G^{-1}P^{\rm int}$, so it will suffice to show that this latter map is an isomorphism.  It is surjective because the map $F^{-1}P \to G^{-1}P^{\rm int}$ (which factors through it) is surjective since $\pi$ is surjective (and hence so is $\pi|F : F \to G$).  It is injective because both its domain and codomain are integral monoids with the same groupification---namely $P^{\rm gp}$. \end{proof}

Suppose now that $X$ is an arbitrary (locally) monoidal space (\S\ref{section:locallymonoidalspaces}).  We will now define a new locally monoidal space $X^{\rm int}$ and a map of locally monoidal spaces $i : X^{\rm int} \to X$.  As a topological space, $X^{\rm int}$ is defined to be the subspace of $X$ consisting of those $x \in X$ for which the monoid homomorphism $\M_{X,x} \to \M_{X,x}^{\rm int}$ is local.  The resulting embedding $i : X^{\rm int} \into X$ will be our $\LMS$ morphism $i$ on the level of topological spaces.  The structure sheaf $\M_{X^{\rm int}}$ is defined to be the sheafification of the presheaf of monoids \bne{Xintpre} U & \mapsto & (i^{-1} \M_X)(U)^{\rm int} . \ene  The universal property of sheafification and functoriality of $P \mapsto P^{\rm int}$ yield a map \bne{Xintmap} i^{-1} \M_X & \to & \M_{X^{\rm int}} \ene of sheaves of monoids on $X^{\rm int}$.  We will now see that \eqref{Xintmap} has local stalks, so it will serve as a lift of our map $i$ of topological spaces to an $\LMS$ morphism.  Indeed, the functor $P \mapsto P^{\rm int}$ commutes with filtered direct limits, so the stalk of \eqref{Xintmap} at a point $x \in X^{\rm int} \subseteq X$ is the map \bne{stalkofXintmap} \M_{X,x} & \to & \M_{X,x}^{\rm int}, \ene which is local by definition of $X^{\rm int}$.  (Note also that this shows that $\M_{X^{\rm int}}$ has integral stalks, so $X^{\rm int}$ is indeed integral in the sense of Remark~\ref{rem:integralLMS}.)  We leave it to the reader to verify that the $\LMS$ morphism $i$ is terminal in the category of integral locally monoidal spaces over $X$.  This proves that the inclusion of the full subcategory of $\LMS$ consisting of \emph{integral} locally monoidal spaces has a right adjoint $X \mapsto X^{\rm int}$ (retracting the inclusion).

We further claim that if $X$ is a fan, then $X^{\rm int}$ is also a fan (and hence the $\LMS$ morphism $i$ is a map of fans, thus we obtain a right adjoint to the inclusion of integral fans into fans).  The construction of $X^{\rm int}$ and its map to $X$ are clearly local on $X$ and the property of being a fan is also local, so it is enough to prove this when $X = \Spec P$ for a monoid $P$.  In this case the map \be \Spec \pi : \Spec (P^{\rm int}) & \to & \Spec P = X \ee induced by the natural surjection $\pi : P \to P^{\rm int}$ is an $\LMS$ morphism whose domain has integral stalks, so by the universal property of $i : X^{\rm int} \to X$, we obtain an $\LMS$ morphism \bne{Xintiso} \Spec P^{\rm int} & \to & X^{\rm int} \ene (over $X$), which we claim is an isomorphism.  Indeed, both $i : X^{\rm int} \to X$ and $\Spec \pi$ are embeddings---the former by construction and the latter by Lemma~\ref{lem:specembedding}.  By comparing the construction of $i$ and the decription of the image of $\Spec \pi$ in Proposition~\ref{prop:integration}, we see that these two embeddings into $X$ have the same image, so \eqref{Xintiso} is a homeomorphism on topological spaces.  To see that \eqref{Xintiso} induces an isomorphism on structure sheaves, it is enough to check that it induces an isomorphism on stalks.  Consider a point $G \in \Spec P^{\rm int}$ (which we view as a face of $P^{\rm int}$, as usual).  Let $F = \pi^{-1}G$ be the corresponding face of $P$.  Then by construction of $\M_{X^{\rm int}}$ and Proposition~\ref{prop:Spec}\eqref{stalkMP}, we see that the stalk of the map of structure sheaves in question at $F$ is the map \be \M_{X^{\rm int},F} = \M_{X,F}^{\rm int} = \M_{P,F}^{\rm int} = (F^{-1}P)^{\rm int} & \to & \M_{P^{\rm int},G} = G^{-1} P^{\rm int}, \ee which is an isomorphism by Proposition~\ref{prop:integration}.

Now suppose that $X$ is an \emph{integral} (locally) monoidal space.  Let $P \mapsto P^\bullet$ be one of the functors on integral monoids discussed in Corollaries~\ref{cor:tf}, \ref{cor:sat}, \ref{cor:trc} (i.e.\ $\bullet \in \{ {\rm tf, \; sat, \; trc} \}$) and let us agree to call a torsion-free (resp.\ saturated, torsion-free and saturated) monoid a ``$\bullet$ monoid".  We will now define a new locally monoidal space $X^\bullet$ and a map of locally monoidal spaces $i : X^\bullet \to X$.  As a topological space, $X^\bullet$ is defined to be equal to (the space underlying) $X$.  The structure sheaf $\M_{X^\bullet}$ of $X^\bullet$ is defined to be the sheaf associated to the presheaf \bne{Xboxpre} U & \mapsto & \M_X(U)^\bullet . \ene  There is an obvious map $\M_X \to \M_{X^\bullet}$ of sheaves of monoids on $X$.  Since the functors $P \mapsto P^\bullet$ commute with filtered direct limits, the stalk of this obvious map at $x \in X$ is the map \bne{stalkofXboxmap} \M_{X,x} & \to & \M_{X,x}^\bullet, \ene which is local by part \eqref{tflocal} of the appropriate aforementioned corollary.  Hence our ``obvious map" defines an $\LRS$ morphism $X^\bullet \to X$ which is easily seen to be terminal in the category of locally monoidal spaces over $X$ with $\bullet$ stalks.

We further claim that if $X$ is an integral fan, then the locally monoidal space $X^\bullet$ defined above is also a fan (hence $X \mapsto X^\bullet$ defines a right adjoint to the inclusion of fans with the appropriate stalks into all fans).  Like the construction $X \mapsto X^{\rm int}$ discussed above, the construction of $X \mapsto X^\bullet$ is local in nature, so it is enough to prove this when $X = \Spec P$ for an integral monoid $P$, in which case we claim that $X^\bullet \cong \Spec (P^\bullet)$.  Indeed, since the localization of any $\bullet$ monoid at any face is also a $\bullet$ monoid (by parts \eqref{tfhomeo} and \eqref{tflocalization} in the aforementioned corollaries), $\Spec (P^\bullet)$ has $\bullet$ stalks, so we obtain a natural $\LMS$ morphism \bne{Xboxiso} \Spec (P^\bullet) & \to & X^\bullet \ene (over $X$), which we claim is an isomorphism.  It is a homeomorphism on spaces because both maps to $X=\Spec P$ are homeomorphisms---this is by construction for $X^\bullet$ and for $\Spec (P^\bullet)$ it holds by part \eqref{tfhomeo} of the appropriate corollary.  We can check that the map of structure sheaves is an isomorphism on stalks, where it follows from part \eqref{tflocalization} of the appropriate corollary.

We close this section with some simple remarks about these constructions.

\begin{rem} We can compose $X \mapsto X^{\rm int}$ with $X \mapsto X^\bullet$ to obtain a right adjoint to the inclusion of integral locally monoidal spaces with $\bullet$ stalks (resp.\ integral fans with $\bullet$ stalks) into all locally monoidal spaces (resp.\ all fans).  We usually denote this composition simply by $X \mapsto X^\bullet$, but hopefully the discussion above has made it clear that it is worthwhile to treat the two constructions ``separately". \end{rem}

\begin{rem} If $X$ has locally finite underlying space, then so does $X^{\rm int}$ since the space underlying $X^{\rm int}$ is a subspace of the space underlying $X$.  It follows from this and Proposition~\ref{prop:locallyfinitetypefans} (or from the ``local" description of $X^{\rm int}$ when $X$ is affine) that $X^{\rm int}$ is fine when $X$ is locally finite type.  Similarly, $X^{\rm sat}$ (resp.\ $X^{\rm trc}$) will be fs (resp.\ toric) when $X$ is fine. \end{rem}

\subsection{Classical fans} \label{section:classicalfans}  Throughout \S\ref{section:classicalfans} we will refer to a fan in the usual sense of toric varieties (Definition~\ref{defn:classicalfan}) as a \emph{classical fan} and a fan in the sense of Definition~\ref{defn:fan} as an \emph{abstract fan}.  We will see that every classical fan gives rise to an abstract fan and that the construction yields a fully faithful functor from classical fans to abstract fans, whose essential image is described in Theorem~\ref{thm:classicalfans}.  In the rest of the paper, therefore, we will make little or no distinction between a classical fan and the associated abstract fan.

\begin{defn} \label{defn:classicalfan} Let $N$ be a \emph{lattice} (i.e.\ a group isomorphic to $\ZZ^n$ for some $n \in \NN$).  A \emph{(classical) fan} $\Sigma$ (with lattice $N$) is a finite, non-empty set of sharp, rational cones (Definition~\ref{defn:cone}) in $N_{\RR} := N \otimes \RR$ satisfying the following properties: \begin{enumerate} \item \label{cf1} For any $\sigma \in \Sigma$ and any $\tau \leq \sigma$, we have $\tau \in \Sigma$. \item \label{cf2} No point of $N \subseteq N_{\RR}$ is contained in the interior of more than one cone of $\Sigma$. \end{enumerate}  A morphism $f : \Sigma \to \Sigma'$ from a (classical) fan $\Sigma$ with lattice $N$ to a (classical) fan $\Sigma'$ with lattice $N'$ is a map of lattices $f : N \to N'$ such that the corresponding map of vector spaces $f_{\RR} : N_{\RR} \to N'_{\RR}$ takes every $\sigma \in \Sigma$ into some $\sigma' \in \Sigma'$. \end{defn}

\begin{rem} \label{rem:defnofclassicalfan}  Definition~\ref{defn:classicalfan} is convenient for our purposes, but it is more customary (cf.\ \cite[1.4]{F}) to define a \emph{(classical) fan} by replacing our condition \eqref{cf2} with the condition \begin{enumerate} \setcounter{enumi}{2} \item \label{cf3} For any $\sigma, \sigma' \in \Sigma$, we have $\sigma \cap \sigma' \leq \sigma$. \end{enumerate}  To see that this yields the same notion of classical fan, we need to check that \eqref{cf2} and \eqref{cf3} are equivalent when \eqref{cf1} holds.

To see that \eqref{cf2} implies \eqref{cf3} when \eqref{cf1} holds, suppose, towards a contradiction, that there are $\sigma, \sigma' \in \Sigma$ with $\sigma \cap \sigma'$ not a face of $\sigma$.  After possibly replacing $\sigma$ with the smallest face of $\sigma$ containing $\sigma \cap \sigma'$ (which will be in $\Sigma$ by \eqref{cf1}), we can assume no proper face of $\sigma$ contains $\sigma \cap \sigma'$.  So, for each proper face $\tau < \sigma$, there is some $v_\tau \in (\sigma \cap \sigma') \setminus \tau$.  The sum $v$ of the $v_\tau$ is in $\sigma' \cap \sigma^\circ$ because if it were in some $\tau < \sigma$, then $v_\tau$ would have to be in $\tau$ since $\tau \leq \sigma$ and all the $v_\tau$ are in $\sigma$.  After possibly replacing $\sigma'$ with the smallest face of $\sigma'$ containing $v$ (which will be in $\Sigma$ by \eqref{cf1}), we can assume $v \in (\sigma')^\circ$, so that $(\sigma')^\circ \cap \sigma^\circ \neq \emptyset$.  Then $(\sigma')^\circ \cap \sigma^\circ = (\sigma \cap \sigma')^\circ$ by Theorem~\ref{thm:cones}\eqref{intersectionfaces} and hence there will be some $n \in N \cap (\sigma')^\circ \cap \sigma^\circ$ because $\sigma \cap \sigma'$ is a rational cone by Theorem~\ref{thm:cones}\eqref{sumandintersection}, so its interior contains an integral point by Theorem~\ref{thm:cones}\eqref{integralpointininterior}.  This contradicts \eqref{cf2}.

To see that \eqref{cf3} implies \eqref{cf2}, suppose, toward a contradiction, that there is some $n \in N \cap \sigma^\circ \cap (\sigma')^\circ$ for distinct $\sigma, \sigma' \in \Sigma$.  After possibly exchanging $\sigma$ and $\sigma'$ we can assume that $\sigma$ is not contained in $\sigma'$.  But then, by \eqref{cf3}, $\sigma \cap \sigma'$ is a proper face of $\sigma$ containing $n$, contradicting $n \in \sigma^\circ$. \end{rem}

Now we will explain how to construct an abstract fan from a classical fan $\Sigma$ with lattice $N$.  Let $M := \Hom(N,\ZZ)$ be the dual lattice.  We topologize $\Sigma$ by declaring a subset $\Sigma' \subseteq \Sigma$ to be \emph{open} iff $\tau \in \Sigma'$ whenever $\sigma \in \Sigma'$ and $\tau \leq \sigma$.  In other words, the open subsets of $\Sigma$ are the \emph{subfans} of $\Sigma$ and the empty set.  Evidently, the smallest open subset of a cone $\sigma \in \Sigma$ is the set $[ \sigma ]$ of all faces of $\sigma$, so that $\tau$ is a generalization of $\sigma$ in $\Sigma$ (i.e.\ $\sigma \in \{ \tau \}^-$) iff $\tau \leq \sigma$.   We equip the topological space $\Sigma$ with a sheaf of monoids $\M_{\Sigma}$ as follows:  First of all, $\Sigma$ is finite, so to describe $\M_{\Sigma}$ we need only describe its stalks and the (functorial) generalization maps between them.  The stalk $\M_{\Sigma,\sigma}$ of $\M_{\Sigma}$ is defined to be the monoid of integral points in the dual cone of $\sigma$: \bne{classicalstalkformula} \M_{\Sigma,\sigma} & := & M \cap \sigma^\lor. \ene  Note that $\M_{\Sigma,\sigma}$ is a toric monoid with $\M_{\Sigma,\sigma}^{\rm gp} = M$.  The generalization map $\M_{\Sigma,\sigma} \to \M_{\Sigma,\tau}$ is defined to be the inclusion of submonoids $M \cap \sigma^\lor \subseteq M \cap \tau^\lor$.

To see that the monoidal space $(\Sigma,\M_\Sigma)$ is an abstract fan (in fact, a \emph{toric} abstract fan), we claim that, for each $\sigma \in \Sigma$, the open subspace $([\sigma],\M_{\Sigma}|[\sigma]) = ( [\sigma], \M_{[\sigma]})$ of $(\Sigma,\M_\Sigma)$ is isomorphic to $\Spec (M \cap \sigma^\lor)$.  To see this, first note that the inclusion $(M \cap \sigma^\lor) \subseteq \sigma^\lor$ yields a homeomorphism of topological spaces \bne{firsthomeo} \Spec \sigma^\lor & \to & \Spec (M \cap \sigma^\lor) \ene by Theorem~\ref{thm:monoidsandcones}\eqref{monoidsandcones1}.  The inclusion-reversing bijection in Theorem~\ref{thm:cones}\eqref{inclusionreversingbijection} may be viewed as a homeomorphism \bne{secondhomeo} \Spec \sigma^\lor & \to & [ \sigma ]. \ene  From the formulas for \eqref{firsthomeo} and \eqref{secondhomeo}, we see that \bne{mainhomeo} [ \sigma ] & \to & \Spec (M \cap \sigma^\lor) \\ \nonumber \tau & \mapsto & M \cap \tau^\perp \cap \sigma^\lor \ene is a homeomorphism.  For $\tau \leq \sigma$, the computation \be \M_{ [\sigma] , \tau} & = & M \cap \tau^\lor \\ & = & M \cap (\sigma^\lor + \tau^\perp) \\ & = & M \cap (\tau^\perp \cap \sigma^\lor)^{-1} \sigma^\lor \\  & = & (M \cap \tau^\perp \cap \sigma^\lor)^{-1} (M \cap \sigma^\lor) \ee using Theorem~\ref{thm:monoidsandcones}\eqref{monoidsandcones3} shows that the stalk of $\M_{ [\sigma] }$ at $\tau$ is the localization of $M \cap \sigma^\lor$ at the face corresponding to $\tau$ under the homeomorphism \eqref{mainhomeo}---this is also the stalk of the structure sheaf $\M_{M \cap \sigma^\lor}$ of $\Spec (M \cap \sigma^\lor)$ at the latter face by Proposition~\ref{prop:Spec}\eqref{stalkMP}.  Since this is all natural in $\tau$, we obtain the claimed isomorphism.

Given a map of classical fans $f : \Sigma \to \Sigma'$ determined by a map of lattices $f : N \to N'$, we define a continuous map of topological spaces $|f| : \Sigma \to \Sigma'$ by mapping a cone $\sigma \in \Sigma$ to the smallest cone $\sigma' \in \Sigma'$ with $f_{\RR}(\sigma) \subseteq \sigma'$ (such a cone exists by definition of a morphism of classical fans).  For $\tau \leq \sigma$ in $\Sigma$, we clearly have $\tau' \leq \sigma'$, and the diagram \bne{stalkmapdiagram} & \xym{ M \cap \sigma^\lor = \M_{\Sigma,\sigma} \ar[r]^-{\subseteq} & \M_{\Sigma,\tau} = M \cap \tau^\lor \\ M' \cap (\sigma')^\lor = \M_{\Sigma',\sigma'} \ar[r]^-{\subseteq} \ar[u]^{f^\lor | M' \cap (\sigma')^\lor} & \M_{\Sigma',\tau'} = M' \cap (\tau')^\lor \ar[u]_{f^\lor | M' \cap (\tau')^\lor} } \ene clearly commutes.  Furthermore, the vertical arrows are local maps of monoids.  To see this, suppose $m' \in M' \cap (\sigma')^\lor$ satifies $f^\lor m' \in \M_{\Sigma,\sigma}^* = M \cap \sigma^\perp$---i.e.\ $\sigma \subseteq N_{\RR}$ is contained in the kernel of the linear functional $m'_{\RR} f_{\RR} : N_{\RR} \to \RR$.  We need to show that $m' \in M' \cap (\sigma')^\perp$---i.e.\ that $\sigma'$ is contained in the kernel of the linear functional $m'_{\RR} : N'_{\RR} \to \RR$.  If this weren't the case, then $\sigma'' := \sigma' \cap \Ker(m'_{\RR})$ would be a proper face of $\sigma'$ containing $f_{\RR}(\sigma)$, contradicting the definition of $\sigma'$.  Since $\Sigma$, $\Sigma'$ are finite spaces, the commutativity of the diagrams \eqref{stalkmapdiagram} determines a unique map $f^\dagger : |f|^{-1} \M_{\Sigma'} \to \M_\Sigma$, whose stalk generalization maps are the diagrams \eqref{stalkmapdiagram}, hence $|f|$ and $f^\dagger$ yield a map of abstract fans $(\Sigma,\M_{\Sigma}) \to (\Sigma',\M_{\Sigma'})$.  

The construction described above yields a functor \bne{classicalfanstoabstractfans} \ClassicalFans & \to & {\bf (Abstract)}\Fans. \ene

\begin{thm} \label{thm:classicalfans} The functor \eqref{classicalfanstoabstractfans} is fully faithful.  Its essential image consists of the toric (abstract) fans $X$ satisfying the following two properties: \begin{enumerate} \item \label{classicalfan1} The underlying topological space of $X$ is finite and irreducible.  \item \label{classicalfan2} Whenever $f,g : \Spec \NN \to X$ are maps of fans with $f|\Spec \ZZ = g|\Spec \ZZ$, we have $f=g$.  \end{enumerate} \end{thm}

\begin{proof}  To see that \eqref{classicalfanstoabstractfans} is faithful, just note that we can recover the map of lattices $f : N \to N'$ (hence the map of classical fans $f : \Sigma \to \Sigma'$) from the corresponding map $(|f|,f^\dagger)$ of abstract fans as the $\ZZ$-linear dual of the stalk of $f^\dagger$ at $0 \in \Sigma$.  

To see that \eqref{classicalfanstoabstractfans} is full, suppose $(g,g^\dagger) : (\Sigma,\M_{\Sigma}) \to (\Sigma',\M_{\Sigma'})$ is a map of abstract fans (i.e.\ a morphism of locally monoidal spaces).  Since $g_0^\dagger : \M_{\Sigma',g(0)} \to \M_{\Sigma,0}=M$ is local, $\M_{\Sigma',g(0)} = M' \cap g(0)^\perp$ must be a group, so we must have $g(0)=0$ so that $g_0 : M' \to M$.  Let $f := g_0^\dagger : N \to N'$.  The map $g^\dagger : g^{-1} \M_{\Sigma'} \to \M_{\Sigma}$ gives rise to a commutative diagram of monoids \bne{moncomdia} & \xym{ M \cap \sigma^\lor = \M_{\Sigma,\sigma} \ar[r]^-{\subseteq} & \M_{\Sigma,0} =M \\ M' \cap g(\sigma)^\lor = \M_{\Sigma',g(\sigma)} \ar[u]^{g^\dagger_\sigma} \ar[r]^-{\subseteq} & \M_{\Sigma',0}=M' \ar[u]_{ g^\dagger_0 = f^\lor } } \ene for any $\sigma \in \Sigma$.  Since the cone over $M \cap \sigma^\lor$ is $\sigma^\lor$ and similarly for $M' \cap (\sigma')^\lor$ (Theorem~\ref{thm:monoidsandcones}\eqref{monoidsandcones2}), we see that $f^\lor_{\RR}$ takes $g(\sigma)^\lor$ into $\sigma^\lor$, hence $f_{\RR}$ must take $\sigma$ into $g(\sigma)$, so that $f : N \to N'$ defines a map of classical fans $f : \Sigma \to \Sigma'$.  In fact, $g(0)$ must be the \emph{smallest} cone $\sigma' \in \Sigma'$ containing $\sigma$, otherwise we would see that $g^\dagger_\sigma$ is not local by reversing the argument in the previous paragraph.  Once this is known, the commutativity of the diagrams \eqref{moncomdia} ensures that $g$ agrees with the map associated to $f$ by our functor \eqref{classicalfanstoabstractfans}.

Let $\Sigma$ be a classical fan with lattice $N$.  It is clear from the definition of the topology on $\Sigma$ that $(\Sigma,\M_{\Sigma})$ satisfies \eqref{classicalfan1} with the zero cone $0 \in \Sigma$ the unique generic point of $\Sigma$.  To see that $(\Sigma,\M_{\Sigma})$ satisfies \eqref{classicalfan2}, note that a map of fans $n : \Spec \ZZ \to (\Sigma,\M_{\Sigma})$ is the same thing as a map of groups $n : \M_{\Sigma,0} = M \to \ZZ$, which is the same thing as an element of $N$.  To extend $n$ to a map of fans $f : \Spec \NN \to (\Sigma,\M_{\Sigma})$ is to give a point $\sigma \in \Sigma$ such that the composition of $\M_{\Sigma,\sigma} \to \M_{\Sigma,0}=M$ and  $n : M \to \ZZ$ takes $\M_{\Sigma,\sigma}$ into $\NN \subseteq \ZZ$ via a \emph{local} map of monoids.  This is equivalent to saying that $n \in N \subseteq N_{\RR}$ is in the \emph{interior} of the cone $\sigma$.  By definition of a classical fan (Definition~\ref{defn:classicalfan}), no point of $N$ can be in the interior of more than one cone of $\Sigma$, so there is at most one such $f$.  (Alternatively, one can just note that $\Spec \ZZ$ and $\Spec \NN$ ``are" classical fans and use the full faithfulness of \eqref{classicalfanstoabstractfans} to establish \eqref{classicalfan2}, but we have spelled out the argument for clarity.) 

It remains to show that any toric fan $X$ satisfying \eqref{classicalfan1} and \eqref{classicalfan2} is in the essential image of \eqref{classicalfanstoabstractfans}.  Since the space underlying $X$ is sober (Proposition~\ref{prop:fans}\eqref{fansaresober}), condition \eqref{classicalfan1} ensures that $X$ has a unique generic point $\eta$.  Since $\eta$ has no non-trivial generalizations, $M := \M_{X,\eta}$ must be a group (Proposition~\ref{prop:fans}\eqref{fangeneralization}).  In fact $M$ must be a lattice because $X$ is toric.  Let $N := \Hom(M,\ZZ)$ be the dual lattice.  For $x \in X$, let $\sigma^\lor(x) \subseteq M \otimes \RR$ be the cone over the image of the generalization map $\M_{X,x} \to \M_{X,\eta}=M$.  The cone $\sigma^\lor(x)$ is rational since the toric monoid $\M_{X,x}$ is finitely generated, so its dual cone $\sigma(x) \subseteq N \otimes \RR$ is also rational by Theorem~\ref{thm:cones}\eqref{FarkasTheorem}.  Since $\M_{X,x}$ is toric (hence saturated), we have $M \cap \sigma^\lor(x) = \M_{X,x}$ by Theorem~\ref{thm:monoidsandcones}.  Since $\M_{X,x} \to M$ is the groupification of $\M_{x,x}$ (Proposition~\ref{prop:fans}\eqref{fangeneralization}), the cone $\sigma^\lor(x) \subseteq M \otimes \RR$ spans $M \otimes \RR$, so its dual cone $\sigma(x)$ is sharp by Theorem~\ref{thm:cones}\eqref{sharpcriterion}.  Since the set $X$ is finite and non-empty by assumption, \bne{thefanwewant} \Sigma(X) & := & \{ \sigma(x) : x \in X \} \ene is a finite, non-empty set of sharp, rational cones in $N_{\RR}$, which we claim is a classical fan (Definition~\ref{defn:classicalfan}).

For every $x \in X$ and every face $\rho \leq \sigma^\lor(x)$, $M \cap \rho$ is a face of $M \cap \sigma^\lor(x) = \M_{X,x}$.  Since $X$ is a fan, Proposition~\ref{prop:fans}\eqref{fangeneralization} says that there must be some generalization $y$ of $x$ in $X$ for which $\M_{X,x} \to \M_{Y,y}$ is the localization at this face.  Using Theorem~\ref{thm:monoidsandcones}\eqref{monoidsandcones3} we then find \be \M_{Y,y} & = & (M \cap \rho)^{-1}(M \cap \sigma^\lor(x)) \\ & = & M \cap \rho^{-1} \sigma^\lor(x), \ee hence $\sigma^\lor(y) = \rho^{-1} \sigma^\lor(x)$.  By Theorem~\ref{thm:cones}\eqref{inclusionreversingbijection}, we then see that $\sigma(y)$ is a face of $\sigma(x)$ and that every face of $\sigma(x)$ arises in this manner for some (unique) $\rho \leq \sigma^\lor(x)$.  Hence $\Sigma(X)$ satisfies condition \eqref{cf1} in Definition~\ref{defn:classicalfan}.

Let $x,y \in X$ and suppose there is some $n \in N \cap \sigma(x)^\circ \cap \sigma(y)^\circ$.  By reversing the discussion in the first paragraph of the proof, one sees that $n$ can be viewed as a map of fans $n : X \to \Spec \ZZ$ which can be extended to a map $\Spec \NN \to X$ by taking the closed point of $\Spec \NN$ either to $x$ or to $y$, so \eqref{classicalfan2} implies $x=y$.  This proves that $\Sigma(X)$ satisfies condition \eqref{cf2} in Definition~\ref{defn:classicalfan}, hence it is a classical fan.  To see that $X$ is isomorphic to the image of $\Sigma(X)$ under \eqref{classicalfanstoabstractfans} it is enough to note that $\M_{X,x} = M \cap \sigma^\lor(x)$ by Theorem~\ref{thm:monoidsandcones}\eqref{monoidsandcones1} because $\M_{X,x}$ is toric, hence saturated.  \end{proof}

\begin{rem} \label{rem:classicalfans1}  Condition \eqref{classicalfan2} in Theorem~\ref{thm:classicalfans} says that $X$ is \emph{valuatively separated} in a sense that we will define and study in \S\ref{section:propermaps}. \end{rem}

\begin{rem} \label{rem:classicalfans}  Assume $X$ is a toric fan satisfying \eqref{classicalfan1} in Theorem~\ref{thm:classicalfans}.  Let $\eta$ be the generic point of $X$, so $M := \M_{X,\eta}$ is a group.  The ``valuatively separated" condition \eqref{classicalfan2} in Theorem~\ref{thm:classicalfans} implies that the map taking $x \in X$ to the cone $\sigma(\M_{X,x}) \subseteq M \otimes \RR$ over the image of the generalization map $\M_{X,x} \to \M_{X,\eta}$ is one-to-one.  This latter condition, however, does \emph{not} imply that $X$ is valuatively separated. \end{rem}

\subsection{Inverse limits} \label{section:inverselimitsoffans}  The unusually nice behaviour of inverse limits of locally monoidal spaces (\S\ref{section:inverselimits}) carries over to fans:

\begin{thm} \label{thm:fansinverselimits} The category $\Fans$ of fans has all finite inverse limits.  All of the forgetful functors \be \Fans & \to & \LMS \\ \Fans & \to & \MS \\ \Fans & \to & \Top \ee preserve finite inverse limits. \end{thm}

\begin{proof} First one proves that $\LMS$ has finite inverse limits and that the inverse limit $X$ of a finite diagram of fans $i \mapsto X_i$, calculated in $\LMS$, is again a fan and hence serves as the inverse limit of $i \mapsto X_i$ in $\Fans$.  This can all be done independently of any results from \S\ref{section:inverselimits}.  Furthermore, the analogous statements are true for rings and both the monoid and ring results can be proved by the same general nonsense.  This is all explained carefully in \cite[3.4]{GM1}.  The other statements then follow from the fact that the forgetful functors $\LMS \to \MS$ and $\MS \to \Top$ preserves all inverse limits (\S\ref{section:inverselimits}). \end{proof}

\begin{rem} Theorem~\ref{thm:fansinverselimits} implies that if we have a finite diagram $i \mapsto X_i$ of fans, then its ``naive" inverse limit $X$ (i.e.\ the one calculated in $\MS$) is actually a fan.  There appears to be no way to see this ``directly."  \end{rem}

\subsection{Monoid schemes} \label{section:monoidschemes}  Here we explain how the theory of \emph{monoid schemes} of \cite{CDH} is related to ``our" (really: ``Kato's") theory of fans.  In \cite{CDH}, the authors work with \emph{pointed monoids}.  

\begin{defn} \label{defn:pointedmonoid} A \emph{pointed monoid} is a monoid $P$ (in the sense of \S\ref{section:monoids}) such that there is an element $\infty \in P$ (necessarily unique), called the \emph{infinity element}, such that $p + \infty = \infty$ for all $p \in P$.  A \emph{morphism} of pointed monoids is a monoid homomorphism preserving $\infty$.  The category of pointed monoids is denoted $\Mon_*$.  \end{defn}

\begin{rem} If one uses multiplicative notation for monoids instead of our additive notation, then our $0$ and $\infty$ should instead be denoted $1$ and $0$, respectively.  If one thinks of passing from additive notation to multiplicative notation as a kind of ``exponentiation," then one might prefer to think of $\infty$ as $- \infty$, though this causes a notational conflict with the use of $-$ for the inverse of a unit in a monoid. \end{rem}

Note that a pointed monoid $P$ is never integral in the usual sense (unless $P = \{ 0 =\infty \}$), but there is a reasonable alternative notion of \emph{pointed integral} where one requires $$ p + q = p' + q \; \implies \; p = p' \quad \forall p,p' \in P, \; q \in P \setminus \{ \infty \}.$$  There is a functor \bne{MonMon} \Mon & \to & \Mon_* \\ \nonumber P & \mapsto & P_* \ene defined by setting $P_* := P \coprod \{ \infty \}$ with the unique multiplication making $\infty$ the infinity element and the inclusion $P \to P_*$ a monoid homomorphism.  (Notice that this makes sense even if $P$ already \emph{has} infinity---the infinity element for $P$ will no longer be infinity in $P_*$.)  The functor \eqref{MonMon} is faithful but not full.  It is left adjoint to the obvious forgetful functor $\Mon_* \to \Mon$ (this forgetful functor is also faithful but not full), hence it preserves direct limits.  

If $P$ is a pointed integral monoid, then $P \setminus \{ \infty \}$ is a submonoid of $P$, integral in the usual sense.\footnote{In general this is not true:  There is a unique pointed monoid structure on $\{0,1,\infty \}$ where $0$ is the zero, $\infty$ is the infinity, and $1+1=\infty$.  This shows that the functor \eqref{MonMon} is not essentially surjective.}  (If $p+q = \infty$ for $p,q \in P \setminus \{ \infty \}$, then we have an equality $p+q = \infty + q$ violating pointed integrality.)  Clearly then $P = (P \setminus \{ \infty \})_*$, so there is a loose sense in which ``pointed integral monoids are the same thing as integral monoids," but the \emph{raison d'etre} for the category of pointed monoids is the non-fullness of \eqref{MonMon}---the maps $Q_* \to P_*$ \emph{not} induced by a map $Q \to P$ (i.e.\ the maps of pointed monoids $h:Q_* \to P_*$ for which $h^{-1}(\infty) \neq \{ \infty \}$) can be quite handy, as we now explain.

Let $P$ be a pointed monoid.  A \emph{pointed face} of $P$ is a pointed submonoid $F$ of $P$ (i.e.\ a submonoid containing $\infty$) satisfying $$ p+p' \in F \setminus \{ \infty \} \; \implies \; p,p' \in F \setminus \{ \infty \} \forall p,p' \in P. $$  The set of pointed faces of $P$ is denoted $\Spec_* P$; it is topologized in the same manner as its ``unpointed" analogue $\Spec P$ (\S\ref{section:Spec}).  For any monoid $P$, one has a homeomorphism \be \Spec P & \to & \Spec_* P_* \\ F & \mapsto & F_*. \ee  If $h : Q \to P$ is a map of pointed monoids and $F \leq P$ is a pointed face of $P$, then one checks that \be (\Spec_* h)(F) & := & h^{-1}(F \setminus \{ \infty \} ) \coprod \{ \infty \} \ee is a pointed face of $Q$ (the union is disjoint because we require $h(\infty) = \infty$ in the definition of a map of pointed monoids) and that the resulting map \be \Spec_* h : \Spec_* P & \to & \Spec_* Q \ee is continuous.  

For a pointed face $F \leq P$, we can define a map of pointed monoids $P \to F$ by taking $p$ to $p$ if $p \in F$ and to $\infty$ if $p \in P \setminus F$.  This map, which we call the \emph{canonical retract} of $F$ sits in a diagram of pointed monoids \bne{retract} & F \to P \to F \ene where the first map is the inclusion and the composition is the identity.  This retract has no analogue in the unpointed situation: If $F$ is a face of a monoid $P$, then $F_*$ is a pointed face of the pointed monoid $P_*$ and we have a retract diagram \bne{retract2} & F_* \to P_* \to F_* \ene as a special case of the previous construction.  The retract map $P_* \to F_*$, however, is \emph{not} in the image of \be \Hom_{\Mon}(P,F) & \to & \Hom_{\Mon_*}(P_*,F_*) \ee unless $F=P$.

Given a pointed monoid $P$, the associated \emph{pointed monoid algebra} $\ZZ_*[P]$ is defined by \be \ZZ_*[P] & := & \ZZ[P]/\langle [\infty] \rangle. \ee  This defines a functor \bne{Zslot3} \ZZ_*[ \slot ] : \Mon_* & \to & \An . \ene  The functors \eqref{Zslot3} and \eqref{MonMon} sit in an obvious $2$-commutative diagram with the usual monoid algebra functor \eqref{Zslot} of \S\ref{section:monoidalgebra}.

The category $\MonSch$ of \emph{monoidal schemes} of \cite{CDH} is defined in analogy with the category of fans by systematically replacing the word ``monoid" with ``pointed monoid."  There is an obvious forgetful functor $\MonSch \to \Fans$ admitting a right adjoint \bne{FansMonSch} \Fans & \to & \MonSch \\ \nonumber X & \mapsto & X_* \ene taking a fan $X$ to the monoid scheme $X_*$ obtained by ``pointing the structure sheaf of $X$."  This functor is not full.

As mentioned in the Introduction, the category of (locally finite type) fans is the initial object in an appropriate $2$-category of ``categories of spaces."  The category $\MonSch$ of monoidal schemes (with its distinguished monoid object $\Spec \NN_*$) is an object of this $2$-category and \eqref{FansMonSch} is the (essentially unique) map out of the initial object.  Thus $\MonSch$ is just one of the many ``categories of spaces" receiving a ``realization functor" from $\Fans$.  In \cite[\S4]{CDH} the authors explain how to view a classical fan as a ``monoid scheme."  Their construction is of course nothing but the composition of the functor \eqref{classicalfanstoabstractfans} in \S\ref{section:classicalfans} and the functor \eqref{FansMonSch}.  The fact that their functor from classical fans to monoid schemes isn't full \cite[Theorem~4.4(2)]{CDH} results only from the non-fullness of \eqref{FansMonSch}.

\begin{rem} \label{rem:pointedmonoids} The category of manifolds with corners defined by Kottke and Melrose in \cite{KM} is related to the category of (positive) log differentiable spaces defined by Molcho and I in \cite{GM1} in a similar way.  Their notion of morphism (``$b$-map") ``is" (in an appropriate sense) the same thing as a map of \emph{pointed} log differentiable spaces; what they call an ``interior $b$-map" is a morphism of log differentiable spaces. \end{rem}

\begin{rem} \label{rem:boundary} In \cite[\S4.4]{GM1}, we introduced the idea of the ``boundary" $\Delta X$ of a ``log space" $X$.  This is another log space equipped with a map \emph{of spaces} (but not log spaces) $\Delta X \to X$.  If $X = X(P)$ is the log space attached to a (finitely generated) monoid $P$, then $\Delta X(P)$ is the disjoint union of the log spaces $X(F)$ associated to codimension one faces $F \leq P$ and the map $\Delta(X) \to X$ is the disjoint union of the ``wrong way maps" $X(F) \to X(F)$; since the latter only make sense in some categories of spaces (e.g.\ in monoidal schemes or in schemes, but not in fans or in log schemes), the map of spaces $\Delta X \to X$ is not always available.  In any case, using the boundary construction is very nearly the same thing as working in the pointed setting. \end{rem}

It should also be remarked that most of the results of the present paper are easily adapted from fans to monoid schemes; with few exceptions the resulting adapted results will be ``new" (i.e.\ not in \cite{CDH}).

\subsection{Coherent sheaves} \label{section:coherentsheaves}  The notion of a module over a monoid (\S\ref{section:modules}) gives rise to a corresponding notion of a module over a sheaf of monoids.  For a locally monoidal space $X$, we let $\Mod(X)$ (or $\Mod(\M_X)$) denote the category of $\M_X$-modules.  It is not abelian, but it does have a tensor product making it symmetric monoidal.  

Given a monoid $P$ and a $P$-module $M$, we obtain (functorially in $M$) an $\M_P$-module (object of $\Mod(\Spec P)$) \be M^{\sim} & := & \u{M} \otimes_{\u{P}} \M_P \ee by making use of the map \eqref{PtoMP}.  If $X$ is a fan, then an $\M_X$-module locally isomorphic to one obtained via this construction is called \emph{quasi-coherent}.  We reserve \emph{coherent} for the case where $X$ is locally finite type and $M$ can be taken finitely generated.  For a fan $X$, an \emph{ideal sheaf} of $X$ is a quasi-coherent $\M_X$-submodule $I \subseteq \M_X$.  The finiteness argument of \S\ref{section:monoidalgebra} shows that an ideal sheaf in a locally finite type fan $X$ is always coherent.

If $X$ is a fan, $M$ is a quasi-coherent $\M_X$-module and $U = \Spec P \subseteq X$ is an affine open subfan, we have a canonical isomorphism $M|U = M_P^\sim$, where $M_P \in \Mod(\M_{X,P}) = \Mod(P)$ is the stalk of $M$ at the closed point of $U$.  

\subsection{Relative Spec and Proj} \label{section:relativeSpecandProj}  For the convenience of the reader, we now summarize some basic facts about relative $\Spec$ and $\Proj$ functors, in analogy with the familiar versions from algebraic geometry.  Details and proofs can be found in \cite[5.12]{GM1}, but, at least if we stick to fans and quasi-coherent sheaves, everything is parallel with the corresponding results and constructions for schemes.

Given any $X \in \LMS$ and any map of sheaves of monoids $\M_X \to A$ on $X$, there is an $\LMS$-morphism $\Spec_X A \to X$.  The construction yields a functor \be \Spec_X : (\M_X / \Mon(X))^{\rm op} & \to & \LMS / X \ee right adjoint to the functor \be \LMS / X & \to & (\M_X / \Mon(X))^{\rm op} \\ (f : Y \to X) & \mapsto & (\M_X \to f_* \M_Y). \ee  If $X$ is a fan and $A$ is quasi-coherent as an $\M_X$-module (\S\ref{section:coherentsheaves}), then $\Spec_X A$ is a fan.  The construction of $\Spec_X$ is compatible with base change in an obvious sense.  If $A$ is a \emph{graded} sheaf of monoids under $A$ (meaning we have a map of sheaves of monoids $A \to \u{\NN}$), there is a variant $\Proj_X A$ enjoying similar properties.  In particular, if $I \subseteq \M_X$ is a quasi-coherent ideal sheaf on a fan $X$, the Rees monoid \be R(I) & := & \M_X \coprod I \coprod (I+I) \coprod \cdots \ee is quasi-coherent sheaf of graded monoids under $\M_X$, and we define the \emph{blowup} of $X$ along $I$ to be $\Proj_X R(I)$.

\subsection{Group objects, actions, and quotients} \label{section:groupobjects} Any (abelian) group $G$ can be viewed as an abelian group object in $\Mon^{\rm op}$ (naturally in $G$) via the comultiplication map $G \to G \oplus G$ given by $g \mapsto (g,g)$ and the coinverse map $g \mapsto -g$.  Since $\Spec : \Mon^{\rm op} \to \Fans$ preserves inverse limits, it takes group objects to group objects and we thus obtain a functor \bne{GG} \GG : \Ab^{\rm op} & \to & \Ab(\Fans) \ene from abelian groups to abelian group objects in $\Fans$.  The fan underlying $\GG(G)$ is $\Spec G$, but we often prefer to write $\GG(G)$ instead of $\Spec G$ to emphasize that we are regarding $\Spec G$ as a group object ``in the usual way."  (Though we will see momentarily that there is actually no ``other" group object structure on $\Spec G$.)  We sometimes use the notation $\GG(P) := \GG(P^{\rm gp})$ for any monoid $P$.

\begin{lem} \label{lem:groupobjects} Any group object in $\Mon^{\rm op}$ is of the form described above, so the functor \eqref{GG} is an equivalence of categories between $\Ab^{\rm op}$ and the category of affine group fans.  In fact, any connected group fan is of the form $\GG(G)$ for an abelian group $G$. \end{lem}

\begin{proof} Suppose $G$ is a cogroup object in $\Mon$ with comultiplication $(f,g) : G \to G \oplus G$, coinverse $i : G \to G$, and coidentity $e : G \to 0$.  Since $0$ is both initial and terminal in $\Mon$, we of course have $e=0$.  The fact that $e$ is a two-sided identity for multiplication is the requirement that $ef+g = \Id_G$ and $f+eg = \Id_G$ which means $f = \Id_G$, $g = \Id_G$, since $e=0$.  The fact that $i$ is a two-sided inverse means $if+g = e$ and $f+ig = e$, which, in light of the other equalities just says $i + \Id = 0$ and $\Id+i=0$, so $G$ has to be a group and $i$ has to be the inverse map.  The associativity of multiplication is not needed. 

For the ``In fact," suppose $X$ is a connected group fan.  Let $e : \Spec 0 \to X$ be the identity map for the group structure on $X$ and let $1 \in X$ be the image of the unique point of $\Spec 0$.  The fact that $e$ is a map of \emph{locally} monoidal spaces implies that $\M_{X,1} = \M_{X,1}^*$ is a group.  I claim that $1$ is a closed point of $X$.  To see this, first note that since the underlying space functor preserves products (cf.\ \S\ref{section:inverselimits}), the group structure on $X$ also gives rise to a group structure on its underlying space.  In particular, for any $x \in X$, we have a homeomorphism $x : X \to X$ given by ``(left) multiplication by $x$."  Now suppose $x \in \{ 1 \}^{-}$.  Then ``multiplication by $x^{-1}$" shows that $1 \in \{ x^{-1} \}^{-}$.  Let $\Spec P$ be an affine open neighborhood of $1$ (hence also of $x^{-1}$) in $X$.  Since $\M_{X,1}$ is a group and $1$ is in the closure of $x^{-1}$ in $\Spec P$, Proposition~\ref{prop:Spec}\eqref{generalization} implies that $1=x^{-1}$.  Then ``multiplication by $x$" shows that $1=x$, so the claim is proved.  

Now that the claim is proved, we know $1$ is closed in $\Spec P$, so it must correspond to the unique closed point $P^*$ of $\Spec P$ (cf.\ Proposition~\ref{prop:Spec}).  But then $P=\M_{P,P^*}=\M_{X,1}$ is a group, so our affine open neighborhood $\Spec P$ consists solely of the point $1$, so $1$ is both open and closed in $X$, hence $X = \{ 1 \} = \Spec \M_{X,1}$ because $X$ is connected.  The group structure on $X$ is the ``usual" one by what we proved above.  \end{proof}  

For any monoid $P$, there is an obvious coaction $a : P \to P^{\rm gp} \oplus P$ of $P^{\rm gp}$ on $P$ given by $a(p) := (p,p)$, suppressing notation for $P \to P^{\rm gp}$.  This action is the \emph{universal action} of a group object of $\Mon^{\rm op}$ on $P$ in the sense that any group coaction $P \to G \times P$ gives rise to a group homomorphism $g : P^{\rm gp} \to G$, so that the action of $G$ on $P$ in $\Mon^{\rm op}$ is induced by the action of $P^{\rm gp}$ and the map of group objects $g^{\rm op} : G \to P^{\rm gp}$ in $\Mon^{\rm op}$.  Notice that the universal coaction map makes $P^{\rm gp} \oplus P$ a free $P$-module with basis $P^{\rm gp} \oplus 0 \subseteq P^{\rm gp} \oplus P$.  In particular, the universal action is flat.

Let us describe the universal action \be \Spec a : \Spec (P^{\rm gp} \oplus P) & \to & \Spec P \ee in geometric terms.  First of all, since $\Spec P^{\rm gp}$ has one point, Theorem~\ref{thm:Specinverselimits} ensures that $\Spec$ of the inclusion $i_2 : P \to P^{\rm gp} \oplus P$ of the second summand and $\Spec$ of the projection $\pi_2 : P^{\rm gp} \oplus P \to P$ induce homeomorphisms on the level of topological spaces.  Since we clearly have $a^{-1} \pi_2^{-1}(F) = i_2^{-1} \pi_2^{-1}(F)$ for each $F \leq P$, the map $\Spec a$ must, on the level of topological spaces, \emph{be the same homeomorphism as} $\Spec i_2$.  Let us hence suppress this common homeomorphism from the notation and regard $\Spec P$ as the underlying topological space of $\Spec (P^{\rm gp} \oplus P)$ by identifying a face $F \leq P$ with the corresponding face $P^{\rm gp} \oplus F \leq P^{\rm gp} \oplus P$.  Then the structure sheaf of $\Spec (P^{\rm gp} \oplus P)$ is just $\u{P}^{\rm gp} \oplus \M_P$, where $\u{P}^{\rm gp}$ is the constant sheaf of abelian groups associated to $P^{\rm gp}$.  The map $a^\dagger : \M_P \to \u{P}^{\rm gp} \oplus \M_P$ is given by $s \mapsto (s_P,s)$, where $s_P \in P^{\rm gp} = \M_{P,P}$ is the stalk at the generic point.

In the remainder of \S\ref{section:groupobjects} we shall study the formation of quotients by group actions in the category of fans.  Our results are analogous to the basic facts about formation of quotients of schemes by (finite and flat) group schemes discussed in great generality in \cite[Expos\'e V]{SGA3} (see \cite[V.4.1]{SGA3} in particular).  

Let us first make some remarks about quotients in the category $\MS$ of monoidal spaces (\S\ref{section:locallymonoidalspaces}).  Much as in the category $\RS$ of ringed spaces, direct limits in $\MS$ are formed in the ``obvious" way:  To construct ``the" direct limit $X$ of a functor $i \mapsto X_i$, we first form ``the" direct limit $X$ of the underlying functor $i \mapsto X_i$ to topological spaces.  We let $e_i : X_i \to X$ denote the structure map to this topological direct limit and we endow the topological space $X$ with the structure sheaf $\M_X$ defined as the inverse limit of the functor $i \mapsto (e_i)_* \M_{X_i}$ to the category of sheaves of monoids on $X$.

\begin{lem} \label{lem:quotients} Let $P$ be a monoid, $A$ a torsion abelian group, $a : P^{\rm gp} \to A$ a group homomorphism.  This $a$ yields an action of the group fan $\GG(A)$ on the fan $\Spec P$.  Any action of $\GG(A)$ on $\Spec P$ arises in this manner.  Suppressing notation for $P \to P^{\rm gp}$, set \be P' & := & \{ p \in P : a(p) = 0 \}, \ee so that $P'$ is a submonoid of $P$. \begin{enumerate} \item \label{quotients1} The map $\Spec P \to \Spec P'$ induced by $P' \into P$ is a homeomorphism of topological spaces and is the categorical quotient of the $\GG(A)$ action on $\Spec P$ in the category of fans, the category of locally monoidal spaces, and the category of monoidal spaces.  \item \label{quotients2} The formation of $P'$ ``commutes with localization" in the sense that for any submonoid $S \subseteq P$, if we set $S' := S \cap P'$, then the natural inclusion of $(S')^{-1} P'$ into \be (S^{-1}P)' & := & \{ p \in S^{-1}P : a(p)=0 \} \ee is bijective. \end{enumerate} \end{lem}

\begin{proof} For $p \in P$, let $n=n(p) \in \ZZ_{>0}$ be the order of the element $a(p)$ in the torsion group $A$.  Then $np \in P'$ for any $p \in \Spec P$, so the inclusion $P' \into P$ is dense (Definition~\ref{defn:dense}) and hence the map $\Spec P \to \Spec P'$ is a homeomorphism by Theorem~\ref{thm:densespecinvariance}.  For \eqref{quotients2}, note that a typical element of $(S^{-1}P)'$ takes the form $p-s$ for some $p \in P$, $s \in S$ with $a(p) = a(s) =: a$.  Let $n$ be the order of $a$ in $A$.  Then $p+(n-1)s \in P'$, $ns \in S'$, and $p-s$ is also the difference of the latter two elements, so it is in $(S')^{-1}P'$ as desired.  To prove the rest of \eqref{quotients1} it is enough to prove that $\Spec P \to \Spec P'$ is the quotient of $\Spec P$ by $\GG(A)$ in the category $\MS$ of monoidal spaces---i.e.\ that $\Spec$ of the equalizer diagram of monoids \bne{monoidequalizer} & \xym{ P' \ar[r] & P \ar@<0.5ex>[r]^-{p \mapsto (p,0)} \ar@<-0.5ex>[r]_-{p \mapsto (p,a(p))} & P \oplus A } \ene is a coequalizer diagram \bne{coequalizer} & \xym{ \Spec P \oplus A \ar@<0.5ex>[r] \ar@<-0.5ex>[r] & \Spec P \ar[r] & \Spec P' } \ene in $\MS$.  (The meaning of this is discussed just before the statement of the lemma.)  The parallel arrows agree on the level of topological spaces (in fact both are the same homeomorphism), so saying that \eqref{coequalizer} is a coequalizer diagram on the level of topological spaces is equivalent to saying that $\Spec P \to \Spec P'$ is a homeomorphism, which we proved above.  It remains to show that \bne{monoidequalizer2} & \xym{ \M_{P'} \ar[r] & \M_P \ar@<0.5ex>[r]^-{p \mapsto (p,0)} \ar@<-0.5ex>[r]_-{p \mapsto (p,a(p))} & \M_{P \oplus A} } \ene is an equalizer diagram of sheaves of monoids on $\Spec P'$ (notation for pushforwards along the various homeomorphisms is suppressed).  This can be checked on stalks.  The stalk of \eqref{monoidequalizer2} at a point $F' \leq P'$ of $\Spec P'$ with corresponding point $F \leq P$ in $\Spec P$ is \bne{equalizer3} & \xym{ (F')^{-1}P' \ar[r] & F^{-1}P \ar@<0.5ex>[r]^-{p \mapsto (p,0)} \ar@<-0.5ex>[r]_-{p \mapsto (p,a(p))} & F^{-1}P \oplus A, } \ene which is an equalizer diagram by \eqref{quotients2} (applied with $S=F$). \end{proof}

\begin{lem} \label{lem:quotients2} The action of any connected group fan on any fan is trivial on the level of topological spaces. \end{lem}

\begin{proof} By Lemma~\ref{lem:groupobjects} any such group fan is isomorphic to $\GG(A)$ for an abelian group $A$ so, in particular, its underlying topological space has a single point.  Since the underlying space functor $\Fans \to \Top$ preserves products (Theorem~\ref{thm:fansinverselimits}), it takes group objects to group objects and actions to action, so such an action cannot help but be trivial because the topological group underlying $\GG(A)$ must be trivial. \end{proof}

\begin{thm} \label{thm:quotients} Let $A$ be a torsion abelian group, $X$ a fan equipped with an action of the group fan $\GG(A)$.  Then the categorical quotient $X \to X / \GG(A)$ of $X$ by $\GG(A)$, calculated in the category $\MS$ of monoidal spaces, is a map of fans which is a homeomorphism on underlying topological spaces and is also the quotient of $X$ by $\GG(A)$ in the categories of locally monoidal spaces and fans. \end{thm}

\begin{proof} The action of $\GG(A)$ on $X$ is trivial on the level of topological spaces by Lemma~\ref{lem:quotients2}.  The fact that $X \to X / \GG(A)$ is a homeomorphism on the level of topological spaces is then clear from the construction of direct limits in $\MS$.  Furthermore, any affine neighborhood $U = \Spec P$ of a point $x \in X$ is automatically $\GG(A)$ invariant, so we have an induced $\GG(A)$ action on $U$.  It is clear from the construction of quotients in $\MS$ that the $\MS$ quotient $U \to U/\GG(A)$ is a neighborhood of $x$ in $X \to X / \GG(A)$.  Lemma~\ref{lem:quotients} ensures that $U \to U/\GG(A)$ is a map of fans which is the quotient of $U$ by $\GG(A)$ in $\MS$, $\LMS$, and $\Fans$.  This proves that $X \to X / \GG(A)$ is a map of fans, since this can be checked locally on $X \to X/ \GG(A)$.  Since it is a map of fans which is the quotient by $\GG(A)$ in $\MS$, it follows easily that it is also the quotient by $\GG(A)$ in $\LMS$ and $\Fans$. \end{proof}

\subsection{The torus of a fan} \label{section:torus}  Let $X$ be a fan.  The subset \be T_X & := & \{ x \in X : \M_{X,x}^* = \M_{X,x} \} \ee of $X$ is stable under generalization, the topology it inherits from $X$ is the discrete topology, and every point $x \in X$ has a generalization in $T_X$ (in particular, $T_X$ is dense in $X$).  Indeed, if $\Spec P \subseteq X$ is an open affine subset of $X$, then it follows from Proposition~\ref{prop:Spec} that the intersection of $T_X$ and $\Spec P$ consists of the single face $P$ itself, which is stable under generalization in $\Spec P$ and is a generalization of each other point of $\Spec P$.  By abuse of notation, we view the discrete topological space $T_X$ as a fan by setting \be T_X & := & \coprod_{x \in T_X} \Spec \M_{X,x}, \ee where the coproduct is over the previously defined set $T_X$.  There is an evident morphism of fans $T_X \to X$ which is manifestly a strict embedding (Definition~\ref{defn:embedding}).  We refer to the fan $T_X$ as the \emph{torus} of $X$.

If we assume that $X$ is locally finite as a topological space (Definition~\ref{defn:locallyfinite}) (this holds if $X$ is a locally finite type fan by Proposition~\ref{prop:locallyfinitetypefans}),  then stability of $T_X \subseteq X$ under generalization implies that $T_X \subseteq X$ is an open subfan.  

For each $x \in T_X$, the corresponding component $\Spec \M_{X,x}$ of $T_X$ acts on $X$ via a map of sheaves \be \M_X & \to & \u{\M}_{X,x} \oplus \M_X \\ s & \mapsto & (s_x,s). \ee  Here $s_x$ denotes the stalk of $s$ at $x$ if $x$ is in the open set over which $s$ is defined, or $0$ otherwise.  This action is evidently trivial away from the closure of $\{ x \}$.

The torus $T_X$ and the action of its components on $X$ is natural in $X$.  The point is that $f : X \to Y$ must take $T_X$ to $T_Y$ because if $\M_{X,x}$ is a group, then $\M_{Y,f(x)}$ has to be a group because $f^\dagger_x : \M_{Y,f(x)} \to \M_{X,x}$ is local.  The map of affine group fans $\Spec f^\dagger_x$ intertwines the action of $\Spec \M_{X,x}$ on $X$ and the action of $\Spec \M_{X,f(x)}$ on $Y$.

\section{Algebraic realization} \label{section:algebraicrealization}  To any fan $X$, one can associate a scheme $\AA(X)$, called the \emph{algebraic realization} of $X$.  The construction yields a functor \bne{algebraicrealization} \AA : \Fans & \to & \Sch, \ene called the \emph{algebraic realization functor} from the category $\Fans$ of fans (\S\ref{section:fans}) to the category $\Sch$ of schemes.  In this section we will establish some basic properties of $\AA(X)$ and the functor \eqref{algebraicrealization} needed elsewhere in the paper.

In fact, the scheme $\AA(X)$ also comes equipped with a natural log structure and the functor $\AA$ can be viewed as a functor to the category of \emph{log schemes}.  Indeed, there are a whole host of possible ``realization functors" taking values in various ``categories of spaces."  This general point of view is described in \cite[\S3]{GM1}.

\subsection{Categorical interpretation} \label{section:categoricalinterpretation}  The algebraic realization functor \eqref{algebraicrealization} can be characterized by a ``categorical universal property" similar to the characterization of the functor $\Spec$ of \S\ref{section:Spec}.  Here we give a brief summary of this point of view---for generalizations and further details, we refer the reader to \cite[\S4]{GM1}.

There is an obvious forgetful functor \bne{LRStoLMS} \LRS & \to & \LMS \\ \nonumber X & \mapsto & |X| \ene defined by regarding a locally ringed space $X = (X,\O_X)$ as a locally monoidal space $|X| = (X,\O_X)$ whose structure sheaf is $\O_X$, regarded as a sheaf of monoids under multiplication.  (Later we will simply write $X$ instead of $|X|$, but we will retain the notation $|X|$ for the rest of the paragraph for the sake of clarity.)  Quite generally, given $Y \in \LMS$, one can form a locally ringed space $\AA(Y) \in \LRS$ equipped with an $\LMS$-morphism $\tau_Y : |\AA(Y)| \to Y$ (we write ``$\tau$" instead of ``$\tau_Y$" if $Y$ is clear from context), called the \emph{orbit map}, with the following universal property:  Given any $X \in \LRS$ and any $\LMS$-morphism $f : |X| \to Y$, there is a unique $\LRS$-morphism $g : X \to \AA(Y)$ such that $f = \tau |g|$.  This construction yields a functor \bne{LMStoLRS} \AA : \LMS & \to & \LRS \ene right adjoint to \eqref{LRStoLMS}, so that there is a bijection \bne{Aadjunction} \Hom_{\LRS}(X,\AA(Y)) & = & \Hom_{\LMS}(|X|,Y) \ene natural in $X \in \LRS$ and $Y \in \LMS$.  See \cite[\S8.1]{GM1} for details.

It turns out that, if $Y \in \LMS$ is a \emph{fan} (Definition~\ref{defn:fan}), then the locally ringed space $\AA(Y)$ is in fact a \emph{scheme} (locally finite type over $\ZZ$ when $Y$ is locally finite type), called the \emph{algebraic realization} of $Y$, so that the functor \eqref{LMStoLRS} can be viewed as a functor which will be our algebraic realization functor \eqref{algebraicrealization}.  The adjunction bijection \eqref{Aadjunction} specializes to a bijection \bne{ASchadjunction} \Hom_{\LRS}(X,\AA(Y)) & = & \Hom_{\LMS}(|X|,Y), \ene natural in $X \in \LRS$, $Y \in \Fans$, which characterizes the algebraic realization functor \eqref{algebraicrealization} up to unique isomorphism.  (One only needs \eqref{ASchadjunction} for $X \in \Sch$ to get this characterization.)  In other words, for a fan $Y$, the scheme $\AA(Y)$ represents the presheaf \bne{AY} \Sch^{\rm op} & \to & \Sets \\ \nonumber X & \mapsto & \Hom_{\LMS}(|X|,Y). \ene The natural bijection \eqref{ASchadjunction}, which is a sort of ``right adjointness," implies, for formal reasons, that:

\begin{prop} \label{prop:algebraicrealization} The algebraic realization functor \eqref{algebraicrealization} preserves inverse limits. \end{prop}

\subsection{Explicit construction} \label{section:explicitconstruction}  Just as the $\Spec$ functor \eqref{SpecLMS} can be described explicitly, so too can the algebraic realization functor \eqref{algebraicrealization}.  It is helpful to keep in mind both the abstract categorical characterization and the explicit construction.

For a monoid $P$, the algebraic realization of the (affine) fan $\Spec P$, will be the (affine) scheme $\AA(P) := \Spec \ZZ[P]$, where $\ZZ[P]$ is the monoid algebra associated to $P$ as in \S\ref{section:monoidalgebra}.  In this case, the orbit map \bne{tauP} \tau_P : | \AA(P) | & \to & \Spec P \ene is the $\LMS$-morphism corresponding to the obvious monoid homomorphism $P \to \ZZ[P]$ under the natural bijections \be \Hom_{\LMS}(|\AA(P)|,\Spec P) & = & \Hom_{\Mon}(P,\O_{\AA(P)}(\AA(P))) \\ & = & \Hom_{\Mon}(P,\ZZ[P]), \ee where the first bijection is a special case of \eqref{Specadjunction} and the second comes from the definition of $|\AA(P)|$ and the formula for global sections of the structure sheaf of an affine scheme.  The adjunction bijection \eqref{Aadjunction} in this case is the bijection \bne{affineAadjunction} \Hom_{\LRS}(X,\AA(P)) & = & \Hom_{\LMS}(|X|,\Spec P) \ene obtained as the composition \be \Hom_{\LRS}(X,\AA(P)) & = & \Hom_{\An}(\ZZ[P],\O_X(X)) \\ & = & \Hom_{\Mon}(P,\O_X(X)) \\ & = & \Hom_{\LMS}(|X|,\Spec P) \ee of ``the universal property of the affine scheme $\Spec \ZZ[P]$" \cite[Err I.1.8.1]{EGA}, the adjointness property of the monoid algebra \eqref{monoidalgebraadjunction}, and the ``universal property of the affine fan $\Spec P$" \eqref{Specadjunction}.  One can check, for example, by running through the explicit formulas for these bijections, that \eqref{affineAadjunction} can de described directly in terms of the orbit map \eqref{tauP} as in \S\ref{section:categoricalinterpretation}.

Now, in general, if the fan $Y$ is obtained by ``gluing together affine fans $Y_i = \Spec P_i$," then $\AA(Y)$ can be obtained by gluing the corresponding affine schemes $\AA(P_i) = \Spec \ZZ[P_i]$ in ``the same way."  A detailed explanation of this point is given in \cite[\S4.5]{GM1}.  What makes the whole thing work is the following:

\begin{prop} \label{prop:algebraicrealization2} The algebraic realization functor \eqref{algebraicrealization} preserves coproducts and takes an open embedding of fans $U \into Y$ (resp.\ open covers of fans) to the open embeddings of schemes $\AA(U) = \tau_Y^{-1}(U) \into \AA(Y)$; the orbit map for $\AA(U)$ is the restriction of the orbit map $\tau_Y$ for $Y$.  In particular, $\AA$ takes Zariski covers of fans to Zariski covers of schemes. \end{prop}

\begin{proof} \cite[Lemmas~4.5.6, 4.5.7]{GM1} \end{proof}

Actually, in some sense, one \emph{first} proves Proposition~\ref{prop:algebraicrealization2} and \emph{uses} this result to help construct the functor $\AA$.  Indeed, notice that the construction of the realization functor $\AA$ amounts to showing that the presheaf \eqref{AY} is representable for every fan $Y$.  One can prove, for example, using the ``universal property of an open embedding" that if $\AA(Y)$ is representable for some fan $Y$, and $U \subseteq Y$ is an open subfan of $Y$, then $\AA(U)$ is representable by the open subfan $\tau_Y^{-1}(U)$ of $\AA(Y)$.

\subsection{The orbit map} \label{section:orbitmap}  In particular, Proposition~\ref{prop:algebraicrealization2} says that for an affine open subfan $U = \Spec P \subseteq Y$, $\tau^{-1}(U) = \Spec \ZZ[P]$ is always an affine open subscheme of $\AA(Y)$, so that the morphism $\tau$ is in some sense ``affine."  In the case of an affine fan $\Spec P$, the orbit map \bne{affineorbitmap} \tau_P : \Spec \ZZ[P] & \to & \Spec P \ene (we now drop the $|\slot|$ notation) is given concretely as follows.  We view $\Spec P$ as the set of faces of the monoid $P$.  We view the points of $\Spec \ZZ[P]$ as ring homomorphisms $x : \ZZ[P] \to k$ ($k$ a field), up to the usual notion of equivalence.  Such a ring map yields a monoid homomorphism $x|P : P \to k$ (or just $x : P \to k$ if there is no chance of confusion) by restriction to $P \subseteq \ZZ[P]$.  Then $\tau(x) = (x|P)^{-1}(k^*)$.

Let $k$ be a field.  Then for any scheme $X$, we have a map $X(k) \to X$ from the set $X(k)$ of $k$-points of $X$ to (the set underlying) $X$ given by taking a $k$-point $x : \Spec k \to X$ to the point $x(\Spec k) \in X$.  The image of this map is the set of points $x \in X$ for which the residue field $k(x)$ embeds in $k$.  For a fan $X$, we write $\tau(k) : \AA(X)(k) \to X$ for the composition of $\AA(X)(k) \to \AA(X)$ and $\tau : \AA(X) \to X$.

Let $P$ be a monoid.  We define a map (of sets!) \be s : \Spec P \to (\Spec \ZZ[P])(k) \ee by taking a face $F$ to the unique ring map $s(F) : \ZZ[P] \to k$ with $s(F)([p]) = 1$ for $p \in F$ and $s(F)([p]) = 0$ for $p \in P \setminus F$.  (Here $[p]$ denotes the image of $p \in P$ in $\ZZ[P]$.  One uses the fact that $F$ is a face to see that $p \mapsto s(F)([p])$ is actually a monoid homomorphism.)  Notice that $\tau(s(F)) = (s(F)|P)^{-1}(k^*)=F$, so that $s$ defines a section of $\tau(k)$.  

If $F \leq P$ is a finitely generated face, then we have a commutative diagram $$ \xym{ \Spec F^{-1}P \ar[r] \ar[d]_{\tau} & \Spec P \ar[d]_\tau \\ \AA(F^{-1}P) = \Spec \ZZ[F^{-1}P] \ar[r] & \Spec \ZZ[P] = \AA(P) }$$ where the horizontal arrows are open embeddings (cf.\ Proposition~\ref{prop:algebraicrealization2}).  One can check that the map $s$ defined for $F^{-1}P$ is just the restriction of the one defined for $P$.  It follows that we can glue the maps $s$ defined for affine fans to construct a section $s : X \to \AA(X)(k)$ of $\tau(k)$ for \emph{any} fan $X$ and any field $k$.  A better way to see this is to give a direct ``global" definition of this map $s$:  Given a fan $X$ and a point $x \in X$, let $t(x) : \M_{X,x} \to k$ be the monoid homomorphism defined by setting $t(x)(m) := 1$ if $m \in \M_{X,x}^*$ and $t(x)(m) := 0$ if $m \in \M_{X,x} \setminus \M_{X,x}^*$.  This $t(x)$ is a local map of monoids, so we can view it as a map of locally monoidal spaces $t(x) : \Spec k \to X$.  As a special case of \eqref{ASchadjunction}, we have \bne{specialcase} \AA(X)(k) & = & \Hom_{\LMS}(\Spec k, X). \ene  We let $s(x) \in \AA(X)(k)$ be the element corresponding to $t(x)$ under this bijection.

If there is any ambiguity in $k$ will we write $s(k)$ instead of $s$, though there is probably never any chance of confusion because if $k \subseteq K$ is a field extension, then $s(K) : X \to \AA(X)(K)$ is just the composition of $s(k) : X \to \AA(X)(k)$ and the natural map $\AA(X)(k) \to \AA(X)(K)$.

\begin{defn} \label{defn:canonicalsection} For a fan $X$, the map $s(k) : X \to \AA(X)(k)$ defined above will be called the \emph{canonical section} of $\tau(k)$.  As mentioned above, we often write $s$ instead of $s(k)$.  Composing $s(k)$ with $\AA(X)(k) \to \AA(X)$ yields a section $s : X \to \AA(X)$ of $\tau : \AA(X) \to X$, called the \emph{canonical section} of $\tau$.  \end{defn}

\begin{rem} The map $s : X \to \AA(X)$ is \emph{not} generally a continuous map of topological spaces.  For example, if $k=\QQ$, $P = \NN$ and $U \subseteq \Spec \ZZ[\NN] = \Spec \ZZ[x]$ is the basic affine open subset associated to $1-x \in \ZZ[x]$, then $s_k^{-1}(U) = \{ \NN \}$ consists solely of the closed point of $\Spec \NN$; this is not open in $\Spec \NN$.  One should view $s$ only as a map of \emph{sets} and nothing more.  \end{rem}

\begin{lem} \label{lem:orbitmapsurjective} For any fan $X$, the orbit map $\tau : \AA(X) \to X$ is surjective.  If $f : X \to Y$ is a map of fans whose algebraic realization $\AA(f)$ is surjective, then $f$ is surjective. \end{lem}

\begin{proof} The orbit map is surjective because it has a set-theoretic section (the canonical section of Definition~\ref{defn:canonicalsection}).  The second statement follows by considering the commutative diagram $$ \xym{ \AA(X) \ar[r]^-{\tau_X} \ar[d]_{\AA(f)} & X \ar[d]^f \\ \AA(Y) \ar[r]^-{\tau_Y} & Y } $$ with surjective horizontal arrows. \end{proof}

\subsection{Stratification} \label{section:stratification} For a fan $X$, the structure map $\tau : \AA(X) \to X$ and the canonical section $s : X \to \AA(X)$ can be viewed as a more general, more structured version of the \emph{orbit-cone correspondence} familiar from classical toric geometry.  This can be most easily explained by first examining the affine case where $X = \Spec P$.  For each face $F \leq P$, we have a ring map $\ZZ[F] \into \ZZ[P]$ induced from the inclusion $F \into P$.  This ring map has a canonical retract $\ZZ[P] \to \ZZ[F]$ which Molcho and I call the \emph{wrong way map} \cite[\S4.4]{GM1}.  To define it, the key point is that $\ZZ[F]$, viewed as a monoid is a pointed monoid (Definition~\ref{defn:pointedmonoid}) with $0 \in \ZZ[F]$ is the ``infinity" element.  There is a monoid homomorphism $P \to \ZZ[F]$ defined by taking $p$ to $[p]$ if $p \in F$ and to $0$ if $p \in P \setminus F$.  (One uses that $F$ is a face to check that this is a monoid homomorphism.)  The wrong map map is the ring homomorphism $\ZZ[P] \to \ZZ[F]$ corresponding to the monoid homomorphism $P \to \ZZ[F]$ under the adjunction isomorphism \be \Hom_{\An}(\ZZ[P],\ZZ[F]) & = & \Hom_{\Mon}(P,\ZZ[F]). \ee  Since the wrong way map takes $[f] \in \ZZ[P]$ to $[f] \in \ZZ[F]$ for every $f \in F$, it retracts the inclusion $\ZZ[F] \into \ZZ[P]$.

\begin{rem} The retract diagram of rings $$ \ZZ[F] \to \ZZ[P] \to \ZZ[F] $$ constructed above arises a bit more naturally in the context of \emph{pointed} monoids (\S\ref{section:monoidschemes}).  It is just the retract diagram of rings obtained by applying the pointed monoid algebra functor \eqref{Zslot3} to the retract diagram of pointed monoids \eqref{retract2}. \end{rem}

For a face $F \leq P$, we can consider the composition \bne{comp} \ZZ[P] \to \ZZ[F] \to \ZZ[F^{\rm gp}] \ene of the wrong way map and the map of monoid algebras induced by the structure map $F \to F^{\rm gp}$.  Taking $\Spec$ of \eqref{comp} we obtain a map of schemes \bne{stratum} \Spec \ZZ[F^{\rm gp}] & \to & \Spec \ZZ[P]. \ene

\begin{prop} \label{prop:affinestratification} Let $P$ be a monoid. \begin{enumerate} \item \label{partition} For each $F \leq P$, the map \eqref{stratum} is one-to-one on the level of topological spaces.  For distinct $F$, these maps have disjoint images.  Their images cover $\Spec P$, so we have a partition \be \Spec \ZZ[P] & = & \coprod_{F \leq P} \Spec \ZZ[F^{\rm gp}] \ee on the level of underlying sets. \item \label{orbitmaprelationship} For $F \in \Spec P$, the subset $\Spec \ZZ[F]^{\rm gp}$ of $\Spec \ZZ[P]$ appearing in the partition of \eqref{partition} is nothing but the preimage of $F$ under the orbit map $\tau_P : \Spec \ZZ[P] \to \Spec P$, so the partition in \eqref{partition} is nothing more than the partition of $\Spec \ZZ[P]$ given by ``pulling back" the partition \be \Spec P & = & \coprod_{F \leq P} \{ F \} \ee of $\Spec P$ along $\tau_P$. \item \label{locallyclosed} If $P$ is finitely generated, then for each $F \leq P$, the map \eqref{stratum} is a locally closed embedding.  Furthermore, there are only finitely many faces $F \leq P$, so the partition in \eqref{partition} is a partition of the topological space $\Spec \ZZ[P]$ into finitely many locally closed subspaces. \item \label{stratification} If $P$ is fine, then the partition in \eqref{partition} is a finite \emph{stratification} of $\Spec \ZZ[P]$ into finitely many locally closed subspaces (``strata") in the sense that the closure of any stratum is a union of strata. \end{enumerate} \end{prop}

\begin{proof} \eqref{partition}:  View a point of $\Spec \ZZ[P]$ as an equivalence class of ring homomorphisms $x : \ZZ[P] \to k$, $k$ a field, up to the equivalence relation generated by extending the field $k$, as usual.  From such an $x$ we get a ring homomorphism $x_{\ZZ} : \ZZ \to k$ and, by composing with $P \to \ZZ[P]$, a monoid homomorphism $x_P : P \to k$.  The latter determines a face $F_x := x^{-1}(k^*)$ of $P$ and a group homomorphism $\ov{x} : F_x^{\rm gp} \to k^*$.  By reversing these steps, we see that the data of $x_{\ZZ}$, $\ov{x}$ determine a point $y$ of $\Spec \ZZ[F_x^{\rm gp}]$ whose image under \eqref{stratum} is $x$.  We can clearly recover $x$ from $y$ and it is clear from the same considerations that any point $y$ of $\Spec \ZZ[F^{\rm gp}]$ will map, under $\Spec \ZZ[F^{\rm gp}] \to \Spec \ZZ[P]$, to a point $x$ of $\Spec \ZZ[P]$ with $F_x = F$.

\eqref{orbitmaprelationship}:  This is clear from the description of the orbit map at the beginning of \S\ref{section:orbitmap}.

\eqref{locallyclosed}:  Recall (Corollary~\ref{cor:faces}) that every face $F$ of the finitely generated monoid $P$ is finitely generated, hence the map $\ZZ[F] \to \ZZ[F^{\rm gp}]$ appearing in \eqref{comp} ``is" the localization of $\ZZ[F]$ at a single element (if $f_1,\dots,f_n$ generate $F$, it is the localization at $[f_1+\cdots+f_n]$), so it is an open embedding.  The retract map $\ZZ[P] \to \ZZ[F]$ is surjective, so, taking $\Spec$, we find that \eqref{stratum} factors as an open embedding followed by a closed embedding.   Corollary~\ref{cor:faces} also gives the finiteness of $\Spec P$.

\eqref{partition}:  The same Corollary says that every face $F$ of $P$ is also fine, hence $F \into F^{\rm gp}$ is injective, hence $\ZZ[F]$ is a domain (since it is contained in $\ZZ[F^{\rm gp}]$, which is certainly a domain), hence $\Spec \ZZ[F]$ is irreducible, hence the non-empty open subspace $\Spec \ZZ[F^{\rm gp}]$ is dense in $\Spec \ZZ[F]$.  Since $\Spec \ZZ[F]$ is closed in $\Spec \ZZ[P]$ (since $P \to F$ is surjective), the closure of the stratum $\Spec \ZZ[F^{\rm gp}]$ is $\Spec \ZZ[F]$, which is the union of the strata $\Spec \ZZ[G^{\rm gp}]$ for $G \leq F$. \end{proof}

\begin{example} \label{example:stratification} It is worth thinking about the partition of Proposition~\ref{prop:affinestratification} for a finitely generated monoid $P$ which is not integral.  Recall the monoid \be P & := & \langle a,b : a+b = a \rangle \ee from Example~\ref{example:SpecPnonintegral}.  We have $\ZZ[P] = \ZZ[x,y]/(x(1-y))$ where $x=[a]$, $y=[b]$.  The ``strata" of $\Spec \ZZ[P]$ associated to the three faces $0$, $\NN b$, $P$ are, respectively, the loci where $(x,y)$ is $(0,0)$ (resp.\ $(0, \neq 0)$, $( \neq 0 , 1)$).  The points in the second ``stratum" where $(x,y) = (0,1)$ (there is one such point for each point of $\Spec \ZZ$) are in the closure of the third ``stratum", but the other points of the second ``stratum" are not in this closure, so the closure of the third ``stratum" is not a union of strata.  Note that the third stratum is $\Spec \ZZ[P^{\rm gp}] = \Spec \ZZ[\ZZ]$.  It is not dense in $\Spec \ZZ[P]$. \end{example}

Since the construction of $\AA(Y)$ and the orbit map $\tau_Y$ is compatible with passing to an open subfan (Proposition~\ref{prop:algebraicrealization2}), Proposition~\ref{prop:affinestratification} easily ``globalizes" as follows:

\begin{prop} \label{prop:stratification} Let $X$ be a fan. \begin{enumerate} \item \label{partition2} For each $x \in X$, we have a natural map of schemes $\AA(\M_{X,x}^*) \to \AA(X)$ which is one-to-one on the level of topological spaces.  (Note $\AA(\M_{X,x}^*) = \Spec \ZZ[\M_{X,x}^*]$.)  For distinct $x$, these maps have disjoint images.  Their images cover $\AA(X)$, so we have a partition \be \AA(X) & = & \coprod_{x \in X} \AA(\M_{X,x}^*) \ee on the level of underlying sets. \item \label{orbitmaprelationship2} For $x \in X$, the subset $\AA(\M_{X,x}^*)$ of $\AA(X)$ appearing in the partition \eqref{partition} is nothing but the preimage of $x$ under the orbit map $\tau_X : \AA(X) \to X$, so the partition in \eqref{partition} is nothing more than the partition of $\AA(X)$ given by ``pulling back" the ``trivial" partition \be X & = & \coprod_{x \in X} \{ x \} \ee of $X$ along $\tau_X$. \item \label{locallyclosed2} If $X$ is locally finite type, then for each $x \leq X$, the map in \eqref{partition2} is a locally closed embedding.  Furthermore, the partition in \eqref{partition} is a locally finite partition of the topological space $\AA(X)$ into finitely many locally closed subspaces. \item \label{stratification2} If $X$ is fine and quasi-compact, then the partition in \eqref{partition} is a finite \emph{stratification} of $\AA(X)$ into finitely many locally closed subspaces (``strata") in the sense that the closure of any stratum is a union of strata. \end{enumerate} \end{prop}

Since $\AA$ preserves inverse limits, it takes group objects to group objects.  In particular, for any abelian group $A$, $\AA$ takes the (abelian) group fan $\GG(A) = \Spec A$ of \S\ref{section:groupobjects} to an abelian group object in schemes, which we shall also abusively denote $\GG(A)$.  The scheme underlying $\GG(A)$ is $\Spec \ZZ[A]$.  A group scheme of this form is often called a ``diagonalizable group scheme".  The scheme realization of the ``universal action" of \S\ref{section:groupobjects} yields a natural action of $\GG(P):=\GG(P^{\rm gp})$ on $\AA(P)$.  We have a commutative $\LMS$ diagram relating these actions: $$ \xym{ \GG(P) \times \AA(P) \ar[r] \ar[d]_{\tau} & \AA(P) \ar[d]^{\tau} \\ (\Spec P^{\rm gp}) \times (\Spec P) \ar[r] & \Spec P } $$  Recall that the universal action is trivial on the topological spaces of fans so the bottom horizontal arrow is a homeomorphism of topological spaces.  In other words, on the level of topological spaces, $\GG(P)$ acts on $\AA(P)$ fiberwise for $\tau$.  For a field $k$, the induced action of $\GG(P)(k)$ on $\AA(P)(k)$ can be described explicitly as follows:  A point $x \in \AA(P)(k)$ may be viewed as a monoid homomorphism $x : P \to k$ and a point $u \in \GG(P)(k)$ may be viewed as a group homomorphism $u : P^{\rm gp} \to k^*$.  Then $u \cdot x \in \AA(P)(k)$ corresponds to the monoid homomorphism $u \cdot x : P \to k$ defined by $(u \cdot x)(p) := u(p)x(p)$.  This group action lies at the heart of the theory of toric varieties.

Here is a ``global" variant of the above construction:  Given a fan $X$ and a point $x \in X$, set $\GG(x) := \Spec \ZZ[\M_{X,x}^{\rm gp}]$.  Recall from Proposition~\ref{prop:algebraicrealization2} that $\GG(x) = \tau_X^{-1}(x)$ is just the fiber of the orbit map over $x$.  Then we have an action $(u,\tilde{x}) \mapsto u \cdot \tilde{x}$ of \be \GG(x)(k) & = & \Hom_{\Ab}(\M_{X,x}^{\rm gp},k^*) \ee on $\tau(k)^{-1}(x) \subseteq \AA(X)(k)$, defined as follows:  View $\tilde{x}$ as a local map of monoids $\tilde{x} : \M_{X,x} \to k$ by using the adjunction bijection \eqref{specialcase}, as we did above.  Then $u \cdot \tilde{x}$ ``is" the local map of monoids $\M_{X,x} \to k$ defined by $(u \cdot \tilde{x})(m) := u(m) \tilde{x}(m)$.

\begin{lem} \label{lem:orbitmap2} Let $X$ be a fan, $k$ a field, $x \in X$, $\tau(k) : \AA(X)(k) \to X$ the map defined above.  \begin{enumerate} \item The set $\tau(k)^{-1}(x)$ is non-empty, hence the map $\tau(k)$ is surjective. \item  For any $\tilde{x} \in \tau(k)^{-1}(x)$, the stabilizer $\Stab \tilde{x}$ of $\tilde{x}$ for the action of $\GG(x)(k)$ on $\tau(k)^{-1}(x)$ defined above is the kernel of \be \GG(x)(k) = \Hom_{\Ab}(\M_{X,x}^{\rm gp},k^*) & \to & \Hom_{\Ab}(\M_{X,x}^*,k^*). \ee \end{enumerate}  Assume that $\M_{X,x}$ is integral\footnote{It is enough to assume that $\M_{X,x}$ is \emph{quasi-integral}, meaning the natural group homomorphism $\M_{X,x}^* \to \M_{X,x}^{\rm gp}$ is injective.} and that at least one of the following holds: \begin{enumerate} \item $k$ is algebraically closed \item $\M_{X,x}$ is fs \end{enumerate}  Then the action of $\GG(x)(k)$ on $\tau(k)^{-1}(x)$ is transitive, the stabilizer of any $\tilde{x} \in \tau(k)^{-1}(x)$ is identified with $\Hom_{\Ab}(\ov{\M}_{X,x}^{\rm gp},k^*)$, and we have natural bijections \be \tau(k)^{-1}(x) & = & \Hom_{\Ab}(\M_{X,x}^{\rm gp},k^*) / \Hom_{\Ab}(\ov{\M}_{X,x}^{\rm gp},k^*) \\ & = & \Hom_{\Ab}(\M_{X,x}^*,k^*). \ee  \end{lem}

\begin{proof} The set $\tau(k)^{-1}(x)$ is non-empty since $s(x) \in \tau(k)^{-1}(x)$, where $s : X \to \AA(X)(k)$ is the canonical section of Definition~\ref{defn:canonicalsection}.  The second statement is clear from the definition of the action in question.  For the final statement, the assumption that $\M_{X,x}$ is integral (quasi-integral will do) implies that we have a short exact sequence of abelian groups \bne{SESA} & 0 \to \M_{X,x}^* \to \M_{X,x}^{\rm gp} \to \ov{\M}_{X,x}^* \to 0. \ene Here we used the natural isomorphism $$\ov{\M}_{X,x}^{\rm gp} = (\M_{X,x}/\M_{X,x}^*)^{\rm gp} = \M_{X,x}^{\rm gp} / \M_{X,x}^*$$ coming from the fact that groupification preserves direct limits.  If $k$ is algebraically closed, then $k^*$ is divisible, and if $\M_{X,x}$ is fs, then $\ov{\M}_{X,x}$ is sharp fs and hence $\ov{\M}_{X,x}^{\rm gp}$ is free \cite[1.2.3]{GM1}.  Under either assumption, the sequence obtained by applying $\Hom_{\Ab}(\slot,k^*)$ to \eqref{SESA} is short exact and the third statement follows.  To make the final bijection truly canonical, we use $s(x) \in \tau(k)^{-1}(x)$ to make the identification \be \tau(k)^{-1}(x) & = & \left ( \GG(x)(k)  / \Stab s(x) \right ) \cdot s(x). \ee  \end{proof}

Of course, the above lemma is hardly more than a rehash of Proposition~\ref{prop:algebraicrealization2}.

\begin{lem} \label{lem:orbitmap1} For any fan $X$, the orbit map $\tau : \AA(X) \to X$ is surjective and quasi-compact as a map of topological spaces.  If $X$ is locally finite type, it is also open. \end{lem}

\begin{proof}  Certainly $\tau$ is surjective because for any field $k$, even the composition $\tau(k)$ of $\AA(X)(k) \to \AA(X)$ and $\tau$ is surjective (Lemma~\ref{lem:orbitmap2}).  The map $\tau$ is quasi-compact because it is ``affine:"  The preimage of an open affine $U = \Spec P$ of $F$ under $\tau$ is $\Spec \ZZ[P]$, so the result follows from the fact that $\Spec$ of any monoid or ring is quasi-compact (same proof for monoids as for rings).

To see that $\tau$ is open when $X$ is locally finite type, we note that the question is local on $X$, so we can reduce to the case where $X = \Spec P$ for a finitely generated monoid $P$ so that $\AA(X) = \Spec \ZZ[P]$.  It suffices to prove that $\tau(U_f)$ is open when $U_f$ is the usual basic open subset of $\Spec \ZZ[P]$ associated to an element \be f & = & \sum_{i=1}^n a_i [p_i] \ee of $\ZZ[P]$.  A point of $\AA(X)$ can be represented by a ring homomorphism $x : \ZZ[P] \to k$ with $k$ an algebraically closed field.  Such a point is in $U_f$ iff $x(f) \neq 0$.  Since $X$ is a finite topological space, we can prove $\tau(U_f)$ is open by proving it is stable under generalization.  

We thus reduce to proving the following:  Suppose $x : \Spec \ZZ[P] \to k$ is a ring homomorphism ($k$ an algebraically closed field) such that $x(f) \neq 0$ and $G$ is a face of $P$ containing $F := (x|P)^{-1}k^*$.  Then we can find a ring homomorphism $y : \Spec \ZZ[P] \to k$ such that $y(f) \neq 0$ and $(y|P)^{-1} k^* = G$.  We can assume that $p_1,\dots,p_n$ are listed so that $p_1,\dots,p_k \in F$ and $p_1,\dots,p_k, p_{k+1},\dots,p_l \in G$, so that \be x(f) & = & \sum_{i=1}^k x(a_i) x(p_i) \\ & \neq & 0. \ee Start by choosing any group homomorphism $x_0 : G^{\rm gp} \to k^*$ extending $x : F^{\rm gp} \to k^*$.  (This can be done since $k^*$ is divisible.)  We can define $y$ by using the ring map $x|\ZZ : \ZZ \to k$ and the monoid homomorphism $y|P : P \to k$ taking $p$ to $x_0(p)$ if $p \in G$ and zero otherwise (this is well-defined because $G$ is a face of $P$).  Obviously we have $(y|P)^{-1}k^* = G$, but the trouble is that \be y(f) & = & x(f) + \sum_{j=k+1}^l x(a_j) x_0(p_j) \ee might be zero.  The claim is that we can modify our choice of $x_0$ so that this doesn't happen.  Obviously there is no trouble if $k=l$, for then the second term above doesn't appear, so we can assume now that $l > k$.  Let $H$ be the quotient of $G^{\rm gp} / F^{\rm gp}$ by its torsion subgroup, so $H \cong \ZZ^r$ is a free abelian group of finite rank $r$ and we have a natural surjection $\pi : G^{\rm gp} \to \ZZ^r \cong H$ which kills $F^{\rm gp}$.  If we choose any group homomorphism $t : \ZZ^r \to k^*$, then we get a new $y' : \ZZ[P] \to k$ with $(y'|P)^{-1} k^* = G$ by replacing $x_0$ with $x_0 \cdot t \pi$.  We need to argue that we can find such a $t$ so that \be y'(f) & = & x(f) + \sum_{j=k+1}^l x(a_j) x_0(p_j) t\pi(p_j) \ee is non-zero.  The key point is that $\pi(p_{k+1}), \dots, \pi(p_l)$ must all be non-zero in $\ZZ^r \cong H$  because if one of these $p_j$ were torsion in $G^{\rm gp} / F^{\rm gp}$ then we could write $np_j + f  = f'$ in $P$ for some $n > 0$ and $f,f' \in F$, which would imply that $p_j \in F$ because $F$ is a face.  In particular, we must have $r > 0$.  Now think of the choice of $t$ as a choice of point of $\GG_m^r$ (we work now over $k$ in the language of classical algebraic geometry, so a point is a $k$-point).  Then the set of $t$ for which $y'(f) = 0$ is a hypersurface $Z$ in $\GG_m^r$ (defined over $k$) whose defining equation has non-zero constant term $x(f)$ because the key point says the exponents $\pi(p_j)$ of the other monomials in this defining equation are non-zero.  So this $Z$ is a non-trivial hypersurface in $\GG_m^r$, $r > 0$, and $k$ is algebraically closed, so there will certainly be some $t$ not lying on $Z$. \end{proof}

\begin{lem} \label{lem:surjectivity} Let $f : X \to Y$ be a surjective map of fans such that $\M_{Y,f(x)}^{\rm gp} \to \M_{X,x}^{\rm gp}$ is injective for every $x \in X$.  Then the map of schemes $\AA(f) : \AA(X) \to \AA(Y)$ is surjective.  In fact, $\AA(f)(k) : \AA(X)(k) \to \AA(Y)(k)$ is surjective for any algebraically closed field $k$. \end{lem}

\begin{proof} Fix $\ov{y} \in \AA(Y)$.  By choosing an algebraic closure $k(\ov{y}) \into k$, we obtain a point $\tilde{y} \in \AA(Y)(k)$ whose image under $\AA(Y)(k) \to \AA(Y)$ is $\ov{y}$.  Set $y := \tau_Y(\ov{y}) = \tau_Y(k)(\tilde{y}) \in Y$.  Since $f$ is surjective, we can choose some $x \in f^{-1}(y) \subseteq X$.  Pick any $\tilde{x} \in \tau_X(k)^{-1}(x)$ (the set $\tau_X(k)^{-1}(x)$ is non-empty by Lemma~\ref{lem:orbitmap2}).  Let $\tilde{y}'$ be the image of $\tilde{x}$ under $\AA(f)(k)$.  Since the diagram $$ \xym{ \AA(X)(k) \ar[r]^-{\tau_X(k)} \ar[d]_{\AA(f)(k)} & X \ar[d]^-f \\ \AA(Y)(k) \ar[r]^-{\tau_Y(k)} & Y } $$ commutes we have $\tilde{y}' \in \tau_Y(k)^{-1}(y)$.  Since $k$ is algebraically closed, we have $u \cdot \tilde{y}' = \tilde{y}$ for some $u \in \GG(y)(k) = \Hom_{\Ab}(\M_{Y,y}^{\rm gp},k^*)$ by Lemma~\ref{lem:orbitmap2}.  Since we assume $\M_{Y,y}^{\rm gp} \to \M_{X,x}^{\rm gp}$ is injective, and $k$ is algebraically closed (hence $k^*$ is divisible), we can extend $u$ to a group homomorphism $v \in \GG(x)(k) = \Hom_{\Ab}(\M_{X,x}^{\rm gp}, k^*)$.  Since the diagram of actions $$ \xym{ \GG(x)(k) \times \tau_X(k)^{-1}(x) \ar[r] \ar[d] & \tau_X(k)^{-1}(x) \ar[d] \\ \GG(y)(k) \times \tau_Y(k)^{-1}(y) \ar[r] & \tau_Y(k)^{-1}(y) } $$ commutes, the point $v \cdot \tilde{x} \in \tau_X(k)^{-1}(x)$ maps to $\tilde{y}$ under $\AA(f)(k) : \AA(X)(k) \to \AA(Y)(k)$ (this argument proves the second statement) hence the image of $v \cdot \tilde{x}$ in $\AA(X)$ maps to $\ov{y}$ under $\AA(f)$. \end{proof}

For clarity and later use, we record an ``affine" version of the above lemma:

\begin{lem} \label{lem:affinesurjectivity} Let $h : Q \to P$ be an injective map of integral monoids such that $\Spec h$ is surjective.  Then $\Spec \ZZ[h]$ is also surjective. \end{lem}

\begin{proof} Since $h$ is injective and $Q,P$ are integral, $h^{\rm gp} : Q^{\rm gp} \to P^{\rm gp}$ is injective.  For a face $F \in \Spec P$, with image $G := h^{-1}F$ under $\Spec h$, the map $\M_{Q,G} \to \M_{P,F}$ induced by $\Spec h$ is the map $G^{-1}Q \to F^{-1}P$ induced by $h$.  The groupification of this map is just $h^{\rm gp} : Q^{\rm gp} = (G^{-1}Q)^{\rm gp} \to P^{\rm gp} = (F^{-1}P)^{\rm gp}$.  This reduces us to the previous lemma. \end{proof}

\subsection{Pullback of modules along the orbit map}  We can jazz up the construction of the functor \eqref{Zslot2} from \S\ref{section:modules} as follows:  For any fan $X$, there is a functor \be \tau^* = \ZZ[ \slot ]  : \Mod(X) & \to & \Mod(\AA(X)) \ee left adjoint to the ``forgetful functor" \be \tau_* : \Mod(\AA(X)) & \to & \Mod(X). \ee  Both of these functors take quasi-coherent sheaves to quasi-coherent sheaves.  The functor $\tau^*$ takes ideal sheaves to ideal sheaves.  The relative Spec and Proj constructions---as well as the blowup construction---of \S\ref{section:relativeSpecandProj} are compatible with these functors in an obvious sense \cite[5.18.3]{GM1}.

\section{Maps of fans} \label{section:mapsoffans} In this section we give an ``\'etude globale \'el\'ementaire de quelques classes de morphismes" for fans something like \cite[II]{EGA}, but nowhere near as thorough.

\subsection{Generalities} \label{section:generalities}  Before launching into the general study of maps of fans that will occupy the remainder of \S\ref{section:mapsoffans} it is helpful to make some general remarks.

When studying a map $f : X \to Y$ of locally monoidal spaces (or of locally ringed spaces) it is common to look at the stalk maps $f_x : \M_{Y,f(x)} \to \M_{X,x}$ for points $x \in X$.  For example, in \S\ref{section:flatmaps} we will define flatness in terms of these maps, just as one does in the case of locally ringed spaces.  This ``looking at stalks" approach works particularly well for (reasonable) fans, because the map $f$ itself will look Zariksi locally like $\Spec f_x$ near $x \in X$.  This is just a reflection of the ``naturality" in the final sentence of Proposition~\ref{prop:locallyfinitetypefans}, but for the sake of clarity, let us formulate this carefully as follows:

\begin{prop} \label{prop:localpicture} Let $f : X \to Y$ be a map of fans.  For each $x \in X$, we have a natural commutative diagram of fans $$ \xym{ \Spec \M_{X,x} \ar[r] \ar[d]_{\Spec f_x} & X \ar[d]^-f \\ \Spec \M_{Y,f(x)} \ar[r] & Y }$$ where the horizontal arrows are strict embeddings (Definition~\ref{defn:embedding}).  The image of the top (resp.\ bottom) horizontal arrow contains $x$ (resp.\ $f(x)$).  If the topological space underlying $X$ (resp.\ $Y$) is locally finite, then the top (resp.\ bottom) horizontal arrow is an open embedding. \end{prop}

\begin{proof} Let $U=\Spec P \subseteq Y$ (resp.\ $V=\Spec Q$) be an affine open neighborhood of $f(x)$ in $Y$ (resp.\ of $x$ in $f^{-1}(U)$).  If $X$ (resp.\ $Y$) is locally finite, we can choose $\Spec Q$ (resp.\ $\Spec P$) finite.  The map of fans $f|V : V \to U$ is $\Spec$ of a map of monoids $P \to Q$ and fits in an obvious commutative square with $f$ where the horizontal arrows are the open embeddings $V \to X$ and $U \to Y$.  So we can replace $f$ with $f|V$, in which case the result follows from Propositions~\ref{prop:Spec} and \ref{prop:Specfinite} applied to the faces of $P$ and $Q$ corresponding to $f(x)$, $y$. \end{proof}

Next we make a general observation about properties of maps of fans which is frequently useful.  Notice that the situation here is much simpler than that for schemes.

\begin{prop} \label{prop:stableunderbasechange} Let ${\bf P}$ (resp.\ ${\bf Q}$) be a property of maps of topological spaces (resp.\ monoids) stable under fibered products (resp.\ pushouts).  Define a property ${\bf R}$ of maps of fans (or of locally monoidal spaces) by saying that $f : X \to Y$ has property ${\bf R}$ iff the underlying map of spaces has property ${\bf P}$ and $f_x : \M_{Y,f(x)} \to \M_{X,x}$ has property ${\bf Q}$ for every $x \in X$.  Then ${\bf R}$ is stable under base change.  Similarly, if ${\bf P}$ and ${\bf Q}$ are stable under composition, so is ${\bf R}$.  \end{prop}

We often apply Proposition~\ref{prop:stableunderbasechange} with ${\bf Q}$ the property of being surjective---this is stable under composition and base change.

\begin{proof}  This follows from the fact that finite inverse limits of fans (or of locally monoidal spaces) commute with the forgetful functor to monoidal spaces (see \S\ref{section:inverselimits} and \S\ref{section:inverselimitsoffans}), so the map of spaces underlying a base change $X' \to Y'$ of $X \to Y$ in fans is just the base change (fibered product) of the map of spaces underlying $X \to Y$.  Similarly, the map $f'_{x'} : \M_{Y',f'(x')} \to \M_{X',x'}$ is a pushout of $f_x : \M_{Y,f(x)} \to \M_{X,x}$, where $x \in X$ is the image of $x'$ under $X' \to X$.  (For schemes one only knows that $f'_{x'}$ is a pushout of $f_x$, followed by a localization at a prime ideal.)  The ``Similarly" is easy. \end{proof}

\begin{rem} \label{rem:localproperties}  In the rest of \S\ref{section:mapsoffans} we will often define a property ${\bf P'}$ of maps of fans by saying that $f : X \to Y$ has property ${\bf P'}$ iff the algebraic realization $\AA(f) : \AA(X) \to \AA(Y)$ (\S\ref{section:algebraicrealization}) has some property ${\bf P}$ of maps of schemes.  Having done this, we will always try to give purely ``fan theoretic" characterizations of those $f$ enjoying property ${\bf P'}$, but this will often require further assumptions (typically that $X$ and $Y$ be fine and $f$ be quasi-compact), so it will be convenient to at least have a general definition of property ${\bf P'}$ that makes sense for all $f$.  The algebraic realization $\AA$ takes open embeddings to open embeddings and Zariski covers to Zariski covers, so if the property ${\bf P}$ is ``local" in some sense (local on the domain, codomain, or both), then ${\bf P'}$ will also be ``local" in the same sense.  In particular, if ${\bf P}$ is local on both the domain and codomain, then it follows from Proposition~\ref{prop:localpicture} that a map $f : X \to Y$ between fans \emph{with locally finite underlying spaces} will have property ${\bf P}$ iff the map of affine fans \be \Spec \M_{X,x} & \to & \Spec \M_{Y,f(x)} \ee has property ${\bf P}$ for each $x \in X$. \end{rem}

\subsection{Quasi-compact maps} \label{section:quasicompactmaps}  Recall that a map of topological spaces $f : X \to Y$ is called \emph{quasi-compact} iff $f^{-1}(U)$ is quasi-compact for each open, quasi-compact subspace $U \subseteq Y$.  Similarly, a map of locally ringed spaces, monoidal spaces, etc.\ is called \emph{quasi-compact} iff the underlying map of topological spaces is quasi-compact in the previous sense.

\begin{lem} \label{lem:quasicompact} A map of fans $f : X \to Y$ is quasi-compact iff $f^{-1}(U)$ can be covered by finitely many open affines for each open affine subset $U$ of $Y$.  Quasi-compact maps of fans are stable under composition and base change. \end{lem}

\begin{proof} This is proved in the same manner as the analogous statements for schemes by making use of the fact that any affine fan is quasi-compact (Proposition~\ref{prop:Spec}\eqref{uniqueclosedpoint}) and any quasi-compact open subset of a fan, being itself a fan, is covered by finitely many open affines. \end{proof}

\begin{rem} \label{rem:quasicompact} I do not know whether quasi-compact maps of arbitrary topological spaces are stable under base change, though I doubt this is true.  \end{rem}

\begin{lem} \label{lem:orbitmap1a} In any commutative diagram of topological spaces $$ \xym{ X' \ar[r]^-i \ar[d]_{f'} & X \ar[d]^f \\ Y' \ar[r]^-j & Y } $$ where $i$ is surjective and $f'$ and $j$ are quasi-compact, $f$ is also quasi-compact. \end{lem}

\begin{proof} Exercise. \end{proof}

\begin{thm} \label{thm:quasicompact} A map of fans is quasi-compact iff its scheme realization is quasi-compact.  A map of locally finite type fans is quasi-compact iff it has finite (but not necessarily discrete!) fibers. \end{thm}

\begin{proof} The implication $(\implies)$ follows easily by thinking in terms of covers by finitely many open affines.  For the other implication, apply Lemma~\ref{lem:orbitmap1a} to the square $$ \xym{ \AA(X) \ar[r]^-{\tau_X} \ar[d]_{\AA(f)} & X \ar[d]^f \\ \AA(Y) \ar[r]^-{\tau_Y} & Y }$$ using Lemma~\ref{lem:orbitmap1} to check the necessary hypotheses.  For the second statement one uses the fact that $\Spec P$ is finite when $P$ is finitely generated (Corollary~\ref{cor:faces}) and the fact that all quasi-compact opens are covered by finitely many affines.  \end{proof}

\subsection{Affine maps} \label{section:affinemaps}  A map $f : X \to Y$ of fans is called \emph{affine} iff $f^{-1}(U)$ is affine for each affine open subfan $U \subseteq Y$.

\begin{prop} \label{prop:affinemaps} The property of being an affine map of fans or schemes is local on the base.  If $f : X \to Y$ is an affine map of fans, then the scheme realization $\AA(f)$ of $f$ is an affine map of schemes. \end{prop}

\begin{proof} For the first statement, affine maps (in either case) can be characterized as those $f : X \to Y$ for which 1) $f_* \O_X$ (replace $\O_X$ with $\M_X$ for fans) is quasi-coherent and 2) the natural map of $Y$-schemes $X \to \Spec_Y f_* \O_X$ is an isomorphism (see \S\ref{section:relativeSpecandProj} for discussion of $\Spec_Y$).  The properties 1) and 2) are clearly local on $Y$.  The second statement is local on $Y$ in light of the first statement, so we can assume $Y = \Spec P$ is affine.  But then $X = \Spec Q$ is affine because $f$ is affine, so $f$ corresponds to a monoid homomorphism $P \to Q$ and $\AA(f)$ is nothing but $\Spec \ZZ[Q] \to \Spec \ZZ[P]$, which is certainly affine.  \end{proof}

\begin{rem} \label{rem:affinemaps} The converse of Proposition~\ref{prop:affinemaps} does not hold.  A disjoint union of two affine fans is not affine even though its scheme realization is affine. \end{rem}

\subsection{Boundary construction} \label{section:boundaryconstruction}  If $X$ is a scheme and $Z$ is a closed subspace of the topological space underlying $X$, then $Z$ can be given a reduced-induced scheme structure making $Z \into X$ a closed embedding of schemes.  There is no analog of this for fans.  Indeed, if $X$ is a fan and $Z$ is a closed subspace of its underlying topological space, there may not even be any fan whose underlying topological space is homeomorphic to $Z$ (even if $X$ is an affine fan)!

\begin{example} \label{example:closedsubspacesoffans} Let $X = \Spec \NN^2$ and let $Z$ be the closed subspace of (the topological space underlying) $X$ given by the complement of the unique open point $\NN^2$ of $X$ (cf.\ Proposition~\ref{prop:Specfinite}\eqref{uniqueopenpoint}).  There can be no fan whose underlying space is homeomorphic to $Z$.  Indeed, since $Z$ itself is the smallest neighborhood of $0 = (\NN^2)^*$ in $Z$ (since $X$ itself is the smallest neighborhood of $0$ in $X$ by Proposition~\ref{prop:Spec}\eqref{uniqueclosedpoint}), such a fan would have to be affine.  The space $Z$ has precisely three points, precisely two of which are open.  But no affine fan can have finite underlying topological space and more than one open point (Proposition~\ref{prop:Specfinite}\eqref{uniqueopenpoint}). \end{example}

Suppose $X$ is a fan and $Z$ is a closed subspace of the topological space underlying $X$.  We are going to show momentarily that it \emph{is} possible to make $Z$ into a fan under a certain additional (purely topological!) assumption (which is also \emph{necessary}).  The fan structure on $Z$ will be determined in a natural way by the fan $X$, but we will \emph{not} generally be able to lift the closed embedding of spaces $Z \into X$ to a map of fans, as we could do with schemes.  On the other hand, we will construct a natural closed embedding of algebraic realizations $\AA(Z) \into \AA(X)$ as is a kind of substitute.

\begin{thm} \label{thm:boundary} For a fan $X$ and a closed subspace $Z$ of the topological space underlying $X$, the following are equivalent: \begin{enumerate} \item \label{fanexists} There exists a fan $(Z,\M_Z)$ with underlying topological space $Z$. \item \label{topologicalcondition} For every $z \in Z$, there is a pair $(U,z')$ consisting of a neighborhood $U$ of $z$ in $X$ and a point $z' \in U \cap Z$ such that $U \cap Z$ is the closure of $z'$ in $U$. \item \label{fancondition} For every $z \in Z$, there is a pair $(V,z')$ consisting of an \emph{affine} neighborhood $V$ of $z$ in $X$ and a point $z' \in V \cap Z$ such that $V \cap Z$ is the closure of $z'$ in $V$. \end{enumerate} When these equivalent conditions hold, we can construct a certain fan $Z(X)$ with underlying topological space $Z$.  This construction enjoys the following properties, which characterize it up to unique isomorphism: \begin{enumerate} \item \label{ZXopenembedding} Whenever $X$ and $Z$ satisfy the equivalent conditions above and $U$ is an open subspace of $X$, the fan structure $(Z \cap U)(U)$ on $Z \cap U$ obtained by applying our construction to the fan $U$ and the closed subspace $Z \cap U$ (these will also satisfy the equivalent conditions above) coincides with the restriction of the fan structure $Z(X)$ on $Z$ to $Z \cap U$.  \item \label{ZXaffine} If $X = \Spec P$ and $Z = Z_F = \{ F \}^{-}$ for some $F \leq P$, then $Z(X) = \Spec F$. \end{enumerate}  Furthermore, there is a closed embedding $\AA(Z(X)) \into \AA(X)$ of algebraic realizations characterized by the following properties: \begin{enumerate} \item \label{ZXdiagram} Let $\tau_Z$ and $\tau_X$ denote the orbit maps for the fans $Z(X)$ and $X$.  Then, on the level of topological spaces, we have a commutative diagram $$ \xym{ \AA(Z(X)) \ar[d]_{\tau_Z}  \ar[r] & \AA(X) \ar[d]_{\tau_X} \\ Z \ar[r] & X. } $$ \item \label{ZXlocalembedding} If $X = \Spec P$ and $Z = Z_F = \{ F \}^{-}$ for some $F \leq P$, so $Z(X) = \Spec F$, then the closed embedding $\AA(Z(X)) \into \AA(X)$ is $\Spec$ of the surjection of rings $\ZZ[P] \to \ZZ[F]$ taking $[p] \in \ZZ[P]$ to $[p] \in \ZZ[F]$ (resp.\ $0 \in \ZZ[F]$) when $p \in F \leq P$ (resp.\ $p \in P \setminus F$). \end{enumerate} \end{thm}

\begin{proof} \eqref{fanexists}$\implies$\eqref{topologicalcondition}:  Fix $z \in Z$.  Pick an affine open neighborhood $W$ of $z$ in $(Z,\M_Z)$.  By Proposition~\ref{prop:Spec}\eqref{Specclosure} there is a point $z' \in W$ which is dense in $W$.  Since $Z$ is a subspace of $X$ there is a neighborhood $U$ of $z$ in $X$ such that $U \cap Z = W$.  Since $Z$ is closed in $X$, the pair $(U,z')$ has the property required in \eqref{topologicalcondition}.

Let us now make a simple observation.  Suppose $z \in Z$, $(U,z')$ is as in \eqref{topologicalcondition} and $V$ is a neighborhood of $z$ in $U$.  Then we have $z' \in V$ because $z$ is in the closure of $z'$ in $U$ and it is clear that $(V,z')$ is also as in \eqref{topologicalcondition}.  To see that \eqref{topologicalcondition}$\implies$\eqref{fancondition} we need only apply this observation, taking $V$ affine.

Obviously \eqref{fancondition}$\implies$\eqref{topologicalcondition}.  The implication \eqref{fancondition}$\implies$\eqref{fanexists} will of course follow once we prove the rest of the theorem with the understanding that the ``equivalent conditions" are just the conditions \eqref{topologicalcondition} and \eqref{fancondition}.  To do this, fix an $X$ and $Z$ satisfying \eqref{topologicalcondition} and \eqref{fancondition} and let us construct the fan structure $Z(X)$ on $Z$.  Fix a point $z \in Z$ and a pair $(V,z')$ as in \eqref{fancondition}.  Since $V = \Spec P$ is affine, $z'$ corresponds to some face $F \leq P$ and the closure $Z \cap V$ of $z'$ in $V$ corresponds to the irreducible closed subset $$Z_F = \{ F \}^- = \{ G \leq F \}$$ of $\Spec P$ (cf.\ Proposition~\ref{prop:Spec}\eqref{irreducibles}).  As a topological space, $Z_F$ is identified with $\Spec F$ (Remark~\ref{rem:ZF}), so we give $Z \cap V$ a fan structure by giving it the usual structure sheaf $\M_F$ of $\Spec F$ (as we must in order for our construction to satisfy the two properties).  Similarly, we \emph{define} our closed embedding of algebraic realizations $\AA(Z \cap V) \into \AA(V)$ as in \eqref{ZXlocalembedding} (as we must).   We can do this for any point $z \in Z$ to define a fan structure on $Z$ near $z$ (and an a local version of $\AA(Z(X)) \into \AA(X)$).  To show that these locally defined fan structures (and closed embeddings of realizations) glue to a global fan structure (which will be our $Z(X)$) on $Z$ (and a closed embedding $\AA(Z(X)) \into \AA(X)$) we need to prove that for any fixed $z \in Z$ and any two pairs $(V_i,z'_i)$ ($i=1,2$) as in \eqref{fancondition} with $z \in V_1 \cap V_2$, the fan structures on $Z \cap V_i$ coming from our construction agree on some neighborhood of $z$ in $Z \cap V_1 \cap V_2$ (and similarly for our locally defined closed embeddings of realizations).  To see this, start by picking an affine open neighborhood $V$ of $z$ in $V_1 \cap V_2$.  We have $z_1',z_2' \in V$ by the observation above.  Furthermore, we must have $z_1'=z_2'=: z'$ because $z_1'$ and $z_2'$ both have the same closure in $V$ (namely $Z \cap V$, again using the observation above) and $V$ is sober (Proposition~\ref{prop:Spec}\eqref{irreducibles}).  The pair $(V,z')$ is another pair as in \eqref{fancondition} by the usual observation.  It is now enough to show that for each fixed $i$, the fan structure on $Z \cap V_i$ coming from $(V_i,z')$ agrees with the fan structure on $Z \cap V$ coming from $(V,z')$ near $z$ (and similarly for the closed embeddings of realizations).

Since $V_i$ is affine, we have $V_i = \Spec P$, where $P :=\M_X(V_i)$.  Similarly, $V = \Spec Q$, where $Q:=\M_X(V)$, and the open embedding $V \into V_i$ is $\Spec$ of the restriction map $h : P \to Q$.  By Proposition~\ref{prop:affineopensubfan}\eqref{aos1}, $h$ identifies $Q$ with $F^{-1} P$ for some face $F \leq P$.  Let $G$ be the face of $P$ corresponding to $z'$ under $V_i = \Spec P$.  Since $z'$ is also in $V \subseteq V_i$, we have $F \leq G$ and $z'$ corresponds to the face $F^{-1}G$ of $Q=F^{-1}P$ under $V = \Spec Q$ because we know $\Spec h$ is an open embedding and Lemma~\ref{lem:boundaryfaces} below ensures that $F^{-1}G$ is a face of $Q=F^{-1}P$ whose preimage under the localization map $h$ is $G$.  By construction, the fan structure on $Z$ near $z$ coming from $(V_i,z')$ is defined by giving $V_i \cap Z$ the usual structure sheaf $\M_G$ of $\Spec G$, whereas the fan structure on $Z$ near $z$ coming from $(V,z')$ is defined by giving $V \cap Z$ the usual structure sheaf $\M_{F^{-1}G}$ of $\Spec F^{-1} G$.  The open embedding $V \cap Z \into V_i \cap Z$ is the map of topological spaces underlying the map of fans $\Spec F^{-1}G \to \Spec G$, so to complete the proof we just need to know that $\M_{G}$ restricts to $\M_{F^{-1}G}$ on $V \cap Z$ under this map of fans (i.e.\ that this map of fans is an open embedding)---this is Proposition~\ref{prop:Spec}\eqref{localizationembedding}.  This completes the construction of $Z(X)$; the other assertions of the theorem are clear from the construction. \end{proof}

\begin{lem} \label{lem:boundaryfaces} Let $P$ be a monoid, $F \leq G \leq P$ faces of $P$.  Then $F^{-1}G$ is a face of $F^{-1}P$ and the diagram of monoids $$ \xym{ G \ar[r] \ar[d] & F^{-1}G \ar[d] \\ P \ar[r] & F^{-1} P } $$ is both cartesian and cocartesian. \end{lem}

\begin{proof} Exercise. \end{proof}

In the following example we will examine an important special case of Theorem~\ref{thm:boundary} which we will use frequently later and which seems to encompass very nearly the full generality of that theorem.  

\begin{example} \label{example:ZXx}  Let $X$ be a fan, $x \in X$, $Z := \{ x \}^{-}$ the irreducible closed subspace of (the topological space underlying) $X$ given by the closure of $x$ in $X$.  The equivalent conditions of Theorem~\ref{thm:boundary} are clearly satisfied for this $X$ and $Z$ because we can take $(U,z') = (X,x)$ for any $z$ in \eqref{topologicalcondition}.  The construction of Theorem~\ref{thm:boundary} gives us a fan structure $Z(X)$ on the topological space $Z$ which we shall denote by $Z(X,x)$ from now on.  In this special case, the structure sheaf $\M_Z$ of the fan $Z(X,x)$ can be described directly as follows.  For a monoid $P$, write $\u{P}$ for the constant sheaf associated to $P$ on the space $Z$.  We have a map $\M_X|Z \to \u{\M_{X,x}}$ of sheaves of monoids on $Z$ defined by taking a section $s$ of $\M_X|Z$ over a (non-empty) open subset $U$ of $Z$ (note that such a $U$ contains $x$) to the constant function $(s_x :U \to \u{\M_{X,x}}) \in \u{\M_{X,x}}(U)$ with value $s_x \in (\M_X|Z)_x=\M_{X,x}$.  We define $\M_Z$ by the cartesian diagram \bne{cartdiagram} & \xym{ \M_Z \ar[r] \ar[d] & \u{ \M_{X,x}^* } \ar[d] \\ \M_X|Z \ar[r] & \u{ \M_{X,x} } } \ene of sheaves of monoids on $Z$.  

To see that $(Z,\M_Z)$ is a fan equal to the fan $Z(X)$ constructed in Theorem~\ref{thm:boundary}, consider an affine open subset $V = \Spec P$ of $X$ containing $x$.  Let $G \leq P$ be the face of $P$ corresponding to $x$, so that $Z \cap V$ corresponds to the irreducible closed subset $Z_G$ of $\Spec P$ given by the closure of $G$ in $\Spec P$ (cf.\ Proposition~\ref{prop:Spec}\eqref{irreducibles}).  We have $\M_{X,x} = \M_{P,G} = G^{-1}P$ (Proposition~\ref{prop:Spec}\eqref{stalkMP}), so $\M_{X,x}^* = G^{\rm gp}$.  According to the construction of Theorem~\ref{thm:boundary}, the restriction of the structure sheaf of the fan $Z(X)$ to the open subspace $Z_G = Z \cap V$ is the usual structure sheaf $\M_G$ of $\Spec G$ (suppressing the natural identification of $\Spec G$ and $Z_G$ as topological spaces discussed in Remark~\ref{rem:ZF}).  To see that $\M_Z|Z_G = \M_G$, first note that $\M_X|V = \M_P$ and formation of finite inverse limits of sheaves commutes with pullback (restriction from $Z$ to $Z_G$ in our case), so $\M_Z|Z_G$ is defined by the cartesian diagram \bne{cartdiagramV} & \xym{ \M_Z|Z_G \ar[r] \ar[d] & \u{ G^{\rm gp} } \ar[d] \\ \M_P|Z_G \ar[r] & \u{ G^{-1} P } } \ene of sheaves of monoids on $Z_G$.  Consequently, our desired equality will follow by showing that the evident diagram \bne{cartdiagramV2} & \xym{ \M_G \ar[r] \ar[d] & \u{ G^{\rm gp} } \ar[d] \\ \M_P|Z_G \ar[r] & \u{ G^{-1} P } } \ene of sheaves of monoids on $Z_G$ is cartesian.  This can be checked on stalks.  A point of $Z_G$ is a face $F$ of $G$ and, using Proposition~\ref{prop:Spec}\eqref{stalkMP}, we see that the stalk of \eqref{cartdiagramV2} at $F$ is \bne{cartdiagramV2stalk} & \xym{ F^{-1}G \ar[r] \ar[d] & G^{\rm gp} \ar[d] \\ F^{-1}P \ar[r] & G^{-1} P. } \ene  The diagram of monoids \eqref{cartdiagramV2stalk} is cartesian by Lemma~\ref{lem:boundaryfaces} (applied with $F \leq G \leq P$ there given by $F^{-1}G \leq F^{-1}G \leq F^{-1}P$ here). 

Since formation of finite inverse limits commutes with stalks, we see from the definition of $\M_Z$ via the cartesian diagram \eqref{cartdiagram} that $\M_{Z,x} = \M_{X,x}^*$ is a group---i.e.\ $x$ is in the torus (\S\ref{section:torus}) of the fan $Z(X,x)$.

The construction of $Z(X,x)$ is functorial in $(X,x)$ in the sense that a map of fans $f : X \to X'$ and a chosen point $x \in X$ yield a map of fans \be Z(f,x) : Z(X,x) & \to & Z(X',f(x)) \ee which, on the level of topological spaces, sits in a commutative diagram \bne{Zfxdiagram} & \xym{ Z(X,x) = \{ x \}^- \ar[r] \ar[d]_{Z(f,x)} & X \ar[d]^f \\ Z(X',f(x)) = \{ f(x) \}^- \ar[r] & X'. } \ene  (There is a similar functoriality for the general $Z(X)$ construction of Theorem~\ref{thm:boundary}, but it is cumbersome to formulate and of questionable additional value.)  Set $Z := Z(X,x)$, $x' := f(x)$, $Z' := Z(X',x')$ to ease notation; let us also abusively write $f : Z \to Z'$ for the restriction of $f$.  The map of structure sheaves $f^{-1} \M_{Z'} \to \M_Z$ which is part of the data of $Z(f,x)$ is defined as follows:  Since $f^{-1}$ commutes with finite inverse limits, $f^{-1} \M_{Z'}$ is the inverse limit of \bne{MZprime} & \xym{ & \u{ \M_{X',x'}^* } \ar[d] \\ f^{-1} \M_{X'} | Z \ar[r] & \u{ \M_{X',x'}, } } \ene where the underlines denote constant sheaves on $Z$.  There is an obvious map from the diagram \eqref{MZprime} to the diagram \bne{MZ} & \xym{ & \u{ \M_{X,x}^* } \ar[d] \\ \M_{X} | Z \ar[r] & \u{ \M_{X,x}, } } \ene and $f^{-1} \M_{Z'} \to \M_Z$ is defined to be the induced map from the inverse limit of \eqref{MZprime} to the inverse limit of \eqref{MZ} ($=\M_Z$).  \end{example}

\subsection{Separated and proper maps} \label{section:propermaps}  In this section we develop a theory of \emph{separated} (resp.\ \emph{proper}) maps of fans in terms of ``valuative criteria" analogous to those considered in algebraic geometry.  These valuative criteria only behave well for maps of \emph{fine} fans (cf.\ Example~\ref{example:x}), so they should not be used to \emph{define} separated (resp.\ proper) maps between arbitrary fans.  In general, in accordance with the philosophy of \S\ref{section:generalities}, we \emph{define} separated and proper maps for (reasonable) maps of fans in terms of the algebraic realization (see Definition~\ref{defn:proper}).  When \emph{separated} and \emph{proper} are thus defined, one can show (Theorem~\ref{thm:diffproper}) that various ``other" realization of a separated (resp.\ proper) map of fans are separated (resp.\ proper) in an appropriate sense.  For quasi-compact maps of fine fans all of the possible notions of ``separated" and ``proper" (the valuative ones, and the ones in terms of various realizations) coincide (Theorem~\ref{thm:diffproper}).

\begin{lem} \label{lem:closure} Let $X$ be a fine fan, $x \in X$, $x'$ a point of the torus $T_X \subseteq X$ (\S\ref{section:torus}).  Then $x \in \{ x' \}^-$ iff there is a map of fans $h : \Spec \NN \to X$ taking the closed point to $x$ and the generic point to $x'$. \end{lem}

\begin{proof} The condition is sufficient on the grounds that $h$ is continuous on spaces.  For the difficult implication, the question is local, so we can assume $X = \Spec P$ for some fine monoid $P$, so that $x$ and $x'$ correspond to faces $F \subseteq G$ of $P$.  The condition $x' \in T_X$ means we must have $G=P$ because $\M_{X,x'} = G^{-1} P$ must be a group.  Such an $h$ is then the same thing as an $h$ as in Lemma~\ref{lem:duality}. \end{proof}

\begin{defn}  \label{defn:valuativecriterion}  A map of fans $f : X \to Y$ is called \emph{valuatively separated} (resp.\ \emph{valuatively proper}) iff there is at most one (resp.\ a unique) lift as indicated in any solid diagram of fans as below. \bne{fandiagramtolift} & \xym{ \Spec \ZZ \ar[d] \ar[r]^-g & X \ar[d]^f \\ \Spec \NN \ar[r]^-g \ar@{.>}[ru] & Y } \ene  If $f : X \to Y$ is a map of schemes (and $U \subseteq X$ is an open subscheme), then we say that $f$ satisfies the \emph{discrete valuative criterion for separatedness} (resp.\ \emph{properness}) (\emph{for maps generically into} $U \subseteq X$) iff there is at most one (resp.\ a unique) lift as indicated in any solid diagram of schemes \bne{schemediagramtolift} & \xym{ \Spec K \ar[d] \ar[r]^-g & X \ar[d]^f \\ \Spec R \ar[r] \ar@{.>}[ru] & Y } \ene (where $g(\Spec K) \in U$) for any field $K$ with discrete valuation $\nu : K^* \to \ZZ$ and valuation ring $R = \{ a \in K^* : \nu(a) \geq 0 \} \cup \{ 0 \}$. \end{defn}

\begin{lem} \label{lem:valuativecriterion}  Valuatively separated (resp.\ valuatively proper) maps of fans are stable under composition and base change.  The property of being valuatively separated (resp.\ valuatively proper) is local on the base. \end{lem}

\begin{proof} The first statement is a formal exercise with the definition in terms of a lifting property.  For the second statement, we just note that, given a diagram \eqref{fandiagramtolift} and a neighborhood $U$ of $g(0)$ in $Y$, a lift in \eqref{fandiagramtolift} is the same thing as a lift in the diagram below. $$ \xym{ \Spec \ZZ \ar[d] \ar[r]^-g & f^{-1}(U) \ar[d]^f \\ \Spec \NN \ar[r]^-g \ar@{.>}[ru] & U } $$ \end{proof}

Here are some ``sanity checks:"

\begin{lem} \label{lem:propersanity} An affine morphism of fans is valuatively separated.   If $h : Q \to P$ is a dense (Definition~\ref{defn:dense}) or finite (Definition~\ref{defn:flat}) map of monoids, then $\Spec h$ is valuatively proper. \end{lem}

\begin{proof} Since the question is local on the base (Lemma~\ref{lem:valuativecriterion}), it is enough to prove that $\Spec h$ is valuatively separated when $h : Q \to P$ is an arbitrary map of monoids.  We need to show that there is at most one completion in a solid diagram of monoids as below. $$ \xym{ \ZZ & \ar[l]_-g P \ar@{.>}[ld] \\ \NN \ar[u] & \ar[l] Q \ar[u]_h } $$  But this is clear from injectivity of $\NN \to \ZZ$.  For the second statement, we need to show that such a completion exists when $h$ is dense or finite, or, equivalently, that $g$ must take values in $\NN \subseteq \ZZ$.  First suppose $h$ is dense.  Then for any $p \in P$, $np \in h(Q)$ so $ng(p) = g(np) \in \NN$, hence $g(p) \in \NN$.  Now suppose $h$ is finite.  Let $S$ be a finite subset of $P$ generating $P$ as a $Q$-module, so we can write any element $p \in P$ as $p = h(q)+s$ for some $q \in Q$, $s \in S$.  Since $g(h(Q)) \subseteq \NN$, we just need to show that $g(s) \in \NN$ for each $s \in S$.  Since $S$ is finite, we can certainly choose $N \in \ZZ$ such that $N < g(s)$ for all $s \in S$.  Now suppose, toward a contradiction, that $g(s) < 0$ for some $s \in S$.  Then $g(Ns) \leq N$.  On the other hand, we can write $Ns = h(q)+s'$ for some $s' \in S$, $q \in Q$ and we then get the contradiction: $$ N \geq g(Ns) = g(h(q)) + g(s') \geq g(s').$$   \end{proof}

\begin{example} \label{example:propernotclosed} When $h : Q \to P$ is a surjection of monoids, $\Spec h$ is an embedding (Lemma~\ref{lem:specembedding}), but it need not be a closed embedding even if $Q$ and $P$ are fine (Example~\ref{example:notclosed}).  For such a map $h$, the map of fine fans $\Spec h$ is valuatively proper by Lemma~\ref{lem:propersanity} even though it is \emph{not} a closed map on topological spaces. \end{example}

\begin{lem} \label{lem:proper1} Let $f : X \to Y$ be a map of fans.  Then $f$ is valuatively separated (resp.\ valuatively proper) iff $\AA(f) : \AA(X) \to \AA(Y)$ satisfies the discrete valuative criterion for separatedness (resp.\ properness) for maps generically into $\AA(T_X) \subseteq \AA(X)$. \end{lem}

\begin{proof} ($\Rightarrow$) Composing with the orbit maps $\tau_X, \tau_Y$, a solid diagram \eqref{schemediagramtolift} yields a solid $\LMS$ diagram \bne{LMSliftdia} & \xym{ \Spec K \ar[r]^-g \ar[d] & X \ar[d]^{f} \\ \Spec R \ar@{.>}[ru] \ar[r]^-g & Y } \ene and by the universal property of the scheme realization $\AA(\slot)$ (\S\ref{section:categoricalinterpretation}), a lift in \eqref{schemediagramtolift} is the same as a lift in \eqref{LMSliftdia}.  Let $x' := g(\Spec K) \in X$, so the hypothesis that $g$ is ``generically into $\AA(T_X)$" means that $x' \in T_X$, so $\M_{X,x'}$ is a group.  The diagram \eqref{LMSliftdia} corresponds to a solid commutative diagram of monoids $$ \xym@C-15pt@R-15pt{ & K & & \ar[ll] \M_{X,x'} \\ R \ar[ru] & & \ar@{.>}[ll] \M_{X,x} \ar@{.>}[ru] \\ & K \ar@{=}[uu]|\hole & & \ar[ll]|>>>>>>>>>\hole \M_{Y,f(x')} \ar[uu]_{f^\dagger_{x'}} \\ R \ar@{=}[uu] \ar[ru] & & \ar[ll] \M_{Y,y} \ar@{.>}[uu] \ar[ru] } $$ where the horizontal and vertical maps are local and the diagonal maps are localizations corresponding to generalizations.  A lift in \eqref{LMSliftdia} corresponds to a point $x \in \{ x' \}^- \cap f^{-1}(y)$ so that the composition \bne{compositiontoR} \M_{X,x} \to \M_{X,x'} \to K \ene actually defines a \emph{local} map of monoids $\M_{X,x} \to R$.  Since $f^\dagger_{x'}$ is local and $\M_{X,x'}$ is a group, so is $\M_{Y,f(x')}$.  The back horizontal arrows are therefore groups homomorphisms to $K^*$ and we can compose with the valuation $\nu$ to obtain group homomorphisms to $\ZZ$.  The commutativity of the bottom square ensures that the resulting composition $$ \M_{Y,y} \to \M_{Y,f(x')} \to K^* \to \ZZ $$ actually takes values in $\NN$ and then the locality of $\M_{Y,y} \to R$ ensures that the resulting monoid homomorphism $\M_{Y,y} \to \NN$ is local because $R^* = \{ a \in K^* : \nu(a) = 0 \}$.  We thus obtain a diagram of fans \bne{Fanliftdia} & \xym{ \Spec \ZZ \ar[r] \ar[d] & X \ar[d]^{f} \\ \Spec \NN \ar@{.>}[ru] \ar[r] & Y } \ene from \eqref{LMSliftdia} by ``composing with the valuation" which is ``the same" as \eqref{LMSliftdia} on the level of topological spaces.  The claim is that ``$x \mapsto x$" is a bijection between lifts in \eqref{Fanliftdia} and lifts in \eqref{LMSliftdia}.  Indeed, a lift in \eqref{Fanliftdia} corresponds to a choice of point $x \in \{ x' \}^- \cap f^{-1}(y)$ so that the composition \bne{fancomposition} \M_{X,x} \to \M_{X,x'} \to \ZZ \ene actually defines a local map $\M_{X,x} \to \NN$ (see Example~\ref{example:SpecN}), and (by the definition of $R \subseteq K$) this will be the case iff \eqref{compositiontoR} defines a local map of monoids $\M_{X,x} \to R$.

($\Leftarrow$)  Pick any field $k$.  Let $K := k((T))$ be the field of formal Laurent series in a variable $T$ with coefficients in $k$.  Equip $K$ with the valuation $\nu = \ord_T$ so that the valuation ring $R = k[[T]]$ of $K$ is the formal power series ring.  The group homomorphism $\ZZ \to K^*$ taking $1$ to $T$ gives rise to a local map of monoids $\NN \to R$.  (This particular choice of valuation ring is not particularly important.)  We thus obtain a diagram of locally monoidal spaces \bne{fg} & \xym{ | \Spec K | \ar[d] \ar[r] & \Spec \ZZ \ar[d] \\ | \Spec R | \ar[r] & \Spec \NN } \ene where the horizontal arrows are homeomorphisms and the notation $| \slot |$ means we regard a scheme as a locally monoidal space using the structure sheaf under multiplication (as in \S\ref{section:categoricalinterpretation}).  Suppose we have a diagram of fans \eqref{Fanliftdia}.  Then we can compose with \eqref{fg} and use the universal property of $\AA( \slot )$ to obtain a diagram of schemes: \bne{Schdia} & \xym{ \Spec K \ar[r] \ar[d] & \AA(X) \ar[d]^-{\AA(f)} \\ \Spec R \ar@{.>}[ru] \ar[r] & \AA(Y) } \ene  The same arguments used in the proof of the other implication show that ``composing with \eqref{fg} and using the universal property of $\tau$" establishes a bijection between lifts in \eqref{Fanliftdia} and lifts in \eqref{Schdia}.  (The inverse is basically given by composing a lift in \eqref{Schdia} with $\tau_X$.)    \end{proof}

\begin{lem} \label{lem:proper2} Let $f : X \to Y$ be a quasi-compact map of fine fans, $x \in X$.  If $f$ is valuatively separated (resp.\ valuatively proper), then so is $Z(f,x) : Z(X,x) \to Z(Y,f(x))$. \end{lem}

\begin{proof} Set $Z := Z(X,x)$, $W := Z(Y,f(x))$, $f := Z(f,x)$ to ease notation.  The boundary construction is local in nature, as is the property of being valuatively separated or proper, so we can assume $Y = \Spec P$ ($P$ a fine monoid) is affine, so $Y$ is finite (Corollary~\ref{cor:faces}) and hence $X$ is also finite by quasi-compactness of $f$ (Theorem~\ref{thm:quasicompact}).  

Consider a solid commutative diagram of fans as below.  \bne{Zdiagramtolift} & \xym{ \Spec \ZZ \ar[r]^-g \ar[d] & Z \ar[d]^f \\ \Spec \NN \ar@{.>}[ru]^-p \ar[r]^-g & W } \ene The point $x \in Z$ is the only point of $Z$ for which $\M_{Z,x} = \M_{X,x}^*$ is a group (because $Z = \{ x \}^-$ and there are no non-trivial localizations of a group), so we have to have $g(\Spec \ZZ) = x$ in such a diagram.  Let $y \in W$ denote the image of the closed point of $\Spec \NN$ under $g : \Spec \NN \to W$.  To give a lift as indicated in \eqref{Zdiagramtolift} is to give a point $p \in \{ x \}^- \cap f^{-1}(y)$ such that the composition \bne{comp1} & \xym{ \M_{Z,p} \ar[r] & \M_{Z,x} = \M_{X,x}^* \ar[r]^-g & \ZZ } \ene takes $\M_{Z,p}$ into $\NN$ via a local map (c.f.\ Example~\ref{example:SpecN}).  The question of whether there is at most one (resp.\ a unique) such $p$ is therefore independent of multiplying $g$ by a positive integer, so after making such a rescaling of $g$ if necessary, we can assume that $g : \M_{X,x}^* \to \ZZ $ lifts to a map also abusively denoted $g : \M_{X,x}^{\rm gp} \to \ZZ$.  Since $\M_{X,x}^{\rm gp}$ is a localization of $\M_{X,x}$, there is some $x' \in X$ for which $x \in \{ x' \}^-$ and the generalization map $\M_{X,x} \to \M_{X,x'}$ is the groupification: $\M_{X,x'} = \M_{X,x}^{\rm gp}$.  By Lemma~\ref{lem:closure} we can find a group homomorphism $h : \M_{X,x'} = \M_{X,x}^{\rm gp} \to \ZZ$ whose restriction to $\M_{X,x}$ defines a \emph{local} map of monoids $h : \M_{X,x} \to \NN$.  Notice that we can multiply $h$ by any positive integer without destroying this property.  Notice also that since $h$ kills $\M_{X,x}^* = \M_{Z,x}$ (and ``$g$ lifts $g$"), the restriction of $g + h : \M_{X,x'} \to \ZZ$ to $\M_{Z,x}$ agrees with the original map $g : \M_{Z,x} \to \ZZ$.  The situation so far is summed up by the solid commutative diagram below.  \bne{bigdiagram} & \xym@C-10pt@R-10pt{ & & & & \NN \ar[r] & \ZZ \\ \\ \ZZ \ar@{=}[dd] & & \ar[ll]_-g \M_{Z,x} = \M_{X,x}^* \ar[rruu]^0 \ar[rr] & &  \M_{X,x} \ar[uu]^h \ar[r] & \M_{X,x'} = \M_{X,x}^{\rm gp} \ar[uu]_h \ar@/_4pc/[lllll]_{g+h} \\ & \M_{Z,p} \ar@{.>}[ru] \ar@{.>}[rr] & &  \M_{X,p} \ar@{.>}[ru] \\ \ZZ & & \ar[ll]_>>>>>>{g}|<<<<<<<<<<<\hole \M_{W,f(x)} = \M_{Y,f(x)}^* \ar[uu]|\hole \ar[rr]|>>>>>>>>>>\hole & &  \M_{Y,f(x)} \ar[uu] \ar[r] & \M_{Y,f(x')} = \M_{Y,f(x)}^{\rm gp} \ar[uu] \\ \NN \ar[u] & \ar[l]_-g \M_{W,y} \ar[ru] \ar@{.>}[uu] \ar[rr] & & \M_{Y,y} \ar[ru] \ar@{.>}[uu] } \ene  Suppose we have a point $p \in \{ x' \}^-$ and an element $m \in \M_{X,p}$ for which the composition \bne{daggercomp} & \xym{ \M_{X,p} \ar[r] & \M_{X,x'} = \M_{X,x}^{\rm gp} \ar[r]^-{g+h} & \ZZ } \ene has $(g+h)(m) \in \NN$ but $h(m) < 0$.  Then we can rescale $h$ by a positive integer to ensure that $(g+h)(m)$ is also negative.  Since $X$ is finite, $\{ x' \}^-$ is certainly finite and each such $\M_{X,p}$ is finitely generated, so after rescaling $h$ by a sufficiently large positive integer, we can assume:

\noindent {\bf (*)} For any $p \in \{ x' \}^-$, if the composition \eqref{daggercomp} takes values in $\NN$, then so does the composition $$ \xym{ \M_{X,p} \ar[r] & \M_{X,x'} = \M_{X,x}^{\rm gp} \ar[r]^-{h} & \ZZ. } $$

By similarly rescaling $h$, we can also assume:

\noindent {\bf (**)} For any non-unit $m \in \M_{X,x}$, $(g+h)(m) > 0$.

Since $f: \M_{Y,f(x)} \to \M_{X,x}$ is local, any $m \in \M_{Y,y}$ mapping to a non-unit in $\M_{Y,f(x)}$ also maps to a non-unit in $\M_{X,x}$ , hence we have $(g+h)(m) > 0$ for all such $m$ by {\bf (**)}.  On the other hand, if $m \in \M_{Y,y}$ is a non-unit, but maps to a unit in $\M_{Y,f(x)}$, then since the square \bne{proper2cartesiansquare} & \xym{ \M_{W,f(x)} = \M_{Y,f(x)}^* \ar[r] & \M_{Y,f(x)} \\ \M_{W,y} \ar[r] \ar[u] & \M_{Y,y} \ar[u] } \ene appearing in \eqref{bigdiagram} is cartesian (Lemma~\ref{lem:boundaryfaces}) $m$ is in fact a non-unit in $\M_{W,y}$ and hence $g(m) > 0$ by locality of $g : \M_{W,y} \to \NN$.  We conclude that $g+h : \M_{Y,y} \to \NN$ is local, and hence we have a solid commutative diagram: \bne{2ndlifting} & \xym@C+40pt{ \Spec \ZZ \ar[d] \ar[r]^-{x', \; g+h} & X \ar[d]^f \\ \Spec \NN \ar@{.>}[ru]^p \ar[r]^-{y,\; f(x), \; g+h} & Y } \ene The proof will be complete upon establishing the

\noindent \emph{Claim.}  There is a bijection between liftings in \eqref{Zdiagramtolift} and liftings in \eqref{2ndlifting}.

Suppose we have a lifting $p$ in \eqref{2ndlifting}.  Such a $p$ is the same thing as a choice of point $p \in \{ x' \}^- \cap f^{-1}(y)$ for which the composition \eqref{daggercomp} yields a \emph{local} map $g+h : \M_{X,p} \to \NN$.  By {\bf (*)}, we know that this implies $h : \M_{X,p} \to \NN$, though we do not know this map $h$ is local.  However, there is of course a unique localization of $\M_{X,p}$ on which $h$ becomes local, so we can find some $p' \in X$ such that $p \in \{ p' \}^-$ and the composition \bne{daggercomp3} & \xym{ \M_{X,p} \ar[r] & \M_{X,p'} \ar[r]^-h & \NN } \ene factors $h : \M_{X,p} \to \NN$ through a \emph{local} map also denoted $h : \M_{X,p'} \to \NN$.  Notice that $p' \in \{ x' \}^-$ because $\M_{X,p'}^{\rm gp} = \M_{X,p}^{\rm gp} = \M_{X,x'}$.  We must have $f(p') = f(x')$ because $f(p')$ and $f(x)$ are both in $\{ f(x') \}^-$ and the composition $$ \xym{ \M_{Y,f(x')} \ar[r]^-f & \M_{X,x'} \ar[r]^-h & \ZZ } $$ restricts to a local map to $\NN$ on both $\M_{Y,f(x)}$ and $\M_{Y,f(p')}$ and the affine fan $Y$ is valuatively separated (Lemma~\ref{lem:propersanity}).  Now we see that both $p'$ and $x$ will provide lifts in the diagram of fans $$ \xym@C+40pt{ \Spec \ZZ \ar[d] \ar[r]^-{x',\; h} & X \ar[d]^f \\ \Spec \NN \ar@{.>}[ru] \ar[r]^-{f(x), \; f(x'), \; h} & Y } $$ and hence $p' = x$ because $f$ is valuatively separated.  We conclude that $p \in \{ p' \}^- = \{ x \}^- = Z$ and $f(p) = y$, so we have the dotted completion in \eqref{bigdiagram} and we can conclude that $p$ provides a lift in \eqref{Zdiagramtolift} provided that we can show that the map $g : \M_{Z,p} \to \ZZ$ appearing in the completed diagram \eqref{bigdiagram} actually yields a local map $g : \M_{Z,p} \to \NN$.  This is the case because we know $g+h : \M_{X,p} \to \NN$ is local and we know $h$ vanishes on $\M_{X,x}^*$, hence on $\M_{Z,p}$.

To go the other way, suppose we have a lifting $p$ in \eqref{Zdiagramtolift}.  This immediately yields the indicated dotted completion in \eqref{bigdiagram} so that the composition $g : \M_{Z,p} \to \ZZ$ there actually defines a local map $g : \M_{Z,p} \to \NN$.  Arguing much as in our previous usage of {\bf (**)}, we see that the cartesianness of the square $$ \xym{ \M_{Z,x} = \M_{X,x}^* \ar[r] & \M_{X,x} \\ \M_{Z,p} \ar[u] \ar[r] & \M_{X,p} \ar[u] } $$ (\S\ref{section:boundaryconstruction}) and property {\bf (**)} imply that $g+h : \M_{X,p} \to \NN$ is local, so this same $p$ provides a lifting in \eqref{2ndlifting}.

\end{proof}

\begin{thm} \label{thm:proper} A map $f : X \to Y$ of locally finite type fans is quasi-compact and valuatively separated (resp.\ quasi-compact and valuatively proper) if its scheme realization $\AA(f)$ is quasi-compact and separated (resp.\ proper).  The converse holds if we assume $f$ is a map of \emph{fine} fans. \end{thm}

\begin{proof}  We already saw in Theorem~\ref{thm:quasicompact} that $f$ is quasi-compact iff $\AA(f)$ is quasi-compact, so we can assume now that $f$ and $\AA(f)$ are quasi-compact.  Then $\AA(f)$ is a finite-type map of noetherian schemes, so by \cite[II.7.2.3]{EGA} (resp.\ \cite[II.7.3.8]{EGA}) $\AA(f)$ is separated (resp.\ proper) iff $\AA(f)$ satisfies the discrete valuative criterion for separatedness (resp.\ properness).

The implication $(\Leftarrow)$ is a consequence of Lemma~\ref{lem:proper1} because if $f$ failed to be separated (resp.\ proper), then $\AA(f)$ couldn't even satisfy the discrete valuative criterion for separatedness (resp.\ properness) for maps generically into $\AA(T_X)$.

For $(\Rightarrow)$ (assuming $X$, $Y$ fine): By the universal property of $\AA( \slot )$ (\S\ref{section:categoricalinterpretation}), a diagram \eqref{schemediagramtolift} as in Definition~\ref{defn:valuativecriterion} has at most one (resp.\ a unique) lift iff the corresponding $\LMS$ diagram \bne{LMSliftA} & \xym{ \Spec K \ar[r]^-g \ar[d] & X \ar[d]^f \\ \Spec R \ar@{.>}[ru] \ar[r]^-g & Y } \ene has at most one (resp.\ a unique) lift (to be precise, we should write $|\Spec K|$, $|\Spec R|$ here).  Let $x := g(\Spec K) \in X$ and let $y \in Y$ be the image of the closed point of $\Spec R$ under $g : \Spec R \to Y$.  Let $Z := Z(X,x)$, $W := Z(Y,f(x))$.  By ``restriction," the diagram \eqref{LMSliftA} yields an $\LMS$ diagram \bne{LMSliftB} & \xym{ \Spec K \ar[r]^-g \ar[d] & Z \ar[d]^{Z(f,x)} \\ \Spec R \ar@{.>}[ru] \ar[r]^-g & W. } \ene 

I claim that lifts in \eqref{LMSliftA} are bijective with lifts in \eqref{LMSliftB}.  A lift in \eqref{LMSliftA} corresponds to a point $p \in Z = \{ x \}^- \cap f^{-1}(y)$ for which the composition $$ \xym{ \M_{X,p} \ar[r] & \M_{X,x} \ar[r]^-g & K } $$ actually yields a local map $\M_{X,p} \to R$.  Since the diagram \bne{DF} & \xym{ \M_{X,x}^* = \M_{Z,x} \ar[r] & \M_{X,x} \\ \M_{Z,p} \ar[r] \ar[u] & \M_{X,p} \ar[u] } \ene is cartesian by the boundary construction (\S\ref{section:boundaryconstruction}), $\M_{Z,p} \to \M_{X,p}$ is local (it is the inclusion of a face), so we see that $g : \M_{Z,p} \to R$ is also local, so that $p$ also serves as a lift in \eqref{LMSliftB}.  To go the other way, suppose $p$ is a lift in \eqref{LMSliftB}.  Then $p \in Z = \{ x \}^- \cap f^{-1}(y)$ and $g : \M_{Z,p} \to R$ is local; the issue is to show that the composition $$ \xym{ \M_{X,p} \ar[r] & \M_{X,x} \ar[r]^-g & K } $$ actually defines a local map $\M_{X,p} \to R$.  If $m \in \M_{X,p}$ maps to a unit in $\M_{X,x}$, then $m \in \M_{Z,p}$ since \eqref{DF} is cartesian, hence $g(m) \in R$ and $g(m) \in R^*$ iff $m \in \M_{Z,p}^*$ iff $m \in \M_{X,p}^*$.  On the other hand, if $m \in \M_{X,p}$ maps to a non-unit in $\M_{X,x}$, then $g(m) = 0$ because $g : \M_{X,x} \to K$ is local.  This proves the claim.

Since $f$ is a valuatively separated (resp.\ valuatively proper) and quasi-compact map of fine fans, Lemma~\ref{lem:proper2} says that $Z(f,x)$ is also valuatively separated (resp.\ valuatively proper), so there is at most one (resp.\ a unique) lift in \eqref{LMSliftB}, as desired.    \end{proof}

\begin{defn} \label{defn:proper}  A map $f : X \to Y$ of locally finite type fans is called \emph{separated} (resp.\ \emph{proper}) iff its scheme realization $\AA(f)$ is separated (resp.\ proper). \end{defn}

\begin{rem} In informal, down-to-earth terms, Theorem~\ref{thm:proper} says that a (quasi-compact) map of fine fans $f : X \to Y$ is separated (resp.\ proper) (i.e.\ $\AA(f)$ is separated (resp.\ proper)) iff, for any ``one-parameter subgroup" $\lambda$ of the torus $T_X \subseteq X$ and any ``limit" $\lambda'$ of the one-parameter subgroup $f \lambda$ of $T_Y$, there is at most one (resp.\ a unique) limit of $\lambda$ lying over $\lambda'$.  For example, if $f : X \to Y$ ``is" a map of classical fans $$f : (\Sigma,N) \to (\Sigma',N')$$ (regarded as a map of fans as in \S\ref{section:classicalfans}), then a map $\Spec \ZZ \to X$ is the same thing as an element of $n$ (a group homomorphism $\ZZ \to N$), and such a map extends uniquely to a map $\Spec \NN \to X$ iff $n$ is in the support of $\Sigma$.  Theorem~\ref{thm:proper} thus specializes to the usual criterion for properness of a map of toric varieties \cite[2.4]{F}.  Most of the effort that goes into the proof of Theorem~\ref{thm:proper} is really a justification (though in greater generality) of the ``one may assume..." assertion of \cite[Page 40]{F} (which is relatively easy to justify in that situation). \end{rem}

\begin{example} \label{example:x} The converse of the first statement of Theorem~\ref{thm:proper} doesn't hold in general.  Consider the same finitely generted, non-integral monoid \be P & := \langle a,b : a+b = a \rangle \ee as in Example~\ref{example:stratification}.  The map $a : \NN \to P$ induces an isomorphism $\NN \to P^{\rm int}$, so $\Spec a$ is valuatively proper because it is an isomorphism as far as maps out of integral fans (such as $\Spec \NN$ and $\Spec \ZZ$) are concerned.  But $\ZZ[a] : \ZZ[\NN] \to \ZZ[P]$ is the map $\ZZ[x] \to \ZZ[x,y]/\langle xy-x \rangle$, which isn't proper because the fiber over $\{ x = 0 \}$ is $\AA^1_y$.  This example shows that there is no way to test the properness of a general map of finite type fans via maps from integral fans. \end{example}

In the remainder of \S\ref{section:propermaps} we will discuss the relationship between the notions of ``separated" and ``proper" (Definition~\ref{defn:proper}) and ``valuatively separated/proper" (Definition~\ref{defn:valuativecriterion}) and the usual topological notion of properness for the map of topological spaces underlying the ``(positive) log differential realization" of a map of locally finite type fans.  For a complete discussion of the latter, we refer to \cite{GM1}[\S4, \S6], though here we shall be concerned only with the underlying topological realization.  We shall now give a brief, self-contained discussion of this for the convenience of the reader.

Let $\Fans$ denote the category of locally finite type fans, $\Top$ the category of topological spaces.  Then the general results of \cite{GM1}[\S4] imply that there exists a functor \bne{differentialrealization} \RR_{\geq 0} : \Fans & \to & \Top \ene satisfying the following properties: \begin{enumerate} \item $\RR_{\geq 0}$ preserves finite inverse limits.  In particular, $\RR_{\geq 0}$ takes monoid objects in $\Fans$ to monoid objects in $\Top$. \item The topological monoid $\RR_{\geq 0}( \Spec \NN )$ is isomorphic to $\RR_{\geq 0}$ with the usual metric topology and multiplication as the monoid operation.  \item $\RR_{\geq 0}$ takes open embeddings of fans to open embeddings of topological spaces. \item $\RR_{\geq 0}$ preserves filtered direct limits whose transition functions are open embeddings as well as pushout diagrams where the maps are open embeddings.  (Note that both $\Fans$ and $\Top$ both \emph{have} such direct limits.) \end{enumerate}  Furthermore, it is shown in \cite{GM1}[\S4] that this functor is ``unique up to unique isomorphism" in an appropriate sense.  For a finitely generated monoid $P$, we will usually write $\RR_{\geq 0}(P)$ instead of $\RR_{\geq 0}(\Spec P)$.  (In the paper \cite{GM1}, the topological space $\RR_{\geq 0}(X)$ is also equipped with a sheaf of $\RR$-algebras making it into a differentiable space, as well as a ``positive log structure," but we shall not make any use of that additional structure here.)

For convenience, we shall now give an explicit construction of the functor $\RR_{\geq 0}$.  Given a fan $X$, we let $\RR_{\geq 0}(X)$ be the set of pairs $(x,f)$, where $x \in X$ and $f : \M_{X,x} \to \RR_{\geq 0}$ is a local monoid homomorphism.  For example, if $X = \Spec P$, then one checks easily that \bne{affinecase} \Hom_{\Mon}(P,\RR_{\geq 0}) & \to & \RR_{\geq 0}(P) \\ \nonumber h & \mapsto & (h^{-1}(0), ``h" : h^{-1}(0)^{-1}P \to \RR_{\geq 0}) \ene is bijective.  We topologize $\RR_{\geq 0}(X)$ by giving it the smallest topology where the subsets \be \UU(U,V,p) & := & \{ (x,f) \in \RR_{\geq 0}(X) : x \in U, f(p_x) \in V \} \ee are open for each open subseteq $U$ of $X$, each open subset $V$ of $\RR_{\geq 0}$ (in the usual metric topology), and each $p \in \M_X(U)$.  The reader can check that the topological space $\RR_{\geq 0}(X)$ thus defined is functorial in $X$ and that the resulting functor satisfies the properties listed in the previous paragraph.  

One can also show that \eqref{affinecase} is a homeomorphism when $\Hom_{\Mon}(P,\RR_{\geq 0})$ is given either of the following topologies: \begin{enumerate} \item the compact-open topology (using the discrete topology on $P$ and the usual metric topology on $\RR_{\geq 0}$) \item the smallest topology such that the ``evaluation at $p$" map $e_p : \Hom_{\Mon}(P,\RR_{\geq 0}) \to \RR_{\geq 0}$ is continuous (for the usual metric topology on $\RR_{\geq 0}$) for each $p \in P$ (the ``weak star topology") \end{enumerate}  It is similarly true that for a finitely generated monoid $P$, the analytic topology on the set \be (\Spec \CC[P])(\CC) & = & \Hom_{\Alg(\CC)}(P,\CC) \\ & = & \Hom_{\Mon}(P,\CC) \ee coincides with the analogous ``weak star topology" defined using the usual topology of $\CC$.  (Depending on how one defines the ``analytic topology" $X^{\rm an}$ on the set of $\CC$-points of a finite type $\CC$-scheme $X$, this may be more or less obvious.)  Since $\RR_{\geq 0}$ is a closed subset of $\CC$, it follows that $\RR_{\geq 0}(P)$ is a closed subspace of $(\Spec \CC[P])^{\rm an}$ for every finitely generated monoid $P$.  Using the properties of \eqref{differentialrealization} in the previous paragraph, together with analogous properties of the (complex) algebraic realization and the ``analytic topology," one sees that $\RR_{\geq 0}(X)$ is a closed subspace of $$\AA(X)^{\rm an} := (\AA(X) \otimes_{\ZZ} \CC)^{\rm an}$$ for any locally finite type fan $X$.  This is an important relationship between the algebraic realization (\S\ref{section:algebraicrealization}) and the ``topological realization" functor \eqref{differentialrealization}.

By definition of the topology on $\RR_{\geq 0}(X)$, we have a continuous ``orbit map" $\tau : \RR_{\geq 0}(X) \to X$ given by $\tau(x,f) = x$, much as in the case of the algebraic realization (\S\ref{section:orbitmap}).

\begin{lem} \label{lem:proper} Consider a commutative diagram of locally finite type fans as on the left below, and the commutative diagram of topological spaces given by its image under \eqref{differentialrealization} (shown on the right below).  $$ \xym{ \Spec \ZZ \ar[r]^-{ g } \ar[d] & X \ar[d]^f \\ \Spec \NN \ar[r] \ar@{.>}[ru]^-{h} & Y } \quad \quad \xym{ \RR_{>0} \ar[r]^-{\RR_{\geq}(g)} \ar[d] & \RR_{\geq 0}(X) \ar[d]^{\RR_{\geq 0}(f)} \\ \Spec \NN \ar[r] & \RR_{\geq 0}(Y) } $$  Then the map $h \mapsto \RR_{\geq 0}(h)$ yields a bijection between the set of completions of the left diagram (as indicated) and the set of completions of the right diagram (as indicated). \end{lem}

\begin{proof}  We will make use of the discussion in Example~\ref{example:SpecN} throughout the proof.  We will also use the explicit description of the functor \eqref{differentialrealization} given just before the statement of the lemma.  Let $x' \in X$ be the image of the unique point of $\Spec \ZZ$ under $g : \Spec \ZZ \to X$, so $\M_{X,x'}$ is a group and $g$ corresponds to a group homomorphism $g : \M_{X,x'} \to \ZZ$ (cf.\ Example~\ref{example:SpecN}).  Let $y' \in Y$ be the image of $x'$ under $f$, so the solid diagram on the left corresponds to a specialization $y \in \{ y' \}^-$ of $y'$ for which the composition $$\M_{Y,y} \to \M_{Y,y'} \to \M_{X,x'} \to \ZZ$$ (the last map being $g$) defines a local monoid homomorphism $\M_{Y,y} \to \NN \subseteq \ZZ$.  A completion $h$ as indicated on the left above corresponds to a choice of point $x \in \{ x' \}^-$ with $f(x)=y$ for which the composition $\M_{X,x} \to \M_{X,x'} \to \ZZ$ (of the generalization map and $g$) yields a local monoid homomorphism $h : \M_{X,x} \to \NN \subseteq \ZZ$.  Given such a completion $h$, the map $\RR_{\geq 0}(h) : \RR_{\geq 0} \to \RR_{\geq 0}(X)$ takes $0 \in \RR_{\geq 0}$ to the point $(x,f_x) \in \RR_{\geq 0}(X)$, where $f_x : \M_{X,x} \to \RR_{\geq 0}$ where $f_x$ is the local monoid homomorphism defined by $f_x(p) := 1$ when $h(p)=0$ (equivalently, when $p \in \M_{X,x}^*$) and $f_x(p) := 1$ when $h(p) > 0$.  Evidently then, we can recover $x$ (and hence $h$) from $\RR_{\geq 0}(h)$ by the formula $x = \tau( \RR_{\geq 0}(h)(0) )$, so $h \mapsto \RR_{\geq 0}(h)$ is injective.  

For surjectivity, suppose $\ell$ is a completion as indicated in the right diagram.  The map $\RR_{\geq 0}(g) : \RR_{>0} \to \RR_{\geq 0}(X)$ takes $\lambda \in \RR_{>0}$ to $(x',f_{\lambda}) \in \RR_{\geq 0}(X)$, where $f_{\lambda} : \M_{X,x'} \to \RR_{\geq 0}$ is the monoid homomorphism given by $f_{\lambda}(p) = \lambda^{ g(p) }$.  (This monoid homomorphism takes values in $\RR_{>0} \subseteq \RR_{\geq 0}$ and is automatically local because $\M_{X,x'}$ is a group.)  Let $x := \tau(\ell(0)) \in X$, so $\ell(0)=(x,f)$ for some local monoid homomorphism $f : \M_{X,x} \to \RR_{\geq 0}$.  Since $\tau$ and $\ell$ are continuous, $\RR_{>0} \subseteq \RR_{\geq 0}$ is dense, and $\tau$ takes $\ell(\RR_{>0})$ to $x'$ (because $\ell | \RR_{>0} = \RR_{\geq 0}(g)$ and $\tau( \RR_{\geq 0}(g)(\lambda) ) = x'$ for every $\lambda \in \RR_{>0}$), we must have $x \in \{ x' \}^{-}$.  Commutativity of the diagram on the right (at $0 \in \RR_{\geq 0}$) ensures that $f(x)=y$.  We claim that the composition \bne{localcomp} & \M_{X,x} \to \M_{X,x'} \to \ZZ \ene (we will often suppress notation for the first map) defines a local monoid homomorphism $h : \M_{X,x} \to \NN \subseteq \ZZ$ and that $\RR_{\geq 0}(h)=\ell$, which will complete the proof.  Continuity of $\ell$ and the equality $\ell | \RR_{>0} = \RR_{\geq 0}(g)$ together ensure that for any $p \in \M_{X,x}$, $f(p) \in \RR_{\geq 0}$ must be the limit as $\lambda \to 0$ (``from the right") of $f_{\lambda}(p) = \lambda^{g(p)} \in \RR_{>0} \subseteq \RR_{\geq 0}$.  This limit does not exist if $g(p) < 0$, is given by $1$ if $g(p)=0$, and is given by $0$ if $g(p) > 0$, so, using locality of $f$, we conclude that the composition \eqref{localcomp} indeed defines a local monoid homomorphism $h : \M_{X,x} \to \NN \subseteq \ZZ$ and that $f$ coincides with the monoid homomorphism $f_x$ defined in the previous paragraph, so $\RR_{\geq 0}(h)=\ell$, as desired.  \end{proof}

\begin{thm} \label{thm:diffproper} Let $f : X \to Y$ be a quasi-compact map of locally finite type fans.  Then each of the following conditions implies the next: \begin{enumerate} \item \label{diffproper1} $f$ is separated (resp.\ proper) (i.e.\ the quasi-compact map of locally finite type $\ZZ$-schemes $\AA(f)$ is separated (resp.\ proper)). \item \label{diffproper2} $\AA(f) \otimes_{\ZZ} \CC$ is separated (proper) \item \label{diffproper3} the map of analytic spaces $\AA(f)^{\rm an} := (\AA(f) \otimes_{\ZZ} \CC)^{\rm an}$ is separated (resp.\ proper). \item \label{diffproper4} the map of topological spaces $\RR_{\geq 0}(f)$ is separated (resp.\ proper). \end{enumerate} These conditions are equivalent if $X$ and $Y$ are fine. \end{thm}

\begin{proof} \eqref{diffproper1} implies \eqref{diffproper2} because separated (resp.\ proper) maps of schemes are stable under base change and \eqref{diffproper2} implies \eqref{diffproper3} by standard ``GAGA" results \cite[\S7]{GAGA}.  To see that \eqref{diffproper3} implies \eqref{diffproper4}, note that we have a commutative diagram of topological spaces $$ \xym{ \RR_{\geq 0}(X) \ar[r] \ar[d]^{\RR_{\geq 0}(f)} & \AA(X)^{\rm an} \ar[d]^{ \AA(f)^{\rm an} } \\ \RR_{\geq 0}(Y) \ar[r] & \AA(Y)^{\rm an} } $$ where the horizontal arrows are closed embeddings (see the discussion before Lemma~\ref{lem:proper}) and one can see (exercise!) that, in this situation, if the right vertical arrow is separated (resp.\ proper) then the left vertical arrow is also separated (resp.\ proper).  Using Lemma~\ref{lem:proper} one can see that condition \eqref{diffproper4} implies that $f$ is valuatively separated (resp.\ valuatively proper), so when $X$ and $Y$ are fine, $f$ is separated (resp.\ proper) by Theorem~\ref{thm:proper}. \end{proof}

\subsection{Alternative approach} \label{section:alternativeapproach} In this section we give an alternative approach to separatedness and properness, using non-integral monoids in an essential way.  This is based on the valuative approach to properness in \cite[\S8]{CDH}.

Suppose $K$ is a field with a discrete valuation $\nu : K^* \to \ZZ$.  We can extend $\nu$ to a (local) map of monoids $\nu : K \to \ZZ_* = \ZZ \coprod \{ \infty \}$ by setting $\nu(0) := \infty$.  Recall from \S\ref{section:monoidschemes} that the addition on $\ZZ_*$ extends the usual addition on $\ZZ$ by the rule $p + \infty = \infty$ for all $p \in \ZZ_*$.  The valuation ring $R$ of $K$ then sits in a \emph{cartesian} diagram of monoids: \bne{valdiagram} & \xym{ R \ar[r]^-{\nu} \ar[d] & \NN_* = \NN \coprod \{ \infty \} \ar[d] \\ K \ar[r]^-{\nu} & \ZZ_* = \ZZ \coprod \{ \infty \} } \ene  

\begin{thm} \label{thm:properalternative} Let $f : X \to Y$ be a map of fans.  Then the map of schemes $\AA(f)$ satisfies the discrete valuative criterion for separatedness (resp.\ properness) (Definition~\ref{defn:valuativecriterion}) iff there is at most one (resp.\ a unique) lift in any solid diagram of fans as below. \bne{dtol} & \xym{ \Spec \ZZ_* \ar[r] \ar[d] & X \ar[d]^f \\ \Spec \NN_* \ar@{.>}[ru] \ar[r]^-g & Y } \ene \end{thm}

\begin{proof} $(\Rightarrow)$ From the solid diagram \eqref{dtol}, we construct a solid $\LMS$ diagram \bne{dtol2} & \xym{ | \Spec \, k((T)) | \ar[r] \ar[d] & X \ar[d]^f \\ | \Spec k[[T]] | \ar@{.>}[ru] \ar[r] & Y } \ene ($k$ any field) as follows:  Let $y \in Y$ (resp.\ $y' \in Y$) be the image of the closed (resp.\ generic) point of $\Spec \NN_*$ so we have a local map of monoids $g : \M_{Y,y'} \to \ZZ_*$ restricting to a local map of monoids $g : \M_{Y,y} \to \NN_*$.   This $g$ gives rise to a local map of monoids $\M_{Y,y'} \to k((T))$ restricting to a local map of monoids $\M_{Y,y} \to k[[T]]$ by taking $m$ to $0$ if $g(m) = \infty$ and to $T^{g(m)}$ otherwise.  The top horizontal arrow of \eqref{dtol2} is defined similarly.  It is straightforward to see that lifts in \eqref{dtol2} are in bijective correspondence with lifts in \eqref{dtol}.  But the universal property of $\AA( \slot )$ ensures that lifts in \eqref{dtol2} are in bijective correspondence with lifts in the diagram of schemes \bne{dtol3} & \xym{  \Spec k((T))  \ar[r] \ar[d] & \AA(X) \ar[d]^-{\AA(f)} \\  \Spec k[[T]]  \ar@{.>}[ru] \ar[r] & \AA(Y) } \ene and we know by hypothesis that there is at most one (resp.\ a unique) such lift.

$(\Leftarrow)$ Suppose $K$ is a field with discrete valuation $\nu : K^* \to \ZZ$ and valuation ring $R$.  We need to prove that there is at most one (resp.\ a unique) lift in any solid diagram of schemes as below. \bne{dtol4} & \xym{  \Spec K  \ar[r] \ar[d] & \AA(X) \ar[d]^-{\AA(f)} \\  \Spec R  \ar@{.>}[ru] \ar[r] & \AA(Y) } \ene  By the universal property of $\AA(\slot)$, it is equivalent to prove that there is at most one (resp.\ a unique) lift in the corresponding solid square on the right of the $\LMS$ diagram below.  \bne{dtol5} & \xym{  \Spec \ZZ_* \ar[r] \ar[d] & \Spec K  \ar[r] \ar[d] & X \ar[d]^-{f} \\  \Spec \NN_* \ar[r] & \Spec R  \ar@{.>}[ru] \ar[r] & Y } \ene  By hypothesis, there is at most one (resp.\ a unique) lift in the ``big" square of \eqref{dtol5}.  The result then follows by arguing that the left square of \eqref{dtol5} is a pushout in $\LMS$---this follows easily from the fact that the horizontal arrows in that square are homeomorphisms and \eqref{valdiagram} is cartesian. \end{proof}

\begin{example} The improperness of the scheme realization of the map $x : \NN \to P$ of Example~\ref{example:x} is detected by the non-existence of a completion in the solid diagram of monoids below. $$ \xym@C+20pt{ \ZZ_* & \ar[l]_-{x \mapsto \infty}^-{y \mapsto -1} \ar@{.>}[ld] P \\ \NN_* \ar[u] & \ar[l]_-{\infty}  \NN \ar[u]_x} $$ \end{example}

\subsection{Equidimensional maps} \label{section:equidimensionalmaps}

\begin{defn} \label{defn:equidimensional} A scheme $X$ of finite type over a field $k$ is said to be of \emph{pure dimension} $n$ iff every irreducible component of $X$ is of dimension $n$.  A finite type map of schemes $f : X \to Y$ is called \emph{equidimensional of relative dimension} $n$ (``e.r.d.\ $n$" for short) iff, for each $y \in Y$, the scheme theoretic fiber $X_y = f^{-1}(y) := X \times_Y \Spec k(y)$ is of pure dimension $n$.  We adopt the convention that the empty scheme is of dimension $n$ for every $n$, hence the question of whether $f$ is e.r.d.\ $n$ depends only on the non-empty fibers of $f$.  We say that a quasi-compact map $f : X \to Y$ of locally finite type fans is \emph{e.r.d.\ n} iff the finite type map of schemes $\AA(f)$ is e.r.d.\ $n$ in the previous sense.  Similarly, we say that a map $h$ of finitely generated monoids is e.r.d.\ $n$ iff $\Spec h$ is e.r.d.\ $n$ in the previous sense.  \end{defn}

\begin{rem} \label{rem:equidimensional} It is worth emphasizing that the topological space underlying the scheme theoretic fiber $f^{-1}(y)$ is nothing but the subspace of $X$ denoted in the same way, so equidimensionality for maps of schemes is a purely ``topological" concept.  If $\{ U_i \}$ is an open cover of a scheme $X$ of finite type over a field $k$, then it is an easy exercise in ``noetherian topology" to see that $X$ is of pure dimension $n$ iff each $U_i$ is of pure dimension $n$.  It follows that the question of whether a (finite type) map of schemes $f :X \to Y$ is e.r.d.\ $n$ is local on both $X$ and $Y$.  It then follows from Remark~\ref{rem:localproperties} that ``e.r.d.\ $n$" for quasi-compact maps of locally finite type fans is also similarly local and that such a map $f :X \to Y$ is e.r.d.\  $n$ iff $f_x : \M_{Y,f(x)} \to \M_{X,x}$ is e.r.d.\ $n$ for each $x \in X$. \end{rem}

\begin{lem} \label{lem:puredimension} Suppose $f : X \to Y$ is a flat, finite, surjective map of finite type schemes over a field $k$.  Then $X$ is of pure dimension $n$ iff $Y$ is of pure dimension $n$.  In particular, if $k \into K$ is a finite field extension, $X$ is of pure dimension $n$ iff $X_K$ is of pure dimension $n$. \end{lem}

\begin{proof} Combine \cite[IV.2.3.4]{EGA} and \cite[IV.5.4.2]{EGA}. \end{proof}

\begin{lem} \label{lem:equidimensionality} Suppose $f : X \to Y$ is a finite type map of schemes and $Y' \to Y$ is a surjective, quasi-finite map of schemes.  Then $f$ is equidimensional of relative dimension $n$ iff the base change $f' : X' \to Y'$ of $f$ along $Y' \to Y$ is e.r.d.\ $n$. \end{lem}

\begin{proof} The assumptions on $Y' \to Y$ ensure that for any $y \in Y$ there is at least one $y' \in Y$ lying over $y$ and that for any such $y'$, the field extension $k(y) \to k(y')$ is finite.  It follows that the fibers of $f'$ and the fibers of $f$ are related by base change along finite field extensions, so the result follows from the previous lemma. \end{proof}

\begin{rem} Presumably the assumption that $k \into K$ be finite in Lemma~\ref{lem:puredimension} (and hence also the assumption that $Y' \to Y$ be quasi-finite in Lemma~\ref{lem:equidimensionality}) can be removed by ``limit arguments," but we shall have no need for this so we will not pursue it. \end{rem}

\begin{lem} \label{lem:torusequidimensionality} For any finitely generated abelian group $A$ and any field $k$, the finite type $k$-scheme $\Spec k[A]$ is of pure dimension $\rk A$.  More generally, for any map $f : A \to B$ of such groups, the corresponding map of schemes \be ( \Spec \ZZ[f] : \Spec \ZZ[B] \to \Spec \ZZ[A] ) & = & ( \AA(f) : \AA(B) \to \AA(A) ) \ee is equidimensional of relative dimension $\rk \Cok f$.  (That is, $f$ is e.r.d.\ $\rk \Cok f$ as a map of monoids.) \end{lem}

\begin{proof} For the first statement, set $\ov{A} := A / A_{\rm tor}$, $d := \rk A = \rk \ov{A}$.  Choose a splitting $A = \ov{A} \oplus A_{\rm tor}$.  Certainly $\Spec k[\ov{A}]$ is of pure dimension $d$.  (It is isomorphic to the open subscheme $\GG_m^d$ of $\AA_k^d$.)  Since $k[\ov{A}] \to k[A]$ is faithfully flat and finite, we conclude that $\Spec k[A]$ is also of pure dimension $d$ by Lemma~\ref{lem:puredimension}.

For the second statement, factor $f : A \to B$ as a surjection $g : A \to A'$ followed by an injection $i : A' \into A$.  Then $\Spec \ZZ[f]$ factors through the closed embedding $\Spec \ZZ[g]$ via the map $\Spec \ZZ[i]$, so the (non-empty) fibers of $\Spec \ZZ[f]$ are the same as those of $\Spec \ZZ[i]$.  Of course we also have $\Cok f = \Cok i$, thus we reduce to the case where $f$ is injective.  If we could choose a splitting $B = A \oplus \Cok f$ of $f$, then $\Spec \ZZ[f]$ would be the projection $$\Spec \ZZ[A] \times \Spec \ZZ[ \Cok f] \to \Spec \ZZ[A]$$ and the result would follow from what we proved above.  In general, we can reduce to the split case by first finding an injective map $A \into A'$ with finite cokernel such that the pushout map $f' : A' \to B' := B \oplus_A A'$ is \emph{split} injective.  (We will see that this can be done in Proposition~\ref{prop:CZEtorsors}.)  We then conclude by applying Lemma~\ref{lem:equidimensionality} with $Y' \to Y$ there given by the finite, surjective (and flat) map $\Spec \ZZ[A'] \to \Spec \ZZ[A]$.   \end{proof}

We next need a technical result to the effect that the stratifications of realizations of fine fans from Proposition~\ref{prop:stratification} are ``compatible with passing to fibers."  The basic issue is that if \be X & = & \coprod_i X_i \ee is a stratification of a topological space $X$ and $Z$ is a subspace of $X$ (even a closed subspace), then \be Z & = & \coprod_i (Z \cap X_i) \ee need not be a stratification of $Z$.  Of course the $Z \cap X_i$ will be locally closed in $Z$ if the $X_i$ are locally closed in $X$, but the problem is that the closure (in $Z$) of a stratum $Z \cap X_i$ may fail to be a union of such strata.  For example, if we consider the stratification of $\AA^2_{x,y} = \AA(\Spec \NN^2)$ coming from that proposition (the usual stratification of the plane by the four torus orbits) and we take \be Z & := & \{ x = 1 \} \cup \{ y = 0 \}, \ee then this problem arises for the ``generic stratum" $Z \cap \{ x,y \neq 0 \}$.  We need to make sure this doesn't happen when $Z$ arises as a fiber of a realization of a ``reasonable" map of fans.

\begin{lem} \label{lem:fiberstratification}  Let $f : X \to Y$ be a quasi-compact map of fine fans, $q \in \AA(Y)$, $y := \tau(q) \in Y$.  Then \bne{fiberstratification} \AA(f)^{-1}(q) & = & \coprod_{x \in f^{-1}(y) } \AA(f)^{-1}(q) \cap \AA(\M_{X,x}^*) \ene is a finite stratification of the scheme $\AA(f)^{-1}(q)$ into locally closed subschemes such that the closure of the stratum indexed by $x \in f^{-1}(y)$ is the union of the strata indexed by the points $x' \in f^{-1}(y)$ with $x'$ in the closure of $x$ in $X$.  In fact, suppose $x,x' \in f^{-1}(y)$, $x'$ is in the closure of $x$ in $X$, and $p'$ is a point of $\AA(f)^{-1}(q) \cap \AA(\M_{X,x'}^*)$.  Then there is a finite field extension $k(p') \subseteq K$ and a map of schemes $\AA^1_K \to \AA(f)^{-1}(q)$ taking the origin to $p'$ and the compliment of the origin into $\AA(f)^{-1}(q) \cap \AA(\M_{X,x}^*)$. \end{lem}

\begin{proof} The statement we want to prove is local on $Y$ so we reduce to the case where $Y = \Spec Q$ for a finitely generated monoid $Q$, hence the topological space underlying $Y$ is finite by Corollary~\ref{cor:faces} and the assumption that $f$ is quasi-compact then also ensures that the space underlying $X$ is finite (Theorem~\ref{thm:quasicompact}).  In this case, we know from Proposition~\ref{prop:stratification} that \be \AA(X) & = & \coprod_{x \in X} \AA(\M_{X,x}^*) \ee is a finite stratification of $X$ by locally closed subschemes.  We see by elementary topological considerations that \eqref{fiberstratification} will be as desired provided we can prove that \be \AA(f)^{-1}(q) \cap \AA(\M_{X,x'}^*) & \subseteq & \ov{ \AA(f)^{-1}(q) \cap \AA(\M_{X,x}^*) } \ee whenever $x$ and $x'$ are as in the statement of the proposition.  Evidently then, it is enough to prove the ``In fact".

To do this, first note that, by using Theorem~\ref{thm:boundary} and Example~\ref{example:ZXx}, we can assume, after possibly replacing $f : X \to Y$ with $Z(X,x) \to Z(Y,y)$, that $\M_{X,x}=\M_{X,x}^*$ and $\M_{Y,y}=\M_{Y,y}^*$ are groups.  The point $q$ corresponds to a group homomorphism $q : \M_{Y,y} \to  k(q)^*$.  A point $p' \in \AA(f)^{-1}(q) \cap \AA(\M_{X,x'}^*)$ corresponds to a local monoid homomorphism $p' : \M_{X,x'} \to k(p')$.  Since $\M_{X,x}$ is a group, the generalization map $\M_{X,x'} \to \M_{X,x}$ must be the groupification of $\M_{X,x'}$ since it is a localization at some face by Proposition~\ref{prop:fans}\eqref{fangeneralization}.  Since the group homomorphism $\M_{X,x'}^* \to \M_{X,x}$ is injective with finitely generated cokernel (because $X$ is fine), we can find a finite field extension $k(p') \subseteq K$ such that $(p')^* : \M_{X,x'}^* \to k(p')^* \subseteq K^*$ extends to a group homomorphism $p : \M_{X,x} \to K^*$.  The situation thus far is summed up in the commutative diagram of monoids below.

\bne{fiberstratdiagram} & \xym{ K^* \\ k(p')^* \ar[u]^-{\subseteq} \ar[r]^-{\subseteq} & k(p') & \ar[l]_-{\supseteq} k(q)^* \\ \M_{X,x'}^* \ar[u]^{(p')^*} \ar[r]^-{\subseteq} & \M_{X,x'} \ar[u]^{p'} \ar[rr] & & \M_{X,x} \ar@/_3pc/[llluu]_-{p} \\ & & \M_{Y,y} \ar[uu]|\hole_<<<<<{q} \ar[ul]^{f_{x'}} \ar[ur]_{f_x} } \ene

By Lemma~\ref{lem:duality}, we can find a local map of monoids $g : \M_{X,x'} \to \NN$.  Write $G : \M_{X,x} \to \ZZ$ for the groupification of $g$.  Consider the ring homomorphism $h : \ZZ[\M_{X,x'}] \to K[T]$ ($T$ an indeterminate) defined by $h([m]) := p'(m)T^{g(m)}$ for $m \in \M_{X,x'}$.  Since $g(f_{x'}(m))=0$ for each $m \in \M_{Y,y} = \M_{Y,y}^*$, we see from commutativity of \eqref{fiberstratdiagram} that the top square in the diagram of rings \bne{ringdia} & \xym@C+40pt{ \ZZ[\M_{Y,y}] \ar[r]^-q \ar[d]_{\ZZ[f_{x'}]} & k(q) \ar[d]^-{ k(q) \subseteq k(p') \subseteq K } \\ \ZZ[\M_{X,x'}] \ar[r]^-h \ar[d]_{\subseteq} & K[T] \ar[d]^-{\subseteq} \\ \ZZ[\M_{X,x}] \ar[r]^-{[m] \mapsto p(m)T^{G(m)}} & K[T,T^{-1}] } \ene commutes.  Since $p$ extends $(p')^*$ and $G$ extends $g$, the bottom square also commutes.   The commutativity of \eqref{ringdia} ensures that $\Spec h$ defines a map $\AA^1_K \to \AA(f)^{-1}(q) \cap \AA(\M_{X,x'})$ taking the complement of the origin into $\AA(f)^{-1}(q) \cap \AA(\M_{X,x}^*)$.  For $m \in \M_{X,x'}$ we have \be h([m])|_{T=0} & = & \left \{ \begin{array}{ccc} p'(m), & \quad & g(m)=0 \\ 0, & \quad & g(m)>0 \end{array} \right . \\ & = &  \left \{ \begin{array}{ccc} p'(m), & \quad & m \in \M_{X,x'}^* \\ 0, & \quad & m \in \M_{X,x'} \setminus \M_{X,x'}^* \end{array} \right . \\ & = & p'(m) \ee by locality of $g$ and $p'$.  This ensures that $\Spec h$ takes the origin to $p'$.   \end{proof}

\begin{cor} \label{cor:fiberstratification} Let $f : X \to Y$ be a quasi-compact map of fine fans, $q \in \AA(Y)$, $y := \tau(q) \in Y$, $f^{-1}_{\max}(y)$ the set of $x \in f^{-1}(y)$ with no generalizations in $f^{-1}(y)$.  Then $$ \coprod_{x \in f^{-1}_{\max}(y) } \AA(f)^{-1}(q) \cap \AA(\M_{X,x}^*) $$ is an open, dense subset of $\AA(f)^{-1}(q)$ and for any irreducible component $Z$ of $\AA(f)^{-1}(q)$ there is a unique $x=x(Z) \in f^{-1}_{\max}(y)$ such that $Z \cap \AA(\M_{X,x}^*) \neq \emptyset$.  For this $x$, $Z \cap \AA(\M_{X,x}^*)$ is both an open, dense subset of $Z$ and an irreducible component of \be \AA(f)^{-1}(q) \cap \AA(\M_{X,x}^*) & = & \AA(f_x^* : \M_{Y,y}^* \to \M_{X,x}^*)^{-1}(q). \ee  We have $$ \dim Z = \dim Z \cap \AA(\M_{X,x}^*) = \rk \Cok f_x^*.$$ \end{cor}

\begin{proof} The hypotheses ensure that $\AA(f)$ is a finite type map of schemes, so $\AA(f)^{-1}(q)$ is a finite type scheme over the field $k(q)$, hence a noetherian topological space of finite dimension.  The first assertion is immediate from Lemma~\ref{lem:fiberstratification} and the other topological assertions then follow by easy exercises in noetherian topology.  The dimension formula is obtained from Lemma~\ref{lem:torusequidimensionality}. \end{proof}

\begin{thm} \label{thm:equidimensionality} Let $f : X \to Y$ be a quasi-compact map of fine fans, $q \in \AA(Y)$, $y := \tau(q) \in Y$, $f^{-1}_{\max}(y)$ the set of $x \in f^{-1}(y)$ with no generalizations in $f^{-1}(y)$.  Then condition \eqref{equidimensionality1} below implies condition \eqref{equidimensionality2} below. \begin{enumerate} \item \label{equidimensionality1} For every point $x \in f^{-1}_{\max}(y)$, we have $n  = \rk \Cok f_x^*$.  \item \label{equidimensionality2} The fiber $\AA(f)^{-1}(q)$ is of pure dimension $n$.  \end{enumerate}  These two conditions are equivalent under either of the following hypotheses: \begin{enumerate} \setcounter{enumi}{2} \item \label{hyp1} $f_x^* : \M_{Y,y} \to \M_{X,x}^*$ is injective for all $x \in f^{-1}_{\max}(y)$. \item \label{hyp2} $q$ is in the image of the identity map $e : \Spec \ZZ \to \AA(\M_{Y,y}^*) \subseteq \AA(Y)$ for the group scheme structure on $\AA(\M_{Y,y}^*) = \GG(\M_{Y,y}^*)$. \end{enumerate}  In particular, $f$ is equidimensional of relative dimension $n$ iff \eqref{equidimensionality1} holds for every $y \in Y$.  \end{thm}

\begin{proof}  The fact that \eqref{equidimensionality1} implies \eqref{equidimensionality2} is immediate from Corollary~\ref{cor:fiberstratification}.  Either of the hypotheses \eqref{hyp1}, \eqref{hyp2} ensures that \be \AA(f)^{-1}(q) \cap \AA(\M_{X,x}^*) & = & \AA(f_x^* : \M_{Y,y}^* \to \M_{X,x}^*)^{-1}(q) \ee is non-empty for every $x \in f^{-1}(y)$.  This is because $\AA(f_x^*) = \GG(f_x^*)$ is a map of group schemes which is surjective when $f_x^*$ is injective.  To that \eqref{equidimensionality2} implies \eqref{equidimensionality1} in this case, consider a point $x \in f^{-1}_{\max}(y)$.  Then $\AA(f_x^*)^{-1}(q)$ is \emph{non-empty} and of pure dimension $m := \rk \Cok f_x^*$ by Lemma~\ref{lem:torusequidimensionality}, so any irreducible component $W$ of $\AA(f_x^*)^{-1}(q)$ is of dimension $m$.  By Corollary~\ref{cor:fiberstratification}, the closure $Z$ of such a $W$ in $\AA(f)^{-1}(q)$ is an irreducible component of $\AA(f)^{-1}(q)$, hence $\dim Z = n$ by \eqref{equidimensionality1}.  But then we find that $$m = \dim W = \dim Z =n,$$ which establishes \eqref{equidimensionality2}. \end{proof}

\begin{cor} \label{cor:equidimensionality} For a quasi-compact map of fine fans $f : X \to Y$ the following are equivalent: \begin{enumerate} \item \label{fequidimensional} $f : X \to Y$ is equidimensional of relative dimension $n$. \item \label{ftfequidimensional} $f^{\rm tf} : X^{\rm tf} \to Y^{\rm tf}$ (cf.\ \S\ref{section:integration}) is equidimensional of relative dimension $n$. \item \label{fsatequidimensional} $f^{\rm sat} : X^{\rm sat} \to Y^{\rm sat}$ (cf.\ \S\ref{section:integration}) is equidimensional of relative dimension $n$. \item \label{ftrcequidimensional} $f^{\rm trc} : X^{\rm trc} \to Y^{\rm trc}$ (cf.\ \S\ref{section:integration}) is equidimensional of relative dimension $n$. \end{enumerate} \end{cor}

\begin{proof} By construction, the maps $f$, $f^{\rm tf}$, $f^{\rm sat}$, and $f^{\rm trc}$ are all the same on the level of topological spaces (cf.\ \S\ref{section:integration}), so by Theorem~\ref{thm:equidimensionality} it is enough to show that \bne{rankequality} \rk \Cok f_x^* & = & \rk \Cok (f^\bullet)_x^* \ene for each $\bullet \in \{ {\rm tf, \; sat, \; trc} \}$ and each $x$ in the common topological space underlying $X$ and $X^\bullet$.  For each such $\bullet$, we have a commutative diagram of fine fans \bne{cefinefandiagram} & \xym{ X^\bullet \ar[r]^{f^\bullet} \ar[d] & Y^\bullet \ar[d] \\ X \ar[r]^-f & Y } \ene (where the vertical arrows are homeomorphisms), and hence for each such $x$ we have a commutative diagram of local maps of fine monoids \bne{cebulletdiagram} & \xym{  \M_{Y^\bullet,f(x)} = \M_{Y,f(x)}^\bullet \ar[r]^-{f^\bullet_x} & \M_{X^\bullet,x} = \M_{X,x}^\bullet \\ \M_{Y,f(x)} \ar[u] \ar[r]^-{f_x} & \M_{X,x} \ar[u] } \ene and a corresponding diagram of finitely generated abelian groups \bne{cebulletdiagramunits} & \xym{  (\M_{Y,y}^\bullet)^* = \M_{Y^\bullet,f(x)}^* \ar[r]^-{(f^\bullet_x)^*} & \M_{X^\bullet,x}^* = (\M_{X,x}^\bullet)^* \\ \M_{Y,f(x)}^* \ar[u] \ar[r]^-{f_x} & \M_{X,x}^* \ar[u]. } \ene  From the explicit construction of $P \to P^{\rm tf}$ for an integral monoid $P$ in Corollary~\ref{cor:tf}, we see that $P^* \to (P^{\rm tf})^*$ is surjective with kernel $P^* \cap P^{\rm gp}_{\rm tor}$, which is finite when $P$ is fine.  So when $\bullet = {\rm tf}$, each vertical arrow in \eqref{cebulletdiagramunits} is surjective with finite kernel, hence \eqref{rankequality} holds.  When $\bullet = {\rm sat}$, then by Corollary~\ref{cor:sat} the vertical arrows in \eqref{cebulletdiagram} are injective and finite (and local).  In general, if $Q \into P$ is an injective, local map of monoids and $S \subseteq P$ generates $P$ as a $Q$-module, then one checks easily that $S \cap P^*$ generates $P^*$ as a $Q^*$-module (i.e.\ $S \cap P^*$ surjects onto $P^*/Q^*$).  It follows that when $\bullet = {\rm sat}$, the vertical arrows in \eqref{cebulletdiagramunits} are injective with finite cokernel, hence \eqref{rankequality} holds.  When $\bullet = {\rm trc}$, we obtain \eqref{rankequality} from the previously proved results by using the fact that $f^{\rm trc} = (f^{\rm sat})^{\rm tf} = (f^{\rm tf})^{\rm sat}$. \end{proof}

Corollary~\ref{cor:equidimensionality} reduces questions of equidimensionality for quasi-compact maps of fine fans to the case of quasi-compact maps of \emph{toric} fans.  Furthermore, we can reduce these questions to the analogous questions for maps of toric \emph{monoids} by looking at stalks as in Remark~\ref{rem:equidimensional}.  We can further reduce to an \emph{injective} map of toric monoids by factoring an arbitrary map $h : Q \to P$ of such monoids as a surjection $k : Q \to R$ followed by an injection $g : R \to P$.  Since $\AA(k)$ is a closed embedding, any fiber of $\AA(h)$ is either empty, or a fiber of $\AA(g)$, so $h$ will be e.r.d.\ $n$ iff $g$ is e.r.d.\ $n$.  It is not generally true in this situation that $R$ is toric---it is torsion-free since it sits inside the torsion-free monoid $P$, but it may not be saturated.  But we know from Corollary~\ref{cor:equidimensionality} that $g$ is e.r.d.\ $n$ iff $g^{\rm sat} : R^{\rm sat} \to P$ is e.r.d.\ $n$ and $g^{\rm sat}$ is an injection of toric monoids, so the reduction is complete.

An injective map of toric monoids $h : Q \to P$ is the same thing as a map of affine classical fans (the corresponding map of maximal cones will be $\sigma(P)^\lor \to \sigma(Q)^\lor$) for which the corresponding map of lattices $N := (P^{\rm gp})^{\lor} \to (Q^{\rm gp})^\lor =: N'$ has finite cokernel (namely $\Ext^1(P^{\rm gp}/Q^{\rm gp},\ZZ)$).  The equidimensional maps of classical fans for which the corresonding map of lattices $N \to N'$ has finite cokernel have a nice characterization due to Abramovich and Karu \cite[4.1]{AK} which we shall give as Theorem~\ref{thm:toricequidimensionality} below, after a few preliminaries.  We will see that equidimensionality is closely related to the following fan-theoretic notion, due originally to K.~Kato and studied further by C.~Nakayama and A.~Ogus in \cite{NO}:

\begin{defn} \label{defn:exact} A map of fine fans $f : X \to Y$ is called \emph{exact} iff the diagram of monoids $$ \xym{ \M_{Y,f(x)} \ar[r]^-{f_x} \ar[d] & \M_{X,x} \ar[d] \\ \M_{Y,f(x)}^{\rm gp} \ar[r]^-{f_x^{\rm gp}} & \M_{X,x}^{\rm gp} }$$ is cartesian for each $x \in X$.  Similarly, a map of fine monoids $h : Q \to P$ is called \emph{exact} iff $\Spec h$ is exact in the previous sense\footnote{This notion is actually called ``locally exact" in \cite[Definition~2.1]{NO}.}---equivalently $$ \xym{ h^{-1}(F)^{-1}Q \ar[r]^-{``h"} \ar[d] & F^{-1}P \ar[d] \\ Q^{\rm gp} \ar[r]^-{h^{\rm gp}} & P^{\rm gp} }$$ is cartesian for each face $F \leq P$. \end{defn}

\begin{lem} \label{lem:exactness} Suppose $h : Q \to P$ is a map of integral monoids such that \bne{exactnessdiagram0} & \xym{ Q \ar[r]^-{h} \ar[d] & P \ar[d] \\ Q^{\rm gp} \ar[r]^{h^{\rm gp}} & P^{\rm gp} } \ene is cartesian.  Recall the functors $P \mapsto P^{\rm tf}$, $P \mapsto P^{\rm sat}$, and $P \mapsto P^{\rm trc}$ from Corollaries~\ref{cor:tf}, \ref{cor:sat}, and \ref{cor:trc}.  Then: \begin{enumerate} \item \label{exactness1} $h$ is local. \item \label{exactness2} If $P$ is saturated then $Q$ is saturated.  \item \label{exactness3} If $Q$ is saturated, then the diagram \be & \xym{ Q^{\rm tf} \ar[r]^-{h^{\rm tf}} \ar[d] & P^{\rm tf} \\ (Q^{\rm tf})^{\rm gp} = Q^{\rm gp} / Q^{\rm gp}_{\rm tor} \ar[r]^-{(h^{\rm tf})^{\rm gp} } & P^{\rm gp} / P^{\rm gp}_{\rm tor} = (P^{\rm tf})^{\rm gp} } \ee is cartesian.   \item \label{exactness4} The diagrams \be \xym{ Q^{\rm sat} \ar[r]^-{h} \ar[d] & P^{\rm sat} \ar[d] \\ Q^{\rm gp} \ar[r]^{h^{\rm gp}} & P^{\rm gp} } & {\rm \; and \;} & \xym{ Q^{\rm trc} \ar[r]^-{h^{\rm trc}} \ar[d] & P^{\rm trc} \ar[d] \\ (Q^{\rm trc})^{\rm gp} = Q^{\rm gp} / Q^{\rm gp}_{\rm tor} \ar[r]^-{(h^{\rm trc})^{\rm gp} } & P^{\rm gp} / P^{\rm gp}_{\rm tor} = (P^{\rm trc})^{\rm gp} } \ee are cartesian.  \end{enumerate}  It follows that if $f :  X \to Y$ is an exact map of integral fans, then $f^{\rm sat}$ and $f^{\rm trc}$ are exact maps of integral fans and, if $X$ is also saturated, then $f^{\rm tf}$ is also exact.  \end{lem}

\begin{proof}  Statements \eqref{exactness1}-\eqref{exactness3} are straightforward exercises with the definitions, as is the exactness of the left diagram in \eqref{exactness4}.  Exactness of the right diagram in \eqref{exactness4} then follows from $P^{\rm trc} = (P^{\rm sat})^{\rm tf}$.  The consequences for maps of integral fans follow from the constructions of $f^{\rm tf}$, $f^{\rm sat}$, and $f^{\rm trc}$ in \S\ref{section:integration} by applying the corresponding monoid result to the stalks of $f$. \end{proof}

\begin{rem} \label{rem:exactness} Lemma~\ref{lem:exactness}\eqref{exactness3} does not hold without assuming $Q$ is saturated, even if $Q$ and $P$ are sharp and fine.  For example, consider the submonoid $P$ of $\ZZ \oplus \ZZ/3\ZZ$ generated by $(1,0)$ and $(2,2)$.  Note $P^{\rm gp} = \ZZ \oplus \ZZ/3\ZZ$.  \emph{Define} $Q$ and $h$ so that $$ \xym{ Q \ar[r]^-{h} \ar[d] & P \ar[d] \\ \ZZ \ar[r]^-{ (1,1) } & \ZZ \oplus \ZZ/3 \ZZ }$$ is cartesian.  Note that $Q$ is the submonoid of $\ZZ$ generated by $2$ and $3$, so $Q^{\rm gp} = \ZZ$ and the bottom arrow in the diagram above is the groupification of the top arrow.  Then $Q = Q^{\rm tf}$, $P^{\rm tf} = \NN$ and $h^{\rm tf}$ is just the obvious inclusion of $Q$ into $\NN$, which is not an isomorphism despite the fact that its groupification is an isomorphism, so the diagram in Lemma~\ref{lem:exactness}\eqref{exactness3} is not cartesian. \end{rem}

\begin{lem} \label{lem:toricexactness} Let $f : N \to N'$ be a map of lattices, $f^\lor : M' \to M$ the map of dual lattices, $\sigma \subseteq N_{\RR}$, $\sigma' \subseteq N'_{\RR}$ rational cones with $f_{\RR}(\sigma) \subseteq \sigma'$.  Then the diagram \bne{exactnessdiagram} & \xym{ M' \cap (\sigma')^\lor \ar[d] \ar[r]^-{f^\lor} & M \cap \sigma^\lor \ar[d] \\ M' \ar[r]^-{f^\lor} & M } \ene is cartesian iff $f_{\RR}(\sigma)=\sigma'$. \end{lem}

\begin{proof} The ``if" is immediate from the definitions.  For the ``only if," suppose the containment $f_{\RR}(\sigma) \subseteq \sigma'$ is not an equality.  Since $\sigma$ is rational, so is $f_{\RR}(\sigma)$, so by Theorem~\ref{thm:cones}\eqref{FarkasTheorem}, we have a proper containment $(\sigma')^\lor \subseteq f_{\RR}(\sigma)^\lor$ of rational cones, so there is some $m' \in M' \cap f_{\RR}(\sigma)^\lor$ (hence $f^\lor m' = m' f \in M \cap \sigma^\lor$) not in $M' \cap (\sigma')^\lor$, violating \eqref{exactnessdiagram} cartesian. \end{proof}

\begin{lem} \label{lem:conesequidimensionality} Let $f : N \to N'$ be a map of lattices, $\sigma \subseteq N_{\RR}$ a sharp cone, $x \in f_{\RR}(\sigma)^{\rm int}$.  Then there is a face $\tau \leq \sigma$ such that $x \in f_{\RR}(\tau)^\circ$ and $f_{\RR}|\tau : \tau \to f_{\RR}(\tau)$ is bijective. \end{lem}

\begin{proof}  This is clear if $f$ is injective since we can simply take $\tau=\sigma$.  By factoring $f$ as a surjection followed by an injection, we reduce to the case where $f$ is surjective.  In this case we proceed by induction on $\dim \sigma$.  Pick $\tilde{x} \in \sigma$ with $f_{\RR}(\tilde{x})=x$.  Let $\sigma' \leq \sigma$ be the unique face of $\sigma$ with $\tilde{x} \in (\sigma')^\circ$.  Since $f_{\RR} : \sigma' \to f_{\RR}(\sigma')$ is surjective we have $x \in f_{\RR}(\sigma')^\circ$.  If $\sigma' < \sigma$, then we conclude by induction, so now assume that $\sigma' = \sigma$---i.e.\ $\tilde{x} \in \sigma^\circ$.  If $f_{\RR}|\sigma : \sigma \to f_{\RR}(\sigma)$ is injective, then it is bijective and we can take $\tau = \sigma$, so we can now assume that the vector subspace $V := \Ker f_{\RR} \cap \Span \sigma$ of $N_{\RR}$ has positive dimension.  

There must be some $v \in V$ for which $\tilde{x} + v \notin \sigma$, otherwise we would see that $V \subseteq \sigma$ (use the fact that $\sigma$ is closed in $N_{\RR}$ and any $v \in V$ is the limit of $\lambda^{-1}(\tilde{x}+\lambda v)$ as $\lambda \to + \infty$), contradicting the assumption that $\sigma$ is sharp.  Since $V$ is connected and \be V' & := & \{ v \in V : \tilde{x}+v \in \sigma \} \ee is a non-empty (since $0 \in V'$) proper closed subset of $V$, there must be some point $v'$ on the boundary of $V'$ in $V$.  Since $V \subseteq \Span \sigma$ and $\sigma^\circ$ is the topological interior of $\sigma$ in $\Span \sigma$, we see that $\tilde{x}+v'$ must lie in a proper face $\sigma'' < \sigma$.  Since $V \subseteq \Ker f_{\RR}$, we have $x = f_{\RR}(\tilde{x}+v')$, and we see as in the previous paragraph that, in fact, $x \in f_{\RR}(\sigma'')^\circ$, so we can conclude by induction.    \end{proof}

\begin{thm} \label{thm:toricequidimensionality} {\bf (Abramovich-Karu)} Let $f : \Sigma \to \Sigma'$ be a map of classical fans such that the cokernel of the corresponding map of lattices $f : N \to N'$ is finite.  Set $n := \rk N$, $n' := \rk N'$.  For $\sigma \in \Sigma$, let $\sigma' \in \Sigma'$ denote the smallest cone of $\Sigma'$ containing $f_{\RR}(\sigma)$.  Then the following are equivalent: \begin{enumerate} \item \label{AKequidimensionality1} The map $\AA(f) : \AA(\Sigma) \to \AA(\Sigma')$ is equidimensional of relative dimension $n-n'$. \item \label{AKequidimensionality2} $\dim \sigma = \dim \sigma'$ for any cone $\sigma \in \Sigma$ with the property that $\tau' < \sigma'$ for all $\tau < \sigma$. \item \label{AKequidimensionality3} For every $\sigma \in F$ we have $f_{\RR}(\sigma) = \sigma'$. \item \label{AKequidimensionality4} The corresponding map of abstract fans $f : (\Sigma,\M_{\Sigma}) \to (\Sigma',\M_{\Sigma'})$ (\S\ref{section:classicalfans}) is exact in the sense of Definition~\ref{defn:exact}. \end{enumerate} \end{thm}

\begin{proof}  The assumption on $f$ ensures that $f_{\RR}$ is surjective and the map of dual lattices $f^\lor : M' \to M$ is injective, hence the map \bne{fsigma} f^\lor : M' \cap (\sigma')^\perp & \to & M \cap \sigma^\perp \ene is injective for any $\sigma \in \Sigma$.  When $\Sigma$ and $\Sigma'$ are viewed as abstract fans and $f$ is viewed as a map of abstract fans as in \S\ref{section:classicalfans}, the map \eqref{fsigma} is identified with the map \bne{fsigmaabstract} f_\sigma^* : \M_{\Sigma',f(\sigma)}^* \to \M_{\Sigma,\sigma}^*. \ene  We have \be \rk \sigma^\perp \cap M & = & n - \dim \sigma \ee and similarly for cones of $\Sigma'$, so by injectivity of \eqref{fsigmaabstract} we have \bne{equiformula} \rk \Cok f_\sigma^* & = & n-n'+\dim \sigma'-\dim \sigma. \ene  Furthermore, the points of the abstract fan $(\Sigma,\M_{\Sigma})$ are precisely the cones $\sigma \in \Sigma$ and the map of abstract fans $f$, on the level of sets, takes $\sigma$ to the cone $\sigma'$ defined in the statement of the theorem (cf.\ \S\ref{section:classicalfans}).  The diagram $$ \xym{ \M_{\Sigma',\sigma'} \ar[r] \ar[d] & \M_{\Sigma,\sigma} \ar[d] \\ \M_{\Sigma',\sigma'}^{\rm gp} \ar[r] & \M_{\Sigma,\sigma}^{\rm gp} } $$ considered in the definition of exactness (Definition~\ref{defn:exact}) is nothing but the diagram \eqref{exactnessdiagram} considered in Lemma~\ref{lem:toricexactness}, so \eqref{AKequidimensionality3} and \eqref{AKequidimensionality4} are equivalent by that lemma.

Fix $\sigma' \in \Sigma'$.  Then the cones $\sigma \in f^{-1}(\sigma')$ with no generalizations in $f^{-1}(\sigma')$ are precisely the \emph{minimal} cones of $f^{-1}(\sigma')$.  The equivalence of \eqref{AKequidimensionality1} and \eqref{AKequidimensionality2} follows from Theorem~\ref{thm:equidimensionality} and \eqref{equiformula}.  

To show that \eqref{AKequidimensionality2} implies \eqref{AKequidimensionality3} we need to show that, assuming \eqref{AKequidimensionality2}, the containment $f_{\RR}(\sigma) \subseteq \sigma'$ is actually an equality for every $\sigma \in \Sigma$.  Obviously it is enough to show this when $\sigma$ is minimal in $f^{-1}(\sigma')$, in which case we have \bne{dimeq} \dim \sigma & = & \dim \sigma' \ene by \eqref{AKequidimensionality2}.  Suppose, toward a contradiction, that $f_{\RR}(\sigma) \subsetneq \sigma'$ for some such $\sigma$.  By Theorem~\ref{thm:cones}\eqref{cones:lemma1} and the definition of $\sigma'$, there must be some point $x \in f_{\RR}(\sigma)^\circ \cap (\sigma')^\circ$.  By Lemma~\ref{lem:conesequidimensionality}, there is some $\tau \leq \sigma$ with $x \in f_{\RR}(\tau)$ and $f_{\RR}|\tau : \tau \to f_{\RR}(\tau)$ bijective.  Since $x \in f_{\RR}(\tau) \cap (\sigma')^\circ$ we must have $\tau' = \sigma'$ for such a $\tau$, so by minimality of $\sigma$ we must have $\tau = \sigma$.  We conclude that $f_{\RR}|\sigma : \sigma \to f_{\RR}(\sigma)$ bijective, hence $\dim f_{\RR}(\sigma) = \dim \sigma$, hence \eqref{dimeq} implies \bne{vecteq} \Span f_{\RR}(\sigma) & = & \Span \sigma' . \ene

Since $f_{\RR}(\sigma) \subsetneq \sigma'$, there must be a point $y \in (\sigma')^\circ \setminus f_{\RR}(\sigma)$ on topological grounds.  (This is because $f_{\RR}(\sigma)$ is a finitely generated cone, so it is closed in $N'_{\RR}$ by Theorem~\ref{thm:cones}\eqref{conetopology}, while $(\sigma')^\circ$ is dense in $\sigma'$ by Theorem~\ref{thm:cones}\eqref{integralpointininterior}.)  We have $y \in \Span f_{\RR}(\sigma)$ by \eqref{vecteq}, so by Theorem~\ref{thm:cones}\eqref{cones:lemma2}, there must be some $s \in (0,1)$ such that $v := (1-s)x+sy$ is in a proper face of $f_{\RR}(\sigma)$.  Since $f_{\RR}$ is surjective, the proper face in question must be $f_{\RR}(\tau)$, for some $\tau < \sigma$.  On the other hand, the fact that $x$ and $y$ are in $(\sigma')^{\circ}$ ensures that $v \in (\sigma')^\circ$, so we must have $\tau' = \sigma'$, contradicting the minimality of $\sigma$ in $f^{-1}(\sigma')$.   

To see that \eqref{AKequidimensionality3} implies \eqref{AKequidimensionality2}, we first show (using Lemma~\ref{lem:conesequidimensionality} exactly as we did above) that $f_{\RR}| \sigma : \sigma \to f_{\RR}(\sigma)$ must be bijective for any $\sigma \in \Sigma$ as in \eqref{AKequidimensionality2}, so we certainly have $\dim \sigma = \dim f_{\RR}(\sigma)$.  But when \eqref{AKequidimensionality3} holds, we have $f_{\RR}(\sigma) = \sigma'$, so we find that $\dim \sigma = \dim \sigma'$, as desired.   \end{proof}

\begin{rem} \label{rem:torusequidimensionality}  The assumption that $\Cok f$ is finite in Theorem~\ref{thm:toricequidimensionality} is equivalent to surjectivity of $\AA(f)$ and to surjectivity of the corresponding map $\Spec \ZZ[f^\lor]$ of tori.  One can of course use Theorem~\ref{thm:equidimensionality} to give a criterion for equidimensionality of $\AA(f)$ without this assumption---the criterion can be phrased in terms of the dimensions of $\sigma$, $\sigma'$, and $f_{\RR}(\sigma)$ for the cones $\sigma \in \Sigma$.  (The reader can work out the precise statement.)   In general, however, equidimensionality is certainly not equivalent to condition \eqref{AKequidimensionality3} in Theorem~\ref{thm:toricequidimensionality} because, for example, every closed embedding is equidimensional of relative dimension zero, but the diagonal embedding $\Delta: \AA^1 \into \AA^2$ is a map of toric varieties for which the corresponding map of fans does not satisfy condition \eqref{AKequidimensionality3}. \end{rem}

\subsection{Finite and quasi-finite maps} \label{section:finitemaps}  Recall the following:

\begin{defn} \label{defn:schemefinite} \cite[I.6.4.5]{EGA} A scheme $X$ of finite type over a field $k$ is called \emph{finite} over $k$ iff it satisfies the following equivalent conditions: \begin{enumerate} \item The topological space underlying $X$ is zero-dimensional. \item The topological space underlying $X$ is finite. \item The topological space underlying $X$ is finite and discrete. \item $X$ is isomorphic to $\Spec A$ for a $k$-algebra $A$ which is finite dimensional as a $k$ vector space. \end{enumerate} \end{defn}

\begin{defn} \label{defn:quasifinite} \cite[II.6.2.3]{EGA} A map of schemes $f : X \to Y$ is called \emph{quasi-finite} iff it is of finite type and for every $y \in Y$, the scheme-theoretic fiber $X_y = f^{-1}(y) := X \times_Y \Spec k(y)$ is finite over $k(y)$.  In other words, $f : X \to Y$ is quasi-finite iff it is equidimensional of relative dimension zero in the sense of Definition~\ref{defn:equidimensional}. \end{defn}

\begin{defn} \label{defn:finite} Recall that a map of schemes $f : X \to Y$ is called \emph{finite} iff it is affine and for any affine open subset $U \subseteq Y$, $\O_X(f^{-1}(U))$ is a finitely generated $\O_Y(U)$-module.  Similarly, a map of fans $f : X \to Y$ is called \emph{finite} iff $f$ is affine and for any affine open subset $U \subseteq Y$, $\M_X(f^{-1}(U))$ is a finitely generated $\M_Y(U)$ module. \end{defn}

One can see, as in the case of schemes, that finite maps of fans are stable under base change, and the finiteness of a map of fans $f : X \to Y$ can be checked locally on $Y$.  Clearly a finite map of schemes is quasi-finite.  If $h : Q \to P$ is a map of monoids, then it is clear that $\Spec h$ is a finite map of fans in the sense of Definition~\ref{defn:finite} iff $h$ is a finite map of monoids in the sense of Definition~\ref{defn:flat}.

\begin{prop} \label{prop:finite}  Let $f : X \to Y$ be a map of fans.  If $f$ is finite, then $\AA(f)$ is finite, and the converse holds if we assume $f$ is affine. \end{prop}

\begin{proof} If $f$ is finite then it is affine by definition, so $\AA(f)$ is also affine by Proposition~\ref{prop:affinemaps}.  Furthermore, by definition of a finite map of fans, $Y$ can be covered by open affine subfans $U \subseteq Y$ such that $f^{-1}(U)$ is affine and $\M_X(f^{-1}(U))$ is a finitely generated $\M_Y(U)$ module.  This implies that $\ZZ[ \M_X(f^{-1}(U)) ]$ is finitely generated as a $\ZZ[\M_Y(U)]$ module so the map of schemes \be ( \AA(f^{-1}(U)) \to \AA(U) ) & = & \Spec( \ZZ[\M_Y(U)] \to \ZZ[ \M_X(f^{-1}(U)) ] ) \ee is finite.  Since these $U$ cover $Y$, the open subschemes $\AA(U) \subseteq \AA(Y)$ cover $\AA(Y)$.  We also have $\AA(f^{-1}(U)) = \AA(f)^{-1}(\AA(U))$, so we conclude that $\AA(f)$ is finite because this can be checked locally on $Y$.  The converse for affine $f$ boils down to showing that a module $M$ over a monoid $P$ is finitely generated whenever the $\ZZ[P]$ module $\ZZ[M]$ is finitely generated.  This is an easy exercise, or see \cite[5.2.1]{G2}.    \end{proof}

\begin{lem} \label{lem:quasifinite} For a map $f : A \to B$ of finitely generated abelian groups, the following are equivalent: \begin{enumerate} \item \label{qfA} $\Cok f$ is finite. \item \label{qfAA} $f$ is finite as a map of monoids (i.e.\ $B$ is finitely generated as an $A$-module).  \item \label{qfB} $\Spec \ZZ[f]$ is finite. \item \label{qfC} $\Spec \ZZ[f]$ is quasi-finite. \end{enumerate} \end{lem}

\begin{proof}  The equivalence of \eqref{qfA} and \eqref{qfAA} is discussed in Example~\ref{example:flatmodule}.  The equivalence of \eqref{qfAA} and \eqref{qfB} was noted in the proof of Proposition~\ref{prop:finite}.  Clearly \eqref{qfB} implies \eqref{qfC}.  The equivalence of \eqref{qfC} and \eqref{qfA} is a special case of Lemma~\ref{lem:torusequidimensionality}. \end{proof}

We now show that the ``$n=0$" case of Theorem~\ref{thm:equidimensionality} holds for arbitrary locally finite type (not necessarily fine) fans.

\begin{thm} \label{thm:quasifinite} For a quasi-compact map $f : X \to Y$ of locally finite type fans, the following are equivalent: \begin{enumerate} \item \label{qf1} $f_x^*: \M_{Y,f(x)}^* \to \M_{X,x}^*$ has finite cokernel for every $x \in X$. \item \label{qf2} $\AA(f)$ is a quasi-finite map of schemes. \end{enumerate} \end{thm}

\begin{proof}  Both conditions are local on $Y$ so we can reduce (as in the proof of Lemma~\ref{lem:fiberstratification}) to the case where the topological spaces underlying $X$ and $Y$ are finite.  In this case, by Proposition~\ref{prop:stratification}, we can partition $\AA(X)$ and $\AA(Y)$ into finitely many locally closed subschemes \bne{parts} \AA(X) & = & \coprod_{x \in X} \AA(\M_{X,x}^*) \\ \nonumber \AA(Y) & = & \coprod_{y \in Y} \AA(\M_{Y,y}^*) \ene such that the diagram \bne{qfdiagram} \xym{ \AA(\M_{X,x}^*) \ar[r] \ar[d]_{ \AA(f_x^*) } & \AA(X) \ar[d]^-f \\ \AA(\M_{Y,f(x)}^*) \ar[r] & \AA(Y) } \ene commutes for every $x \in X$.  For every $q \in \AA(Y)$ with $y := \tau(q) \in Y$, we therefore obtain a partition \bne{fiberpartition} \AA(f)^{-1}(q) & = & \coprod_{x \in f^{-1}(y)} \AA(f_x^*)^{-1}(q) \ene of the fiber $\AA(f)^{-1}(q)$ into finitely many locally closed subschemes.  (Note that these partitions may not be \emph{stratifications} when $X$ and $Y$ are not fine.)  By Lemma~\ref{lem:quasifinite}, \eqref{qf1} implies that each ``part" in the partition \eqref{fiberpartition} is of dimension zero, so $\AA(f)^{-1}(q)$ must also be of dimension zero, so \eqref{qf2} holds.  Conversely, suppose \eqref{qf2} holds, but assume, towards a contradiction, that \eqref{qf1} fails for some $x$.  Then by Lemma~\ref{lem:quasifinite}, $\AA(f_x^*)$ is not quasi-finite, so at least one of its (non-empty) fibers, say $\AA(f_x^*)^{-1}(q)$, will have positive dimension.  (In fact \emph{all} of its fibers will have pure dimension $\rk \Cok f_x^* > 0$ by Lemma~\ref{lem:torusequidimensionality}, but we don't need to know this.)  Since we have the partition \eqref{fiberpartition}, the fact that $\dim \AA(f_x^*)^{-1}(q) > 0$ certainly implies that $\dim \AA(f)^{-1}(q) > 0$, contradicting \eqref{qf2}.  \end{proof}

\begin{defn} A map $f : X \to Y$ of locally finite type fans is called \emph{quasi-finite} iff it is quasi-compact and satisfies the equivalent conditions of Theorem~\ref{thm:quasifinite}. \end{defn}

\begin{rem} \label{rem:finite} For any monoid $P$, the obvious map of fans $f : \Spec P \coprod \Spec P \to \Spec P$ is not affine (Remark~\ref{rem:affinemaps}), so it is not a finite map of fans, even though $\AA(f)$ is clearly a finite map of schemes.  Furthermore, when $P$ is finitely generated, $f$ is proper in the sense of Definition~\ref{defn:proper} (even if $P$ isn't finitely generated, $\AA(f)$ is a proper map of schemes and $f$ is valuatively proper) and $f$ is quasi-finite in the sense of the above definition. \end{rem}

\begin{prop} \label{prop:quasifinite} Let $f : X \to Y$, $g : Y \to Z$, $X' \to X$ be maps of locally finite type fans, $P$ a finitely generated monoid, $F$ a face of $P$.  \begin{enumerate} \item \label{prop:quasifinite1} If $f$ and $g$ are quasi-finite, then $gf$ is quasi-finite. \item \label{prop:quasifinite2} If $gf$ is quasi-finite, then $f$ is quasi-finite. \item If $f$ is quasi-finite, then so is its base change $f' : X' \times_X Y \to X'$ along $X' \to X$. \item If $f$ is an open embedding, then $f$ is quasi-finite. \end{enumerate} \end{prop}

\begin{proof} If we take Theorem~\ref{thm:quasifinite}\eqref{qf2} as the definition of \emph{quasi-finite}, then all of these statements follow from analogous properties of quasi-finite maps of schemes \cite[II.6.2.4]{EGA}, though it is also a good exercise for the reader to prove these statements taking Theorem~\ref{thm:quasifinite}\eqref{qf1} as the definition. \end{proof}

To motivate our next result, recall that there is a corollary of Zariski's Main Theorem that says any reasonable\footnote{separated and finitely presented with quasi-compact and quasi-separated codomain} quasi-finite map of schemes $f : X \to Y$ can be factored as an open embedding $X \into X'$ followed by a finite map $X' \to Y$ \cite[IV.8.12.6]{EGA}.

\begin{lem} \label{lem:quasifinite2} Suppose $h : Q \to P$ is an injective, quasi-finite map of fine monoids.  Then $\Spec h$ is an open embedding \emph{on the level of topological spaces}.  If, furthermore, $\Spec h$ is a homeomorphism (equivalently: is surjective), then $h$ is finite.  It follows that $h$ can be factored as $Q \to F^{-1}Q \to P$, where $F = h^{-1}(P^*) \leq Q$, with $F^{-1}Q \to P$ finite and injective. \end{lem}

\begin{proof} There are various ways to prove this, but the strategy we shall adopt is to reduce to the toric case.  By Corollary~\ref{cor:trc} we have a commutative diagram \bne{tordiagram} & \xym{Q \ar[r]^-{h} \ar[d] & P \ar[d] \\ Q^{\rm trc} \ar[r]^-{h^{\rm trc}} & P^{\rm trc} } \ene where $Q^{\rm trc}$ and $P^{\rm trc}$ are toic, the vertical arrows are finite, and $\Spec$ of either vertical arrow is a homeomorphism, so it suffices to prove that $\Spec h^{\rm trc}$ is an open embedding on the level of spaces and that $h^{\rm trc}$ is finite when $\Spec h$ (equivalently $\Spec h^{\rm trc}$) is a homeomorphism.  

To do this, we first check that the ``injective and quasi-finite" assumption on $h$ is inherited by $h^{\rm trc}$.  The clearest way to see this is to proceed in two steps:  Since $h^{\rm trc} = (h^{\rm sat})^{\rm tf}$, it is enough to show that the ``injective and quasi-finite" assumption on a map $h$ of fine monoids is inherited by $h^{\rm tf}$ (see Corollary~\ref{cor:tf}) and by $h^{\rm sat}$ (see Corollary~\ref{cor:sat}).  For injectivity, note that $h$ injective implies $h^{\rm gp}$ injective (hence $h^{\rm sat}$ is also injective), which implies that $Q^{\rm gp}/Q^{\rm gp}_{\rm tor} \to P^{\rm gp}/P^{\rm gp}_{\rm tor}$ is injective, hence injectivity of $h^{\rm tf}$ is clear from its construction.  The fact that quasi-finiteness of $h$ is inherited by $h^{\rm tf}$, $h^{\rm sat}$, and $h^{\rm trc}$ is the special case ``$n=0$" of Corollary~\ref{cor:equidimensionality}.

The above results reduce us to the case where $Q$ and $P$ are toric.  In this case $h^{\rm gp} : Q^{\rm gp} \into P^{\rm gp}$ is a finite index inclusion of lattices because $h$ is injective and quasi-finite.  Consider the map of (rational) cones $\sigma(h) : \sigma(Q) \to \sigma(P)$ induced by $h$.  Since $Q$ and $P$ are toric, Theorem~\ref{thm:monoidsandcones}\eqref{monoidsandcones1} ensures that we recover $h$ as the map on ``integral points" \bne{h} Q = Q^{\rm gp} \cap \sigma(Q) & \to & P^{\rm gp} \cap \sigma(P) = P \ene induced by $h^{\rm gp}$ and $\sigma(h)$.  In other words, the map of (affine) ``toric varieties" $\Spec \ZZ[P] \to \Spec \ZZ[Q]$ coincides with the one induced by the map of (classical) fans $[ \sigma(P)^\lor ] \to [ \sigma(Q)^\lor ]$ induced by the finite index inclusion of lattices $(h^{\rm gp})^\lor$.  That map of ``toric varieties" is quasi-finite by definition (since $h$ is assumed quasi-finite).  By Theorem~\ref{thm:toricequidimensionality} we conclude that $(h^{\rm gp})^\lor \otimes \RR$ takes $\sigma(P)^\lor$ bijectively onto some face $\tau \leq \sigma(Q)^{\lor}$.  Dualizing again and noting that $\tau^\lor$ is the localization of $\sigma(Q)=\sigma(Q)^{\lor \lor}$ (Theorem~\ref{thm:cones}\eqref{FarkasTheorem}) at the face $\tau^\perp \cap \sigma(Q) \leq \sigma(Q)$ (Theorem~\ref{thm:cones}\eqref{facesofcones}), we see that $\Spec \sigma(h)$ is the composition of the isomorphism $\Spec \sigma(P) \to \Spec \tau^\lor$ and the map $\Spec \tau^\lor \to \Spec \sigma(Q)$.  Since $\Spec \sigma(Q)$ is finite (Theorem~\ref{thm:cones}\eqref{finitelymanyfaces}), this latter map is an open embedding by Proposition~\ref{prop:Specfinite}\eqref{facelocalizationembedding}, so we conclude that $\Spec \sigma(h)$ is an open embedding on the level of topological spaces.  We have a commutative diagram $$ \xym{ Q \ar[r]^-h \ar[d] & P \ar[d] \\ \sigma(Q) \ar[r]^-{\sigma(h)} & \sigma(P) }$$ where $\Spec$ of each vertical arrow is a homeomorphism by Theorem~\ref{thm:monoidsandcones}\eqref{monoidsandcones1}, so we conclude that $\Spec h$ is an open embedding on the level of topological spaces.  We also see that when $\Spec h$ is a homeomorphism, $\Spec \sigma(h)$ is also a homeomorphism, so we must have $\tau = \sigma(P)$ and hence $\sigma(h) : \sigma(Q) \to \sigma(P)$ must be bijective.  If we suppress notation for the isomorphisms $h^{\rm gp}_{\RR}$ and $\sigma(h)$, then the description \eqref{h} of $h$ just says that $P$ (resp.\ $Q$) is the set of integral points in the cone $\sigma(P) = \sigma(Q)$ with respect to $P^{\rm gp}$ (resp.\ with respect to the finite index subgroup $Q^{\rm gp}$ of $P^{\rm gp}$), so it is clear that $h$ is dense in the sense of Definition~\ref{defn:dense}, hence it is finite by Theorem~\ref{thm:dense}\eqref{dense2}.   

For the ``It follows," first note that the image of $\Spec h$ is clearly contained in the smallest open subset \be U_F & = & \{ F' \in \Spec Q : F \leq F' \} \\ & = & \Spec F^{-1}Q \ee of $\Spec Q$ containing $F$ (cf.\ Proposition~\ref{prop:Specfinite}).  Since $\Spec h$ is an open embedding, it must therefore be a homeomorphism onto $U_F$.  Since $\Spec h$ factors topologically through the open subfan $\Spec F^{-1}Q$ of $\Spec Q$, it also factors ``fan theoretically" through it (Remark~\ref{rem:embedding}), so $h$ factors as $Q \to F^{-1}Q \to P$.  The second map is a quasi-finite (Proposition~\ref{prop:quasifinite}\eqref{prop:quasifinite2}), injective map of fine monoids and $\Spec$ of it is a homeomorphism, so it is finite.  \end{proof}

\begin{thm} \label{thm:finite} For an injective map $h : Q \to P$ of fine monoids, the following are equivalent: \begin{enumerate} \item \label{hdense} $h$ is dense (Definition~\ref{defn:dense}). \item \label{hfinite} $h$ is finite ($P$ is finitely generated as a $Q$-module). \item \label{hfiniteand} $h$ is finite and $\Spec h$ is a homeomorphism.  \item \label{hquasifiniteand} $h$ is quasi-finite and $\Spec h$ is surjective. \item \label{hquasifinitelocal} $h$ is quasi-finite and local. \end{enumerate} \end{thm}

\begin{proof} The equivalence of \eqref{hdense} and \eqref{hfinite} follows from Theorem~\ref{thm:dense}.  Obviously \eqref{hfiniteand} implies \eqref{hquasifiniteand} and we saw in Lemma~\ref{lem:quasifinite2} that \eqref{hquasifiniteand} and \eqref{hquasifinitelocal} are equivalent, and that these equivalent conditions imply \eqref{hfinite}.

It remains to show that \eqref{hfinite} implies \eqref{hfiniteand}.  For this, we need to show that when $h$ is finite, then in the factorization $Q \to F^{-1}Q \to P$ of $h$ in Lemma~\ref{lem:quasifinite2}, we must have $F=Q^*$.  The key point is to note that $Q \into F^{-1}Q$ must be dense since $h$ is dense---but this implies that for every $f \in F$ we have $-nf \in Q$ for some $n \in \ZZ_{>0}$, hence $nf \in Q^*$, hence $f \in Q^*$.  We conclude that $F=Q^*$ as desired.\end{proof}

\subsection{Flat maps} \label{section:flatmaps}  

\begin{defn} \label{defn:flatmapoffans} An $\LMS$-morphism (in particular, a map of fans) $f : X \to Y$ is called \emph{flat} (resp.\ a \emph{flat cover}) iff the map of monoids $f_x : \M_{Y,f(x)} \to \M_{X,x}$ is flat (Definition~\ref{defn:flat}) for every $x \in X$ (resp.\ and $f$ is surjective).  A map of monoids $h : Q \to P$ is called \emph{faithfully flat} iff $\Spec h$ is a flat cover. \end{defn}

The next lemma implies that a faithfully flat map of monoids is indeed flat, as the terminology suggests.

\begin{lem} \label{lem:flat} A map of monoids $h : Q \to P$ is flat iff $\Spec h$ is a flat map of fans.  Furthermore, when these equivalent conditions hold, the map $h^{\rm gp} : Q^{\rm gp} \to P^{\rm gp}$ is injective.  It follows that any flat map of \emph{integral} monoids is injective. \end{lem}

\begin{proof} ($\Rightarrow$) Assuming $h$ is flat, we need to show that for each face $F \in \Spec P$ with image $G \in \Spec Q$, the map $\M_{Q,G} \to \M_{P,F}$ is flat.  Here $G = h^{-1}(F)$ and the stalk map in question is the map $G^{-1}Q \to F^{-1} P$ induced by $h$.  Let $F' := h(G)$.  Then $F'$ is a submonoid of $P$ and we see from the universal property of localization that $$ \xym{ Q \ar[r]^-h \ar[d] & P \ar[d] \\ G^{-1}Q \ar[r] & (F')^{-1} P } $$ is a pushout diagram, hence $G^{-1} Q \to (F')^{-1} P$ is flat since $h$ is flat and flat maps are stable under pushout.  The map of interest to us is the composition of the latter flat map and a localization $(F')^{-1} P \to F^{-1} P$ (note $F' \subseteq F$), so it is flat.

($\Leftarrow$)  By assumption, the stalk $\M_{Q,G} \to \M_{P,P^*}$ of $\Spec h$ at the closed point $P^*$ of $\Spec P$ is flat.  Here $G = h^{-1}(P^*)$ and this stalk map is the map $G^{-1} Q \to P$ induced by $h$.  But $h$ factors as the localization $Q \to G^{-1}Q$, which is flat, followed by that flat stalk map, so $h$ is flat. 

For the furthermore, note that when $h$ is flat, so is $h^{\rm gp}$ because one sees much as in the above discussion that $h^{\rm gp}$ is a localization of a pushout of $h$.  It is easy to see that a flat map of groups must be injective (cf.\ \cite[\S5.5]{G2} or Corollary~\ref{cor:faithfullyflatimpliesinjective} below).  \end{proof}

\begin{prop} \label{prop:flatmaps} Flat $\LMS$ morphisms and flat covers are stable under composition and base change.  The question of flatness of an $\LMS$ morphism $f$ is local on $f$. \end{prop}

\begin{proof} Since flat maps of monoids are stable under composition and base change, the first statement follows from Proposition~\ref{prop:stableunderbasechange}.  The second assertion holds since flatness is defined in terms of stalks. \end{proof}

\begin{lem} \label{lem:formula}  Suppose $f : X \to Y$ is a map of fans such that $(f^\dagger_x)^{\rm gp} : \M_{Y,f(x)}^{\rm gp} \to \M_{X,x}^{\rm gp}$ is injective for all $x \in X$.  (This condition holds when $f$ is flat by Lemma~\ref{lem:flat}.)  Then for any subset $U \subseteq X$, we have \be \tau_Y^{-1}(f(U)) & = & \AA(f)(\tau_X^{-1}(U)).\ee  It follows that if $\AA(f)$ is surjective, then $f$ is surjective.  \end{lem}

\begin{proof} The containment $\supseteq$ is clear from commutativity of the diagram $$ \xym{ \AA(X) \ar[d]_{\AA(f)} \ar[r]^-{\tau_X} & X \ar[d]^f \\ \AA(Y) \ar[r]^-{\tau_Y} & Y } $$ of topological spaces.  For the other containment, suppose $p \in \tau_Y^{-1}(f(U))$.  Then we can write $\tau_Y(p) = f(u)$ for some $u \in U \subseteq X$.  For clarity, let us think of $p$ as a $k$-point of $\AA(Y)$ for some fixed algebraically closed field $k$.  By Lemma~\ref{lem:orbitmap1}, we can find some $k$-point $v \in \AA(X)(k)$ with $\tau_X(v) = u$.  Then the $k$-points $\AA(f)(v), p \in \AA(Y)(k)$ have the same image $f(u)$ under $\tau_Y$, so by Lemma~\ref{lem:orbitmap2} we can find some $$g \in \GG(\M_{Y,f(u)}^{\rm gp})(k) = \Hom_{\Ab}(\M_{Y,f(u)}^{\rm gp},k^*)$$ such that $g \cdot \AA(f)(v) = p$.  By hypothesis, $\M_{Y,f(u)}^{\rm gp} \to \M_{X,u}^{\rm gp}$ is injective, so we can lift $g$ to $\ov{g} \in \GG(\M_{X,u}^{\rm gp})(k)$ because $k^*$ is divisible.  From the $\Spec(\M_{Y,f(u)}^{\rm gp} \to \M_{X,u}^{\rm gp})$-equivariance property of $f$ (or rather, the scheme realization of this), we compute $\AA(f)(\ov{g} \cdot v) = g \cdot  \AA(f)(v) = p$ and we have $\tau_X(\ov{g} \cdot v) = \tau_X(v) = u$ since $\GG(\M_{X,u}^{\rm gp})$ acts on $\AA(X)$ fiberwise for $\tau_X$.  This shows that $p \in \AA(f)(\tau_X^{-1}(U))$ as desired. 

For the ``It follows," note that if $\AA(f)$ is surjective, then when we apply the formula with $U=X$, we find $\tau_Y^{-1}(f(X)) = \AA(Y)$.  Since $\tau_Y : \AA(Y) \to Y$ is surjective (Lemma~\ref{lem:orbitmapsurjective}), we conclude that $f(X)=Y$. \end{proof}

\begin{thm} \label{thm:flatmaps}  The scheme realization $\AA(f)$ of a flat map (resp.\ flat cover) of fans $f$ is a flat map (resp.\ flat and surjective map) of schemes.  The converse(s) holds for integral fans. \end{thm}

\begin{proof} As discussed in \S\ref{section:modules}, the map $\ZZ[h] : \ZZ[Q] \to \ZZ[P]$ is flat whenever $h : Q \to P$ is a flat map of monoids.   The converse also holds when $Q$ and $P$ are integral, by a result of Kato (see my discussion in \cite[5.4.7]{G2}).  Together with Lemma~\ref{lem:formula}, this proves that $\AA( \Spec h )$ is flat when $\Spec h$ is flat and that the converse holds for when $Q$ and $P$ are integral.

We conclude that for a map of fans $f : X \to Y$, the map of schemes $\AA(f)$ is flat if $f$ is flat (and conversely if $X$ and $Y$ are integral) by noting that the question is local (cf.\ Proposition~\ref{prop:flatmaps}), so we can reduce the situation of the previous paragraph.

For the parenthetical statements, we just need to note that if $f : X \to Y$ is a map of fans for which $\AA(f)$ is surjective, then $f$ itself must be surjective by Lemma~\ref{lem:orbitmapsurjective} and the converse holds when $f$ is flat by Lemma~\ref{lem:formula}. \end{proof}

\begin{cor} \label{cor:faithfullyflatimpliesinjective} A faithfully flat map of monoids is injective. \end{cor}

\begin{proof} Suppose $h : Q \to P$ is faithfully flat.  Then $\ZZ[h]$ is a faithfully flat map of rings by the theorem, so it is injective (its kernel $I$ is zero by faithful flatness because $I \otimes \ZZ[P] = I \ZZ[P]$ is clearly zero), hence $h : Q \to P$ must be injective. \end{proof}

We will now derive some topological consequences of flatness and properness.

\begin{thm} \label{thm:openclosed}  Suppose $f : X \to Y$ is a map of locally finite type fans such that: \begin{enumerate} \item $\AA(f)$ is open (resp.\ closed) and \item $(f^\dagger_x)^{\rm gp}$ is injective for all $x \in X$. \end{enumerate}  Then $f$ is open (resp.\ closed). \end{thm}

\begin{proof} Consider an open (resp.\ closed) subset $U \subseteq X$.  We want to show that $f(U) \subseteq Y$ is open (resp.\ closed).  Since $\tau_Y : \AA(Y) \to Y$ is open surjective (Lemma~\ref{lem:orbitmap1}), it suffices to show that $\tau_Y^{-1}(f(U))$ is open (resp.\ closed).   Since $\AA(f)$ is open (resp.\ closed) by assumption, this is immediate from the formula of Lemma~\ref{lem:formula}. \end{proof}

\begin{cor} \label{cor:flatimpliesopen}  If $f : X \to Y$ is a flat map (resp.\ flat cover) of locally finite type fans then $f$ is open (resp.\ open and surjective) on topological spaces. \end{cor}

\begin{proof} The scheme realization $\AA(f)$ is a flat (by Proposition~\ref{prop:flatmaps}) map of locally finite type $\ZZ$-schemes, so it is flat and of locally finite presentation, hence it is open by \cite[IV.2.4.6]{EGA}.  Flatness of $f^\dagger_x : \M_{Y,f(x)} \to \M_{X,x}$ implies flatness (hence injectivity) of its groupification (Lemma~\ref{lem:flat}), so the result now follows from the theorem.  \end{proof}

\begin{cor} \label{cor:localandflatimpliesfaithfullyflat} A flat, local map of finitely generated monoids is faithfully flat, hence also injective by Corollary~\ref{cor:faithfullyflatimpliesinjective}. \end{cor}

\begin{rem} This should hold without the finite generation assumption. \end{rem}

\begin{proof}  Suppose $h : Q \to P$ is such a map.  Then $h^{-1}(P^*)=Q^*$ by definition of ``local," so $Q^* \in \Spec Q$ is in the image of $\Spec h$.  But $\Spec h$ is an open map by the theorem and $\Spec Q$ is the smallest open neighborhood of $Q^*$ in $\Spec Q$ (Proposition~\ref{prop:Spec}\eqref{uniqueclosedpoint}), so $h$ must be surjective. \end{proof} 

\begin{cor} Suppose $f : X \to Y$ is a proper (Definition~\ref{defn:proper}) map of locally finite type fans such that $(f^\dagger)^{\rm gp}_x$ is injective for all $x \in X$.  Then $f$ is closed. \end{cor}

\begin{proof} $\AA(f)$ is closed, so this is clear from the theorem. \end{proof}

Recall (Example~\ref{example:propernotclosed}) that for a map of fine fans $f$, $\AA(f)$ can be proper (equivalently, $f$ can be quasi-compact and valuatively proper) even when $f$ is not closed.  However:

\begin{cor} Suppose $f : X \to Y$ is a quasi-compact, valuatively proper map of fine fans such that $(f^\dagger)^{\rm gp}_x$ is injective for all $x \in X$ (this holds, for example, if $f$ is flat).  Then $f$ is closed. \end{cor}

\begin{proof} The hypotheses imply that $\AA(f)$ is proper (Theorem~\ref{thm:proper}), so this is a special case of the previous corollary.  For completeness, we also give an alternative proof that doesn't appeal to the scheme realization.  The question is local, so we can assume $Y$ is affine, hence finite (Corollary~\ref{cor:faces}).  Suppose $Z \subseteq X$ is closed and $y \in f(Z)^-$.  Since $Y$ is finite, this means $y \in \{ f(x) \}^-$ for some $x \in Z$.  We want to show $y \in f(Z)$.  Since it certainly suffices to prove that $y \in f( \{ x \}^- )$, we can assume $Z = \{ x \}^-$.  If necessary, we can replace $f : X \to Y$ with $Z(f,x) : Z=Z(X,x) \to Z(Y,f(x))$ without destroying the hypotheses (Lemma~\ref{lem:proper2}), so we can assume $x \in T_X$---i.e.\ $\M_{X,x} = \M_{X,x}^*$ is a group---and $X=Z$.  Then $\M_{Y,f(x)}$ is also a group since $f^\dagger_x$ is local, so by Lemma~\ref{lem:closure} we can find a map of fans $g : \Spec \NN \to Y$ taking the generic point to $f(x)$ and the closed point to $y$.  Since $\M_{Y,f(x)} \to \M_{X,x}$ is an injective map of groups by assumption (and its quotient is finitely generated since $X$ and $Y$ are fine), we can assume, after possibly multiplying $g$ by a large enough positive integer, that $g : \M_{Y,f(x)} \to \ZZ$ lifts to a map also abusively denoted $g : \M_{X,x} \to \ZZ$.  Then we get a solid commutative diagram $$ \xym{ \Spec \ZZ \ar[d] \ar[r]^-g & X \ar[d]^f \\ \Spec \NN \ar@{.>}[ru] \ar[r]^-g & Y } $$ which can be completed as indicated by the ``valuatively proper" hypothesis---the image $z$ of the closed point under such a completion lies in $Z=X$ and has $f(z)=y$. \end{proof}

\subsection{Flat descent} \label{section:descent}  In this section we will show that the usual theory of flat descent in algebraic geometry carries over to the category of fans.  I will not bore the reader by writing out all possible analogs of the usual algebraic descent results; I'll just give some of the highlights.

\begin{lem} \label{lem:Amitsur} If $h : Q \to P$ is a faithfully flat map of monoids then a map of $Q$-modules $f : M \to M'$ is an isomorphism iff $f \otimes_{Q} P$ is an isomorphism and \bne{equalizer} & \xym@C+20pt{ Q \ar[r]^-h & P \ar@<.5ex>[r]^-{p \mapsto [p,0]} \ar@<-.5ex>[r]_-{p \mapsto [0,p]} & P \oplus_Q P } \ene is an equalizer diagram of monoids (equivalently: of sets) which remains an equalizer diagram of $Q$-modules after tensoring with any $Q$-module $M$. \end{lem}

\begin{proof}  The key points are that $\ZZ[h]$ is faithfully flat (see Lemma~\ref{lem:flat} and Theorem~\ref{thm:flatmaps}) and the monoid algebra functor $\ZZ[ \slot ]$ commutes with tensor products.  If $f \otimes_{Q} P$ is an isomorphism, then so is $\ZZ[f \otimes_Q P] = \ZZ[f] \otimes_{\ZZ[Q]} \ZZ[P]$, hence so is $\ZZ[f]$ by faithful flatness of $\ZZ[h]$ and it is clear that isomorphy of $\ZZ[f]$ implies isomorphy of $f$.  To prove the second statement, use the first statement and the fact that tensoring with a flat map of monoids preserves equalizers \cite[5.2.4]{G2} to reduce to proving that $M \otimes_Q \eqref{equalizer} \otimes_Q P$ is an equalizer, then observe that $\eqref{equalizer} \otimes_Q P$ is (the truncation of) the contractible augmented cosimplicial monoid $$ \xym{ P \ar[r] & P \oplus_Q P \ar@<.5ex>[r] \ar@<-.5ex>[r] & P \oplus_Q P \oplus_Q P \ar@<.5ex>[r] \ar@<-.5ex>[r] \ar[r] & \cdots } $$ via the usual homotopy operator of Amitsur's Lemma, so it stays contractible after applying any functor, such as $M \otimes_Q \slot$.  Alternatively, use the usual analog of this statement for rings \cite[4.22]{FAG} to conclude that $\ZZ[ \slot ]$ of $M \otimes_Q \eqref{equalizer} \otimes_Q P$ is an equalizer diagram, from which the result follows easily. \end{proof}

\begin{defn} \label{defn:effectivedescent} A map $p : X \to Y$ in a category $\C$ is called an \emph{effective descent morphism} iff the fiber product $X \times_Y X$ exists and \bne{equalizer2} & \xym{ \Hom_{\C}(Y,Z) \ar[r]^-{p^*} & \Hom_{\C}(X,Z) \ar@<.5ex>[r]^-{\pi_1^*} \ar@<-.5ex>[r]_-{\pi_2^*} & \Hom_{\C}(X \times_Y X,Z) } \ene is an equalizer diagram of sets for every object $Z$ of $\C$.  A map is called a \emph{universal effective descent morphism} iff any base change of it exists and is an effective descent morphism. \end{defn}

\begin{thm} \label{thm:flatdescent} If $h : Q \to P$ is a faithfully flat map of monoids then $h$ is a universal effective descent morphism in $\Mon^{\rm op}$.  If $p : X \to Y$ is a flat cover of locally finite type fans then $p$ is a universal effective descent morphism in the category of locally finite type fans. \end{thm}

\begin{proof}  The sorts of maps we are considering are stable under base change (Proposition~\ref{prop:flatmaps}) so we reduce to proving the theorem with ``universal" deleted.

The first statement is immediate from the fact that the diagram in Lemma~\ref{lem:Amitsur} is an equalizer diagram.  For the second statement, suppose $f : X \to Z$ is a map of fans such that $f \pi_1 = f \pi_2 : X \times_Y X \to Z$.  We want to show that there is a unique map of fans $g : Y \to Z$ with $f = gp$.  Fibers products of fans commute with passing to underlying topological spaces and $p$ is an open (Corollary~\ref{cor:flatimpliesopen}) surjective map of topological spaces.  Since an open surjective map of topological spaces is an effective descent morphism, there is a unique such $g$ on the level of topological spaces.  We now need to show that there is a unique map $g^\dagger : g^{-1} \M_Z \to \M_Y$ of sheaves of monoids on $Y$ with $f^\dagger = p^\dagger (p^{-1} g^\dagger) : f^{-1} \M_Z \to \M_X$.  There is at most one such $g^\dagger$ because $p$ is surjective and the stalks $p_x : \M_{X,x} \to \M_{Y,f(x)}$ are flat, local maps of finitely generated monoids hence injective by  Corollary~\ref{cor:localandflatimpliesfaithfullyflat}, so we can construct $g^\dagger$ locally near a given point $y \in Y$.  Pick an affine open neighborhood $\Spec R$ of $z := g(y)$ in $Z$.  Let $Q := \M_{Y,y}$ so that $\Spec Q$ is the smallest neighborhood of $y$ in $Y$.  (Here we use that $Y$ is locally finite type so that its topological space is locally finite.)  In particular $\Spec Q$ is contained in $g^{-1}(\Spec R)$.  Since $p$ is surjective we can find some $x \in p^{-1}(y)$.  Pick an affine open neighborhood $\Spec P$ of $x$ in $p^{-1}(\Spec Q)$.  Since $p$ is open, $p(x)=y$, and $\Spec Q$ is the smallest open neighborhood of $y$ in $Y$, $p | \Spec P : \Spec P \to \Spec Q$ is surjective so the corresponding monoid homomorphism $p^\dagger : Q \to P$ (we will continue the slight abuse of notation introduced here) is faithfully flat.  The original equality $f \pi_1 = f \pi_2$ implies that the monoid homomorphism $f^\dagger : R \to P$ (same slight abuse of notation) corresponding to $f | \Spec R : \Spec P \to \Spec R$ yields the same map to $P \oplus_Q P$ upon composing with either natural map $P \rightrightarrows P \oplus_Q P$.  The first part of the theorem therefore implies that there is a (unique) monoid homomorphism $g^\dagger : R \to Q$ with $f^\dagger = p^\dagger g^\dagger$.  (Note that the map on topological spaces $\Spec g^\dagger$ has to agree with our previously constructed $g$ because $p|\Spec P : \Spec P \to \Spec Q$ is surjective.)  Up to the abusive notation, this $g^\dagger$ is our desired map of sheaves of monoids. \end{proof}

For a map of monoids $h : Q \to P$, let $\Mod(Q \to P)$ denote the category of pairs $(N,\phi)$, where $N \in \Mod(P)$ and $\phi : \pi_1^* N \to \pi_2^* N$ is a ``descent datum:" an isomorphism of $(P \oplus_Q P)$-modules satisfying the usual cocycle condition (c.f. \cite[Page 80]{FAG}).  There is a ``tautological descent datum" functor \bne{tautdescent}  \Mod(Q) & \to & \Mod(P) \\ \nonumber M & \mapsto & (M \otimes_Q P, \phi_{\rm taut}) \ene admitting a right adjoint $G$, where $G(N,\phi) \in \Mod(Q)$ is defined by the equalizer diagram of $Q$-modules $$ \xym{ G(N,\phi) \ar[r] & N \ar@<.5ex>[r] \ar@<-.5ex>[r] & \pi_2^* N,} $$ where the parallel arrows are the natural map $N \to \pi_2^* N$ and the composition of the natural map $N \to \pi_1^*N$ and $\phi$.

\begin{thm} \label{thm:flatdescentofmodules} If $h : Q \to P$ is faithfully flat, then \eqref{tautdescent} is an equivalence of categories with inverse $G$.  If $f : X \to Y$ is a flat cover of locally finite type fans, then $f^*$ induces an equivalence of categories between $\Qco(Y)$ and the category $\Qco(X \to Y)$ of quasi-coherent sheaves on $X$ equipped with a descent datum relative to $f$. \end{thm}

\begin{proof}  The first statement is proved in exactly the same manner as the analogous statement for rings:  The proof of \cite[4.21]{FAG} carries over \emph{mutatis mutandis} by replacing ``$A \to B$" with ``$Q \to P$," \cite[4.22]{FAG} with Lemma~\ref{lem:Amitsur}, and ``kernel of the difference" with ``equalizer."  The geometric variant is proved by reducing to the local affine situation exactly as in the proof of Theorem~\ref{thm:flatdescent}. \end{proof}

Once one has the theory of flat descent for quasi-coherent sheaves (Theorem~\ref{thm:flatdescentofmodules}), the theory of flat descent for ideal sheaves, sections of quasi-coherent sheaves, affine morphisms, projective morphisms, etc.\ follows formally as in \cite[Exp.\ VIII]{SGA1}.

\subsection{Reduced fibers} \label{section:reducedfibers}  Given a ring $A$, a monoid $P$ and an ideal $I$ of $P$, we can consider the $A$-algebra $A[P] / A[I] = A \otimes \ZZ[P]/ \ZZ[I]$.  A basic goal of this section is to characterize---in purely ``monoid theoretic" terms---those $A$, $I$, and $P$ for which this $A$-algebra is \emph{reduced} (i.e.\ any nilpotent element is zero).  A closely related goal---which we shall also pursue---is to give a ``monoid theoretic" classification of the monoid homomorphisms $h : Q \to P$ and the fields $k$ for which $k[h] : k[Q] \to k[P]$ has reduced fibers.  Our mail tool for this purpose is the following variant of ``Gaussian elimination," useful for many questions in the theory of monoids:

\begin{lem} \label{lem:Gaussianelimination} {\bf (Gaussian Elimination)} Fix an integral monoid $P$, elements $p_1,\dots,p_m \in P$, and a submonoid $Q \subseteq P$.  Suppose  \be (A,\ov{q}) & = & \bp a_{11} & a_{12} & \cdots & a_{1m} & | & q_1 \\ a_{21} & a_{22} & \cdots & a_{2m} & | & q_2 \\ \vdots & \vdots & & \vdots & & \vdots \\ a_{m1} & a_{m2} & \cdots & a_{mm} & | & q_m \ep \ee is an ``augmented matrix" with $a_{ij} \in \ZZ$, $q_i \in Q$ satisfying the following conditions: \begin{enumerate} \item For $i \neq j$, $a_{ij}$ is a non-positive integer. \item For each $i$, $\sum_{j=1}^m a_{ij}$ is a positive integer (in particular, each $a_{ii}$ must be positive). \item For each $i$, $ \sum_{j=1}^m a_{ij} p_j = q_i$ in $P^{\rm gp}$. \end{enumerate}  Then there is another augmented matrix $(A',\ov{q}')$ satisfying all the same properties where the off-diagonal entries of $A'$ are zero. \end{lem}

\begin{proof} We perform ``Gaussian elimination."  The point is that matrices satisfying the three conditions are invariant under the following operation: For any $i \neq j$, we can replace row $i$ with $$ a_{jj} \row i - a_{ij} \row j.$$  Indeed, $a_{jj} >0 $ and $-a_{ij} \geq 0$, so this clearly preserves the last two conditions.  For $k \neq i,j$, the new entry in row $i$ column $k$ will be $$a_{jj} a_{ik} - a_{ij} a_{jk},$$ which is $\leq 0$ because $$a_{jj} > 0, \; a_{ik} \leq 0, \; a_{jk} \leq 0.$$ The key point is that the new entry in row $i$ column $j$ is \be a_{jj}a_{ij}-a_{ij}a_{jj} & = & 0. \ee  Now we perform the following algorithm: \begin{enumerate} \item Set $c := 1$ and proceed to the next step.  \item Assume at this step that we are given a matrix satisfying the three properties above, together with the property: For any $j < c$, the only nonzero entry in column $j$ is $a_{jj}$.  Then, for each $i \neq c$, replace row $i$ with $$ a_{cc} \row i - a_{ic} \row c.$$  The resulting matrix still satisfies the three properties above, and now, for any $j < c+1$, the only nonzero entry in column $j$ is $a_{jj}$.  \item If $c=m$, stop, otherwise increase $c$ by one and return to Step 2. \end{enumerate} This algorithm clearly terminates after finitely-many steps to yield the desired $(A',\ov{q}')$. \end{proof}

Although it has nothing to do with ``reduced fibers," it seems appropriate to mention the following ``application" of Gaussian Elimination:

\begin{thm} \label{thm:saturationequalsintegralclosure} Let $Q$ be a submonoid of an integral monoid $P$.  Let \be L & := & \{ p \in P : \exists \; n >0 \; {\rm such \; that \;} np \in Q \} \ee be the saturation of $Q$ in $P$.  Then $\ZZ[L] \subseteq \ZZ[P]$ is the integral closure of $\ZZ[Q]$ in $\ZZ[P]$.  \end{thm}

\begin{proof} For $l \in L$, we have $nl = q$ for some $n \in \ZZ_{>0}$, $q \in Q$, so $[l]^n = [q]$ in $\ZZ[P]$, hence $[l] \in \ZZ[P]$ is a root of the monic polynomial $X^n - [q] \in \ZZ[Q][X]$, so $[l]$ is integral over $\ZZ[Q]$, hence $\ZZ[L]$ is integral over $\ZZ[Q]$ because the set of integral elements always forms a subring.  We have proved that the integral closure of $\ZZ[Q]$ in $\ZZ[P]$ at least \emph{contains} $\ZZ[L]$.

Showing that the integral closure of $\ZZ[Q]$ in $\ZZ[P]$ is contained in $\ZZ[L]$ is more difficult.  Consider a typical element $f = \sum_{i=1}^m a_i [p_i]$ of $\ZZ[P]$.  We can assume each integer $a_i$ is non-zero.  Suppose $f$ is integral over $\ZZ[Q]$, so $f$ satisfies \bne{integralcondition} g_0 + g_1 f + \cdots + g_{n-1}f^{n-1} + f^n & = & 0 \ene in $\ZZ[P]$ for some $g_i \in \ZZ[Q].$  We want to prove $f \in \ZZ[L]$.  I.e., we want to show that each $p_i$ is in $L$.  Think about the coefficient of $[np_i]$ in \eqref{integralcondition}.  There is an obvious (nonzero!) contribution of $a_i^n$ to this coefficient coming from the $f^n$ term.  The other contributions to this coefficient occur only when we can write \bne{piexp} np_i & = & B_{i1}p_1+B_{i2}p_2+ \cdots + B_{im}p_m + q_i \ene for some $B_{ij} \in \NN$ with $\sum_{j=1}^m B_{ij} < n$ and some $q_i \in Q$.  In particular, in order for \eqref{integralcondition} to hold, we must have at least one expression of the form \eqref{piexp} for each $i \in \{ 1, \dots, m \}$.  I claim that this alone will be enough to conclude that each $p_i$ is in $L$.  For $i \neq j$, set $a_{ij} := -B_{ij}$, and set $a_{ii} := n - B_{ii}$.  Set $A := (a_{ij})$, $\ov{q} := (q_1,\dots,q_m)$ (transpose).  The equations \eqref{piexp} imply that the ``augmented matrix" $(A,\ov{q})$ satisfies the conditions in Lemma~\ref{lem:Gaussianelimination}.  By that lemma, we can find another augmented matrix $(A',\ov{q}')$ satisfying the same conditions, but with $A'$ diagonal.  Then the diagonal entries $a_1',\dots,a_m'$ of $A'$ are positive integers and we have $a_i p_i = q_i'$ for $i=1,\dots,m$, hence each $p_i \in L$, as desired. \end{proof}

The next step towards our goal is to handle the case of \emph{groups}:

\begin{lem} \label{lem:reducedfibers}  Let $B$ be an abelian group, $p$ a ``prime," possibly zero, $\FF_p$ the prime field of characteristic $p$ (i.e.\ $\FF_0 = \QQ$, $\FF_p = \ZZ / p \ZZ$ when $p \in \{ 2,3,5, \dots \}$).  The following are equivalent: \begin{enumerate} \item \label{rf1} For any reduced $\FF_p$-algebra $A$, the ring $A \otimes_{\FF_p} \FF_p[B] = A[B]$ is reduced. \item \label{rf2} There exists a non-zero $\FF_p$-algebra $A$ for which the ring $A[B]$ is reduced. \item \label{rf3} $B$ does not contain an element of order $p$. \end{enumerate}  In particular, these equivalent conditions hold when $B$ is torsion-free or $p=0$. \end{lem}

\begin{proof} Obviously \eqref{rf1} implies \eqref{rf2} and obviously \eqref{rf2} implies \eqref{rf3} when $p=0$ since \eqref{rf3} holds automatically in that case.  To see that \eqref{rf2} implies \eqref{rf3} when $p>0$, suppose $b \in B$ has order $p$.  Then, since $A$ is an $\FF_p$-algebra we have $p=0$ in $A$, hence $[b]-[0] = [b]-1$ will nilpotent in $A[B]$ (and non-zero when $A \neq 0$) because $$([b]-[0])^p = [b]^p - [0]^p = [pb]-[p0]=[0]-[0]=0.$$  To see that \eqref{rf3} implies \eqref{rf1}, first note that we can reduce to the case where $B$ is finitely generated because the hypothesis on $B$ in \eqref{rf3} is inherited by any subgroup of $B$ and because any nilpotent $f=\sum_{b \in B} a_b [b]$ in $A[B]$ will be in the subring $A[B']$, where $B'$ is the subgroup of $B$ generated by the finitely many $b \in B$ for which $a_b \neq 0$.  Suppose $A$ is a reduced $\FF_p$-algebra.  Let $\ov{\FF}_p$ be an algebraic closure of $\FF_p$, $\ov{A} := A \otimes_{\FF_p} \ov{\FF}_p$.  The ring $\ov{A}$ is reduced because $A$ is reduced and the prime field $\FF_p$ is perfect.  Next notice that it suffices to show that $$A[B] \otimes_{\FF_p} \ov{\FF}_p = \ov{A} \otimes_{\ov{\FF}_p} \ov{\FF}_p[B] = \ov{A}[B]$$ is reduced, since this ring contains $A[B]$ as a subring because $\FF_p \into \ov{\FF}_p$ is faithfully flat.    By the classification of finitely generated abelian groups, any such group $B$ is of the form $$ \ZZ^r \oplus \ZZ/n_1 \ZZ \oplus \cdots \oplus \ZZ / n_m \ZZ $$ for some $r \in \NN$ and some positive integers $n_1, \dots, n_m$.  The hypothesis on $B$ in \eqref{rf3} ensures that $p$ does not divide any $n_i$, hence $n_1,\dots,n_m \in \ov{\FF}_p^*$.  Since $\ov{\FF}_p$ is algebraically closed, we have \be \ov{\FF}_p[\ZZ/n_i \ZZ] & = & \ov{\FF}_p[x]/(x^{n_i}-1) \\ & \cong & \ov{\FF}_p^{n_i} \ee (product of $n_i$ copies of $\ov{\FF}_p$) because $x^{n_i}-1$ splits and has no repeated roots in $\ov{\FF}_p$ since $x^{n_i}-1$ and its derivative $n_ix^{n_i-1}$ generate the unit ideal in $\ov{\FF}_p[x]$.  The ring $\ov{\FF}_p[B]$ is hence a finite product of rings of the form $\ov{\FF}_p[\ZZ^r] = \ov{\FF}_p[x_1^{\pm 1},\dots,x_r^{\pm 1}]$ and hence $\ov{A}[B]$ is a finite product of rings of the form $\ov{A}[x_1^{\pm 1},\dots,x_r^{\pm 1}]$.  Such a ring is clearly reduced because $\ov{A}$ is reduced.  \end{proof}

\begin{defn} \label{defn:saturatedideal} Let $P$ be a monoid.  An ideal $I \subseteq P$ (\S\ref{section:modules}) is called \emph{saturated} iff for all $p \in P$, if $np \in I$ for a positive integer $n$, then $p \in I$.  A map $h : Q \to P$ of monoids is called \emph{reduced} iff the ideal \be I(h) & := & h(Q \setminus Q^*)+P \ee of $P$ generated by $h(Q \setminus Q^*)$ is saturated.  (If no confusion can result, we often write $I$ for $I(h)$.) \end{defn}

\begin{lem} \label{lem:reduced} An ideal $I$ of a monoid $P$ is saturated iff the image $\ov{I}$ of $I$ in $\ov{P}$ is saturated.  A monoid homomorphism $h : Q \to P$ is reduced iff $\ov{h} : \ov{Q} \to \ov{P}$ is reduced. \end{lem}

\begin{proof} Exercise. \end{proof}

\begin{lem} \label{lem:toricreduced} Suppose $h : Q \into P$ is an injective, dense, reduced, and local map of sharp, fine monoids.  Then $h$ is an isomorphism. \end{lem}

\begin{proof}  We need to show that $h$ is surjective.  Suppose not.  By Lemma~\ref{lem:duality} we can find a local map of monoids $f : P \to \NN$.  Among all $p \in P \setminus Q$, choose one with $f(p)$ minimal.  Note that $f(p) > 0$ for any such $p$ because $p$ cannot be in $P^* = \{ 0 \}$ and $f$ is local.  By density, $np \in Q \setminus Q^*$ for some $n \in \ZZ_{>0}$, so $np \in I(h)$.  If we could write $p = q + p'$ for some $q \in Q \setminus Q^*$ and some $p' \in P$, then we would have $p' \in P \setminus Q$ and $f(q) > 0$ by locality of $f$ and $h$, hence $f(p') < f(p)$ would contradict our choice of $p$.  But then $p \notin I(h)$, contradicting the assumption that $h$ is reduced. \end{proof}

\begin{thm} \label{thm:reducedfibers} Let $P$ be an integral monoid, $I \neq P$ a non-unit ideal of $P$.  Let $p$, $\FF_p$ be as in Lemma~\ref{lem:reducedfibers}.  The following are equivalent: \begin{enumerate} \item \label{reducedfibers1} For any reduced $\FF_p$-algebra $A$, the ring $A \otimes_{\FF_p} \FF_p[P]/\FF_p[I] = A[P]/A[I]$ is reduced. \item \label{reducedfibers2} There exists a non-zero $\FF_p$-algebra $A$ such that $A[P]/A[I]$ is reduced. \item \label{reducedfibers3} $I$ is saturated in $P$ and $P^*$ does not contain an element of order $p$. \end{enumerate}  \end{thm}

\begin{proof}  For any ring $A$, the $A$ module $A[P]$ (resp.\ $A[I]$) is free with basis $P$ (resp.\ $I$), so the $A$-module $A[P]/A[I]$ is free as an $A$-module with basis $P \setminus I$.  We shall use these facts without further comment in the rest of the proof.

Obviously \eqref{reducedfibers1} implies \eqref{reducedfibers2}.  To see that \eqref{reducedfibers2} implies \eqref{reducedfibers3}, suppose that $np \in I$ for some $p \in P$, $n \in \ZZ_{>0}$.  Then $[p]$ is nilpotent in $A[P]/A[I]$, so it is zero by \eqref{reducedfibers2}, so $p \in I$ (since $A \neq 0$).  This proves that $I$ is saturated.  Since $I$ is not the unit ideal, $P^* \cap I = \emptyset$, hence the natural map $A[P^*] \to A[P]/A[I]$ is injective, hence $A[P^*]$ is reduced by \eqref{reducedfibers2}, hence $P^*$ does not contain an element equal to the characteristic of $k$ by Lemma~\ref{lem:reducedfibers}.

To prove that \eqref{reducedfibers3} implies \eqref{reducedfibers1}, we proceed as follows:

\noindent {\bf Step 1:}  We reduce to the case where $P$ is fine.  If $A[P]/A[I]$ contains a non-trivial nilpotent $f = \sum_{p \in P \setminus I} a_p[p]$, then $f$ is also a non-trivial nilpotent in $A[Q]/A[I \cap Q]$, where $Q$ is the fine submonoid of $P$ generated by the finite set of $p \in P \setminus I$ for which $a_p \neq 0$.  Since the hypotheses on $I$ and $P$ are inherited by $I \cap Q$ and $Q$, this completes the reduction.

\noindent {\bf Step 2:} We show that \eqref{reducedfibers3} implies \eqref{reducedfibers1} when $P^*$ is torsion-free.  Suppose $A$ is a reduced $\FF_p$-algebra and $f = \sum_{i=1}^m a_i [p_i] \in A[P]$ is nilpotent in $A[P]/A[I]$.  We want to show that $f \in A[I]$.  After subtracting an appropriate element of $A[I]$ from $f$ if necessary, we can assume that every $a_i$ is nonzero, $p_1, \dots, p_m \notin I$, and the $p_i$ are distinct.  Since $f$ is nilpotent in $A[P]/A[I]$, there is a positive integer $n$ so that $f^n = \sum_{j=1}^s b_j [r_j]$ in $A[P]$ for some $r_1,\dots,r_s \in I$.  Since $I$ is saturated, no $np_i$ appears in the list $r_1,\dots,r_s$, so it must be that, for every $i$, the coefficient of $[np_i]$ in $f^n$ is zero.  On the other hand, since $A$ is reduced and $a_i \neq 0$, $a_i^n \neq 0$, so there is an obvious nonzero contribution of $a_i^n$ to this coefficient when we expand $f^n$.  The only other contributions to this coefficient in $f^n$ occur when we can write \bne{qq} np_i & = & \sum_{j \neq i} n_{ij} p_j \ene for some $n_{ij} \in \NN$ with $\sum_{j \neq i} n_{ij} < n$.  Choose one such expression \eqref{qq} for each $i$, and set $a_{ij} = n$ when $i=j$, $a_{ij} := -n_{ij}$ when $i \neq j$.  Then $A := (a_{ij})$ is an $m \times m$ integer matrix such that the ``augmented matrix" $(A,0)$ satisfies the properties of Lemma~\ref{lem:Gaussianelimination} for $Q := 0$.  By that lemma, we find another augmented matric $(A',0)$ satisfying the same properties, but with $A'$ diagonal with positive integer entries $a_1',\dots,a_m'$ on the diagonal.  Then we have $a_i'p_i = 0$ for $i=1,\dots,m$, so $p_1=\cdots=p_m=0$ because we assume $P^*$ is torsion-free.  But the $p_i$ are distinct and $f = a[0]$ clearly can't be nilpotent unless $a=0$ (or $I=P$), so we must have $f=0$ as desired. (Note that we do not need the result of Step 1 in Step 2---we could have done those steps in the opposite order.)

\noindent {\bf Step 3:} We show that \eqref{reducedfibers3} implies \eqref{reducedfibers1} when $P \cong P^* \oplus \ov{P}$ splits (\S\ref{section:splittings}), so we have $\FF_p[P] \cong \FF_p[P^*] \otimes_{\FF_p} \FF_p[\ov{P}]$ since the monoid algebra functor preserves direct limits.  By Lemma~\ref{lem:reduced}, the assumption in \eqref{reducedfibers3} that $I$ is saturated in $P$ implies that the image $\ov{I} \subseteq \ov{P}$ of $I$ in $\ov{P}$ is a saturated ideal of $\ov{P}$.  The fact that $I \subseteq P$ is an ideal implies that, under our chosen splitting, we have $I = P^* \times \ov{I}$.  In other words, $I$ must be the ideal of $P$ generated by the image of $\ov{I} \subseteq \ov{P}$ under the splitting map $\ov{P} \to P$.  Hence the ideal $\FF_p[I] \subseteq \FF_p[P]$ is the ideal generated by the image of the ideal $\FF_p[\ov{I}]$ of $\FF_p[\ov{P}]$ under $\FF_p[\ov{P}] \to \FF_p[P]$, so we have \be \FF_p[P]/\FF_p[I] & = & \FF_p[\ov{P}] / \FF_p[\ov{I}] \otimes_{\FF_p[\ov{P}]} \FF_p[P] \\ & \cong & \FF_p[\ov{P}] / \FF_p[\ov{I}] \otimes_{\FF_p[\ov{P}]} \FF_p[\ov{P}] \otimes_{\FF_p} \FF_p[P^*] \\ & = & \FF_p[\ov{P}]/\FF_p[\ov{I}] \otimes_{\FF_p} \FF_p[P^*]. \ee  Hence for any $\FF_p$-algebra $A$ we have \bne{reducedformula} A[P]/A[I] & \cong & A[\ov{P}]/A[\ov{I}] \otimes_{\FF_p} \FF_p[P^*]. \ene  If $A$ is reduced, then $A[\ov{P}]/A[\ov{I}]$ is reduced by Step 2, since $\ov{P}^* = 0$ is certainly torsion-free.  Then the assumption on $P^*$ in \eqref{reducedfibers3} and \eqref{reducedformula} imply that $A[P]/A[I]$ is reduced by Lemma~\ref{lem:reducedfibers}.

\noindent {\bf Step 4:} We reduce to the case where $P$ splits.  We can assume $P$ is fine by Step 1, so by Theorem~\ref{thm:splittingmonoids} we can find an inclusion $P \into P'$ with the properties listed in that theorem.  Property \eqref{splittingmonoids1} there ensures that the hypothesis on $P^*$ in \eqref{reducedfibers3} is inherited by $(P')^*$.  Let $I'$ be the ideal of $P'$ generated by the image of $I$.  Property \eqref{splittingmonoids2} in Theorem~\ref{thm:splittingmonoids} implies that $P \into P'$ is faithfully flat (even free of finite rank $|(P')^*/P^*| > 0$), hence $k[P] \into k[P']$ is also faithfully flat, and hence $A[P]/A[I] \to A[P']/A[I']$ is faithfully flat (it is a pushout of the previous map) for any $\FF_p$-algebra $A$.  In particular, this map is injective, so $A[P]/A[I]$ will be reduced whenever $A[P']/A[I']$ is reduced.  The same property also says we can find a basis $S \subseteq (P')^* \subseteq P'$ for $P'$ as a $P$-module.  The ideal $I'$ can then be described explicitly as \be I' & = & \{ i+s : i \in I, s \in S \} . \ee  Suppose $np' \in I'$ for some $p' \in P'$, $n \in \ZZ_{>0}$, so $np' = i+s$ for some $i \in I$, $s \in S$.  We can write $p' = p+s'$ for some $p \in P$, $s' \in S$.  We then find that $np = i+s-ns'$ is in $I$, so $p \in I$ (and hence $p' \in I'$) when $I$ is saturated.  This proves that $I'$ is saturated when $I$ is saturated.  Since Property \eqref{splittingmonoids3} in Theorem~\ref{thm:splittingmonoids} says that $P'$ splits, we conclude from Step 3 that $A[P']/A[I']$ (and hence also its subring $A[P]/A[I]$) is reduced whenever $A$ is reduced is $P$ satisfies \eqref{reducedfibers3}.  This completes the reduction. \end{proof}

\begin{cor} \label{cor:reducedfibers} Let $h : Q \to P$ be a local map of integral monoids.  Let $p$, $\FF_p$ be as in Lemma~\ref{lem:reducedfibers}.  The following are equivalent: \begin{enumerate} \item \label{redfib1} For every field $k$ of characteristic $p$ and every local map of monoids $x : Q \to k = (k, \cdot)$, the $k$-algebra \be \ZZ[P] \otimes_{\ZZ[Q]}^x k & = & k[P] \otimes_{k[Q]}^x k \ee defined as the pushout of $\ZZ[h]$ and $``x" : \ZZ[Q] \to k$ (or, equivalently, as the pushout of $k[h]$ and $``x" : k[Q] \to k$) is reduced. \item \label{redfib2} Let $\1 : Q \to \FF_p$ be the unique local map of monoids with $\1 (u)=1$ for every $u \in Q^*$.  Then the $\FF_p$-algebra $\ZZ[P] \otimes_{\ZZ[Q]}^{\1} \FF_p$ is reduced. \item \label{redfib3} $h$ is reduced (Definition~\ref{defn:saturatedideal}) and $\Cok( h^* : Q^* \to P^*)$ does not contain an element of order $p$. \end{enumerate} \end{cor}

\begin{proof} The assumption that $h$ is local ensures that the ideal $I(h) = h(Q \setminus Q^*)+P$ of $P$ is not the unit ideal (exercise!).  Obviously \eqref{redfib1} implies \eqref{redfib2}.  To see that \eqref{redfib2} and \eqref{redfib3} are equivalent, first note that $\1 : Q \to \FF_p$ factors as the composition of the sharpening map $Q \to \ov{Q}$ and the analogous local map $\ov{\1} : \ov{Q} \to \FF_p$ for $\ov{Q}$.  By considering the corresponding factorization of $`` \1 " : \FF_p[Q] \to \FF_p$ we obtain a diagram of $\FF_p$-algebras \bne{kalgdiagram} & \xym{ \FF_p[Q] \ar[r] \ar[d]_{\FF_p[h]} & \FF_p[\ov{Q}] \ar[d] \ar[r]^{ `` \ov{\1} " } & \FF_p \ar[d] \\ \FF_p[P] \ar[r] & \FF_p[P'] \ar[r] & \ZZ[P] \otimes_{\ZZ[Q]}^{\1} \FF_p } \ene where the squares are pushouts and the left square is $\FF_p[ \slot ]$ of the pushout diagram \bne{monpd} & \xym{ Q \ar[r] \ar[d]_h & \ov{Q} \ar[d]^{h'} \\ P \ar[r] & P/h(Q^*) =: P' } \ene of integral monoids.  The map $`` \ov{\1} " : \FF_p[\ov{Q}] \to \FF_p$ is surjective with kernel given by the ideal $\FF_p[ \ov{Q} \setminus \{ 0 \} ]$ of $\FF_p[\ov{Q}]$ associated to the ideal $\ov{Q} \setminus \{ 0 \}$ of $\ov{Q}$.  Therefore \be \ZZ[P] \otimes_{\ZZ[Q]}^{\1} \FF_p & = & \ZZ[P'] \otimes_{\ZZ[\ov{Q}]}^{\ov{\1}} \\ & = & \FF_p[P']/\FF_p[I(h')], \ee where $I(h') = h'(\ov{Q} \setminus \{ 0 \})+P'$ is the ideal of $P'$ generated by the image of $\ov{Q} \setminus \{ 0 \}$ in $P'$.  It is easy to see that $(P')^* = \Cok(h^* : Q^* \to P^*)$ and that $I(h')$ is saturated iff $I(h)$ is saturated (cf.\ Lemma~\ref{lem:reduced}), therefore \eqref{redfib2} and \eqref{redfib3} are equivalent by Theorem~\ref{thm:reducedfibers} (applied to the integral monoid $P'$ and the ideal $I(h')$).  

To see that \eqref{redfib3} implies \eqref{redfib1}, we first reduce to the case where $k$ is algebraically closed by using faithful flatness of ``the" algebraic closure $k \into \ov{k}$ much as in the proof of Lemma~\ref{lem:reducedfibers}.  We next reduce to the case where $h$ is injective by factoring $h$ as a surjection $f : Q \to Q'$ followed by an injection $h' : Q' \into P$.  Since we have $I(h)=I(h')$ and $\Cok h^* = \Cok (h')^*$, the hypotheses \eqref{redfib3} on $h$ are inherited by $h'$.  Since \be \Spec k[h] & = & (\Spec k[f])(\Spec k[h']) \ee and $\Spec k[f]$ is a closed embedding, every ``fiber" of $\Spec k[h]$ is either empty, or a fiber of $\Spec k[h']$.  (Explicitly, if the local map $x : Q \to k$ factors through $f$ (necessarily uniquely) via $x' : Q' \to k$, then \be \ZZ[P] \otimes_{\ZZ[Q]}^{x} k & = & \ZZ[P] \otimes_{\ZZ[Q'}^{x'} k, \ee whereas if $x$ does not factor through $f$ then there is some $u \in Q^*$ with $h(u)=0$ and $x(u) \neq 1$, and then $w := x(u)[0]-[u] \in k[Q]$ is in the kernel of $``x" : k[Q] \to k$, but $(k[h])(w) = (x(u)-1)[0] \in k^* \subseteq k[P]$ is a unit in $k[P]$, so $k[P] \otimes_{k[Q]}^x k$ is the zero ring.)

When $h$ is injective and $k$ is algebraically closed, we can extend $x^* : Q^* \to k^*$ to a group homomorphism $X : P^{\rm gp} \to k^*$ because $k^*$ is a divisible (i.e.\ injective) abelian group.  In this case we have an isomorphism of $k$-algebras \bne{kAlgiso} \ZZ[P] \otimes_{\ZZ[Q]}^{x} k & \cong & \ZZ[P] \otimes_{\ZZ[Q]}^{\1} k \\ \nonumber [p] \otimes \lambda & \mapsto & [p] \otimes X(p)\lambda \\ \nonumber [p] \otimes X(-p) \lambda & \mapsfrom & [p] \otimes \lambda. \ene  The $k$-algebra on the right is obtained from $\ZZ[P] \otimes_{\ZZ[Q]}^{\1} \FF_p$ by tensoring over $\FF_p$ with $k$, so it will be reduced when $\ZZ[P] \otimes_{\ZZ[Q]}^{\1} \FF_p$ is reduced because $\FF_p$ is perfect.  We thus find that \eqref{redfib3} implies \eqref{redfib1} because we already know \eqref{redfib3} implies \eqref{redfib2}. \end{proof}

\begin{cor} \label{cor:reducedfibers2} Let $f : X \to Y$ be a map of fine fans.  Let $p$, $\FF_p$ be as in Lemma~\ref{lem:reducedfibers}.  The following are equivalent: \begin{enumerate} \item For every $x \in X$, the monoid homomorphism $f_x : \M_{Y,f(x)} \to \M_{X,x}$ is reduced and $p$ does not divide the order of $\Cok (f^*_x : \M_{Y,f(x)}^* \to \M_{X,x}^*)_{\rm tor}$.  \item For every field $k$ of characteristic $p$ and every $k$-point $q \in \AA(Y)(k)$, the scheme $\AA(X) \times_{\AA(Y)}^q \Spec k$ is reduced. \item For every point $q \in \AA(Y)$ having residue field $k(q)$ of characteristic $p$, the scheme-theoretic fiber $\AA(f)^{-1}(q)$ of $\AA(f)$ is reduced. \item For every point $q \in \AA(Y)$ with residue field $k(q) = \FF_p$, the scheme-theoretic fiber $\AA(f)^{-1}(q)$ of $\AA(f)$ is reduced. \end{enumerate} \end{cor}

\begin{proof} For a field $k$, a $k$-point $q \in \AA(Y)(k)$ corresponds to a point $\tau(q) = y \in Y$ together with a local monoid homomorphism $q : \M_{Y,y} \to k$.  The corresponding point of $\AA(Y)$ lies in the open subscheme $\AA(\M_{Y,y})$ of $\AA(Y)$ and $q$ corresponds to the $k$-point of $\AA(\M_{Y,y}) = \Spec \ZZ[\M_{Y,y}]$ determined by the ring homomorphism $`` q " : \ZZ[\M_{Y,y}] \to k$.   As $x$ ranges over $f^{-1}(y)$, the affine fans $\AA(\M_{X,x})$ are open subfans of $\AA(X)$ covering $f^{-1}(y)$, hence the schemes \bne{fibcover} \AA(\M_{X,x}) \times_{ \AA( \M_{Y,y} ) }^q \Spec k & = & \Spec( \ZZ[\M_{X,x}] \otimes_{\ZZ[ \M_{Y,y} ]}^q k ) \ene form an open cover of $\AA(X) \times_{\AA(Y)}^q \Spec k$.  Therefore the latter scheme will be reduced iff the scheme \eqref{fibcover} is reduced for every $x \in f^{-1}(y)$.  If we fix $y \in Y$ and consider the local monoid homomorphism $\1_{p,y} : \M_{Y,y} \to \FF_p$ defined by mapping $\M_{Y,y}^*$ to $1 \in \FF_p$ and $\M_{Y,y} \setminus \M_{Y,y}^*$ to $0 \in \FF_p$, then the corresponding point $q \in \AA(Y)(\FF_p)$ may also be viewed as a point of $\AA(Y)$ with residue field $k(q) = \FF_p$.  The equivalence of the four conditions thus follows from Corollary~\ref{cor:reducedfibers}.  \end{proof}

For (certain) maps of \emph{classical} fans, there is a reformulation of Corollary~\ref{cor:reducedfibers2}, due to Abramovich and Karu (at least in characteristic zero), which is often useful:

\begin{cor} \label{cor:toricreducedfibers} {\bf (Abramovich-Karu)} Let $f : \Sigma \to \Sigma'$ be a map of classical fans corresponding to a map of lattices also abusively denoted $f : N \to N'$.  Let $f^\lor : M' \to M$ be the dual map.  Let $p$, $\FF_p$ be as in Lemma~\ref{lem:reducedfibers}.  Assume that $\Cok (f : N \to N')$ is finite and that $\AA(f)$ is equidimensional, so that $\sigma' := f_{\RR}(\sigma) \in \Sigma'$ for every $\sigma \in \Sigma$ (cf.\ Theorem~\ref{thm:toricequidimensionality}).  For a cone $\sigma \in \Sigma$, set $N_\sigma := N \cap \Span \sigma$.  The following are equivalent: \begin{enumerate} \item \label{toricrf1} For every $\sigma \in \Sigma$, the map of lattices $f_\sigma : N_\sigma \to N'_{\sigma'}$ is surjective and $p$ does not divide the order of the finite group $N'/(N'_{\sigma'}+f(N))$.  \item \label{toricrf1prime} For every cone $\sigma \in \Sigma$, the (local) map of (toric) monoids \be f_\sigma : M' \cap (\sigma')^\lor & \to & M \cap \sigma^\lor \ee induced by $f^\lor$ is reduced and $p$ does not divide the order of the torsion subgroup of the cokernel of \be f_\sigma^* : M' \cap (\sigma')^\perp & \to & M \cap \sigma^\perp. \ee \item \label{toricrf2} For every field $k$ of characteristic $p$ and every $k$-point $q \in \AA(Y)(k)$, the scheme $\AA(\Sigma) \times_{\AA(\Sigma')}^q \Spec k$ is reduced. \item \label{toricrf3} For point $q \in \AA(\Sigma')$ with residue field $k(q) = \FF_p$, the scheme-theoretic fiber $\AA(f)^{-1}(q)$ of $\AA(f)$ is reduced. \end{enumerate} \end{cor}

\begin{proof}  We view the classical fans $\Sigma$, $\Sigma'$ as (abstract) fans $(\Sigma,\M_{\Sigma})$, $(\Sigma',\M_{\Sigma'})$ as in \S\ref{section:classicalfans}.  The points of the abstract fan $\Sigma$ (resp.\ $\Sigma'$) are the cones of the classical fan $\Sigma$ (resp.\ $\Sigma'$).  When we view $f$ as a map of abstract fans as in \S\ref{section:classicalfans}, the maps in \eqref{toricrf1prime} correspond, respectively, to the maps \be f_\sigma : \M_{\Sigma',\sigma'} & \to & \M_{\Sigma,\sigma} \\ f_\sigma^* : \M_{\Sigma',\sigma'} & \to & \M_{\Sigma,\sigma}^* \ee considered in condition \eqref{rf1} of Corollary~\ref{cor:reducedfibers2}.  The conditions \eqref{toricrf1prime}-\eqref{toricrf3} are hence equivalent that corollary (even without the assumptions in the sentence ``Assmue that \dots"), so it remains only to show that \eqref{toricrf1} is equivalent to \eqref{toricrf1prime}.

We first handle the case $p=0$ (where the second parts of conditions \eqref{toricrf1} and \eqref{toricrf1prime} hold automatically), closely following (i.e.\ plagiarising) Karu \cite[Lemma~5.2]{Kar}.  

To see that \eqref{toricrf1prime} implies \eqref{toricrf1} (when $p=0$), suppose, towards a contradiction, that \eqref{toricrf1prime} holds, but $f : N_\sigma \to N_{\sigma'}$ is not surjective for some $\sigma \in \Sigma$.  By Lemma~\ref{lem:conesequidimensionality} (and the equidimensionality assumption), we can find a face $\tau \leq \sigma$ such that $f_{\RR} : \tau \to \sigma'$ is bijective.  Then $\tau' = \sigma'$ and $f : N_{\tau} \to N'_{\sigma'}$ is injective but not surjective (since $N_\tau \subseteq N_\sigma$).  Replace $\sigma$ by $\tau$.  Now replace $N$ (resp.\ $N'$) by $N_\sigma$ (resp.\ $N'_{\sigma'}$) and $\Sigma$ (resp.\ $[\sigma']$) with $[\sigma]$ (resp.\ $[\sigma']$).  These replacements don't destroy the assumptions that $f : N \to N'$ has finite cokernel (because the assumptions ensure that $N_{\sigma} \to N_{\sigma'}$ has finite cokernel for any $\sigma \in \Sigma$), that $\AA(f)$ is equidimensional (this is clear from Theorem~\ref{thm:toricequidimensionality}), and that $f_\sigma$ is reduced.  (Indeed, this replacement has the effect of replacing $f_\sigma$ with $\ov{f}_\sigma$, which is reduced when $f_\sigma$ is reduced by Lemma~\ref{lem:reduced}.)  We can therefore assume that $\sigma$ (resp.\ $\sigma'$) is the unique maximal cone of $\Sigma$ (resp.\ $\Sigma'$), that these cones are full-dimensional, and that $f : N \to N'$ is injective, with finite (but non-zero) cokernel.  Then $f^\lor : M' \to M$ is also injective, with finite (but non-zero) cokernel.  Since $f_{\RR}$ is an isomorphism taking $\sigma$ bijectively onto $\sigma'$, the dual map $f_{\RR}^\lor$ is an isomorphism taking $(\sigma')^\lor$ bijectively onto $\sigma^\lor$, hence \be f_\sigma : M' \cap (\sigma')^\lor & \into & M \cap \sigma \ee is an injective, dense, reduced (by \eqref{toricrf1prime}), and local map of sharp, fine monoids which is not an isomorphism because its groupification $f^\lor : M' \into M$ is not an isomorphism---this contradicts Lemma~\ref{lem:toricreduced}. 

To see that \eqref{toricrf1} implies \eqref{toricrf1prime} (when $p=0$), suppose, towards a contradiction, that \eqref{toricrf1} holds, but \bne{fsigmaAKreduced} f_\sigma : M' \cap (\sigma')^\lor & \to & M \cap \sigma^\lor \ene is not reduced for some $\sigma \in \Sigma$.  Then for some $m \in M \cap \sigma^\lor$, we can write \bne{nonredformula} nm & = & f_{\sigma}(m_1)+m_2 \ene for some $n>1$, $m_1 \in M' \cap (\sigma')^\lor \setminus M' \cap (\sigma')^\perp$, $m_2 \in M \cap \sigma^\lor$ but

\noindent {\bf (*)} no such $m_1,m_2$ exist for $n=1$.

Let $\rho$ be the unique face of $\sigma^\lor$ containing $m_2$ in its interior.  Since there are only finitely many faces of $\sigma^\lor$, we can assume the quadruple $(n,m,m_1,m_2)$ is chosen ``minimally," in the sense that there is no other such quadruple $(n',m',m_1',m_2')$ with $m_2'$ lying in a proper face of $\rho$.  Let $\sigma_0 := \sigma \cap m_2^\perp = \sigma \cap (\Span \rho)^\perp$, so that $\sigma_0$ is a face of $\sigma$ with $\sigma_0^\perp = \Span \rho$ (cf.\ Theorem~\ref{thm:cones}).  Next we show that the image of the map \bne{nonsaturatedimage} g : M' \cap (\sigma')^\lor & \to & \ov{(M \cap \sigma_0^\lor)} = (M \cap \sigma_0^\lor)/(M \cap \sigma_0^\perp) \ene induced by \eqref{fsigmaAKreduced} and the inclusion $M \cap \sigma^\lor \subseteq M \cap \sigma_0^\lor$ is not saturated.  Since $$m_2 \in M \cap \rho \subseteq M \cap \Span \rho = M \cap \sigma_0^\perp,$$ formula \eqref{nonredformula} shows that $n\ov{m}$ is in the image of $g$, where $\ov{m}$ is the image of $m \in M \cap \sigma^\lor \subseteq M \cap \sigma_0^\lor$ in $\ov{(M \cap \sigma_0^\lor)}$---but we claim that $\ov{m}$ is not in the image of $g$.  If it were, then we could write \bne{nonredformula2} m & = & f_\sigma(m_1')+m_2' \ene for some $m_1' \in M' \cap (\sigma')^\lor$, $m_2' \in M \cap \sigma_0^\perp$.  Furthermore, we can always arrange that $m_1' \in  M' \cap (\sigma')^\lor \setminus M' \cap (\sigma')^\perp$ because, if \eqref{nonredformula2} holds for some $m_1' \in M' \cap (\sigma')^\perp$, then $$f_{\sigma}(m_1') \in M \cap \sigma^\perp \subseteq M \cap \sigma_0^\perp,$$ so $m \in M \cap \sigma_0^\perp$ and then, using \eqref{nonredformula}, we can write \bne{nonredformula3} m & = & f_\sigma(m_1) + (m_2-(n-1)m) \ene instead of \eqref{nonredformula2}.  Therefore we must have $m_2' \notin M \cap \sigma^\lor$ by {\bf (*)}, so $m_2' \in \Span \rho \setminus \rho$ and hence Theorem~\ref{thm:cones}\eqref{cones:lemma2} ensures that there is some $s \in (0,1) \cap \QQ$ such that $sm_2+(1-s)m_2'$ is in some proper face $\rho' < \rho$.  Clearing denominators, we can find $k,l \in \ZZ_{>0}$ so that $m_2'' := km_2+lm_2' \in \rho'$.  But then \bne{nonredformula4} (kn+l)m & = & f_\sigma(km_1+lm_1')+m_2'' \ene contradicts our choice of the quadruple $(n,m,m_1,m_2)$.  This completes the proof that \eqref{nonsaturatedimage} does not have saturated image.

The map \eqref{nonsaturatedimage} factors as the surjection \be M' \cap (\sigma')^\lor & \to & (M' \cap (\sigma_0')^\lor) / (M' \cap (\sigma_0')^\perp)  \ee followed by the map \bne{secondmap} (M' \cap (\sigma_0')^\lor) / (M' \cap (\sigma_0')^\perp) & \to & (M \cap \sigma_0^\lor) / (M \cap \sigma_0^\perp), \ene hence \eqref{secondmap} does not have saturated image.  Since $f_{\RR}(\sigma_0) = \sigma_0'$, we have a \emph{cartesian} diagram $$ \xym{ M' \cap (\sigma_0')^\perp \ar[r] \ar[d] & M \cap \sigma_0^\perp \ar[d] \\ M' \cap (\sigma_0')^\lor \ar[r] \ar[d] & M \cap \sigma_0^\lor \ar[d] \\ M' \ar[r]^-{f^\lor} & M, } $$ from which we see that \eqref{secondmap} is injective.  Hence \eqref{secondmap} (which is nothing but the sharpening of $f_{\sigma_0} : M' \cap (\sigma_0')^\lor \to M \cap \sigma_0^\lor$) is an injective map of saturated monoids (cf.\ Corollary~\ref{cor:saturatedlocalization}) which does not have saturated image, so its groupification \bne{secondmapgroupification} M' / (M' \cap (\sigma_0')^\perp) & \to & M / (M \cap  \sigma_0^\perp) \ene is easily seen to be a map of lattices which does not have saturated image (i.e.\ its cokernel has non-trivial torsion), hence the dual map of \eqref{secondmapgroupification}, which is $N_{\sigma_0} \to N_{\sigma_0'}$, is not surjective, contradicting \eqref{toricrf1prime}. 

Now that the $p=0$ case is established, it remains only to show that the second condition in \eqref{toricrf1} is equivalent to the second condition in \eqref{toricrf1prime}.  For this, it is enough to show that for any $\sigma \in \Sigma$, there is an isomorphism (not necessarily canonical) of finite abelian groups \bne{twogroups} \Cok( f_{\sigma}^* : M' \cap (\sigma')^\perp \to M \cap \sigma^\perp )_{\rm tor} & \cong &  N'/(f(N)+N'_{\sigma'})  . \ene  The map $f$ induces a map of exact sequences of lattices \bne{latticeSESmap} & \xym{ 0 \ar[r] & N_\sigma \ar[r] \ar[d] & N \ar[d] \ar[r] & N/N_\sigma \ar[d]^-g \ar[r] & 0 \\ 0 \ar[r] & N'_{\sigma'} \ar[r] & N' \ar[r] & N' / N'_{\sigma'} \ar[r] & 0 } \ene where the group on the right of \eqref{twogroups} is the cokernel of the map labelled $g$.  Dualizing \eqref{latticeSESmap} we obtain a map of exact sequences of lattices \bne{latticeSESmapdual} &  \xym{ 0 \ar[r] & M' \cap (\sigma')^\perp \ar[r] \ar[d]_-{g^\lor = f_{\sigma}^* } & M' \ar[r] \ar[d] & M'/(M' \cap (\sigma')^\perp) \ar[r] \ar[d] & 0 \\ 0 \ar[r] & M \cap \sigma^\perp \ar[r] & M \ar[r] & M / (M \cap \sigma^\perp) \ar[r] & 0, } \ene hence we obtain a (canonical!) isomorphism \be \Cok( f_{\sigma}^* : M' \cap (\sigma')^\perp \to M \cap \sigma^\perp )_{\rm tor} & = & \Ext^1( N'/(f(N)+N'_{\sigma'}), \ZZ ) \ee by Lemma~\ref{lem:toricreducedfibers}, below.  We get a (non-canonical) isomorphism \eqref{twogroups} by using the fact that $\Ext^1(B,\ZZ)$ is non-canonically isomorphic to $B$ for any finite abelian group $B$.  \end{proof}

\begin{lem} \label{lem:toricreducedfibers} Let $f : L \to L'$ be a map of lattices (finitely generated free abelian groups) with finite cokernel $B$, so the map of dual lattices $f^\lor : (L')^ \to L^\lor$ is injective.  Then we have a natural isomorphism of finite abelian groups \be (\Cok f^\lor)_{\rm tor} & = & \Ext^1(B,\ZZ). \ee \end{lem}

\begin{proof}  By applying $\Hom( \slot, \ZZ)$ to the appropriate exact sequence, we easily establish the result when $f$ is either injective or surjective.  In general we factor $f$ as a surjection $L \to L''$ followed by an injection $L'' \into L'$.  Since $L'' \into L'$ is injective with cokernel $B$, we have $((L'')^\lor / (L')^\lor))_{\rm tor} = \Ext^1(B,\ZZ)$.  The inclusions $(L')^\lor \into (L'')^\lor \into L^\lor$ yield a short exact sequence $$ 0 \to (L'')^\lor / (L')^\lor \to L^\lor / (L')^\lor \to L^\lor / (L'')^\lor \to 0.$$  The right term is torsion-free because $L \to L''$ is surjective, so the middle group has the same torsion subgroup as the left group, namely $\Ext^1(B,\ZZ)$, as desired.  \end{proof}

\subsection{CZE maps} \label{section:CZEmaps}  The purpose of this section is to introduce \emph{CZE maps} and the \emph{CZE topology}---these are analogous to \'etale maps and the \'etale topology of schemes.

\begin{thm} \label{thm:CZEmaps} For a map $h : Q \to P$ of fine monoids, the following are equivalent: \begin{enumerate} \item \label{CZE1} $h^*$ is injective with finite cokernel and $$ \xym{ Q^* \ar[d] \ar[r]^-{h^*} & P^* \ar[d] \\ Q \ar[r]^-h & P } $$ is a pushout diagram of monoids. \item \label{CZE2} $h$ is injective, $\ov{h} : \ov{Q} \to \ov{P}$ is an isomorphism, and $P/Q$ is finite. \item \label{CZE3} There exists a finite subset $S \subseteq P^*$ which is a $Q$-basis for $P$ (i.e.\ the map of sets $Q \times S \to P$ defined by $(q,s) \mapsto h(q)+s$ is bijective) and any other $Q$-basis $S' \subseteq P$ for $P$ is contained in $P^*$ and has the same cardinality as $S$.  \item \label{CZE4} There is a finite subset $S \subseteq P^*$ which is a $Q$-basis for $P$.  \item \label{CZE5} $\Spec h$ is a homeomorphism on the level of topological spaces and $\Spec k[h]$ is a finite \'etale cover for any field $k$ of characteristic zero. \item \label{CZE5a} $\Spec k[h]$ is a finite \'etale cover for any field $k$ of characteristic zero.  \item \label{CZE6} There exists a field $k$ for which $k[h]$ is finite \'etale. \item \label{CZE7} $h$ is finite, flat, reduced (Definition~\ref{defn:saturatedideal}), and local. \item \label{CZE8} $h$ is finite, flat, and reduced.  \end{enumerate} \end{thm}

\begin{proof} \eqref{CZE1}$\implies$\eqref{CZE2}: One readily checks that that any pushout of an injective map of groups in injective.  To see that $\ov{h}$ is an isomorphism, just note that $\ov{h^*}$ is trivially an isomorphism since $\ov{Q^*} = \ov{P^*} = 0$, then use the fact that sharpening preserves pushouts because it is a left adjoint.  From the pushout diagram in \eqref{CZE1}, we find that $P/Q = P^*/Q^*$ is finite.

\eqref{CZE2}$\implies$\eqref{CZE1}:  Since $\ov{h}$ is an isomorphism, the diagram in \eqref{CZE1} is a pushout by \cite[Lemma~1.2.11]{GM1}.  Obviously $h^*$ is injective since $h$ is injective; $P^*/Q^*$ is finite because $P/Q$ is finite and the pushout diagram in \eqref{CZE1} yields $P/Q = P^*/Q^*$.

\eqref{CZE1}$\implies$\eqref{CZE3}: Pick any finite subset $S \subseteq P^*$ taken bijectively onto $P^*/Q^*$ by the quotient projection $P^* \to P^*/Q^*$.  Then $S$ is clearly a $Q^*$-basis for $P^*$.  The pushout diagram in \eqref{CZE1} then implies that $S \subseteq P^* \subseteq P$ is also a $Q$ basis for $P$.  Now suppose $S' \subseteq P$ is another $Q$-basis for $P$.  For $s' \in S'$ we can write $s' = q+s$ for some $q \in Q$, $s \in S$.  We can also write $s = q' +s''$ for some $q' \in Q$, $s'' \in S$.  Combining these we find $s' = q+q'+s''$, hence $s'=s''$ and $q+q'=0$ because $S'$ is a basis.  Since $s \in P^*$, the formula $s' = q+s$ now shows that $s' \in P^*$.  This proves that $S' \subseteq P^*$.  I claim that $S'$ is a $Q^*$ basis for $P^*$.  The only issue is to show that, for $u \in P^*$, when we write $u = q' + s'$ for $q' \in Q$ and $s' \in S'$ we must actually have $q' \in Q^*$.  Since $s' \in P^*$, we can write $s' = q + s$ for some $q \in Q^*$, $s \in S$.  We then find $u = q'+q+s$.  But then, since $S$ is a $Q^*$ basis for $P^*$ and a $Q$ basis for $P$, we must have $q'+q \in Q^*$, hence $q' \in Q^*$ as desired.  Since $S'$ is a $Q^*$ basis for $P^*$, it is clear that the quotient map $P^* \to P^*/Q^*$ takes $S'$ bijectively onto $P^*/Q^*$, hence $|S'|=|P^*/Q^*|=|S|$, as desired.

Obviously \eqref{CZE3}$\implies$\eqref{CZE4}.

\eqref{CZE4}$\implies$\eqref{CZE2}:  Since $P$ is a free $Q$-module, $h$ is certainly injective and the fact that the $Q$-basis $S$ for $P$ lies in $P^*$ easily implies that $\ov{h}$ is an isomorphism.  Since $P \to P/Q$ will take the basis $S$ bijectively onto $P/Q$, we find that $P/Q$ is finite.

\eqref{CZE1}$\implies$\eqref{CZE5}: $\Spec h$ is a homeomorphism because $\Spec \ov{P} \to \Spec P$ is a homeomorphism for any monoid $P$ (Corollary~\ref{cor:sharpening}) and $\ov{h}$ is an isomorphism (we already proved \eqref{CZE1}$\implies$\eqref{CZE2}).  Since finite \'etale maps are stable under base change and $\Spec k[ \slot ]$ takes pushout diagrams of monoids to cartesian diagrams of $k$-schemes, we just have to show that $\Spec k[h^*]$ is a finite \'etale cover when $h^* : Q^* \into P^*$ is an injective map of finitely generated abelian groups with finite cokernel.  Notice that the $k$-algebra $k[P^*]$ is obtained from the $k$-algebra $k[Q^*]$ by adjoining $n^{\rm th}$ roots (for varying $n$) of various elements of the form $[q] \in k[Q^*]$ for $q \in Q^*$.  Since these elements $[q]$ are units in $k[Q^*]$ and our base field $k$ has characteristic zero, this proves that $k[h^*]$ is \'etale.  Certainly $k[h^*]$ is finite because if $S$ is a $Q^*$ basis for $P^*$ then $S$ is also a $k[Q^*]$ basis for $k[P^*]$.  Certainly $\Spec h^*$ is surjective since $\Spec Q^*$ and $\Spec P^*$ are one-point spaces.  Since $h^*$ is also injective, the map $\Spec \ZZ[h^*]$ is surjective by Lemma~\ref{lem:surjectivity}, hence $\Spec k[h^*]$ is also surjective because surjective maps of schemes are stable under base change.

Obviously \eqref{CZE5}$\implies$\eqref{CZE5a}$\implies$\eqref{CZE6}.

\eqref{CZE6}$\implies$\eqref{CZE7}:  Since $k[h]$ is \'etale it is, in particular, flat, hence $h : Q \to P$ is also flat by \cite[4.1]{KK}.  Similarly, since $\Spec k[h]$ is finite, we see easily (cf.\ \cite[5.2.1]{G2}) that $h : Q \to P$ is finite.  Since $h$ is a flat map of integral monoids, it is injective by Lemma~\ref{lem:flat}.  Since $h$ is a finite, injective map of fine monoids, $\Spec h$ is surjective by Theorem~\ref{thm:finite}, but then $h$ must be local because otherwise $Q^* \in \Spec Q$ would not be in the image of $\Spec h$.  Since $h$ is local and $k[h]$ is \'etale (hence has reduced fibers), $h$ must be reduced by Corollary~\ref{cor:reducedfibers}.

Obviously \eqref{CZE7}$\implies$\eqref{CZE8}.

\eqref{CZE8}$\implies$\eqref{CZE4}:  Since $h$ is a finite, flat map of fine monoids, there is a finite basis $S \subseteq P$ for $P$ as a $Q$ module by the ``Monoidal Quillen-Suslin Theorem" \cite[Theorem~5.6.2]{G2}.  It remains to use the further assumption that $h$ is reduced (i.e.\ that $I := h(Q \setminus Q^*)+P$ is a saturated ideal of $P$) to prove that any such basis $S$ must in fact be contained in $P^*$.  Suppose, toward a contradiction, that some $s \in S$ is not in $P^*$.  A finite map $h$ of fine monoids is dense (Definition~\ref{defn:dense}) by Theorem~\ref{thm:dense}\eqref{dense3}, so $ns \in h(Q)$ for some $n \in \ZZ_{>0}$.  We cannot have $ns \in h(Q^*) \subseteq P^*$ because $s \notin P^*$, so we have $ns \in I$.  But $I$ is saturated, so $s \in I$ and hence we can write $s = h(q) + p$ for some $q \in Q \setminus Q^*$ and some $p \in P$.  Since $S$ is a $Q$-basis for $P$, we can write $p = h(q')+s'$ for some $q' \in Q$, $s' \in S$, hence $h(0)+s = h(q+q')+s'$, so, since $S$ is a basis we find that $q+q'=0$ (and that $s=s'$, which is not important), which is a contradiction because $q \notin Q^*$. \end{proof}

\begin{rem} \label{rem:CZEfiniteness} It is possible to consider various generalizations of CZE maps of fine monoids and fine fans.  For example, Condition \eqref{CZE1} in Theorem~\ref{thm:CZEmaps} serves as a good definition of CZE maps between \emph{general} monoids.  For a prime number $p$, one can also consider ``CpE" (``characteristic $p$ \'etale") maps of monoids by further demanding in Theorem~\ref{thm:CZEmaps}\eqref{CZE1} that the order of $\Cok (h^*)$ be prime to $p$.  Without finiteness assumptions, the ``best" way to define a CZE map of fans $f$ is to ask that $f$ be Zariski locally (on its domain and codomain) given by $\Spec$ of a CZE map of monoids---we shall see in Proposition~\ref{prop:CZEimpliesZariski} that for fine fans this manner of defining CZE maps corresponds with our Definition~\ref{defn:CZEmap} in terms of stalks.  For simplicity, we will stick to fine monoids, fine fans, and ``characteristic zero" in this section, but we invite the reader to think about these more general notions if desired. \end{rem}

\begin{defn} \label{defn:CZEmap}  A map of fine monoids $h :Q \to P$ is called \emph{CZE (characteristic zero \'etale)} iff it satisfies the equivalent conditions in Theorem~\ref{thm:CZEmaps}.  A map of fine fans $f : X \to Y$ is called \emph{CZE} iff $f_x : \mathcal{M}_{Y,f(x)} \to \mathcal{M}_{X,x}$ is a CZE map of monoids for every $x \in X$.  A \emph{CZE cover} is a surjective CZE map of fine fans.  The \emph{CZE topology} on the category of fine fans is the topology generated by the CZE covers (cf.\ \cite[IV.4.2]{SGA3}). \end{defn}

\begin{rem} \label{rem:CZEtopology}  Since a CZE map of monoids is flat, every CZE map of fine fans is flat (in the sense of Definition~\ref{defn:flatmapoffans}) and every CZE cover is, in particular, a flat cover.  In particular, by Theorem~\ref{thm:flatdescent}, every CZE cover is a universal effective descent morphism---i.e.\ the CZE topology on fine fans is subcanonical. \end{rem}

We now collect some basic facts about CZE maps of fine monoids:

\begin{lem} \label{lem:CZEmaps} CZE maps of fine monoids are stable under composition and base change.  Suppose $h : Q \into P$ is a CZE map of fine monoids and $F \leq P$ is a face of $P$.  Then $(F \cap Q) \into F$ and $(F \cap Q)^{-1}Q \to F^{-1}P$ are also CZE maps of fine monoids.  It follows that $\Spec h$ is a CZE map of fine fans which is a homeomorphism on topological spaces---in particular $\Spec h$ is a CZE cover. \end{lem}

\begin{proof} To see that CZE maps of fine monoids are stable under composition, use Theorem~\ref{thm:CZEmaps}\eqref{CZE1} and note that a composition of pushout squares is a pushout and a composition of injective maps of abelian groups with finite cokernel is again injective with finite cokernel.  Stability of CZE maps under pushout follows easily from any one of several characterizations of such maps in Theorem~\ref{thm:CZEmaps}.  For example, using \eqref{CZE6}, one need only note that the functor $\Spec k[ \slot ]$ takes pushouts of monoids to pullbacks of schemes and finite \'etale maps of schemes are stable under pullback.  Since $h : Q \into P$ is CZE, we can find a finite $Q$-basis $S \subseteq P^*$ for $P$.  Since $F \leq P$ is a face, we have $P^* \subseteq F$, so, in particular, $S \subseteq F^*$ and we claim that $S$ is an $(F \cap Q)$-basis for $F$.  Indeed, given $f \in F$ we know that it is possible to write $f = q+s$ for a unique $q \in Q$, $s \in S$ because $S$ is a $Q$-basis for $P$, so the issue is to show that this $q$ must actually be in $F \cap Q$---but that is clear from the definition of a face since $q+s = f \in F$.  Since we already showed that CZE maps are stable under pushout, we can prove that $(F \cap Q)^{-1}Q \to F^{-1}P$ is CZE by showing that $$ \xym{Q \ar[r]^-h \ar[d] & P \ar[d] \\ (F \cap Q)^{-1}Q \ar[r] & F^{-1}P } $$ is a pushout.  By the universal property of localization, this amounts to showing that $(F \cap Q)^{-1}P = F^{-1}P$.  For this it is enough to note that any $f \in F$ can be written as $f=q+s$ with $q \in F \cap Q$, $s \in P^*$, as we saw above. \end{proof}

\begin{prop} \label{prop:CZEmaps} Suppose $h : Q \to P$ is a map of fine monoids for which the map of fine fans $\Spec h$ is CZE (i.e.\ $h^{-1}(F)^{-1}Q \to F^{-1}P$ is a CZE map of fine monoids for every face $F \leq P$).  Then $h$ is injective, quasi-finite, and flat, and if we let $G := h^{-1}(P^*)$, then the induced map $G^{-1} Q \to P$ is a CZE map of fine monoids.  The following are equivalent: \begin{enumerate} \item \label{CZEhlocal} $G=Q^*$ (i.e.\ $h$ is local). \item \label{CZEhdense} $h$ is dense (Definition~\ref{defn:dense}). \item \label{CZEhfinite} $h$ is finite ($P$ is finitely generated as a $Q$-module). \item \label{CZEhfiniteand} $h$ is finite and $\Spec h$ is a homeomorphism.  \item $\Spec h$ is a homeomorphism. \item \label{CZEhquasifiniteand} $\Spec h$ is surjective. \item \label{CZEmapofmonoids} $h$ is a CZE map of fine monoids \item \label{CZEcover} $\Spec h$ is a CZE cover. \end{enumerate} \end{prop}

\begin{proof} Since $h^{-1}(F)^{-1}Q \to F^{-1}P$ is a CZE map of fine monoids for every $F \leq P$, $h$ is injective (take $F=P$ to see that $h^{\rm gp}$ is injective since a CZE map of fine monoids is injective), quasi-finite, and flat, since these conditions can be checked at each face of $P$ and a CZE map of fine monoids is finite (hence quasi-finite) and flat.  Obviously $G^{-1}Q \to P$ is CZE because it is just the map $h^{-1}(F)^{-1}Q \to F^{-1}P$ in the case $F=P^*$.  The conditions \eqref{CZEhlocal}-\eqref{CZEhquasifiniteand} are equivalent by Theorem~\ref{thm:finite}.  It is immediate from the definitions (i.e.\ from the characterization of CZE maps of fine monoids in Theorem~\ref{thm:CZEmaps}) that these conditions are equivalent to the other two. \end{proof}

\begin{prop} CZE maps (resp.\ CZE covers) of fine fans are closed under composition and base change. \end{prop}

\begin{proof} CZE maps of fine monoids are stable under composition and base change by Lemma~\ref{lem:CZEmaps}, so this follows from Proposition~\ref{prop:stableunderbasechange}. \end{proof}

\begin{prop} \label{prop:CZEimpliesZariski}  Suppose $f : X \to Y$ is a CZE map of fine fans.  Then: \begin{enumerate} \item \label{CZElocalpicture} Zariski locally on $f$, $f$ is $\Spec$ of a CZE map of fine monoids.  \item \label{CZElocalhomeo} The map of topological spaces underlying $f$ is a local homeomorphism.  That is, every point $x \in X$ has an open neighborhood mapping homeomorphically under $f$ onto an open neighborhood of $f(x)$ in $Y$. \end{enumerate} In particular, any CZE cover of fine fans is a Zariski cover on the underlying spaces. \end{prop}

\begin{proof} \eqref{CZElocalpicture} follows from Proposition~\ref{prop:localpicture} because the topological space underlying a fine fan is locally finite (Proposition~\ref{prop:locallyfinitetypefans}).  Since $\Spec h$ is a homeomorphism for any CZE map $h$, \eqref{CZElocalhomeo} follows from \eqref{CZElocalpicture}. \end{proof}

\begin{prop} \label{prop:CZErealizationetale} Suppose $f : X \to Y$ is a CZE map (resp.\ CZE cover) of fine fans.  Then the scheme realization $\AA(f)$ of $f$ is a flat map (resp.\ an fppf cover) of locally finite type schemes which becomes \'etale (resp.\ an \'etale cover) after base change along $\ZZ \to \QQ$. \end{prop}

\begin{proof}  As noted in Remark~\ref{rem:CZEtopology}, $f$ is, in particular, a flat map (resp.\ flat cover) of fine fans, so the first part of the proposition follows from Theorem~\ref{thm:flatmaps} and it remains only to check that $\AA(f)$ becomes \'etale after base change along $\ZZ \to \QQ$.  This can be checked locally on $\AA(f)$, so it follows from Proposition~\ref{prop:CZEimpliesZariski}\eqref{CZElocalpicture} and the characterization of CZE maps of fine monids in Theorem~\ref{thm:CZEmaps}.  \end{proof}

Recall from \S\ref{section:groupobjects} that a (finitely generated abelian) group $A$ gives rise (contravariantly in $A$) to an abelian group object $\GG(A)$ in the category of fine fans, with underlying fan $\Spec A$.  

\begin{prop} \label{prop:CZEtorsors} For any exact sequence of finitely generated abelian groups \bne{FGAGSeq} & 0 \to A \to B \to C \to 0, \ene there exists an injective map of abelian groups $A \into A'$ with finite cokernel such that the exact sequence \bne{FGAGSeq2} & 0 \to A' \to B' \to C \to 0 \ene obtained by pushing out \eqref{FGAGSeq} along $A \to A'$ splits.  It follows that the sequence \bne{GGFGAGSeq} & 0 \to \GG(C) \to \GG(B) \to \GG(A) \to 0 \ene obtained by applying $\GG( \slot ) = \Spec( \slot )$ to \eqref{FGAGSeq} is a short exact sequence of (representable) sheaves of abelian groups on the category of fine fans in the CZE topology which can be split after pulling back along a CZE cover $\GG(A') \to \GG(A)$.  Equivalently, the induced action of $\GG(C)$ on $\GG(B)$ makes $\GG(B) \to \GG(A)$ a $\GG(C)$ torsor, locally trivial in the CZE topology.  \end{prop}

\begin{proof} The first statement was proved in Lemma~\ref{lem:splittinggroups}.  For the other statements, note that such a map $A \to A'$ is a CZE map of monoids, hence $\GG(A') \to \GG(A)$ is a CZE cover by Proposition~\ref{prop:CZEmaps} (or just note that the space underlying any $\GG(A)$ is a point).  Since $\Spec$ preserves inverse limits, the pullback of $\GG(B) \to \GG(A)$ (with the action of $\GG(C)$) along this cover ``is" the trivial $\GG(C)$ torsor $\GG(B') \cong \GG(A') \times \GG(C)$ over $\GG(A')$. \end{proof}

The following remarks about Zariski and CZE covers are often useful.

\begin{rem} \label{rem:Zariskicovers} Every affine fan is a ``point for the Zariski topology" in the sense that any ``Zariski cover" (surjective local isomorphism of fans) $f : X \to \Spec P$ has a section.  (That is, the identity map is cofinal in the family of all Zariski covers of $\Spec P$.)  To see this, pick $x \in X$ with $f(x) = P^* \in \Spec P$ and a neighborhood $U$ of $x$ in $X$ mapped by $f$ isomorphically onto a neighborhood $V$ of $P^*$ in $\Spec P$.  Since $\Spec P = U_{P^*}$ is the only open subset of $\Spec P$ containing $P^*$ (Proposition~\ref{prop:Spec}\eqref{uniqueclosedpoint}), we must have $V = \Spec P$, so $f$ takes $U$ isomorphically onto $\Spec P$.  In particular, the higher Zariski cohomology of any sheaf of abelian groups vanishes on any affine fan.  This may also be seen by noting that the global sections functor on $\Spec P$ is identified with the ``stalk at $P^*$" functor, hence is exact. \end{rem}

\begin{rem} \label{rem:GGnotexact} In the situation of Proposition~\ref{prop:CZEtorsors}, the sequence \eqref{GGFGAGSeq} is \emph{not} generally exact in the Zariski topology.  Indeed, if it were, then since no affine fan has higher Zariski cohomology (cf.\ Remark~\ref{rem:Zariskicovers}), the sequence \bne{GH} & 0 \to \GG(C)(\Spec P) \to \GG(B)(\Spec P) \to \GG(A)(\Spec P) \to 0 \ene would be exact for every fine monoid $P$.  But \eqref{GH} is just the sequence \bne{GH2} & 0 \to \Hom_{\Ab}(C,P^*) \to \Hom_{\Ab}(B,P^*) \to \Hom_{\Ab}(A,P^*) \to 0 \ene obtained by applying $\Hom_{\Ab}( \slot, P^*)$ to \eqref{FGAGSeq} and \eqref{GH2} will fail to be exact for some finitely generated abelian group $P=P^*$ unless \eqref{FGAGSeq} splits.  One motivation for considering the CZE topology is that it is ``fine enough" that \eqref{GGFGAGSeq} is exact, but ``coarse enough" that it is still possible to understand CZE covers as, for example, in the next remark. \end{rem}

\begin{rem} \label{rem:CZEcovers} If we fix a fine monoid $P$ and an injection of groups $h^* : P^* \to A$ with finite cokernel, then one checks easily that $(P \oplus_{P^*} A)^* = A$, so we obtain a CZE map of monoids $h : P \to Q := P \oplus_{P^*} A$ by pushing out $h^*$ along the inclusion $P^* \into P$, hence we obtain a CZE cover of fine fans $\Spec h : \Spec Q \to \Spec P$ by Proposition~\ref{prop:CZEmaps}.  We claim that the CZE covers of $\Spec P$ thus obtained are cofinal in the family of all CZE covers of $\Spec P$.  Indeed, suppose $ f : X \to \Spec P$ is a CZE cover.  Choose $x \in X$ with $f(x) = P^* \in \Spec P$.  The map $f_x : P=\M_{P,P^*} \to \M_{X,x}$ is a CZE map of monoids, so it is one of the maps $h$ discussed above (with $A=\M_{X,x}^*$) and we see from Proposition~\ref{prop:localpicture} that $\Spec f_x : \Spec \M_{X,x} \to \Spec P$ factors through $f$.  In particular, we see that if $P$ is a \emph{sharp} monoid, then $\Spec P$ is a point for the CZE topology and hence the higher CZE cohomology of any sheaf of abelian groups vanishes on $\Spec P$.  

More generally, if we were willing to work with non-finitely generated monoids as in Remark~\ref{rem:CZEfiniteness}, then we would see that $\Spec P$ is a point for the CZE topology whenever $P^*$ is a \emph{divisible} (i.e.\ injective) abelian group.  (If $h^*$ above has a retract, so does $h$, so the CZE cover $\Spec h$ is refined by the trivial cover.)  It is also true in some sense that the stalk of the CZE structure sheaf of $X$ at a point $x \in X$ (``the CZE strict henselization of $P := \M_{X,x}$") is given by $P \oplus_{P^*} D(P^*)$, where $P^* \to D(P^*)$ is the divisible hull of $P^*$. \end{rem}

\subsection{CZE cohomology} \label{section:CZEcohomology} In this section, we will try to get a feeling for the CZE topology (\S\ref{section:CZEmaps}) by ``computing" the CZE cohomology groups $\H^i_{\rm CZE}(X,\GG(A))$ for $X$ a fine fan, $A$ a finitely generated abelian group.  By ``compute" we mean that we will interpret them as certain $\Ext$ groups in the category $\Ab(X)$ of sheaves of abelian groups on the topological space underlying $X$ (equivalently, on the (small) \emph{Zariski} site of $X$).  

\begin{thm} \label{thm:CZEcohomology} For any fine fan $X$ and any finitely generated abelian group $A$ there are natural isomorphisms \be \H^i_{\rm CZE}(X,\GG(A)) & = & \Ext^i_{\Ab(X)}(\u{A},\M_X^*) \ee for all $i$.  Here $\u{A} \in \Ab(X)$ is the constant sheaf associated to $A$. \end{thm}

\begin{proof}  This can be proved by surprisingly simple general nonsense, though later (in the proof of Corollary~\ref{cor:CZEcohomology}) we shall also outline a more ``computational" approach.  Consider the following functors (all of which are half exact maps of abelian categories):  \bne{functors} \GG : \Ab^{\rm op} & \to & \Ab(X_{\rm CZE}) \\ \nonumber \Gamma_{\rm CZE} = \Gamma_{\rm CZE}(X, \slot) : \Ab(X_{\rm CZE}) & \to & \Ab \\ \nonumber \Gamma = \Gamma(X,\slot) : \Ab(X) & \to & \Ab \\ \nonumber \pi_* : \Ab(X_{\rm CZE}) & \to & \Ab(X) \\ \nonumber \sHom( \slot , \M_X^*) : \Ab(X)^{\rm op} & \to & \Ab(X) \\ \nonumber \Hom( \slot, \M_X^*) : \Ab(X)^{\rm op} & \to & \Ab \\ \nonumber \u{ \slot }_X : \Ab & \to & \Ab(X). \ene The functor $\GG$ is exact by Proposition~\ref{prop:CZEtorsors}.  (To be precise, we view the domain category of $\GG$ as the opposite of the category of \emph{finitely generated} abelian groups.)  The functor $\pi_*$ is just the forgetful functor arising from the fact that any CZE sheaf is, in particular, a Zariski sheaf.  This functor is left exact but not generally exact.  The functor $\Gamma_{\rm CZE}$ is the ``CZE global sections functor" (whose right derived functors are the CZE cohomology).  The last functor is the constant sheaf functor---it is exact and left adjoint to $\Gamma$ since it is the inverse image functor for the projection from $X$ to a point, whose direct image functor is $\Gamma$.  

Everything will follow from the following ``equalities" (natural isomorphisms) between compositions of functors from our list: \bne{formulaA} \Gamma_{\rm CZE} & = & \Gamma \pi_* \\ \label{formulaB} \Gamma \sHom( \slot , \M_X^*) & = & \Hom( \slot, \M_X^*) \\ \label{formula1} \pi_* \GG( \slot ) & = & \sHom( \u{ \slot }_X, \M_X^*) \ene  The first two equalities here are a matter of definitions.  For the third, we compute \be \sHom( \u{A}_X , \M_X^*)(U) & = & \Hom_{\Ab(U)}( \u{A}_X|U, \M_X^*|U) \\ & = & \Hom_{\Ab(U)}( \u{A}_U, \M_U^*) \\ & = & \Hom_{\Ab}(A,\Gamma(U,\M_U^*)) \\ & = & \Hom_{\Fans}(U,\GG(A)) \\ & = & ( \pi_* \GG(A))(U) \ee (natural in $A \in \Ab^{\rm op}$, $U$ an open subspace of $X$) using the ``modular interpretation" of $\GG(A)$ and the adjointness relationship between constant sheaves and global sections mentioned above.

Since the functors $\GG$ and $\u{ \slot }_X$ are already exact, the equality between the right derived functors $\DAb^{\rm op} \to \DAb(X)$ of the functors in \eqref{formula1} takes the form: \bne{formula2} (\R \pi_*) \GG( \slot ) & = & \R \sHom( \u{ \slot }_X,\M_X^*) . \ene  The key point is that the right derived functor of a composition is the composition of the right derived functors.  If we now compose the equality \eqref{formula2} with the right derived functor $\R \Gamma : \DAb(X) \to \DAb$ and use that same fact together with \eqref{formulaA} and \eqref{formulaB}, we obtain an equality \bne{formula3} \R \Gamma_{\rm CZE} \GG( \slot ) & = &  \R \Hom_{\Ab(X)}( \u{ \slot }_X,\M_X^*), \ene from which the result follows by taking cohomology.  \end{proof}

\begin{rem} \label{rem:CZEcohomoloy} It is clear from the proof of Theorem~\ref{thm:CZEcohomology} that we in fact have \be \H^i_{\tau}(X,\GG(A)) & = & \Ext^i_{\Ab(X)}(\u{A},\M_X^*) \ee for all $i$ and any subcanonical topology $\tau$ \emph{containing} the CZE topology. \end{rem}

On the face of things it might seem as if Theorem~\ref{thm:CZEcohomology} just provides a relationship between two inscrutable constructions, but actually it will allow us to make some fairly down-to-earth statements about CZE cohomology.

\begin{lem} \label{lem:projectives}  Suppose $X$ is a topological space such that the global section functor $\Gamma(X,\slot) : \Ab(X) \to \Ab$ is exact.  (This holds for the space underlying $\Spec P$, $P$ a monoid, by Remark~\ref{rem:Zariskicovers}.)  Then for any $A \in \Ab$ and any $F \in \Ab(X)$ we have natural isomorphisms \be \Ext^i_{\Ab(X)}(\u{A},F) & = & \Ext^i_{\Ab}(A,\Gamma(X,F)) \ee for all $i$.  In particular, the locally constant sheaf $\u{L}$ on $X$ associated to a projective (i.e.\ free) abelian group $L$ is a projective object in $\Ab(X)$. \end{lem}

\begin{proof} Since $A \mapsto \u{A}$ is left adjoint to $\Gamma(X,\slot)$, the computation \be \Hom_{\Ab(X)}(\u{A},F) & = & \Hom_{\Ab}(A,\Gamma(X,F)) \ee shows that $\Hom_{\Ab(X)}(\u{A},\slot)$ is the composition of $\Gamma(X,\slot)$ and $\Hom_{\Ab}(A, \slot)$.  Since $\Gamma(X,\slot)$ is exact the Grothendieck spectral sequence from ``the derived functors of the composities" to ``the derived functors of the composition" degenerates to yield the result (or argue directly \dots). \end{proof}

\begin{cor} \label{cor:CZEcohomology} For any affine fine fan $X = \Spec P$ and any abelian group $A$ we have \be \H^i_{\rm CZE}(X,\GG(A)) & = & \Ext^i_{\Ab(X)}(\u{A},\M_X^*) \\ & = & \Ext^i_{\Ab}(A,P^*) \ee for all $i$.  In particular, these groups vanish for $i>1$ and for $i>0$ when $A$ is a lattice.  \end{cor}

\begin{proof}  Combine Theorem~\ref{thm:CZEcohomology} and Lemma~\ref{lem:projectives}, noting that $\Gamma(X,\M_X^*) = P^*$.  

There is a also a rather nice alternative argument to establish the vanishing \bne{desiredCZEvanishing} \H^{>0}_{\rm CZE}(X,\GG(L)) & = & 0 \ene for a lattice $L$ without appealing to Theorem~\ref{thm:CZEcohomology}.  In fact, once this basic vanishing is known, the whole corollary can be easily deduced.  One can then give an alternative proof of Theorem~\ref{thm:CZEcohomology} by reducing to the affine case by considering a (Zariski) \v{C}ech hypercover $X_\bullet$ of an arbitrary fine fan $X$ with each $X_p$ a disjoint union of affines.  Our alternative method of establishing the vanishing \eqref{desiredCZEvanishing} is analogous to Grothendieck's proof of the vanishing of the higher cohomology of a quasi-coherent sheaf on an affine scheme in \cite[III.1.3.1]{EGA}.  

The first step is to establish the vanishing \bne{desiredCechCZEvanishing} \CechH^{>0}_{\rm CZE}(X,\GG(L)) & = & 0 \ene of the analogous \emph{\v{C}ech} CZE cohomology groups.  Consider a map $h^* : P^* \to A$ and the corresponding CZE cover $\Spec h : X_0 = \Spec Q \to X$ as in Remark~\ref{rem:CZEcovers}.  As mentioned in that remark, these CZE covers are cofinal, so \eqref{desiredCechCZEvanishing} will follow from the vanishing of the \v{C}ech cohomology \bne{desiredCechvanishing} \CechH^{>0}(X,\Spec h,\GG(L)) & = & 0 \ene for the map $\Spec h$.  (The \v{C}ech cohomology with respect to a map has nothing to do with any topology.)  Consider the cosimplicial abelian group \be G^\bullet & := & ( A \rightrightarrows A \oplus_{P^*} A \threerightarrows A \oplus_{P^*} A \oplus_{P^*} A \cdots ) \ee obtained as the ``$0$-(co?)skeleton" of $h^*$.  Since $\Spec$ takes direct limits of monoids to inverse limits of fans, the $0$-coskeleton \be X_\bullet & := & ( \cdots X_0 \times_X X_0 \times_X X_0 \threerightarrows X_0 \times_X X_0 \rightrightarrows X_0 )  \\ & =: & ( \cdots X_2 \threerightarrows X_1 \rightrightarrows X_0) \ee of our map $\Spec h$ is given by \be X_\bullet & = & \Spec (G^\bullet \oplus_{P^*} P ), \ee hence \be \CechH^{i}(X, \Spec h,\GG(L)) & = & \H^i( \Gamma( X_\bullet , \GG(L) )) \\ & = & \H^i( \Hom_{\Ab}( L, (G^\bullet \oplus_{P^*} P)^* )) \\ & = & \H^i( \Hom_{\Ab}(L, G^\bullet) ). \ee  It is well-known that the cohomology of $G^\bullet$ is given by $P^*$ in degree zero and vanishes in higher degrees.  Since the lattice $L$ is a projective abelian group, applying $\Hom_{\Ab}(L, \slot )$ commutes with the formation of the aforementioned cohomology, thus we obtain the desired vanishing \eqref{desiredCechvanishing}. 

To bootstrap up from the \v{C}ech vanishing to the ``usual" vanishing, it suffices to show that \emph{any} sheaf $F$ on the category of fine fans in the CZE topology satisfying \bne{Cechvanishing} \CechH_{\rm CZE}^i(X,F) & = & 0 \ene for all affine $X$ and all $i>0$ also satisfies \bne{desvan} \H^i_{\rm CZE}(X,F) & = & 0 \ene for all affine $X$ and all $i>0$.  This is a standard argument of H.~Cartan.  The analogous variant for topological spaces is formulated and proved in \cite[5.9.2]{God}, but the argument makes sense in the present context as well.  Godement's proof uses the spectral sequence from \v{C}ech cohomology with coefficients in ``the cohomology presheaves" abutting to usual cohomology, but I find the argument in \cite[Cohomology, Lemma~12.9]{SP} to be a little easier---let's run through it quickly for the reader's convenience.  Fix a sheaf $F$ as above.  Choose an exact sequence \bne{injres} & 0 \to F \to I \to B \to 0 \ene of sheaves of abelian groups on the category of fine fans (in the CZE topology) with $I$ injective.\footnote{One could just work with the category of fine affine fans to avoid any set-theoretic issues that may arise in verifying the existence of such a thing.}  Since $\CechH^1$ always agrees with $\H^1$ (for any sheaf on any site), we at least know that \bne{H1vanishing} \H^1_{\rm CZE}(X,F) & = & 0 \ene for every affine $X$.  Now consider a CZE cover $X_0 \to X$ of affine fans as in the previous paragraph.  Since each $X_p$ is also affine, the vanishing \eqref{H1vanishing} ensures that the sequence of cochain complexes of abelian groups \bne{sescomplexes} & 0 \to \Gamma(X_\bullet,F) \to \Gamma(X_\bullet,I) \to \Gamma(X_\bullet,B) \to 0 \ene obtained by applying $\Gamma(X_\bullet, \slot)$ to \eqref{injres} is exact.  The associated long exact sequence of cohomology groups is a long exact sequence relating the \emph{\v{C}ech} cohomology of the sheaves in \eqref{injres} with respect to the map $X_0 \to X$.  Since the map $X_0 \to X$ is a CZE cover and $I$ is an injective CZE sheaf, the higher \v{C}ech cohomology of $I$ with respect to $X_0 \to X$ vanishes.  Putting this together with the hypothesized vanishings \eqref{Cechvanishing}, we see from our long exact cohomology sequence that the higher \v{C}ech cohomology of $B$ with respect to $X_0 \to X$ also vanishes.  Since the maps of the form $X_0 \to X$ that we've been considering are cofinal in all CZE covers of $X$, we conclude that $B$ also satisfies the same vanishing property \eqref{Cechvanishing} enjoyed by $F$.  In particular, as we saw for $F$, we see also that $\H^1_{\rm CZE}(X,B)=0$ for affine $X$.  Now we look at the long exact sequence of CZE cohomology associated to \eqref{injres} on an affine $X$.  Since $\H^1_{\rm CZE}(X,B)=0$ and all higher CZE cohomology of the injective sheaf $I$ vanishes, we see that $\H^2_{\rm CZE}(X,F)=0$.  But these same arguments also apply to $B$ as well, so $\H^2_{\rm CZE}(X,B)=0$ for $X$ affine, then we again look back at the long exact sequence of CZE cohomology associated to \eqref{injres} to conclude $\H^3_{\rm CZE}(X,F)=0$.  Continuing in this manner we establish the desired vanishing \eqref{desvan}.   \end{proof}

\begin{cor} \label{cor:CZEcohomology2} For a fine fan $X$ and a lattice $L$, there are natural isomorphisms \be \H^i_{\rm CZE}(X,\GG(L)) & = & \H^i(X,\GG(L)) \ee between the CZE cohomology of $\GG(L)$ and the Zariski cohomology of $\GG(L)$. \end{cor}

\begin{rem} The special case of the above corollary where $L=\ZZ$ and $i=1$ says that every CZE locally trivial $\GG_m$ torsor over a fine fan is already Zariski locally trivial.  This corollary is similar to the form of Hilbert's Theorem 90 that says the first \'etale cohomology of a scheme with coefficients in $\GG_m$ is the same as the analogous \emph{Zariski} cohomology group. \end{rem}

\begin{proof} Using Lemma~\ref{lem:projectives}, we see that $\u{L}$ is ``locally projective," so that \be \sExt^i(\u{L},\M_X^*) & = & 0 \ee for $i>0$.  Consequently, the local-to-global spectral sequence for $\Ext^i( \u{L}, \M_X^*)$ degenerates to yield \be \Ext^i_{\Ab(X)}(\u{L},\M_X^*) & = & \H^i(X, \sHom( \u{L}, \M_X^*) ) \\ & = & \H^i(X,\GG(L)). \ee  (The second isomorphism here comes from the ``modular interpretation" of $\GG(L)$ like the isomorphism \eqref{formula1} in the proof of Theorem~\ref{thm:CZEcohomology}.)  The corollary now follows from Theorem~\ref{thm:CZEcohomology}.

It is also possible to give an alternative proof using the previous corollary instead of Theorem~\ref{thm:CZEcohomology}.  Since we have $\Gamma_{\rm CZE}  =  \Gamma \pi_*$ (see the proof of Theorem~\ref{thm:CZEcohomology}), the corollary can be obtained by showing that the higher derived functors $\R^{>0} \pi_*$ of the pushforward functor $\pi_* : \Ab(X_{\rm CZE}) \to \Ab(X)$ vanish on $\GG(L)$ ($L$ a lattice).  To see this, one first shows that for $F \in \Ab(X_{\rm CZE})$, the sheaf $\R^i \pi_* F \in \Ab(X)$ is obtained by sheafifying the presheaf \bne{presheaf} U & \mapsto & \H^i_{\rm CZE}(U,F|U_{\rm CZE}) \ene (same argument as the analogous result \cite[8.1]{H} in topology).  Since every point $x \in X$ has a smallest neighborhood $U=U_x$ and $U = \Spec \M_{X,x}$ is affine, we have \bne{pushforwardstalkformula} (\R^i \pi_* F)_x & = & \H^i_{\rm CZE}(U,F|U_{\rm CZE}). \ene  The desired vanishing is obtained by checking that the stalks of $\R^i \pi_* \GG(L)$ are zero when $i>0$ by using \eqref{pushforwardstalkformula} and Corollary~\ref{cor:CZEcohomology}. \end{proof}

\end{document}